\documentclass[journal, onecolumn,12pt, draftcls, doublespace]{IEEEtran}

\usepackage{times}
\usepackage{amsmath}
\usepackage{amssymb}
\usepackage[pagebackref=true,breaklinks=true,letterpaper=true,colorlinks,bookmarks=false]{hyperref}

\usepackage{ifthen}
\usepackage{algorithm}
\usepackage{algorithmic}
\floatstyle{ruled}
\newfloat{alg}{thp}{lop}
\floatname{alg}{Algorithm}
\renewenvironment{algorithm}
{\begin{alg} \vspace{0.01cm} \noindent } {\vspace{0.01cm}
\end{alg}}

\def\Z{\mathcal{Z}}

\newcommand{\Prob}{\mathrm{Prob}}
\newcommand{\e}{\mathbf{e}}
\newcommand{\br}{\mathbf{r}}
\newcommand{\x}{\mathbf{x}}
\newcommand{\y}{\mathbf{y}}
\newcommand{\z}{\mathbf{z}}
\newcommand{\BZ}{\mathbf{Z}}
\newcommand{\bu}{\mathbf{u}}
\newcommand{\bs}{\mathbf{s}}
\newcommand{\bh}{\mathbf{h}}
\newcommand{\wu}{u}

\renewcommand{\v}{\mathbf{v}}

\newcommand{\btheta}{\boldsymbol{\theta}}
\renewcommand{\Re}{\mathbb{R}}
\newcommand{\GBS}{\textbf{GBS}}
\newcommand{\BS}{\mathrm{B}}
\newcommand{\RC}{\textbf{RC}}
\newcommand{\IM}{\mathrm{Im}}
\newcommand{\Rank}{\mathrm{rank}\mbox{ } }
\newtheorem{Definition}{Definition}
\newtheorem{Theorem}{Theorem}
\newtheorem{Problem}{Problem}
\newtheorem{Example}{Example}
\newtheorem{Lemma}{Lemma}

\newtheorem{Remark}{Remark}
\newtheorem{Notation}{Notation}

\newtheorem{Corollary}{Corollary}
\newtheorem{Assumption}{Assumption}

\long\def\comment#1{}
\newcommand{\ie}{{i.e. }}
\newcommand{\eg}{{e.g., }}
\newcommand{\ea}{{et al. }}

\begin{document}

\title{Realization Theory of Stochastic Jump-Markov Linear Systems}

\author{
Mih\'aly Petreczky, \and Ren\'e Vidal
\thanks{Mih\'aly Petreczky is with Centre de Recherche en Informatique, Signal et Automatique de Lille (CRIStAL),  UMR CNRS 9189, Ecole Centrale de Lille, 59651 Villeneuve d’Ascq, France (e-mail:~ mihaly.petreczky@ec-lille.fr).
%He carried out a significant part of the research presented in this paper during his stay as a postdoc at Center for Imaging Science, Department of Biomedical Engineering, Johns Hopkins University.
}
%\thanks{Mih\'aly Petreczky is with the Ecole des Mines de Douai, Department of Computer Science and Automatic Control, 59508 Douai, France (e-mail:~\texttt{mihaly.petreczky@mines-douai.fr}).}
\thanks{Ren\'e Vidal is with the Center for Imaging Science, Department of Biomedical Engineering, Johns Hopkins University, 3400 N. Charles Street, Baltimore MD 21218, USA (e-mail: rvidal@cis.jhu.edu).}
}

\maketitle

\begin{abstract}
 In this paper, we present a complete stochastic realization theory
 for stochastic jump-linear systems. We  present
 necessary and sufficient conditions for the existence of a 
 realization, along with a characterization of minimality
in terms of reachability and observability.  
We also formulate a realization algorithm and argue that 
minimality can be checked algorithmically. 
The main tool for solving the stochastic realization problem for
jump-linear systems is the formulation and solution of 
a stochastic realization problem for a general class of bilinear systems with non-white-noise inputs. The solution to this generalized stochastic bilinear realization problem is based on the theory of formal power series.  Stochastic jump-linear systems represent a special case
 of generalized stochastic bilinear systems.
\end{abstract}

\section{Introduction}

Hybrid systems are dynamical systems that exhibit both continuous and discrete behaviors. Such systems have a wide range of applications, including systems biology, computer vision, flight control systems, etc. While there is a vast amount of literature on stability, reachability, observability, identification, and controller design for hybrid systems, there are relatively fewer results available on realization theory of hybrid systems. 
 
Realization theory is one of the central topics of control and systems theory. Its goals are to study the conditions under which the observed behavior of a system can be represented by a state-space representation of a certain type and to develop algorithms for finding a (preferably minimal) state-space representation of the observed behavior. The study of these problems is not only of theoretical relevance, but also of practical importance in several applications such as model reduction and systems identification. In fact, one can argue that stochastic realization theory is indispensable for the understanding of systems identification.
%On a more pragmatical level, stochastic realization
%algorithms can potentially yield system identification
%algorithms as it was demonstrated by the development
%of subspace identification methods.
 
\subsection{State-of-the-art}
 
%THIS TEXT IS TAKEN FROM "Minimal state-space realization in linear system theory: an overview" by B. De Schutter.
%The origins of the minimal state-space realization problem lie in the early 1960s. The minimal state-space realization problem for (continuous) LTI systems was first stated by Gilbert [18], who gave an algorithm for transforming a transfer function into a system of differential equations (i.e., a state-space description). A second algorithm for the problem was given around the same time by Kalman [32]. The approach of Gilbert was based on partial-fraction expansions and worked under the assumption that each entry of the transfer function matrix has distinct poles. Kalman's algorithm was based on the theory of controllability and observability and reduced a nonminimal state-space realization until it became minimal (cf. Theorem 2.2). Ho and Kalman [26 and 27] approached the minimal realization problem from an entirely new point of view: they solved the problem starting from the sequence of Markov parameters of the system. Their algorithm will be discussed extensively below. All these algorithms assume that the entire sequence of Markov parameters is available. However, many times only a limited number of Markov parameters is available. The corresponding minimal partial state-space realization problem for MIMO systems was first explored by Kalman [34] and Tether [54]. Later, Rissanen [42] gave a recursive solution of the SISO version of this problem (which he claims can easily be extended to the MIMO case). 
 
For the class of linear systems, the realization problem is relatively well understood thanks to the work of Kalman \ea in the sixties \cite{Ho-Kalman:Allerton65,Kalman:69}. For instance, it is well known that all minimal representations, \ie representations such that the dimension of the state-space is minimal, are related by a change of basis of the state-space. Also, it is well known that the rank of a Hankel matrix $H$ formed from the output measurements is related the dimension of all minimal representations and that a realization of the system can be obtained from the factorization of $H$. Such results have lead to a huge literature on identification of linear systems \cite{ljung87}, including the well-known subspace identification methods \cite{Overschee:96}.

For the class of bilinear systems, the realization problem is also relatively well studied thanks to the works of Brockett \cite{Brockett:TAVSS72}, Fliess \cite{Fliess:ASP73}, Isidori \ea \cite{Dalessandro:JCO74,Isidori:TAC73,Isidori:GMST73}, Sontag \cite{MR960665} and Sussman \cite{Sussman:MST75,Sussman:JFI76} in the 1970's. However, realization of stochastic bilinear systems is relatively unstudied, except the case when input is white noise \cite{StochBilin,FrazhoBilin}. On the other hand, there are a number of papers on identification of bilinear systems with inputs which are not white noise, see \eg \cite{MacChenBilSub,Favoreel:PhD99, Verhaegen:INJC, Verhaegen:automatica}. However, all these papers require a number of conditions on the underlying system in order to operate correctly.

For more general nonlinear systems, the realization problem is not as well understood. 
There exists a complete realization theory for 
analytic nonlinear systems \cite{Fliess:Real,Jak:RealAnal,Jak:RealSumm,Celle:Real, Reut:Real,Son:RealAnal} and for general smooth systems 
\cite{Jakub:JCO80,Sussman:MST77}. However, the algorithmic aspects
of this theory are not that well developed.
%% bilinear (multilinear) systems  \cite{Son:Real,Isi:Nonlin,Isi:Bilin,IsiDirect:Bilin,MFliess:Hank,MFliess:FormPow} both in discrete and continuous time. 
There is a substantial amount of work on realization theory of
polynomial systems \cite{MR862318,Son:Resp}, and 
rational systems \cite{SonWang:Obs,SonWang:SIAM,WangGenSer} 
both in continuous and discrete time. However, the issue of minimality for 
polynomial systems is not that well understood.

One of the earliest attempts to characterize realization of deterministic hybrid systems can be found in \cite{GrossmanHybAlg}, though a formal theory is not presented. Since then, most of the work has concentrated on switched linear systems \cite{MP:Real,MP:RealForm}, switched bilinear systems \cite{MP:RealBilin}, 
 linear and bilinear hybrid systems without guards and partially
observed discrete states \cite{MP:HybBilinReal,MP:HybPow}, and
nonlinear analytic hybrid systems without guards \cite{MP:HybNonLin}.
The main assumptions made are that the continuous dynamics evolve in 
continuous-time and the discrete events which initiate the change of the 
discrete states are part of the input. Hence, the discrete states
may (switched systems) or may not (hybrid systems without guards) be
fully observed. For the classes of hybrid systems mentioned above, with
the exception of nonlinear hybrid systems without guards, a
complete realization theory and realization algorithms are available.
\cite{MP:PL} contains partial results on realization theory of piecewise-affine autonomous hybrid systems with guards. In that paper necessary and sufficient conditions for existence of a realization were presented, but the problem of minimality was not
dealt with. As far as the authors know, the only paper dealing with realization theory of stochastic hybrid systems is \cite{MPRV:HSCC}, where only necessary conditions for the existence of a realization were presented.

\subsection{Paper contributions}

In this paper we will present a complete stochastic realization theory of discrete-time stochastic jump-linear systems. Stochastic jump-linear systems have a vast literature and numerous applications (see for example \cite{CostaBook} and the references therein).  For simplicity, we will consider only stochastic jump-linear systems with fully observed discrete state. In addition, we will assume that the continuous state-transition depends not only on the current, but also on the next discrete state and that the continuous state at each time instant lives in a state-space that depends on the current discrete state. In this way we will obtain  a more general model, which we will call \emph{generalized stochastic jump Markov linear systems}. It turns out that the class of classical stochastic  jump-linear systems generates the same class of output processes as the new more general class. However, by looking at more general systems we are able to obtain a neat characterization of minimality as well as necessary and sufficient conditions for the existence of a realization. We will also formulate a realization algorithm and argue that minimality can be checked algorithmically.

The main tool for solving the stochastic realization problem is the solution of a general bilinear realization problem, whose formulation and solution can be described as follows. Consider an output and an input process and imagine you would like to compute recursively the linear projection of the future outputs onto the space of products of past outputs and inputs. Under the assumption that the mixed covariances of the future outputs with the products of past outputs and inputs form a \emph{rational formal power series}, we show that one can construct a bilinear state-space representation of the output process in the forward innovation form. The matrices of this state-space representation are determined by the parameters of the rational representation of the covariance sequence of future and past outputs and inputs. The results on realization theory of stochastic jump-linear systems are then obtained by viewing the discrete state process as an input process. 

To the best of our knowledge, both the solution of the realization problem for
stochastic jump-linear systems, and the formulation and solution of the general bilinear
realization problem are new. 
In comparison the work of \cite{MPRV:HSCC} on stochastic realization of jump-linear systems, the main contribution of this paper is that it presents both necessary and sufficient conditions for the existence of a realization as well as a characterization of minimality. 
In comparison to the work of \cite{StochBilin} on stochastic realization of bilinear systems with observed white-noise input process, the main contribution of this paper is to solve the realization problem for a more general class of bilinear systems, without requiring the input process to be white. 
%The proposed general bilinear filtering/realization problem
%is an extension of  the filtering/realization of 
%bilinear systems with observed white-noise input process and
%additive noise.
%To the best of our knowledge, 
%\cite{StochBilin} is the only paper on the latter topic.
%The results of \cite{StochBilin}
%are a special case of results of this paper. 
%The main difference
%is that it is no longer assumed that the input is white.
 %%More precisely,
%%we formulate general conditions on the output and input
%%processes, which ensure existence of a stochastic bilinear
%%realization in the forward innovation form.
%%The proposed conditions hold for output of linear
%%systems and stochastic bilinear systems with white noise
%%inputs. In addition, they also hold if the considered
%%input is an ergodic Markov process, thus enabling
%%us to construct a jump-linear system in the forward
%%innovation form.
%%To the best of our knowledge,
%%\cite{StochBilin} is the only construction of the
%%foward innovation model available in the literature
%%and it operates with the assumption that the input is
%%white noise.
In comparison with the works of \cite{MacChenBilSub,Favoreel:PhD99, Verhaegen:INJC, Verhaegen:automatica} on identification of bilinear systems with inputs that are not necessarily white noise, there are two main contributions.
%There are a number of papers on identification of bilinear systems with inputs which are not white noise, see \cite{Favoreel:PhD99, Verhaegen:INJC, Verhaegen:automatica, MacChenBilSub}. 
First, the aforementioned papers aim to identify the
parameters of the system from the measurements. 
In contrast, the goal of realization theory is to understand
the conditions, under which a (not necessarily identifiable)
state-space representation exists. Hence, establishing algorithms
for finding the parameters of the system that generate the
process answers the realization problem only partially.
Second, all the aforementioned papers assume that the system to be identified is already in the
forward innovation form and impose a number of observability and stability conditions on the underlying system, which are more restrictive than the conditions assumed here.

\subsection{Paper outline}
%The outline of the paper is as follows. Section \ref{gjmls:def} introduces the concept of stochastic jump-linear systems and formulates the realization problem for this class of systems. Section \ref{sect:pow} presents the background material on the theory of rational formal power series. The results of Section \ref{sect:pow} will be instrumental for solving the generalized bilinear realization problem, which will be formulated and solved in Section \ref{gen:filt}. Section \ref{real} presents the solution of the realization problem for stochastic jump-linear systems.

The outline of the paper is as follows. Section \ref{sect:pow} presents the background material on the theory of rational formal power series. These results will be instrumental for solving the generalized bilinear realization problem, which will be formulated and solved in Section \ref{gen:filt}. Section \ref{sect:GJMLS} formulates the realization problem for stochastic jump Markov linear systems and presents a solution to it based on the results in Section \ref{gen:filt}.

\section{Rational Power Series}
\label{sect:pow}

In this section, we present several results on formal power series, which will be used for solving a general bilinear filtering/realization problem to be presented in Section \ref{gen:filt}. In turn, the solution to this bilinear filtering/realization problem will yield a solution to the realization problem for stochastic jump-linear systems, as we will show in Section \ref{sect:GJMLS}. 

The material and results in Subsections \ref{hscc_pow_basic} and \ref{hscc_pow_stab} can be found in \cite{MP:Phd} and \cite{MPRV:HSCC}, respectively. For more details on the classical theory of rational formal power series, the reader is referred to \cite{Reut:Book,Son:Real,Son:Resp} and the references therein.

\subsection{ Definition and Basic Theory}
\label{hscc_pow_basic}

Let $\Sigma$ be a finite set. We will refer to $\Sigma$ as the \emph{alphabet}. The elements of $\Sigma$ will be called \emph{letters}, and every finite sequence of letters will be called a  \emph{word} or \emph{string} over $\Sigma$.
 Denote by $\Sigma^{*}$ the set of all finite words from elements in $\Sigma$. 
 An element $w\in \Sigma^{*}$ of length $|w|=k\geq 0$ is a finite sequence
 $w=\sigma_{1} \sigma_{2} \cdots \sigma_{k}$ with  $\sigma_{1},\ldots, \sigma_{k} \in \Sigma$.
The empty word is denoted by $\epsilon$ and its length is zero, \ie $|\epsilon|=0$.
Denote by $\Sigma^{+}$ the set of all non-empty words
over $\Sigma$, \ie $\Sigma^{+}=\Sigma^{*}\setminus \{\epsilon\}$.
 The concatenation of two words $v=\nu_{1} \cdots \nu_{m}$ and 
$w=\sigma_{1} \cdots \sigma_{k} \in \Sigma^{*}$
 is the word $vw=\nu_{1} \cdots \nu_{m}\sigma_{1} \cdots \sigma_{k}$.

%In the sequel we will use the following lexicographic ordering on $\Sigma^{*}$. 

\begin{Definition}[Lexicographic ordering]
\label{rem:lex} Let  $<$ be an ordering on $\Sigma$ so that $\Sigma=\{\sigma_1,\ldots,\sigma_{|\Sigma|}\}$ with $\sigma_1 < \sigma_2 < \ldots < \sigma_{|\Sigma|}$. We define a lexicographic ordering $\prec$ on $\Sigma^{*}$ as follows. For any $v=\nu_{1} \cdots \nu_{m}$ and $w=\sigma_{1} \cdots \sigma_{k} \in \Sigma^{*}$, $v \prec w$ if either $|v| < |w|$ or $0 < |v|=|w|$, $v \ne w$ and for some $l \le |w|$, $\nu_l < \sigma_l$ with the ordering $<$ on $\Sigma$ an $\nu_i = \sigma_i$ for $i=1,\ldots, l-1$. 
\end{Definition}

Notice that $\prec$ is a complete ordering and that $\Sigma^{*}=\{v_0,v_1,\ldots \}$ with $v_0 \prec v_1 \prec \ldots $. Therefore, we will call the set $\{v_0,v_1,\ldots\}$ an \emph{ordered enumeration} of $\Sigma^*$. Notice also that $v_0 =\epsilon$ and that for all $i \in \mathbb{N}$ and $\sigma \in \Sigma$, we have $\nu_i \prec \nu_i\sigma$. Moreover, denote by $M(N)$ the number of all non-empty words over $\Sigma$ whose length is at most $N$, \ie $M(N)=|\{w \in \Sigma^{+} \mid |w| \le N\}|$. It then follows
that with the lexicographic ordering defined above,  the set $\{v_0,v_1,\ldots,v_{M(N)}\}$ equals to the set of
all words of length at most $M(N)$, including the empty word.

%For any two sets $J$ and $A$, an
%\emph{indexed subset} of $A$ with the \emph{index set} $J$ is simply a map $Z :J \rightarrow A$, denoted by $Z=\{ a_{j} \in A \mid j \in J\}$, where $a_{j}=Z(j)$ for all  $j \in J$. Notice that
%we do not require the elements $a_{j}$ to be all different.

\def\FPS{\mathbb{R}^{p} \! \ll \!  \Sigma^{*} \!\!\gg}

  A \emph{formal power series} $S$ with coefficients in 
  $\mathbb{R}^{p}$ is a map
  \( S: \Sigma^{*} \rightarrow \mathbb{R}^{p} \).
  We will call the values $S(w) \in \mathbb{R}^{p}$, $w \in \Sigma^{*}$, the \emph{coefficients} of
  $S$.
   We will denote by
   $\FPS$ the set of all formal power series
   with coefficients in $\mathbb{R}^{p}$.
    Consider a family of formal power series
    $\Psi=\{ S_{j}  \in \FPS \mid j \in J\}$
    indexed with a finite index set
    $J$. We will call such an indexed set of formal power
    series a \emph{family of formal power series}. 
   % In the
    %sequel we will work only with \emph{finite families}
    %of power series and hence we can assume that
    %$J$ is finite.

    A family of formal power series $\Psi$
    will be called \emph{rational} if 
    there exists an integer $n \in \mathbb{N}$, a matrix
   $C \in \mathbb{R}^{p \times n} $, a 
   collection of matrices 
   $A_{\sigma} \in \mathbb{R}^{n \times n}$
   indexed by $\sigma \in \Sigma$,  and an indexed set 
   $B=\{ B_{j} \in \mathbb{R}^{n} \mid j \in J\}$ 
   of vectors in $\mathbb{R}^{n}$, such that  for each index 
   $j \in J$ and for all sequences 
   $\sigma_{1}, \ldots, \sigma_{k} \in \Sigma$, $k \ge 0$,  
  \begin{align}
      S_{j}(\sigma_{1}\sigma_{2} \cdots \sigma_{k})=
      CA_{\sigma_{k}}A_{\sigma_{k-1}} \cdots A_{\sigma_{1}}B_{j}.
  \end{align}
  
The 4-tuple $R=(\mathbb{R}^{n},\{A_{\sigma}\}_{\sigma \in \Sigma},B,C)$ will be called a \emph{representation} of $\Psi$ and the number $n = \dim R$ will be called the \emph{dimension} of $R$. A representation $R_{min}$ of $\Psi$ will be called \emph{minimal} if all representations $R$ of $\Psi$ satisfy \( \dim R_{min} \le \dim R \). Two representations of $\Psi$, $R=(\mathbb{R}^{n},\{ A_{\sigma} \}_{\sigma \in \Sigma},B,C)$ and $\widetilde{R}=(\mathbb{R}^{n},\{ \widetilde{A}_{\sigma} \}_{\sigma \in \Sigma}, \widetilde{B},\widetilde{C})$, will be called \emph{isomorphic}, if there exists  a nonsingular matrix $T \in \mathbb{R}^{n \times n}$ such that$T\widetilde{A}_{\sigma}=A_{\sigma}T$ for all $\sigma \in \Sigma$, $T\widetilde{B}_{j}=B_{j}$ for all $j \in J$, and $\widetilde{C}=CT$. 

Let $R=(\mathbb{R}^{n},\{A_{\sigma}\}_{\sigma \in \Sigma},B,C)$ be a representation of $\Psi$. In order to characterize whether this representation is reachable and observable, let us define the following short-hand notation
\begin{Notation}
\label{matrix:prod}
   $A_{w} \!\doteq\! A_{\sigma_{k}}A_{\sigma_{k-1}} \!\!\cdots A_{\sigma_{1}}$ for
   $w=\sigma_{1} \cdots \sigma_{k}\in \Sigma^*$ and
   $\sigma_1,\ldots, \sigma_k \in \Sigma$, $k \ge 0$.
The map $A_{\epsilon}$ will be identified with the identity map. 
\end{Notation}|
Recall the ordered enumeration of $\Sigma^*$, $\{v_0,v_1,\ldots\}$,  fix an enumeration  of $J=\{j_1,\ldots,j_K\}$ and let $\widetilde{B}=\begin{bmatrix} B_{j_1} & \cdots & B_{j_K} \end{bmatrix}$.
Define the following matrices.
  \begin{align}
 \label{sect:pow:form1}
 %%W_{R} &= \quad \SPAN\{ A_{w}B_{j} \mid w \in \Sigma^{*}, |w| \le n-1, j \in J\}  \; \\
 W_{R} &= \begin{bmatrix} A_{v_0}\widetilde{B} &, \ldots, & A_{v_{M(n-1)}}\widetilde{B} \end{bmatrix} \\
\label{sect:pow:form1.1}
 %O_{R} &=  \!\!\!\!\bigcap_{w \in \Sigma^{*}, |w| \le n-1} 
%                 \!\!\!\!\!\!\!\!\!\ker CA_{w}.
   O_R &= \begin{bmatrix} (CA_{v_0})^T & \ldots & (CA_{v_{M(n-1)}})^T \end{bmatrix}^T.
  \end{align}
 We will call the representation $R$ \emph{observable} if $\ker O_{R} = \{0\}$ and \emph{reachable} if $\dim R = \Rank W_{R}$.
Observability and reachability of representations can be checked numerically. 
For instance, one can formulate an algorithm for transforming any 
representation to a minimal representation of the same
family of formal power series (see 
\cite{MP:Phd} and the references therein for details).

Let $\Psi=\{ S_{j} \in \FPS \mid j \in J\}$ 
be a family of formal power series and define $I=\{1,\ldots,p\}$. We define 
the Hankel-matrix $H_{\Psi}$ of $\Psi$ as  the matrix such that the following holds.
%whose entries are given by  $(H_{\Psi})_{(u,i)(v,j)}= (S_{j}(vu))_{i}$.
The rows of $H_{\Psi}$ are indexed by pairs
$(u,i)$ where $u\in\Sigma^*$ is a word over $\Sigma$ and $i$ is
and integer in $I=\{1,2,\ldots,p \}$. Likewise, the columns of
$H_{\Psi}$ are indexed by pairs $(v,j)$, where
$v\in\Sigma^*$ and $j$ is an element of the index set $J$. 
Thus, the element of $H_{\Psi}$ whose
row index is $(u,i)$ and whose
column index is $(v,j)$ is simply the $i$th row
of the vector $S_{j}(vu)\in\Re^p$, i.e. $(H_{\Psi})_{(u,i)(v,j)}= (S_{j}(vu))_{i}$.
      
      The following result on realization of formal power series can be found in
       \cite{Son:Real,Son:Resp,MP:Phd}.
 \begin{Theorem}[Realization of formal power series]
 \label{sect:form:theo1}
\label{sect:pow:theo1}
    Let $\Psi=\{ S_{j} \in \mathbb{R}^{p}\ll \Sigma^{*} \gg \mid j \in J\}$ be a set of formal power series indexed by $J$.
    Then the following holds.
  \begin{itemize}
  \item[(i) ]
  $\Psi$ is rational $\iff$ $\Rank H_{\Psi} < +\infty$.
  \item[(ii)]
   $R$ is  a minimal representation of $\Psi$ 
   $\iff$ $R$ is reachable and observable $\iff$
   $\dim R=\Rank H_{\Psi}$.
  \item[(iii)] All minimal representations of $\Psi$ are isomorphic.
 % \item[(iv)]
 %  If the rank of the Hankel matrix $H_{\Psi}$ is finite,
 %  \ie $n=\Rank H_{\Psi} < +\infty$, then one can construct a 
 %  representation $R=(\mathbb{R}^{n}, \{A_{\sigma}\}_{\sigma \in \Sigma}, B,C)$ of $\Psi$
 %  using the columns of $H_{\Psi}$ (see
 %  \cite{MP:Phd} for details). 
\end{itemize}    
\end{Theorem}

%Notice that $H_{\Psi}$ is an infinite matrix and hence the construction in 
%part (iv) of Theorem \ref{sect:pow:theo1} is not directly computable. However, 
It is possible to compute a minimal representation of 
$\Psi$ from finitely many data. The procedure resembles
very much the partial realization algorithms for linear systems. 
One defines the finite matrix $H_{\Psi,M,N}$ as
the finite upper-left block of the infinite Hankel matrix $H_{\Psi}$
obtained by taking all the rows of
$H_{\Psi}$ indexed by words over $\Sigma$ of length at most $M$, and
all the columns of $H_{\Psi}$ indexed by words of length at most $N$.
If $\Rank H_{\Psi,N,N}=\Rank H_{\Psi}$ holds, then
there exists an algorithm for computing a minimal representation
$R_{N}$ of $\Psi$. The algorithm is essentially a generalization
of the well-known Kalman-Ho algorithm \cite{Ho-Kalman:Allerton65} for linear systems.
The condition $\Rank H_{\Psi,N,N}=\Rank H_{\Psi}$ holds, if,
for example, $N$ is chosen to be 
bigger than the dimension of some representation of $\Psi$.
In practice, this means that $N$ has to be an upper bound on
the estimated dimension of a potential representation of $\Psi$.
More details on the computation of a minimal representation from
a Hankel-matrix can be found in \cite{MP:Phd} and the references
therein. 

For the purposes of this paper we will use a specific version of 
the realization algorithm.
In order to present the algorithm, we define the notion of \emph{$r,N$-selection}:
an \emph{$r,N$-selection} is a pair $(\alpha,beta)$ such that
\begin{enumerate}
\item
$\alpha \subseteq \Sigma^{N} \times \{1,\ldots,p\}$, 
$\beta \subseteq \Sigma^{N} \times J$, $\Sigma^{N}=\{v \in \Sigma^{*} \mid |v| \le N\}$,
\item
 $|\alpha|=|\beta|=r$.
\end{enumerate}
Intuitively, $\alpha$ represents a selection of $r$ rows of $H_{\Psi,N,N}$ and
$\beta$ represents a selection of $r$ columns of $H_{\Psi,N,N}$.
Let $(\alpha,\beta)$ be an $r,N$-selection. The proposed algorithm takes as parameter
the matrix $H_{\Psi,N+1,N}$ and an $r,N$-selection $(\alpha,\beta)$. 
In addition, we assume that the $r,N$-selection $(\alpha,\beta)$ is such that
the following holds. Let $H_{\Psi,\alpha,\beta}$ be the matrix formed by the
intersection of the columns of $H_{\Psi,N,N}$ indexed by elements of $\beta$ with
the rows of $H_{\Psi,N,N}$ indexed by the elements of $\alpha$.
We then assume that $\Rank H_{\Psi,\alpha,\beta}=\Rank H_{\Psi,N,N+1}$.
%$\alpha=\{(u_1,j_1),\ldots, (u_r,j_r)\}$, $\beta=\{(v_1,j_1),\ldots, (v_r,j_r)\}$.
%Define the matrix
%\[   H_{\Psi,\alpha,\beta}=((S_{j_l}(v_lu_k))_{i_k}))_{l,k=1,\ldots,r} \in \mathbb{R}^{r \times r}.
%\]
%We assume that $r=\Rank H_{\Psi,\alpha,\beta}=\Rank H_{\Psi,N+1,N}$.
\begin{algorithm}
\caption{
\label{alg1} 
 \newline
\textbf{Inputs:} matrix $H_{\Psi,N+1,N}$ and  $r,N$-selection $(\alpha,\beta)$ 
 \newline
\textbf{Output:} representation $\widetilde{R}_N$.
}
\begin{algorithmic}
%%\STATE 
%%   %Let $r=\Rank H_{\Psi,N,N}$ and choose
%%   $j_{1},\ldots, j_{r} \in J$, $i_{1},\ldots, i_{r} \in \{1,\ldots,p\}$,
%%   $v_{1},\ldots, v_{r}, u_{1},\ldots, u_{r} \in \Sigma^{*}$
%%   such that  for all $l=1,\ldots, r$, 
%%   $|v_{l}| \le N$ and $|u_{l}| \le N$, and
%%   the matrix
%%   %%\begin{equation}
%%   \(
%%   T=((S_{j_{k}}(v_{k}u_{l}))_{i_{l}}))_{l,k=1,\ldots, r} \in \mathbb{R}^{r \times r}
%%   %%\end{equation}
%%   \) of $H_{\Psi,N,N}$ is of
%%   rank $r$. %%\ie $\Rank T=r$.
\STATE 
    For each symbol 
    $\sigma \in \Sigma$ let $A_{\sigma} \in \mathbb{R}^{r \times r}$
    be such that %%for each $l=1,\ldots, r$
    \[ A_{\sigma}H_{\Psi,\alpha,\beta}=Z_{\sigma} \]
  \mbox{  where }
  $Z_{\sigma}$ is $r \times r$ matrix with row indices from $\alpha$ and column indices from
  $\beta$ such that its entry indexed by $z \in \alpha$, $(v,j) \in \beta$ equals
  the entry of $H_{\Psi,N,N+1}$ indexed by $(z,(v\sigma,j))$.

  %$i$th column of $H_{\Psi,\alpha,\beta}$ corresponds to the column 
  %indexed by $(v,j) \in \beta$, then 
  %the $i$th column of $Z_{\sigma}$ is formed by the column of
  %$H_{\Psi}$ indexed by $(v\sigma,j)$ and the rows which are indexed by elements of $\alpha$,
  %$i=1,\ldots,r$.
\STATE
    Let $B=\{ B_{j} \mid j \in J\}$, where for each 
    index $j \in J$, the vector
    $B_{j} \in \mathbb{R}^{r}$ is formed by those entries of
    the column $(\epsilon,j)$ of $H_{\Psi}$ which are indexed by the elements of $\alpha$.

\STATE
    Let $C \in \mathbb{R}^{p \times r}$ whose $i$th row is the interesection of 
    the row indexed by $(\epsilon,i)$ with the columns of $H_{\Psi}$ indexed by the elements of
    $\beta$, $i=1,2,\ldots,p$.
\STATE Return $\widetilde{R}_N=(\mathbb{R}^{r}, \{A_{\sigma}\}_{\sigma \in \Sigma},B,C)$.
\end{algorithmic}
\end{algorithm}
\begin{Theorem}[\cite{MP:Phd,Son:Real,SontagMod}]
\label{sect:part:theo1}
   If $r=\Rank H_{\Psi,N,N}=\Rank H_{\Psi}$, then 
   there exists an $r,N$-selection $(\alpha,\beta)$ such that
   $\Rank H_{\Psi,\alpha,\beta}=r$ and the 
   the representation
   $\widetilde{R}_{N}$ returned by
   Algorithm \ref{alg1} when applied to $H_{\Psi,N+1,N}$ and $(\alpha,\beta)$
   is minimal representation of $\Psi$.
 Furthermore, if $\Rank H_{\Psi} \le N$, or, equivalently, there exists a representation $R$ of $\Psi$, such that $\dim R \le N$, then
 $\Rank H_{\Psi}=\Rank H_{\Psi,N,N}$, hence 
 $\widetilde{R}_{N}$ is a minimal 
 representation of $\Psi$.
\end{Theorem}
%%
%%From a computational point of view, algorithm 
%%\textsl{ComputeRepresentation} may not be the best way to compute a representation of $\Psi$. 
%%However, we have chosen to present it, because it makes
%%theoretical reasoning easier. The algorithm is essentially a reformulation
%%of the construction presented in \cite{SontagMod}. An 
%% uses the factorization of the 
%%finite Hankel-matrix
%%$H_{\Psi,N,N+1}$ can be found in \cite{MP:Phd}. 

\subsection{ A Notion of Stability for Formal Power Series}
\label{hscc_pow_stab}

Since our goal is to use formal power series to build a stochastic realization theory for jump-linear systems,  we will need to restrict our attention to formal power series that are stable in some sense, similarly to the case of linear systems. In this subsection, we consider the notion of square summability for formal power series, and translate the requirement of square summability into algebraic properties of their representations.

More specifically, consider a formal power series $S \in \mathbb{R}^{p}\ll \Sigma^{*} \gg$ and define the sequence 
 \begin{align} 
 L_{n}=\sum_{k=0}^{n} \sum_{\sigma_{1} \in \Sigma} \cdots 
 \sum_{\sigma_{k} \in \Sigma} 
 ||S(\sigma_{1}\sigma_{2}\cdots \sigma_{k})||^{2}_{2} .
 \end{align}
where $||\cdot||_{2}$ is the Euclidean norm
in $\mathbb{R}^{p}$. 
The series $S$ will be called \emph{square summable}, if the
limit $\lim_{n \rightarrow +\infty} L_{n}$ exists and it
is finite. The family 
 $\Psi=\{ S_{j} \in \mathbb{R}^{p}\ll \Sigma^{*} \gg \mid j \in J\}$
 will be called \emph{square summable}, if
 for each $j \in J$, the formal power series 
 $S_{j}$ is square summable. 

We now characterize square summability of a family of formal power series in terms of the stability of its representation. Let $R=(\mathbb{R}^{n}, \{ A_{\sigma} \}_{\sigma \in \Sigma}, B,C)$ be an arbitrary representation of $\Psi=\{ S_{j} \in \mathbb{R}^{p}\ll \Sigma^{*} \gg \mid j \in J\}$.  Assume that
   $\Sigma=\{\sigma_{1},\ldots, \sigma_{d}\}$, where $d$ is the number of
   elements of $\Sigma$, 
%%  $d=card(\Sigma)$  
 and consider the matrix 
 \( \widetilde{A}=\sum\limits_{i=1}^{d} A_{\sigma_{i}}^{T} \otimes A_{\sigma_{i}}^{T} \),
%%  \begin{align}
%%     \widetilde{A}_{R}=\begin{bmatrix}
%%                    A_{\sigma_{1}} \otimes A_{\sigma_{1}} & 
%%		    A_{\sigma_{2}} \otimes A_{\sigma_{2}} & \cdots & 
%%		    A_{\sigma_{d}} \otimes A_{\sigma_{d}}  \\
%%                    A_{\sigma_{1}} \otimes A_{\sigma_{1}} & 
%%		    A_{\sigma_{2}} \otimes A_{\sigma_{2}}  & \cdots & 
%%		    A_{\sigma_{d}} \otimes A_{\sigma_{d}} \\
%%                   \vdots     & \vdots    & \vdots & \vdots \\
%%                    A_{\sigma_{1}} \otimes A_{\sigma_{1}} & 
%%		    A_{\sigma_{2}} \otimes A_{\sigma_{2}} & \cdots & 
%%		    A_{\sigma_{d}} \otimes A_{\sigma_{d}} 
%%                  \end{bmatrix} \in \mathbb{R}^{n^{2}d \times n^{2}d},
%%  \end{align}
  where $\otimes$ denotes the Kronecker product.
  We will call $R$ \emph{stable}, if the matrix
  $\widetilde{A}$ is stable, \ie
  if all
  its eigenvalues $\lambda$ lie inside the unit disk ($|\lambda| < 1$).
   We have the following.
 %% \begin{Lemma}
 %% \label{hscc_pow_stab:lemma1}
 %%  $R$ is stable if and only if for any $x \in \mathbb{R}^{n}$,
 %%  the formal power series $S(w)=A_{w}x$, $w \in \Sigma^{*}$
 %%  is square summable.
 %% \end{Lemma}
  \begin{Theorem}
  \label{hscc_pow_stab:theo2}
   Consider a family of formal power series $\Psi$. If $\Psi$ admits a 
   stable representation, then $\Psi$ is square summable.
    If $\Psi$ is square summable, then any minimal representation of $\Psi$ is stable.
  \end{Theorem}
  
   Notice the analogy with the case of linear systems, where the minimal
   realization of a stable transfer matrix is also stable. 

\begin{proof}[Proof of Theorem \ref{hscc_pow_stab:theo2}]
Assume that $\Psi$ has a stable representation $R=(\mathbb{R}^{n}, \{ A_{\sigma}\}_{\sigma\in \Sigma},C,B)$. Then all the eigenvalues of the matrix $\widetilde{A}=\sum_{\sigma \in \Sigma} A_{\sigma}^{T} \otimes A_{\sigma}^T$ are inside the unit circle. One can easily see that the matrix $\widetilde{A}$ is in fact a matrix representation of the linear map $\Z : \Re^{n\times n} \to \Re^{n\times n}$ defined as \[  \Z(V)=\sum_{\sigma \in \Sigma} A_\sigma^TVA_\sigma .\]  This result is obtained by identifying $\Re^{n\times n}$ with $\Re^{n^2}$,  as it is done in \cite[Section 2.1]{CostaBook}. As a consequence, the eigenvalues of $\Z$ and $\widetilde{A}$ coincide. Since the eigenvalues of $\Z$ are inside the unit circle, it follows from \cite[Proposition 2.5]{CostaBook} that for each positive semi-definite matrix $V \ge 0$,  the infinite sum
    \(   \sum_{k=0}^{\infty} \|\mathcal{Z}^{k}(V)\|  \)
    is convergent. By noticing that
    \[ \forall x \in \Re^n ~~ x^T\mathcal{Z}^{k}(V)x \le \| x \|_{2}^2  \cdot \| \Z^{k}(V) \|, \]
    we conclude that $\sum_{k=0}^{\infty} x^T\mathcal{Z}^{k}(V)x$
   is convergent for all $x$. It can be shown by induction that
  \begin{equation}
   \label{hscc_pow_stab:lemma2:eq2}
     \mathcal{Z}^{k}(V) = \sum_{w \in \Sigma^{*}, |w|=k} A_{w}^TVA_{w}.
    \end{equation}
    Thus, letting $V=C^TC$ in $\sum_{k=0}^{\infty} x^T\mathcal{Z}^{k}(V)x$, we conclude that 
   \( \sum_{w \in \Sigma^{*}} \|CA_{w}x\|_{2}^{2} \)
   is convergent for all $x$. If we set $x=B_j$, $j \in J$,
   we then obtain that $\sum_{w \in \Sigma^{*}} \|S_{j}(w) \|_{2}^{2}$ is convergent for all $j \in J$, \ie the family $\Psi$
  is square summable.

%%Then for some $Q > 0$, the series
%%   $\sum_{w \in X^{*}} B_j^TA_{w}^TQA_{w}B_j$  will be
%%   convergent for each $j \in J$. There exists a constant
%%   $M > 0$, such that for all $x \in \mathbb{R}^n$,
%%   \[ x^TC^TCx \le M x^TQx \]
%%   and hence 
%%   \[ \sum_{w \in X^{*}} ||S_j(w)||_{2}^2 = 
%%       \sum_{w \in X^{*}} B_j^TA_{w}^TC^TCA_{w}B_j
%%       \le 
%%      \sum_{w \in X^{*}} B_j^TA_{w}^TQA_{w}B_j.
%%   \]
%%   That is, for each $j \in J$,
%%   $\sum_{w \in X^{*}} ||S_{j}(w)||_{2}^{2}$ is 
%%   convergent and hence $\Psi$ is square summable.

Assume now that $\Psi$ is square summable and let $R=(\mathbb{R}^{n}, \{ A_{\sigma}\}_{\sigma \in \Sigma},C,B)$ be a minimal representation of $\Psi$. Also, let $Q=O_{R}^TO_{R} > 0$, where $O_R$ is the observability matrix of $R$, which is full rank because $R$ is observable. 
First we show that
   \begin{equation}
   \label{hscc_pow_stab:lemma2:eq11}
    \sum_{k=0}^\infty x^T \Z^k(Q) x = \sum_{w \in \Sigma^{*}} x^TA_{w}^TQA_{w}x
   \end{equation}  
is convergent for all $x \in \mathbb{R}^{n}$. To see this, notice from the reachability of $R$ that any $x \in \mathbb{R}^{n}$ is a linear combination of vectors of the form $A_vB_j$, $j \in J$, $v \in \Sigma^{*}$. Hence, it is sufficient to prove the convergence of \eqref{hscc_pow_stab:lemma2:eq11} for $x=A_vB_j$. But the latter follow from the fact that \[ \sum_{w\in\Sigma^*}(B_jA_v)^TA_{w}^TQA_wA_vB_j=  \sum_{w\in\Sigma^*} \sum_{i=0}^{M(n-1)} ||S_{j}(vwv_i)||_{2}^{2}\] and that $\sum_{w \in \Sigma^{*}} \|S_{j}(w)\|^2_2$, hence $\sum_{w \in \Sigma^{*}} \sum_{i=0}^{M(n-1)} \|S_{j}(vwv_i)\|_{2}^{2}$ is convergent. Next we show that 
   \begin{equation}
    \sum_{k=0}^\infty x^T \Z^k(V) x = \sum_{w \in \Sigma^{*}} x^TA_{w}^TVA_{w}x
   \end{equation}  
is convergent for all $x \in \mathbb{R}^{n}$ and for all positive semi-definite $n\times n$ matrices $V \geq 0$. To see this, notice that for all $V \ge 0$   and $Q > 0$, there exists $M > 0$ such that $x^TVx \le M x^TQx$ for all $x \in\mathbb{R}^n$. Indeed, we can choose $M=\frac{\|V\|}{m}$, where $0 < m=\inf_{\|x\|=1} x^TQx$, so that $m\|x\|^2 \le x^TQx$ and hence $x^TVx \le \|x\|^{2}\|V\| \le M x^TQx$. Therefore, for any $V \ge 0$,
   \[ 
        \sum_{k=0}^\infty x^T\mathcal{Z}^{k}(V)x =
     \sum_{w \in \Sigma^{*}} x^TA_{w}^TVA_{w}x \le
       M \sum_{w \in \Sigma^{*}} x^TA_{w}^TQA_{w}x = 
       M \sum_{k=0}^\infty x^T\mathcal{Z}^{k}(Q)x,
    \]
and so $\sum_{k=0}^\infty x^T\mathcal{Z}^{k}(V)x$ is convergent for all $x\in\Re^n$ and $V\geq 0$. This implies that
    \[ \lim_{k \rightarrow \infty} x^T\mathcal{Z}^{k}(V)x = 0
    \]
    for all $x \in \mathbb{R}^n$. Therefore, 
     $\lim_{k \rightarrow \infty} \mathcal{Z}^{k}(V) = 0$
    for all $V \ge 0$, 
    which by \cite[Proposition 2.5]{CostaBook} implies
    that all the eigenvalues of $\mathcal{Z}$ (and hence
    of $\widetilde{A}$) have 
    modulus strictly smaller than $1$, \ie $R$ is stable.

\end{proof}

\section{Stochastic Realization of Generalized Bilinear Systems}
\label{gen:filt}
In this section we formulate and solve the realization problem for generalized stochastic
bilinear systems (abbreviated by \GBS). A \GBS\ is stochastic system which is
bilinear in state and inputs and where the inputs is an observed stochastic process.
Informally, the realization problem can be 
formulated as follows: given an output process and input process, 
find a \GBS which is driven by the input process, and whose
output process coincides with the given one. 
 Unlike in \cite{StochBilin}, we will not require the input to be white. In particular, we will allow finite-state Markov processes as inputs, which will allow us to apply the framework to the realization of stochastic jump-linear systems.  Particular cases of this generalized bilinear realization problem include realization of classical linear and bilinear systems, as well as the Kalman filter. In addition, the solution to this general problem provides a solution to the realization of stochastic jump-linear systems, as we will show in Section \ref{sect:GJMLS}.

The motivation of the realization problem stems from system identification and filtering.
The link with system identification is quite clear: the realization problem can be
viewed as a idealized system identification problem. The link with filtering is less
direct. Recall that
filtering one is interested in computing the conditional expectation (or the linear projection) of the current output onto the past outputs. 
%One can either consider finitely or infinitely many past outputs.
The Kalman filter is an algorithm that computes such a projection recursively.
If one considers stationary linear systems,  then the Kalman filter yields a linear stochastic realization in the forward innovation form.  That is, there is a correspondence
between recursive filters and stochastic realizations in forward innovation form.

In the case of bilinear situation, the situation is similar. The main difference is that
the filtering occurs based not only on past outputs but on past inputs too.
%instead of computing a projection 
%onto the past outputs, one computes a projection onto various 
%products of past outputs and past and present observed inputs. 
%
In particular, the correspondence between filters and stochastic realizations carries over
to bilinear systems.
 %a \GBS\ in forward innovation form can be thought of as a 
%recursive algorithm for computing the linear projection of the output on the past 
%outputs and inputs. In fact, the solution of the realization problem we will present in this
%section will yield a quasi-algorithms for computing a \GBS\ realization in forward innovation
%form, i.e. a recursive filter.
%
%
Similarly to the linear case, 
the construction of the recursive filter (i.e. stochastic realization in forward innovation form) relies on the fact that the
covariances of the outputs can be represented as rational formal power series. 
%In fact, a rational representation of the covariances yields the parameters of the 
%realization in forward innovation form.
%if one is allowed to consider past inputs and outputs at any (hence negative) time in
%the past, then the recursive filter yields a stochastic linear/bilinear realization in the forward innovation form. 
%In this section, we formulate and solve a generalized bilinear stochastic realization problem.
% Given an output and an input process, the goal is to find a recursive procedure for computing the linear projection of the current output onto the space of products of past outputs and inputs.
%In order to solve the generalized bilinear realization problem, we will assume that the covariances of the current and past outputs and inputs form a rational formal power series. Therefore, one can find a recursive state-space realization of the covariance sequence (and so of the generalized bilinear system) using the results in Section \ref{sect:pow}. 

The section is organized as follows. In \S \ref{gbs:prob:def} we define the class
of generalized bilinear systems and the corresponding realization problem.
In \S \ref{gbs:prob:sol} we present the solution of the realization problem.
In \S \ref{gbs:prob:alg} we present a realization algorithm. The proofs 
of the results of \S \ref{gbs:prob:sol}--\ref{gbs:prob:alg} are presented in
\S \ref{gbs:prob:proof}.

In what follows, we will work with random variables and stochastic processes. We will use the standard terminology and notation of probability theory
\cite{Bilingsley}. Throughout the paper, we fix a probability space $(\Omega,\mathcal{F},P)$ and
all the random variables and stochastic processes should be understood with respect to this probability space. Here
$\mathcal{F}$ is a $\sigma$-algebra over the set $\Omega$, $P$ is a probability measure on $\mathcal{F}$. 
With a slight abuse of notation, when we want to indicate that a random variable $\z$ takes its values in a set $X$ (i.e. $\z$ is a measurable function $\z:\Omega \rightarrow X$), we will write $\z \in X$.  We denote the expectation of a random variable $\z$ by $E[\z]$.
Let $\mathbb{Z}$ be the set of integers. 
Recall that a discrete-time stochastic process (in the sequel to be referred to as process or stochastic process) taking values in a set $X$ is just a
collection $\{\z(t)\}_{t \in \mathbb{Z}}$ where $\z(t)  \in X$ is a random variable for all $t \in \mathbb{Z}$; $\z(t)$ is referred to as the value of the
stochastic process $\{\z(t)\}_{t \in \mathbb{Z}}$ at time $t \in \mathbb{Z}$.
In the sequel, by abuse of notation, the stochastic process $\{\z(t)\}_{t \in \mathbb{Z}}$  will be denoted by $\z(t)$: whether $\z(t)$ means a stochastic
process or its value at time $t$ will be clear from the context. %If $\{\z(t)\}_{t \in \mathbb{Z}}$ is a process such that $\z(t) \in X$, $t \in \mathbb{Z}$, for some set $X$, then this is denoted by $\z(t) \in X$.  Furthermore, in the sequel, we will use the following definition. 
A stochastic process $\z(t) \in \mathbb{R}^k$ is called zero mean and square integrable, if the expectations $E[\z(t)]$ and $E[\z^T(t)\z(t)]$ exist, and 
$E[\z(t)]=0$ and $E[\z^T(t)\z(t)] < +\infty$. 
Furthermore, recall that a process $\z(t)  \in \mathbb{R}^k$ is wide sense stationary, if for every $s,t,k \in \mathbb{Z}$, 
the expectation $E[\z(t+k)\z^T(s+k)]$ exists and its value is independent of $k$.

\subsection{Stochastic Realization Problem for Generalized Bilinear Systems}
\label{gbs:prob:def}

%Let the \emph{output process} $\y(t)$ be a wide-sense stationary, zero-mean and $\Re^{p}$ valued process. 
Let the \emph{input process} be a collection of $\Re$ valued random processes $\{\bu_\sigma(t)\}_{\sigma\in\Sigma}$ indexed by the elements of a finite alphabet $\Sigma$.  
\begin{Definition}[Generalized Bilinear System]
 A generalized bilinear system (\emph{abbreviated by \GBS}) of 
% the output process $\y(t)$ 
is a system of the form
 \begin{equation}
     \label{gen:filt:bil:def}
     \BS \left \{\begin{split}
      \x(t+1)&=\sum_{\sigma \in \Sigma} (A_{\sigma}\x(t)+
       K_{\sigma}\v(t))\bu_{\sigma}(t) \\
      \y(t)&=C\x(t)+D\v(t),
     \end{split}\right.
     \end{equation}
 where $A_{\sigma} \in \mathbb{R}^{n}$, $K_{\sigma} \in \mathbb{R}^{n \times m}$,
  $C \in \mathbb{R}^{p \times n}$, $D \in \mathbb{R}^{p \times m}$,
 $\y(t)$ is a stochastic process with values in $\mathbb{R}^{p}$,
 called the \emph{state process}, 
 $\x(t)$ is a stochastic process with values in $\mathbb{R}^{n}$, called the \emph{state process} and 
 $\v(t()$ is a stochastic process with values in $\mathbb{R}^{m}$, called the 
 \emph{noise process}. The \emph{dimension} of $\BS$ is defined as the number $n$ of
 state variables.  The system $\BS$ is said to be a \emph{realization} 
 of the process $\widetilde{\y}(t)$ if $\widetilde{\y}(t)=\y(t)$ for all
 $t \in \mathrm{Z}$. The \GBS\ $\BS$ is said to be a \emph{minimal realization} of
 $\y(t)$ if $\BS$ is a realization of $\y(t)$ and it  has the minimal dimension 
 among all possible \GBS\ of $\y(t)$.
\end{Definition}
 Now we are ready to state the realization problem for \GBS{s}.
 \begin{Definition}[Realization problem for generalized bilinear systems]
 Given an output process $\y(t)$ and find conditions for existence of a
 \GBS\ which is a realization of $\y(t)$ and characterize minimality
 for \GBS{s} which are realizations of $\y(t)$.
\end{Definition}
Notice that by choosing $\bu_{\sigma}(t)$ in an appropriate way, \GBS{s}\ 
include linear, bilinear, and as we shall see later, even jump-linear systems.
%Another motivation for the realization problem above is its close relationship to filtering. This will be elaborated in detail later on.
%Before proceeding further with presenting the solution to the realization problem
%formulated above, we would like to mention a number of special cases of generalized
%bilinear systems.
\begin{Example}[Realization of Linear Systems]
\label{lin:example}
  Notice that if $\Sigma=\{\sigma\}$ and $\bu_{\sigma}(t)=1$, then the generalized bilinear stochastic realization problem reduces to the classical stochastic
linear realization problem.
\end{Example}
\begin{Example}[Realization of Bilinear Systems]
\label{bilin:example}
Notice that if $\Sigma=\{1,2\}$, $\bu_{1}(t)=1$ and $\bu_{2}(t)$ is white noise, then the generalized bilinear stochastic realization problem reduces to the classical bilinear realization problem \cite{StochBilin,FrazhoBilin}.
% If we assume that $\bu_{2}(t)$ is a colored noise, then we obtain a realization problem which has not been solved yet, to the best of our knowledge.
\end{Example}
\begin{Example}[Linear Jump-Markov systems with i.i.d discrete-state]
\label{example:jump_markoviid}
 Assume that $\btheta(t) \in \Sigma$ are independent and 
 identically distributed random variables, $P(\btheta(t)=\sigma)=p_{\sigma} > 0$.
 Consider the generalized bilinear system with $\bu_{\sigma}(t)=\chi(\btheta=\sigma)$,
 where $\chi$ is the indicator function. In this case the realization problem for
 \GBS{s} yields the realization of Jump-Markov linear systems 
 where is Markov process is observable and i.i.d. In fact, it can be shown that
 the realization problem of more general type jump-linear systems can also
 be reduced to that of \GBS{s}.
\end{Example}
\begin{Example}[Stochastic LPV systems]
 Let $\Sigma=\{1,\ldots,d\}$ and let $\bu(t)=(\bu_1(t),\ldots,\bu_{d}(t))$ be a stochastic
 process such that $\bu$ and $\v$ are independent. The resulting \GBS\ can be viewed
 as a stochastic linear parameter-varying system (LPV), where $\bu$ plays the role of 
 the scheduling variable. LPV systems represent a widely applied and popular system class.
 Identification of LPV systems is a subject of active research. The results of this paper
 are potentially useful for system identification of LPV systems.
\end{Example}
\begin{Example}[jump-bilinear systems with i.i.d discrete-state]
\label{example:bilin_jump_markoviid}
 Let $Q$ be a finite set and fix an integer $m$.
 Assume that $\btheta(t) \in Q$ are i.i.d random variables, $P(\btheta(t)=q)=p_{q} \ge 0$ for
 all $q\in Q$. Define $\Sigma=Q \times \{0,\ldots,m\}$ and let $\bu(t) \in \mathbb{R}^m$ be a colored noise process.
 Define $\bu_{(q,j)}(t)=\bu_{j}(t)\chi(\btheta=q)$, where $\bu_j(t)$ denotes the
 $j$th entry of $\bu(t)$ for $j=1,\ldots,m$ and $\bu_0(t)=1$.
 With this choice of the input process, we immediately
obtain the following  jump-bilinear system
   %\begin{align}
   \(  \x(t+1) =  \sum_{j=0}^{m} (A_{\btheta(t),j}\x(t)+K_{\btheta(t),j}\v(t))\bu_{j}(t) \) and
    \( \y(t)  =  C\x(t)+D\v(t) \).
   %\end{align}.
That is \GBS{s} do not only describe known system classes, but they also yield new system classes.
\end{Example}
The examples above are intended to demonstrate the versatility of \GBS{s}. \GBS{s}
can be used not only to describe well known system classes, but also system classes
which have not been studied in the literature so far.

\subsection{Hilbert-space of square integrable random variables}
\label{sec:hilbert:rev}
In order to make the realization problem tractable, we need to make additional assumption
on \GBS{s}.  In particular, 
in the sequel, the outputs and inputs at any time instance are mean-square integrable random variables.  Such random variables form a Hilbert-space $\mathcal{H}$ with  covariance playing the role of scalar product. Since $\mathcal{H}$ is an Hilbert-space, we can speak of orthogonal projection of a random variable onto a closed subspace of $\mathcal{H}$. 
Below we recall the framework of the Hilbert-space of random variables in more detail.

 In the sequel, we will identify random variables which differ only on a set 
 of probability zero.
 A scalar random variable $\z \in \mathbb{R}$ is said to
 be mean-square integrable, if the expectation $E[\z^2]$ exists
 and it is finite. The space of scalar mean-square random
 variables forms an Hilbert-space $\mathcal{H}$ 
 with the scalar product
 $<\z,\x>=E[\z\x]$ and the corresponding norm
 $||\z||=\sqrt{E[\z^2]}$. A sequence of random variables
 $\z_n$ is said to converge to in mean-square sense to
 $\z$, if $\lim_{n \rightarrow \infty} E[(\z-\z_n)^2]=0$, or,
 in other words, if $\lim_{n \rightarrow \infty} ||\z_n-\z||=0$
 with the norm $||.||$ defined above. As it is customary in Hilbert-spaces, the scalar product and the norm are continuous operators 
 with respect to the topology induced by mean-square convergence.
 That is, if $\lim_{n \rightarrow \infty} \z_n = \z$ and
 $\lim_{n \rightarrow \infty} \x_n = \x$ in the mean-square sense,
 then $\lim_{n \rightarrow \infty} E[\x_n\z_n]=E[\x\z]$ and
 $\lim_{n \rightarrow \infty} ||\x_n||=||\x||$.

 Suppose that $M$ is a closed linear subset of $\mathcal{H}$. 
 The orthogonal projection of a variable $\z$ onto $M$ 
 the unique element $\z^{*}$ of $M$ which satisfies the following
 two equivalent conditions: \textbf{(a)} 
 $||\z^{*}-\z|| \le ||\x-\z||$ for all $\x \in M$, 
 \textbf{(a)} $\z-\z^{*}$ is orthogonal to $M$, i.e.
  $E[(\z-\z^{*})\x]=0$ for all $\x \in M$. 
 Note that if $M$ is the linear span of finitely many elements,
 then it is automatically closed.

 Consider now a \emph{vector valued} random variable
 $\z=(\z_1,\ldots,\z_p)^T \in \mathbb{R}^{p}$.  
 We will call $\z$ mean-square
 integrable, if the coordinates $\z_i$, $i=1,\ldots,p$ are
 mean-square integrable scalar random variables. Note that if
 we denote by $||.||_{2}$ the Euclidean norm in $\mathbb{R}^{p}$,
 then mean-square integrability of $\z$ is equivalent to 
 existence and finiteness of $E[||z||_{2}^{2}]$.
 If $\z_n=(\z_1^n,\z_2^n,\ldots,\z_p^n) \in \mathbb{R}^{p}$, $n \in \mathbb{N}$ and
 $\z=(\z_1,\ldots,\z_p) \in \mathbb{R}^{p}$ are mean-square integrable random variables,
 then we say that $\z_n$ converges to $\z$ in a mean square sense,
 if for all $i=1,\ldots, p$, the sequence $\z_i^n \in \mathbb{R}$ of 
 $i$ coordinates
 of $\z_n$ converges to the $i$th coordinate $\z_i \in \mathbb{R}$ 
 of $\z$ in the mean-square sense.

 Let $M$ be a closed linear subspace of mean-square integrable
 \emph{scalar} random variables. Let 
 $\z=(\z_1,\ldots,\z_p) \in \mathbb{R}^{p}$ be 
 a \emph{vector valued} mean-square integrable random variable.
 By the orthogonal projection of $\z$ onto $M$ we mean the
 vector valued random variable $\z^{*}=(\z_1^{*},\ldots,\z_p^{*})$
 such that $\z_i^{*} \in M$ is the orthogonal projection of
 the $i$th coordinate $\z_i$ of $\z$ onto $M$, as defined for the
 scalar case. The orthogonal projection 
 $\z^{*}$ has the following property: 
 $E[(\z-\z^{*})\x]=0$ for all $\x \in M$. If $M$ is generated
 by closure of the linear span of
 the coordinates of a subset $S$ of $\mathbb{R}^{k}$ valued
 mean-square integrable random variables, then $\z^{*}$ is
 uniquely determined by the following property: $E[(\z-\z^{*})\x^T]=0$ for all $\x \in S$ and all the coordinates of $\z^{*}$ belong to $M$.

  In fact, by abuse of terminology, we will say \emph{that $\z$ belongs to $M$}, if 
  all its coordinates $\z_1,\ldots,\z_p$ belong to $M$.
  Similarly, let $\x_i \in \mathbb{R}^{k}$, $i \in I$ be a family of vector valued
  mean-square integrable random variables and assume that $I$ is an arbitrary set.
  Then the Hilbert-space generated by $\{x_i\}_{i \in I}$ is understood to be the
  smallest closed subspace $M$ of the Hilbert-space of all square integrable random variables
  such that for $\x_i$, $i \in I$ belongs to $M$ in the above sense (i.e. 
  the components of $\x_i$ belongs to $M$).

 Assume that $\z$ belongs to $M$ and assume that $M$ is the Hilbert-space generated
 by the components some vector values variables $\{\x_i\}_{i \in I}$. In the sequel,
 we will often use the following simple result.
 \begin{Lemma}
 \label{lemma:measurability}
  If the $\mathbb{R}^{p}$-valued random variable $\z$ belongs to $M$, then
  $\z$ is measurable with respect to the $\sigma$-algebra generated by 
  $\mathcal{F}=\{\x_i\}_{i \in I}$.
\end{Lemma}
 Indeed, by \cite[Exercise 34.13]{Bilingsley}, the conditional expectation $E[\z \mid \mathcal{F}]$
 equals the orthogonal projection of $z$ to the close subspace $\mathcal{H}_{\mathcal{F}}$
  generated by all the $\mathcal{F}$ measurable mean square integrable random variables.
 But $M$ is a subspace of $\mathcal{H}_{\mathcal{F}}$ 
 and hence $\z$ already belongs to $\mathcal{H}_{\mathcal{F}}$. Hence,
 the orthogonal projection of $z$ to $\mathcal{H}_{\mathcal{F}}$ equals 
 $\z$ itself. Thus, $\z=E[\z \mid \mathcal{F}]$ and since 
 $E[\z \mid \mathcal{F}]$ is $\mathcal{F}$ measurable by definition,
 Lemma \ref{lemma:measurability} follows.

\subsection{Solution of the realization problem for \GBS} \label{gbs:prob:sol} Below we present the solution of the realization problem for \GBS{s}.  
We will only state the results, their proofs will be presented in \S \ref{gbs:prob:proof}.
In order to state the results,  will introduce the following notation and terminology.
%Denote by $\Sigma^{+}$ the set of all nonempty words over the alphabet $\Sigma$.
\begin{Notation}
\label{ppnot1}
 We fix a collection $\{p_{\sigma} > 0\}_{\sigma \in \Sigma}$ of
 real numbers. For each $w \in \Sigma^{*}$ define the number
 $p_{w}$ as follows: $p_{\epsilon}=1$ and if $w=v\sigma$ for some 
 $v \in \Sigma^{*}$ and $\sigma \in \Sigma$, then let $p_{w}=p_{v}p_{\sigma}$.
\end{Notation}
The roles of $\{p_{\sigma}\}_{\sigma \in \Sigma}$ will become clear later on.
For each word $w=\sigma_1\sigma_2\cdots \sigma_{k} \in \Sigma^{+}$, $k\ge 1$, $\sigma_1,\ldots,\sigma_k \in \Sigma$, define the random variables
\begin{equation}
\label{gen:filt:eqzdef0}
  \bu_{w}(t)=\bu_{\sigma_{1}}(t-k+1)\bu_{\sigma_{2}}(t-k+1)\cdots \bu_{\sigma_{k}}(t)
\end{equation}
Using the notation defined above, we formulate the following assumptions which will be valid for the rest of the section.
\begin{Assumption}[Input process]
\label{input:assumption}
 \begin{enumerate}
 \item
 $\sum_{\sigma \in \Sigma} \alpha_{\sigma} \bu_{\sigma}(t)=1$ for some
 numbers $\{\alpha_{\sigma} \in \mathbb{R}\}_{\sigma \in \Sigma}$.
 \item
  For each $w \in \Sigma^{+}$, all the first and second order moments of the process $\bu_{w}(t)$ are finite.
  %i.e. $E[|u_{w}(t)|^{2}] < +\infty$ for any $k \ge 0$.
 \end{enumerate}
\end{Assumption}
 We mention a number of examples of
 $\bu_{\sigma}(t)$ which satisfies the assumptions above. 
 \begin{Example}[Bilinear systems \cite{StochBilin}]
  $\Sigma=\{0,1\}$, $\bu_{0}(t)=1$, $\bu_{1}(t)$ is a white noise Gaussian process.
  In this case, $\alpha_0=1$, $\alpha_1=0$.
 \end{Example}
 \begin{Example}[Discrete valued input]
 \label{disc:input}
  Assume there exists a process $\btheta(t)$
  takes its values
  from a finite alphabet $\Sigma$ and let $\bu_{\sigma}(t)=\chi(\btheta(t)=\sigma)$.
  Then $E[|\bu_{w}(t)|^{k}]=E[\bu_{w}(t)] = P(\btheta(t-k)=\sigma_1 \cdots \btheta(t-1)=\sigma_k)$ and  with $\alpha_{\sigma}=1$, $\sum_{\sigma \in \Sigma} \bu_{\sigma}(t)=1$.
 \end{Example}
Next, we define a class of stochastic processes which will play an important role in the rest of the paper.
Let $\br(t) \in  \mathbb{R}^{k}$ be a stochastic process and define for each $w \in \Sigma^{+}$
\begin{equation} 
\label{gen:filt:eqzdef1}
\begin{split}
   \z^{\br}_{w}(t)=  \br(t-|w|)\bu_{w}(t-1)\frac{1}{\sqrt{p_w}}.
%\bu_{\sigma_{1}}(t-k)\bu_{\sigma_{2}}(t-k+1)\cdots \bu_{\sigma_{k}}(t-1)\frac{1}{\sqrt{p_{\sigma_1}p_{\sigma_2}\cdots p_{\sigma}}}.
  \end{split}
 \end{equation} 
In the sequel, the process $\z^{\y}_{w}(t)$, obtained from \eqref{gen:filt:eqzdef1} by choosing
$\br(t)=\y(t)$ will play a central role. For this reason, we introduce the following notation
\begin{Notation}
\label{gen:filt:eqzdef1:not}
 In the sequel we denote by
$\z_{w}(t)$ the process $\z^{\y}_{w}(t)$.
\end{Notation}
%\begin{equation} 
%\label{gen:filt:eqzdef11}
%  \z_{w}(t)=  \y(t-k)\bu_{\sigma_{1}}(t-k)\bu_{\sigma_{2}}(t-k+1)\cdots
%    \bu_{\sigma_{k}}(t-1) \frac{1}{\sqrt{p_{\sigma_1}p_{\sigma_2}\cdots p_{\sigma}}}. 
%\end{equation}
Below, we will define a number of properties of $\z_{w}^{\br}(t)$  and we will require that
the noise, state, and output processes $\x(t)$, $\v(t)$ and $\y(t)$ of a \GBS\ are
such that $\z_{w}^{\x}(t)$, $\z^{\v}_{w}(t)$ and $\z_{w}^{\y}(t)$ satisfy those properties.
Intuitively, these properties say that $\z_{w}^{\br}(t)$ is a wide-sense stationary
stochastic process if $w$ is also viewed as multidimensional time.
To this, we introduce the following definitions.
\begin{Definition}[Admissible words]
 \label{gen:filt:ass1:cond1:def1}
 A set $L \subseteq \Sigma^{+}$ is a \emph{set of admissible words}, if the following
 conditions hold. 
 \begin{enumerate}
 \item $\Sigma \subseteq L$ and for all $w \in \Sigma^{+} \setminus L$, $\bu_{w}(t)=0$ almost surely.
 \item  There exists a set $S \subseteq \Sigma \times \Sigma$, such that
        the word $w=\sigma_1\cdots \sigma_k \in \Sigma^{+}$, $\sigma_1,\ldots,\sigma_k \in \Sigma$, $k > 1$ belongs $L$ if and
        only if $(\sigma_i,\sigma_{i+1}) \in S$ for all $i=1,\ldots,k-1$.
 %If for some word $w \in \Sigma^{+}$ and letter $\sigma \in \Sigma$, $w \sigma \in L$ or $\sigma w \in L$,  then $w \in L$.
% is also admissible, \ie if $w \sigma \in L$ or $\sigma w \in L$, then $w \in L$. In addition, every symbol $\sigma$ is an
%    element of $L$.
 \end{enumerate}
\end{Definition}
For the rest of the paper $L$ will denote a fixed set of admissible words.
The motivation behind introducing the set $L$ is that for certain $w \in \Sigma^{+}$,
we might wish to set $\z_{w}(t)$ to zero. This will be the case when we try to use realization
theory for \GBS{s} for jump-markov systems. A simpler motivating example is presented
below.
\begin{Example}[Jump-markov systems with restricted switching]
\label{example:jump_markoviid:const}
 Consider the system described in Example \ref{example:jump_markoviid} but with
 the following modification. We no longer assume that 
 $\btheta$ is an i.i.d process. Instead we assume that there exists a set
 $S \subseteq Q \times Q$ describing the admissible discrete state transitions, and
 $\btheta(t)$ is a stationary Markov process  such that 
  $P(\btheta(t+1)=q_2 \mid \btheta(t)=q_1)=p_{q_2}$ if $(q_1,q_2) \in S$ and 
  $P(\btheta(t+1)=q_2 \mid \btheta(t)=q_1)=0$ if $(q_1,q_2) \in S$. 
  In this case, the $\bu_{w}(t)=0$ almost surely for $w \notin L$, 
  where $L$ is as defined in \ref{gen:filt:ass1:cond1:def1}
\end{Example}
\begin{Definition}[Recursive covariance property]
%In addition in the rest of the paper, we will also fix a set of positive reals
%$\{p_{\sigma} > 0\}_{\sigma \in \Sigma}$ and we will make the following assumption
%about the input process.
\label{def:RC}
 A process $\br(t)$ is said to have \emph{recursive covariance property (abbreviated by \RC)}
 if it satisfies the following conditions.
\begin{enumerate}
\item
\label{RC1}
 The processes $(\br(t), \{ \z^{\br}_{w}(t) \mid w \in \Sigma^{+}\})$ are jointly wide-sense stationary, that is, for all $t,k \in \mathbb{Z}$, and for all $w,v \in \Sigma^{+}$ we have that
$E[\br(t)]=0$, $E[\z^{\br}_{w}(t)]=0$, and
\begin{align*}
E[\br(t+k)(\z_{w}^{\br}(t+k))^r] =E[\br(t)(\z^{\br}_{w}(t))^{T}]  \quad \text{and} \quad 
E[\z_{w}^{\br}(t+k)(\z_{v}^{\br}(t+k))^T]=E[\z^{\br}_{w}(t)(\z_{v}^{\br}(t))^T].
\end{align*}
%\item
%\label{RC2}
%  $w \in \Sigma^{+} \setminus L$,
 %$\forall t \in \mathbb{Z}: \z^{\br}_{w}(t)=0$.
\item
\label{RC3}
 Denote by
 \[ T^{\br}_{w,v}=E[\z^{\br}_{w}(t)(\z_{v}^{\br}(t))^T] \mbox{ and }
    \Lambda^{\br}_{w}=E[\br(t)(\z_{w}^{\br}(t))^T] \]
 Then for any $w,v \in \Sigma^{+}$, $\sigma,\sigma^{'} \in \Sigma$,
 $T_{\sigma,\sigma^{'}}=0$ for $\sigma \ne \sigma^{'}$ and 
 \begin{align}
  & T^{\br}_{w\sigma,v\sigma^{'}}=
 \begin{cases}
  T^{\br}_{w,v} & \mbox{ if } \sigma=\sigma^{'} \mbox{ and } w\sigma \in L \mbox{ or } v\sigma \in L \\
  0  & \mbox{ if } \sigma \ne \sigma^{'}
 \end{cases} 
 \qquad\text{and}\qquad \\
  & T^{\br}_{w \sigma,\sigma^{'}} =
 \begin{cases}
  (\Lambda^{\br}_{w})^T & \mbox{ if } \sigma=\sigma^{'} \\
   0  & \mbox{ if } \sigma \ne \sigma^{'}
   \end{cases}.
  \end{align}
\item
\label{RC4}
In addition, $T^{\br}_{w,v}=0$ if $w \notin L$ or $v \notin L$.
 If $w\sigma  \in L$ then
 for all $v\sigma \notin L$, $T^{\br}_{v,w}=0$, and similarly, if
 $v\sigma \in L$, then for all $w\sigma \notin L$,
 $T^{\br}_{v,w}=0$.
%%\item
%%\label{RC5}
%%For each $t$, and for all $k\ge 0$, the  random variables
%%  $\br(t-k),\z^{\br}_{w}(t-k)$, $w \in \Sigma^{+}$,
%%  belong the the closure (in the mean-square sense) of
%%  the space spanned by the random variables
%%  $\{\z^{\br}_{w}(t)$, $w \in \Sigma^{+}\}$.
\end{enumerate} 
\end{Definition}
\begin{Remark}
\label{rem:RC}
 It can be shown that
 Part \ref{RC4} of Definition \ref{def:RC} is by the other conditions.
% If $w \notin L$ or $v \notin L$, then either $\z_{w}^{\br}(t)=0$ or
% $\z_{v}^{\br}(t)=0$ and hence $T^{\br}_{w,v}=0$. If $v\sigma \notin L$, then 
% $v=\hat{v}\hat{\sigma}$ for some $\hat{v}\in \Sigma^{*}$ and $\hat{\sigma} \in \Sigma$ and either $w \notin L$ or $\hat{\sigma}\sigma \notin S$. Indeed, if
% $v \in L$ and $\hat{\sigma}\sigma \in S$, then by definition of $L$,
% $v\sigma \in L$ which a contradiction. Notice that $w=\hat{w}\sigma^{'}$. Then
% $w\sigma \in L$ implies that $\sigma^{'}\sigma \in S$. If $v \notin L$, then
% by Part \ref{RC2}, $T^{\br}_{w,v}=0$. If $v \in L$, then due to $\hat{\sigma}\sigma \in S$, $\hat{\sigma} \ne \sigma^{'}$, and hence by Part \ref{RC3}
% $T^{\br}_{w,v}=T^{\br}_{\hat{w}\sigma^{'},\hat{v}\hat{\sigma}}=0$. The second claim
% of Part \ref{RC4} follows analogously.
\end{Remark}
\begin{Remark}
  It is clear that if $\br(t) \in \mathbb{R}^r$ is an \RC\ process, then for any matrix $F \in \mathbb{R}^{l \times r}$, $l > 0$, 
  the process $\bs(t)=F\br(t)$, $t \in \mathbb{Z}$, is \RC. 
\end{Remark}
Intuitively, if $\br$ is $\RC$, then the processes $\z_{w}^{\br}$ obtained by multiplying
$\br(t)$ with future inputs $\bu_{w}(t+|w|)$ are zero-mean wide-sense stationary, moreover,
the covariances $T_{w,v}^{\br}$ have a specific recursive structure. This recursive structure
can be interpreted as wide-sense stationarity, if $w$ is viewed as a time instant on the
multidimensional time axis $\Sigma^{+}$. This property coincides with the property required
of multidimensional positive kernels in \cite{Popescu1} and a special instance of
this property was also
used in \cite{FrazhoBilin,StochBilin}.
This property (Part \ref{RC4} of Definition \ref{def:RC}) is crucial for developing
stochastic realization theory, especially for the realization algorithm.
\begin{Example}[Examples of \RC\ processes]
\label{rc:example1}
 %Below we present some simple examples of \RC\ process. These examples will be 
 %relevant later on too.
 %\item
  %Assume that $L=\Sigma^{+}$, $\br(t)$ is a zero-mean wide-sense stationary process, $\br(t)$ is independent of $\bu_{\sigma}(t+k)$, $k \ge 0$. Moreover, assume t
  %$\bu_{0}(t)=1$, $\bu_{1}$ is a Gaussian colored noise.  Then
  %$\y(t)$ is an \RC\ process and $p_0=1,p_1=1$. Note that the output, state and 
  %noise processes of a bilinear system from \cite{StochBilin} all have these properties.
   Assume that $L=\Sigma^{+}$, $\br(t)$ is a zero-mean wide-sense stationary process,
   $\br(t)$ and $\bu_{\sigma}(t+k)$, $k \ge 0$ are independent, 
   $\bu_{\sigma}(t)$ are i.i.d and $E[\bu^{2}_{\sigma}]=p_{\sigma}$, and
   %$\bu_{\sigma}(t)$ is a Gaussian colored noise with variance $p_{\sigma}$, and
   $\{\bu_{\sigma_1}(t)\}_{t \in \mathbb{Z}}$, $\{\bu_{\sigma_2}(t)\}_{t \in \mathbb{Z}}$
   are uncorrelated for all $\sigma_1 \ne \sigma_2$, i.e.
   $E[\bu_{\sigma_1}(t)\bu_{\sigma_2}(l)]=0$, $l,t \in \mathbb{Z}$. Moreover, assume that
   $\bu_{\sigma}(t)$ satisfies Assumption \ref{input:assumption} and that for all $w \in \Sigma^{+}$,
   $E[\br(t-|w|)\bu_w(t-1)\br^T(t)]$ is independent of $t$. 
   Then $\br$ is a \RC\ process.

   One particular examples of the situation is when $\bu_{\sigma}(t)$ is a zero mean i.i.d 
   Gaussian process.  Another example if when $\Sigma=\{0,1\}$, 
   $\bu_{0}(t)=1$ and $\bu_{1}(t)$ is an i.i.d zero mean Gaussian process with variance
   $p_{\sigma}$. This latter example is the one which occurs in bilinear stochastic systems.
   Finally, consider 
   $\bu_{\sigma}(t)$ is as in Example \ref{disc:input}. Assume, moreover that
   $\btheta(t)$ are i.i.d 
   $p_{\sigma}=P(\btheta(t)=\sigma)$. Then with $L=\Sigma^{+}$, $\br(t)$ is an \RC\ process.
\end{Example}
Although Example \ref{rc:example1} covers a lot important cases, the example below
demonstrates that RC processes where $\bu_{\sigma}$ is not an i.i.d process also plays
an important role.
\begin{Example}
\label{rc:example2}
   Consider the process $\btheta$ from Example \ref{example:jump_markoviid:const} and
   assume that $\{\btheta(t+l) \mid l \ge 0\}$ and $\br(t)$ are conditionally independent
   w.r.t. to $\{\btheta(t-l) \mid l \ge 0\}$. Assume that $\br(t)$ is wide-sense
  stationary, square integrable and zero-mean. Then $\br(t)$ is a \RC\ process with
  $L$ defined in Example \ref{example:jump_markoviid:const}.
\end{Example}
Now we are ready to formulate the assumptions we are going
to make about \GBS{s}.
\begin{Assumption}
\label{gbs:def}
 In the sequel, we will only consider \GBS{s} which satisfy the following conditions.
 \begin{enumerate}
 \item  
 \label{gbs:def:prop1}
        The noise process $\v(t)$
        have the \RC\ property.
 \item  
 \label{gbs:def:prop2}
       For every $w,v \in \Sigma^{+}$, $w \ne v$, $\z_v^{\v}(t)$ and
        $\z_{w}^{\v}(t)$ are orthogonal, i.e. $E[\z_w^{\v}(t)(\z_v^{\v}(t))^T]=0$.
        %the covariance $E[\v(t)\v^T(t)\bu_{\sigma}^{2}(t)]=Q_{\sigma}$  of $\v(t)$ is strictly positive definite
 %\item  
 %\label{gbs:def:prop3}
 %       For every $w,v \in \Sigma^{+}$, $|w| \ge |v|$, $\z^{\x}_w(t)$ and
 %       $\z^{\x}_{v}(t)$ are orthogonal, i.e. 
 %       $E[\z^{\x}_{w}(t)(\z^{\x}_{v}(t))^T]=0$.
 \item  
 \label{gbs:def:prop4}
        The state $\x(t)$ belongs to the Hilbert-space generated by the entries of
        $\{\z_{w}^{\v}(t) \mid w \in \Sigma^{+}\}$.
        %$\x(t)$, $\{z_{w}^{\v}(t) \mid w \in \Sigma^{+}\}$ are 
       %jointly wide-sense stationary, i.e. 
        %$E[\x(t+k)(\z_{w}^{\v}(t+k))^T]=E[\x(t)(\z_{w}^{\v}(t))^T]$ for
        %all $k,t \in \mathrm{Z}$, $w \in \Sigma^{+}$.
        
 \item 
 \label{gbs:def:prop5}
      The matrix $\Sigma_{\sigma \in \Sigma} p_{\sigma}A_{\sigma}^T \otimes A_{\sigma}^T$ is
       stable, i.e. all its eigenvalues are inside the unit circle.
%\item
 %\label{gbs:def:prop6}
  %For all $\sigma \in \Sigma$, $DQ_{\sigma}D^{T}$ is strictly positive definite.

\item
\label{gbs:def:prop7}
  For all $\sigma_1,\sigma_2 \in \Sigma$, if $\sigma_1\sigma_2 \notin L$, then
  $A_{\sigma_2}A_{\sigma_1}=0$ and $A_{\sigma_2}K_{\sigma_1}T^{\v}_{\sigma_1,\sigma_1}=0$.
 \end{enumerate}
\end{Assumption}
 Intuitively, Part \ref{gbs:def:prop1} of Assumption \ref{gbs:def} requires that the state and
 noise process are stationary and that they are very loosely correlated with future
 inputs. Parts \ref{gbs:def:prop2}--\ref{gbs:def:prop1} of Assumption \ref{gbs:def} say that the noise processes are uncorrelated.
Parts \ref{gbs:def:prop4}--\ref{gbs:def:prop5} intuitively express the
 assumption that $\x(t)$ is the result of starting at zero initial state at $-\infty$ and
 allowing the system to be driven by the noise process alone. The stability assumption is
 there to guarantee that this can be done.
 In fact, Assumption \ref{gbs:def} yields the following.
 \begin{Lemma}
 \label{gbs:def:new_lemma1}
  If $\BS$ of the form \eqref{gen:filt:bil:def} satisfies Assumption \ref{gbs:def}, then $\begin{bmatrix} \v^T(t),\x^T(t) \end{bmatrix}^T$ is an \RC\ process, and
  hence $\x(t)$ is an \RC\ process. Moreover, $w,v \in \Sigma^{+}$, $|w| \ge |v|$, $\z^{\x}_w(t)$ and $\z^{\v}_v(t)$ are uncorrelated, i.e. 
  $E[\z^{\x}_w(t)(\z^{\v}_v(t))^T]=0$, and 
  \begin{equation} 
  \label{gbs:def:new_lemma1:eq1}
   \forall t \in \mathbb{Z}: 
    \x(t)=\sum_{w \in \Sigma^{*}} \sum_{\sigma \in \Sigma} \sqrt{p_{\sigma w}} A_{w}B_{\sigma}\z_{\sigma w}^{\v}(t).
  \end{equation}
  Here we used Notation \ref{matrix:prod} for the matrix product  $A_{w}$, $w \in \Sigma^{*}$ and convergence is understood in the mean-square sense.
\end{Lemma}
  In fact, we can also show that under some mild conditions, the trajectories 
  of $\BS$ converge to $\x(t)$ as $t$ goes to infinity.
 \begin{Lemma}
 \label{gbs:def:new_lemma1.2}
  With the assumptions of Lemma \ref{gbs:def:new_lemma1}, 
  if $\hat{\x}(t)$ is a process which satisfies the first equation of
  \eqref{gen:filt:bil:def} and for all $w,v \in \Sigma^{*}$, $\sigma_1,\sigma_2 \in \Sigma$,
   $|w|=|v|$,
   $T_{\sigma_1w,\sigma_2v}^{\hat{\x}}=E[\hat{\x}(0)\hat{\x}^T(0)\bu_{w}(|w|-1)\bu_{v}(|v|-1)]$ is such that
   $T_{\sigma_1w,\sigma_2v}^{\hat{\x}}=0$ if $\sigma_1w \ne \sigma_2v$ and
   $T_{\sigma_1w,\sigma_2v}^{\hat{\x}}=E[\hat{\x}(0)\hat{\x}^T(0)\bu_{\sigma}^2(0)]p_{w}$ otherwise, then
   \[ \lim_{t \rightarrow \infty} E[||\x(t)-\hat{\x}(t)||^{2}]=0. \]
 \end{Lemma}
  If $\hat{\x}(0)$ is independent of $\bu_{\sigma}(t)$, $t \ge 0$, $\sigma \in \Sigma$ and
  $E[\bu_{w}(|w|-1)\bu_{v}(|v|-1)]=0$ for $w \ne v$ and 
  $p_{w}=E[\bu_{w}(|w|-1)\bu_{v}(|v|-1)]=0$, then the assumptions of 
  Lemma \ref{gbs:def:new_lemma1.2} are satisfied. In particular,
  the assumptions of Lemma \ref{gbs:def:new_lemma1.2} are the standard ones made for
  the systems described in Examples \ref{lin:example}--\ref{example:jump_markoviid}.
  Finally, note that $\x(t)$ is wide-sense stationary and the
 % if $\x(t)$ is \RC\, then it is wide-sense stationary and in fact the
  following holds.
 \begin{Lemma}
 \label{gbs:def:new_lemma1.3}
 Consider a \GBS\ $\BS$ of the form \eqref{gen:filt:bil:def} and assume that
 $\BS$ satisfies Assumptions \ref{gbs:def}. Consider the 
 equation
 \begin{equation} 
 \label{gen:filt:theo3.2:eq1}
   P_{\sigma}=p_{\sigma} (\sum_{\sigma_1 \in \Sigma, \sigma_1\sigma \in L} A_{\sigma_1}P_{\sigma_1}A_{\sigma_1}^T + K_{\sigma_1}Q_{\sigma_1}K_{\sigma_1}^T)
\end{equation}
 where $Q_{\sigma}=E[\v(t)\v^T(t)\bu_{\sigma}^2(t)]$ and
 $\{P_{q}\}_{q \in Q}$ is a family of matrix-valued indeterminate. Then
 \eqref{gen:filt:theo3.2:eq1} has a unique solution determined by
 $P_{\sigma}=E[\x(t)\x^T(t)\bu_{\sigma}^2(t)]$, $\sigma \in \Sigma$.
 \end{Lemma}
  The proofs of Lemma \ref{gbs:def:new_lemma1}--Lema \ref{gbs:def:new_lemma1.3} 
  require certain technical results,
  for this reason we postpone them to \S \ref{proof:gbs:tech}.
  The  Lemma \ref{gbs:def:new_lemma1} says that the state of
  $\Sigma$ is the one which one would obtain by starting the system at zero at
  $-\infty$. Lemma \ref{gbs:def:new_lemma1.2} says that if we pick any initial state which satisfies some mild conditions, then the resulting state trajectory of $\Sigma$ will converge to
  the stationary trajectory $\x(t)$. In fact, the existence of the right-hand side of
  \eqref{gbs:def:new_lemma1:eq1} does not require Part \ref{gbs:def:prop4} of
  Assumption \ref{gbs:def}. Hence, Lemma \ref{gbs:def:new_lemma1} -- \ref{gbs:def:new_lemma1.2}  can be 
  interpreted as stating that if the system $\Sigma$ satisfies Assumption \ref{gbs:def},
  except Part \ref{gbs:def:prop4}, then it has a state trajectory which satisfies 
  Part \ref{gbs:def:prop4}, moreover any state-trajectory of $\Sigma$ converges to
  that particular one. The situation is similar to that of for stable linear systems: asymptotically, a the state-trajectory of a stable linear system is stationary. 
  Finally, Lemma \ref{gbs:def:new_lemma1.3} provides a formula for the state covariance as a solution of a Lyapunov-like equation. Note that similar formulas are well-known for the linear \cite{Caines1} and even bilinear case \cite{StochBilin,FrazhoBilin}.
 The formula of Lemma \ref{gbs:def:new_lemma1.3} represents a generalization of those
 well-known results.
  
 %Note that in particular, Assumption \ref{gbs:def} this also implies that 
 %that the past states and the future noises are uncorrelated as well, i.e.
 %$E[\x(t)\v^T(t+k)]=0$ for $k \ge 0$. 
 %Indeed, $\x(t)$ belongs to the Hilbert-space
 %generated by $\z_{w}^{\v}(t)$ and by Part \ref{gbs:def:prop2}, 
 %$E[\z_{w}^{\v}(t)\v^{T}(t+k)]=0$\footnote{recall that from 
 %Assumption \ref{input:assumption} it follows $v(t+k)$ is just a linear sum of 
 %$\z_{w}(t+k)$, $|w|=k$, $w \in \Sigma^{+}$.}.
% \begin{Proposition}
% \label{gbs:def:prop3}
%        For every $w,v \in \Sigma^{+}$, $|w| \ge |v|$, $\z^{\x}_w(t)$ and
%        $\z^{\x}_{v}(t)$ are orthogonal, i.e.
%        $E[\z^{\v}_{w}(t)(\z^{\x}_{v}(t))^T]=0$.
% \end{Proposition}.
 We present a number of examples of systems which satisfy
 Assumptions \ref{gbs:def}.
\begin{Example}[Linear systems]
\label{example:lin:gbs_def}
 A stationary stable Gaussian linear system with the standard assumption
 can be viewed as a \GBS\ which satisfies Assumption \ref{gbs:def}. In this case,
 $\Sigma=\{0\}$, $\bu_{0}(t)=1$, $A_0$ is stable, $L=\Sigma^{+}$, 
 $\v(t)$ is an i.i.d process which is Gaussian and zero mean. 
 If we assume that the initial state of the
 system at time $\-\infty$ was zero, then it is easy to see that
 the resulting \GBS\ satisfies Assumption \ref{gbs:def}.
\end{Example}
 \begin{Example}[Bilinear systems]
\label{example:bilin:gbs_def}
  The bilinear systems from \cite{StochBilin,FrazhoBilin} satisfy Assumption \ref{gbs:def}. In that case, $\Sigma=\{0,1\}$, $\bu_0(t)=1$, $\bu_{\sigma}(t)$ is a white noise Gaussian process, $\v(t)$ is also a white noise Gaussian process, $B_{1}=0$ and
the random variables $\v(t)$ and $\bu_{1}(t+l)$, $l \in \mathrm{Z}$ are assumed to be independent (the $\sigma$-algebra generated by them is independent). 
Moreover, it is assumed that $\x(t)$ is zero-mean, 
wide-sense stationary and satisfies \eqref{gbs:def:new_lemma1:eq1}. 
In fact in \cite{StochBilin,FrazhoBilin}  it was not explicitly assumed that $\x(t)$ satisfies \eqref{gbs:def:new_lemma1:eq1}, but from the discussion after Lemma \ref{gbs:def:new_lemma1} it
 follows that this can be assumed without loss of generality. Moreover, the state process of the realization constructed by the algorithm \cite{StochBilin,FrazhoBilin} does 
 satisfy the assumptions of Lemma \ref{gbs:def:new_lemma1}.
 \end{Example}
\begin{Example}[Jump-linear systems driven by i.i.d.]
\label{example:jumplin:gbs_def}
 Consider jump-linear systems driven by an i.i.d process as described in Example
 \ref{example:jump_markoviid}. In this case, $L=\Sigma^{+}$. 
  Assume that $\{\btheta(t)\}_{t \in \mathrm{Z}}$ are independent,
 identically distributed, $p_{\sigma}=P(\btheta(t)=\sigma) > 0$, $\sigma \in \Sigma$.
 Assume that the noise process is $\{\v(t)\}_{t \in \mathrm{Z}}$ is independent of
 $\{\btheta(t)\}_{t \in \mathrm{Z}}$ and that $\v(t)$ is a wide-sense colored noise process
 i.e. $E[\v(t)\bu_{w}(l-1)\v^T(l)]=0$, $l > t$, $w \in \Sigma^{+}$, $|w|=l-t+1$, $E[\v(t)]=0$, $E[\v(t)\v^T(t)]=Q > 0$.
 Assume that  Part \ref{gbs:def:prop4}, Part \ref{gbs:def:prop5} and Part \ref{gbs:def:prop7}
 of Assumption \ref{gbs:def} holds. Then the system satisfies Assumption \ref{gbs:def}.
 Note that the assumptions we made are quite mild, they are similar to the ones of
 \cite{Kotsalis}. 
\end{Example}
 The examples above represent a special case of the following class of \GBS{s}.
\begin{Example}[\GBS\ with independent inputs]
\label{example:iid:gbs_def}
 Consider a \GBS $\BS$ such that $\v$ and $\bu_{\sigma}$, $\sigma \in \Sigma$ satisfy
 Example \ref{rc:example1}. That is, $\bu_{\sigma}$ is an i.i.d process,
 $E[\bu_{\sigma}^2(t)]=p_{\sigma}$, and the $\sigma$-algebras generated by 
  $\{\v(t-l)\}_{l=0}^{\infty}$ and
 $\{\bu_{\sigma}(t+l) \mid \sigma \in \Sigma,l \ge 0\}$ are independent for any $t$.
 Assume moreover that $\v$ is a zero mean wide sense stationary process and
 $E[\v(t-l)\v^T(t)]=0$, $l > 0$, $w \in \Sigma^{+}$, $|w|=l$, $t \in \mathrm{Z}$.
 Let $L=\Sigma^{+}$ and assume
 $\sum_{\sigma \in \Sigma} p_{\sigma}A_{\sigma}^T \otimes A_{\sigma}^T$ is a stable matrix. 
 Assume that the state $\x(t)$ is obtained by starting the
 system in zero initial state at time $-\infty$. Then $\BS$ satisfies Assumption \ref{gbs:def}.
\end{Example}
 Examples \ref{example:lin:gbs_def}--\ref{example:jumplin:gbs_def} represent
 special cases of Example \ref{example:iid:gbs_def}.
 Example \ref{example:iid:gbs_def} can also be used to obtain bilinear jump-markov
systems as described in Example \ref{example:bilin_jump_markoviid}.
Unfortunately, Example \ref{example:iid:gbs_def} does not cover the case of
jump-markov linear systems where the discrete state process is not i.i.d.
Below we show that even such cases can be cast into our framework. Here we only
present a special class of jump-markov linear systems, the general case is dealt
with in \S ref{sect:real}.
\begin{Example}[Jump-markov linear systems with restricted switching]
\label{example:const:gbs_def}
 Consider the input process $\bu_{\sigma}$, $\sigma \in \Sigma$ 
 described in Example \ref{example:jump_markoviid:const}. Consider a \GBS\
 with this input process, such that the following holds.
 Denote by $\mathcal{D}_{t}$ the $\sigma$-algebra generated by $\{\btheta(l) \mid l < t\}$. Assume that $\v(t)$ is a wide-sense stationary zero mean process,
such that $\v(t)$ and $\v(l)$, $l \ne t$ are  $\v(t)$ and $\v(l)$, 
   $l \le t$ are conditionally uncorrelated with respect to the $\sigma$-algebra $\mathcal{D}_{l,t-1}$ generated by $\{\btheta(t)\}_{t=l}^{t_1-1}$, i.e.
    $E[\v(t)\v^T(l)\mid \mathcal{D}_{l,t-1}]=0$. Moreover, assume that the
$\sigma$-algebras generated by
$\{\btheta(t+l)\}_{l=0}^{\infty}$ and $\{\v(t-l)\}_{l=0}^{\infty}$ are
conditionally independent with respect to $\mathcal{D}_t$. 
In addition, assume that Part \ref{gbs:def:prop4}, Part \ref{gbs:def:prop5} and Part \ref{gbs:def:prop7} of Assumption \ref{gbs:def} holds.
 Then the resulting system will again satisfy Assumption \ref{gbs:def}.  
 The \GBS{s} described above can be thought of as a special class of jump-markov linear systems, where the transition probabilities of the discrete state process are either zero or depend only on the final state.  
%Part \ref{gbs:def:prop5} of Assumption \ref{gbs:def} is a mild restriction, in fact, if it is not satisfied, then by enlarging the state-space, this can be achieved.
\end{Example}
%In order to solve the generalized bilinear stochastic realization problem, we will make a number of assumptions. 
  
 Next, we state a number of assumptions on the output process $\y(t)$ which will guarantee
 existence of a \GBS\ realization of $\y$.
% More 
 %precisely, these assumptions will enable us to construct a \GBS\ of $(\y(t),\{\bu_{\sigma}(t)\}_{\sigma \in \Sigma})$.  This \GBS\ will also have a neat system-theoretic interpretation
  %as a filter, capable of recursively computing the best linear estimate of
  %the output $\y(t)$ based on the past values of outputs and inputs.
 To this end, recall that $\z_{w}(t)$ denotes the process $\z^{\y}_{w}(t)$.
%We will call the random variables  $\z_{w}(t), w \in \Sigma^{+}$ the \emph{predictor variables}.
%\begin{Assumption}
%  \label{gen:filt:ass1:cond0}
% We will assume that the output and and predictor variables $(\y(t), \{ \z_{w}(t) \mid w \in \Sigma^{+}\})$ are jointly wide-sense stationary, that is, for all $t,k \in \mathbb{Z}$, and for all $w,v \in \Sigma^{+}$ we have that
%
%\begin{align}
%E[\y(t+k)\z_{w}^{T}(t+k)] &=E[\y(t)\z_{w}^{T}(t)],  \quad \text{and}\\
%E[\z_{w}(t+k)\z_{v}^{T}(t+k)] &=E[\z_{w}(t)\z_{v}^{T}(t)].
%\end{align}
%\end{Assumption}
When constructing a \GBS\ realization of $\y$, 
we will compute the orthogonal projection of the future outputs onto the
Hilbert space formed by the past outputs and inputs. 
In order to simplify the discussion about orthogonal projections, we will use the
following notation.
%Since the space of zero-mean wide-sense stationary processes is a Hilbert space,\footnote{For a brief introduction to this framework see \cite{Caines1} and  the references therein.} 
%The linear projection of future outputs onto the past
%outputs and inputs is interpreted as the orthogonal projection
%in that Hilbert space.
%More precisely, we define the predictor space as the closure (in the square-mean sense) of the linear space spanned by the coordinates of the predictor variables $\{\z_w(t)\}_{w\in\Sigma^+}$. 
%and we will denote by $E_l[\z \! \mid \! M]$ the orthogonal projection of the components of a vector valued square-integrable random variable $\z$ onto a closed subset $M$ of $\mathcal{H}$.
%In addition, in the sequel we will use the following notation.
\begin{Notation}[Orthogonal projection $E_l$]
 Let $Z$ be a set of $\mathbb{R}^{p}$-valued mean-square integrable
 random variables. Let $\z \in \mathbb{R}^{k}$, $k > 0$,
 be another
 mean-square integrable random variable. We denote by
 \( E_{l}[\z \mid Z] \)
 the orthogonal projection of $\z$ onto the
 subspace $M$, where $M$ is the closure of
 the linear space spanned by the coordinates of the elements
 of $Z$.
% That is, $M$ is the smallest (with respect to set inclusion) closed linear subset such that for any $\z=(\z_1,\ldots,\z_p) \in Z$, $M$ contains the scalar random variables $\z_1,\ldots,\z_p$.
\end{Notation}
One can interpret $E_{l}[\z \mid Z]$ as the best approximation (prediction) 
of $\z$ in terms of (infinite) linear combination of elements of $Z$. 
Next, we define the forward innovation process for $\y$.
\begin{Definition}[Forward innovation]
The \emph{forward innovation process} $\e$ of $\y$ is defined as
\begin{equation}
\label{gen:filt:inneq}
  \e(t)=\y(t)-E_{l}[\y(t) \mid \{ \z_{w}(t) \mid w \in \Sigma^{+} \} ].
 \end{equation}
\end{Definition}
That is, the forward innovation is the difference between the predicted
output and the actual one, if the prediction is based on linear extrapolation of
past outputs. The forward innovation process has all the properties 
required of the noise of a \GBS. Below we define a class of \GBS{s} where
$\e$ is the noise.
%% In fact, when constructing a \GBS\ realization of
%$\y$, $\e$ plays the role of the noise. More precisely, we define the following
%class of \GBS{s}.
% The resulting class of \GBS{s} will be 
%important in the sequel, for this reason we assign it a name.
%We are now ready to define the concept of a \GBS\ in forward innovation form.
%
  \begin{Definition}[\GBS\ in forward innovation form]
%Given an output process $\y(t)$ and a collection of input processes $\{\bu_{\sigma}\}_{\sigma \in \Sigma}$ indexed by a given finite alphabet $\Sigma$, find a forward innovation  state-space realization of $\y(t)$ of the
    Let $\BS$ be \GBS\ of the form \eqref{gen:filt:bil:def}.
    Then $\BS$ is in 
    \emph{forward innovation form}, if $D=I_p$, 
     $\v(t)=\e(t)$ for all $t \in \mathbb{Z}$, and
    $\BS$ satisfies Assumption \ref{gbs:def}.
% equals the
     %\begin{equation}
     %\label{gen:filt:stateeq}
     %\begin{split}
    %  \x(t+1)&=\sum_{\sigma \in \Sigma} (A_{\sigma}\x(t)+
     %  K_{\sigma}\e(t))\bu_{\sigma}(t) \\
      %\y(t)&=C\x(t)+\e(t),
     %\end{split}
     %\end{equation}
%where the equalities are assumed to hold in the square-mean sense.
%In equation \eqref{gen:filt:stateeq}, 
     %$\x(t)$ is a random process taking values in $\Re^n$,
%     the forward innovation $\e(t)$ defined as
%and  for all $\sigma \in \Sigma$ the system matrices are of the form
     %$A_{\sigma} \in \mathbb{R}^{n \times n}$,
     %$K_{\sigma} \in \mathbb{R}^{n \times p}$, and
     %$C \in \mathbb{R}^{p \times n}$.
   \end{Definition}    
  That is, if $\Sigma$ is in forward innovation form, then the noise equals 
  $\e$ and  $C\x(t)$ equals
  the linear projection of $\y(t)$ to the space
  $\{\z_{w}(t) \mid w \in \Sigma^{+}\}$, i.e. $C\x(t)$ is the best linear estimate
  of $\y(t)$ in terms of $\{ \z_{w}(t)  \mid w \in \Sigma^{+}\}$.
  %The noise $\e(t)$ of $\BS$ is the difference
  %between the measured and the predicted output. 
  Moreover, due to Part \ref{gbs:def:prop4}
  of Assumption \ref{gbs:def:prop4}, the state $\x(t)$ of $\Sigma$ belongs to the Hilber-space
  generated by the variables $\{ \z_{w}(t) \mid w \in \Sigma^{+}\}$.  Hence,
  a realization in forward innovation form is its own Kalman-filter, and it can be
  viewed as a system which is driven by the past outputs and inputs.

 As we have mentioned before, for realizability by \GBS, 
 the covariances of the outputs and inputs should form a rational formal power series.
 Below, we define these formal power series.
 \begin{Definition}[Family of formal power series $\Psi_{\y}$]
    For each $j\in I = \{1,\ldots, p\}$, $\sigma \in \Sigma$, 
    define the formal power series 
    $S_{(j,\sigma)} \in \FPS$ 
    as
    %\begin{equation}
    %\label{gen:filt:pow1}
    \( S_{(j,\sigma)}(w)=(\Lambda^{\y}_{\sigma w})_{.,j}, \)
    %\end{equation}
    where $(\Lambda^{\y}_{w\sigma})_{.,j}$ denotes the
    $j$th column of the $p \times p$ covariance matrix
    $\Lambda^{\y}_{w\sigma}=E[\y(t)\z_{w\sigma}^{T}(t)]$.
    Define the family of formal power series
    \begin{equation} 
    \label{gen:filt:pow2}
      \Psi_{\y}=\{ S_{(j,\sigma)} 
        \mid j\in I,  \sigma \in \Sigma\}.
    \end{equation}	
 \end{Definition}
   We can now state the following assumptions which guarantee
   existence of a \GBS\ realization.
\begin{Assumption}%[Output process] 
  \label{output:assumptions}
    The process  $\y$ is \RC\ and the family of $\Psi_{\y}$ is 
    is square summable and rational.
\end{Assumption}
%\item
In addition, we will use the following assumption.
Define the random variables
$\z^f_w(t)$, $w \in \Sigma^{+}$, 
 %$\sigma_1,\ldots, \sigma_k \in \Sigma$, $k > 0$, 
\begin{equation} 
\z^{f}_{w}(t)=\y^{T}(t+|w|)\bu_{w}(t+|w|-1)\frac{1}{\sqrt{p_{w}}}.
\end{equation}
\begin{Assumption} 
\label{output:assumptions:extra}
 For each $w \in \Sigma^{+}$, assume that 
 the variable $\z^f_{w}(t)$ is square integrable.
\end{Assumption}
%Finally, we will need the following technical assumption, which is automatically satisfied
%for a large class of processes. To this end, we 
%\end{Assumption}
\begin{Remark}
\label{output:assumptions:extra:rem}
 In many important cases, Assumption \ref{output:assumptions:extra} is automatically
 satisfied if $\y$ satisfies Assumption \ref{output:assumptions}.
  We present below a number such cases.
 \begin{enumerate}
 \item $\bu_{\sigma}(t)$ is essentially bounded for all $\sigma \in \Sigma$, $t \in \mathrm{Z}_m$, i.e. there exists a constant $K > 0$ such that $|\bu_{\sigma}(t)| \le K$ almost everywhere.
  This is the case when for example $\bu_{\sigma}$ arises from a discrete valued 
  process, as described in Example \ref{disc:input}. 
  Then $E[(\z^f_{w}(t))^T\z^f_{w}(t)] \le E[\y^T(t+k)\y(t+k)] K^{2}\frac{1}{p_w} < +\infty$, $k=|w|$.
  %since
  %by Assumption \ref{output:assumptions} $\y$ is \RC\ and hence it is square integrable.

 \item 
   If $\y(t)$, $\bu_{w}(t)$ have finite fourth order moments, then by H\"olders inequality,
   $E[(\z^f_{w}(t))^T\z^f_{w}(t)] \le (E[(\y^T(t+k)\y(t+k))^2]E[\bu_{w}^{4}(t+k)])^{1/2} < +\infty$,
  $k=|w|$.
% The right-hand side of the inequality above is finite by assumption.
   In particular, this assumption was made in \cite{StochBilin}.
 \end{enumerate}
\end{Remark}
Now we can state the main result on existence of a \GBS\ realization.
\begin{Theorem}[Stochastic realization of \GBS{s}: existence]
\label{gen:filt:theo3}
Assume that $\y$ satisfies Assumption \ref{output:assumptions:extra}.
Then $\y$ has a realization by a  \GBS\ which
satisfies Assumption \ref{gbs:def} if and only if $\y$  satisfies Assumption \ref{output:assumptions}.  Moreover, if $\y$ has a realization by a \GBS\ which satisfies Assumption \ref{gbs:def}, then it has a realization by a \GBS\ in forward innovation form.
\end{Theorem}
Recall that by Remark \ref{output:assumptions:extra:rem}, in many cases
Assumption \ref{output:assumptions:extra} follows from Assumption \ref{gbs:def}.
 %In particular, if $\bu_{\sigma}(t)$ arises from discrete-valued process, the following holds.
\begin{Corollary}
 Assume $\bu_{\sigma}(t)=\chi(\btheta(t)=\sigma)$ where $\btheta(t)$ is a $\Sigma$-valued process with $P(\btheta(t)=\sigma)=p_{\sigma} > 0$. Then $\y$ has a realization by a \GBS\ which
satisfies Assumption \ref{gbs:def} if and only if $\y$ satisfies Assumption \ref{output:assumptions}.
\end{Corollary}
Theorem \ref{gen:filt:theo3} is an easy consequence of 
Theorem \ref{gen:filt:theo3.2} -- \ref{gen:filt:theo3.3} which
will be stated below.
Theorem \ref{gen:filt:theo3.2} implies that the condition of Theorem \ref{gen:filt:theo3}
is necessary for existence of a realization, and Theorem \ref{gen:filt:theo3.3} implies
that this condition is sufficient.
 In order to state Theorems \ref{gen:filt:theo3.2} -- \ref{gen:filt:theo3.3}, we need the following definition.
\begin{Definition}[Full rank process]
\label{output:assumptions:part2}
   We will say that $\y$ is a full rank,
    for each $\sigma \in \Sigma$ the covariance
    $E[\e(t)\e^{T}(t)\bu_{\sigma}^{2}(t)]$ is of
    rank $p$, hence strictly positive definite.
\footnote{Note that the concept of a full rank process already has an established
 definition \cite{Caines1}, which is slightly different from the one used in this paper. In the
 linear case, i.e. when $\Sigma=\{z\}$ and $\bu_{z}=1$, the two definitions coincide.
 Hence, our definition represents a slight abuse of terminology.}
\end{Definition}
 Strictly speaking, the concept of a full rank process is not necessary for
 Theorem \ref{gen:filt:theo3}. However, it plays an important role in formulating
 a realization algorithm. For this reason, we prefer to state 
 Theorems \ref{gen:filt:theo3.2}--\ref{gen:filt:theo3.3} in such a way, that
 the concept of a full rank process is already used.
 %If $\y$ is a full-rank process, then we will be able  
%to obtain a nice expression for the matrix $K_{\sigma}$ and to formulate a 
%realization algorithm.
 
 Next, we relate \GBS{s} and rational representations.
\begin{Definition}[Representation associated with \GBS]
\label{gbs2repr}
 Consider the unique collection of $n \times n$ matrices $\{P_{\sigma}\}_{\sigma \in \Sigma}$ which satisfy \eqref{gen:filt:theo3.2:eq1}.
 Define, the matrices 
\begin{equation} 
 \label{gsb2repr:eq1}
  B_{\sigma}=\frac{1}{\sqrt{p_{\sigma}}}(A_{\sigma}P_{\sigma}C^T+K_{\sigma}Q_{\sigma}D^T).
 \end{equation}
Define the \emph{representation associated with $\BS$} as 
$R_{\BS}=(\mathbb{R}^{n}, \{\sqrt{p_{\sigma}}A_{\sigma}\}_{\sigma \in \Sigma},B,C)$,
where $B=\{ B_{\sigma,j} \mid \sigma \in \Sigma,j=1,\ldots,p\}$ and
$B_{\sigma,j}$ denotes the $j$th column of $B_{\sigma}$.
\end{Definition}
\begin{Theorem}[Necessary condition for existence]
\label{gen:filt:theo3.2}
 If $\BS$ is a realization of $\y$ and $\BS$ satisfies Assumption \ref{gbs:def}, then
 the following holds.
 \begin{itemize}
 \item
 The process $\y$ is \RC.
 \item
 The representation $R_{\BS}$ well defined, stable, 
   and $R_{\BS}$ is a representation of $\Psi_{\y}$.
 \item
  $\y$ satisfies Assumption \ref{output:assumptions}.
 \item 
   If, in addition, for all $\sigma \in \Sigma$, $DE[\v(t)\v^T(t)\bu_{\sigma}^2(t)]D^T > 0$, then $\y$ is full rank.
 \end{itemize}
 %$R_{\BS}$ is a rational
 %representation of the family of formal power series $\Psi_{\y}$. Moreover,
\end{Theorem}
\begin{Remark}
 The definition of $R_{\BS}$ implies that it 
 is completely determined by the matrices 
  \( (C,D,\{A_{\sigma},K_{\sigma},Q_{\sigma})\}_{\sigma \in \Sigma}) \).
 %In fact, the unique solution $\{P_{\sigma}\}_{\sigma \in \Sigma}$
%of \eqref{gen:filt:theo3.2} can be interpreted as 
 %$P_{\sigma}=E[\x(t)\x^T(t)\bu_{\sigma}^{2}(t)]$.
\end{Remark}
The first two statements of Theorem \ref{gen:filt:theo3.2} state that if
$\y$ has a realization by a \GBS\ $\BS$, then $\y$ is \RC\ and $R_{\BS}$ is a stable 
representation of $\Psi_{\y}$. The third statement, i.e.
 that $\y$ satisfies Assumption \ref{output:assumptions}, is an easy corollary of the
 previous ones and Theorem \ref{hscc_pow_stab:theo2}
represent necessary conditions for realizability. 
 Theorem \ref{gen:filt:theo3.2} not only shows that Assumption \ref{output:assumptions}
 represent a sufficient condition, but it also described how to obtain a 
 stable representation of the family of formal power series $\Psi_{\y}$.
The last statement of
Theorem \ref{gen:filt:theo3.2} says that under some mild assumptions the output of
a \GBS\ is full rank. This is important, because it shows that the requirement 
that $\y$ is a full rank process is not an unnatural one. In turn, this assumption
allows us to propose a realization algorithm.

Next, we present the result stating the sufficient condition for existence 
of a realization.
\begin{Theorem}[Sufficient condition for existence]
\label{gen:filt:theo3.3}
If $\y(t)$ satisfies Assumption \ref{output:assumptions},
then it has a  \GBS\ $\BS$ in forward innovation and this \GBS\ $\BS$ can be obtained
from a minimal rational representation of $\Psi_{\y}$ as follows.
%Then, $\y$ has a realization by a generalized bilinear system of the form \eqref{gen:filt:stateeq} if and only if the family of formal power series $\Psi$ from \eqref{gen:filt:pow2} is rational. Furthermore, the generalized bilinear stochastic realization problem has a solution of the form 
%
 \begin{equation}
 \label{gen:filt:eq2}
 \BS\left\{\begin{split}
  \x(t+1)&=\sum_{\sigma \in \Sigma} (\frac{1}{\sqrt{p_{\sigma}}}A_{\sigma} \x(t)
    + K_{\sigma}\e(t))\bu_{\sigma}(t)\\
  \y(t)&=C\x(t)+\e(t)
 \end{split}\right.
 \end{equation}
 where 
 \begin{itemize}
 \item
$R=(\mathbb{R}^{n}, \{A_{\sigma}\}_{\sigma \in \Sigma}, B,C)$,
 $B=\{B_{(i,\sigma)} \in \mathbb{R}^{n} \mid \sigma \in \Sigma, i=1,\ldots,p\}$
 is a minimal representation of $\Psi$.
%\item
%$\e(t)$ is the innovation process defined in 
 %(\ref{gen:filt:inneq}).
\item
 Let $O_R$ the observability matrix of $R$.
% $\O=[(CA_{w})^{T} \mid w \in \Sigma^{*},|w| \le n]^{T}$.
 %Denote by $M(n)$ the number of words of length at most $n$
 %over the alphabet $\Sigma$.
 Define the random variable $Y_{n}(t)$ as
 \begin{equation}
 \label{gen:filt:proof:eq1}
 %Y_{n}(t)=[ \z^{f}_{w}(t) \mid w \in \Sigma^{*}, |w| \le n]^{T}
 Y_{n}(t)=
 \begin{bmatrix} \z^{f}_{v_0}(t) & \ldots & \z^{f}_{v_{M(n-1)}} \end{bmatrix}^T
 \end{equation}
 where $\z^{f}_{\epsilon}(t)=\y^{T}(t)$ and
 for all $w \in \Sigma^{+}$, $\z^f$ is as defined in \eqref{gen:filt:proof:eq2}.
 The variable $Y_{n}(t)$ can be thought of as the products of
 future outputs
 and inputs. Notice that $R$ is observable, hence the matrix $O_R$
 is has a left inverse, which we will denote by $O^{-1}_R$.
 Then the state $\x(t)$ is define as 
%the projection of $O^{-1}_R(Y_{n}(t))$ onto the predictor space
 %\[ \x(t)=E_{l}[O^{-1}_R(Y_{n}(t)) \mid \{ \z_{w}(t) \mid w \in \Sigma^{+}\} ]\]
\item
 For each $\sigma \in \Sigma$,
 \begin{equation}
 \label{gen:filt:eq3.1}
  K_{\sigma}(p_{\sigma}T_{\sigma,\sigma}-CP_{\sigma}C^{T})=(B_{\sigma}\sqrt{p_{\sigma}} - \frac{1}{\sqrt{p_{\sigma}}}A_{\sigma}P_{\sigma}C^{T})
 \end{equation}
 where $P_{\sigma}=E[\x(t)\x^{T}(t)\bu_{\sigma}(t)\bu_{\sigma}(t)]$, and
 \begin{align}
  B_{\sigma} &=
  \begin{bmatrix} B_{(1,\sigma)}, & B_{(2,\sigma)}, & \ldots , & B_{(p,\sigma)} \end{bmatrix} \in \mathbb{R}^{n \times p}.
\end{align}
\end{itemize}
 If, in addition, $\y$ is a full-rank process, then
 $(p_{\sigma}T_{\sigma,\sigma}- CP_{\sigma}C^{T})$ is invertible and
 \begin{equation}
 \label{gen:filt:eq3}
  K_{\sigma}=(B_{\sigma}\sqrt{p_{\sigma}} - \frac{1}{\sqrt{p_{\sigma}}}A_{\sigma}P_{\sigma}C^{T})(p_{\sigma}T_{\sigma,\sigma}- CP_{\sigma}C^{T})^{-1}.
 \end{equation}
  %morever, $P_{\sigma}$, $\sigma \in \Sigma$ are the unique solutions of 
%% \begin{equation}
%% \label{gen:filt:eq32}
%% \begin{split}
%%  & P_{\sigma}=p_{\sigma}(\sum_{\sigma_1 \in \Sigma} \{A_{\sigma_1}P_{\sigma_1}A_{\sigma_1}^T+ \\
%%    & (B_{\sigma_1}\sqrt{p_{\sigma_1}} - \frac{1}{\sqrt{p_{\sigma_1}}}A_{\sigma_1}P_{\sigma_1}C^{T})(p_{\sigma}T_{\sigma_1,\sigma_1}- CP_{\sigma_1}C^{T})^{-1}(B_{\sigma_1}\sqrt{p_{\sigma_1}} - \frac{1}{\sqrt{p_{\sigma_1}}}A_{\sigma_1}P_{\sigma_1}C^{T})^T\})
%% \end{split}
%% \end{equation}
 Moreover, the \GBS\ $\BS$ constructed above satisfies Assumption \ref{gbs:def}.
\end{Theorem}
%
%Using the proof of Theorem \ref{gen:filt:theo3}, we can state the following
%characterization of minimality. To this end, we need to introduce additional
%notation and terminology.
%%Assume that $\Sigma=\{\sigma_1,\ldots,\sigma_{\QNUM}\}$.
%%\begin{Definition}[Observability matrix}
%%Define the \emph{observability matrix $\mathcal{O}_{\BS}$} of $\BS$ as follows: define
%%recursively
%%$\mathcal{O}_0=C$, $\mathcal{O}_{i+1}=\begin{bmatrix} \mathcal{O}_i^{T}, & A_{\sigma_1}^T\mathcal{O}_i^T & \ldots & A_{\sigma_{\QNUM}}^T\mathcal{O}_i^T \end{bmatrix}^T$ and let $\mathcal{O}_{\BS}=\mathcal{O}_{n-1}$.
%%\end{Definition}
%%\begin{Definition}[Controllability matrix]
%%Define the \emph{controllability matrix $\mathcal{C}_{\BS}$ as follows:
%%define $\mathcal{C}_0=\begin{bmatrix} B_{\sigma_1}, & \ldots, & B_{\sigma_{\QNUM}} \end{bmatrix}$, $\mathcal{C}_{i+1}=\begin{bmatrix} C_i &, A_{\sigma_1}C_i, & \ldots, & A_{\sigma_{\QNUM}}C_i \end{bmatrix}$, $i  \in \mathbb{N}$, and let
%%$\mathcal{C}_{\BS}=\mathcal{C}_{n-1}$.
%%\end{Definition}
 \begin{Remark}[Algebraic Ricccati equation]
\label{rem:ricc}
By Theorem \ref{gen:filt:theo3.3}, if $\y$ is full rank, then 
the combination of \eqref{gen:filt:eq3} and 
\eqref{gen:filt:theo3.2:eq1} yields an equation of which $\{P_{\sigma}\}_{\sigma \in \Sigma}$
is a unique solution.
This equation is analogous to the well-known algebraic Riccati equation for linear systems.
%solution of
%which can be used to obtain the gain of the Kalman filter
\end{Remark}
Theorem \ref{gen:filt:theo3.3} not only gives a sufficient condition for existence of a \GBS\ realization, but it serves as a starting point of a realization algorithm.  Moreover, it makes the relationship between realization theory and filtering more precise. 
In particular, Remark \ref{rem:ricc} and Theorem \ref{gen:filt:theo3.3} imply that
the data contained in a rational representation of $\Psi_{\y}$ (i.e. of covariances of
outputs and inputs) contains all the necessary information for constructing a \GBS\
realization of $\y$ in forward innovation form. As it was mentioned before, such a \GBS\
can be viewed as recursive filter for computing the best linear estimates of future
outputs based on past outputs. Together with Theorem \ref{gen:filt:theo3.2} and
Theorem \ref{sect:pow:theo1} it yields an algorithm for computing such a filter from an
arbitrary \GBS\ realization of $\y$: we first compute the representation $R_{\BS}$ 
associated with a \GBS\ realization $\BS$  of $\y$, then we use Theorem \ref{gen:filt:theo3.3}
to obtain a \GBS\ in forward innovation form.

Theorem \ref{gen:filt:theo3.2} -- \ref{gen:filt:theo3.3} imply the following
characterization of minimality.
\begin{Definition}[Minimality]
A \GBS\ $\BS$ which satisfies Assumption \ref{gbs:def} 
 is said to be a \emph{minimal realization} of $\y(t)$ if it realizes $\y(t)$ and it has the minimal dimension among all possible\GBS\ realizations of $\y(t)$
 which satisfy Assumption \ref{gbs:def}.
\end{Definition}
\begin{Theorem}[Minimality of \GBS{s}]
\label{gen:filt:min_theo}
 Assume $\BS$ is a \GBS\ which satisfies Assumption \ref{gbs:def} and which is a realization of $\y$.
 The \GBS\ $\BS$ is minimal if and only if $R_{\BS}$ is minimal. If the
 \GBS{s} $\BS_1$ and $\BS_2$ are both minimal realizations of $\y$ and they both satisfy Assumption \ref{gbs:def}, then
 $R_{\BS_1}$ and $R_{\BS_2}$ are isomorphic.
\end{Theorem}
\begin{Remark}
 The isomorphism of $R_{\BS_1}$ and $R_{\BS_2}$ can be directly translated into
 a relationship between the matrices of $\BS_1$ and $\BS_2$. 
 If $\BS_1$ is of the form \eqref{gen:filt:bil:def} and
 the corresponding matrices of $\BS_2$  are $\hat{A}_{\sigma}$,
 $\hat{K}_{\sigma}$ and $\hat{C}$ and $\hat{D}$, then isomorphism of
 $R_{\BS_1}$ and $R_{\BS_2}$ implies that
%%     \[ \BS_2 \left \{\begin{split}
%%      \hat{\x}(t+1)&=\sum_{\sigma \in \Sigma} (\hat{A}_{\sigma}\hat{\x}(t)+
%%       \hat{K}_{\sigma}\hat{\v}(t))\bu_{\sigma}(t) \\
%%      \y(t)&=\hat{C}\hat{\x}(t)+\hat{D}\v(t),
%%     \end{split}\right.,
%%     \]
there exists a non-singular matrix $\mathrm{S} \in \mathbb{R}^{n \times n}$
such that
\( C\mathrm{S}^{-1}=\hat{C}, \forall \sigma \in \Sigma:
   \mathrm{S}A_{\sigma}\mathrm{S}^{-1}=\hat{A}_{\sigma}.
   %\mathrm{S}B_{\sigma}=\hat{B}_{\sigma}
\)
%where
%\( \hat{B}_{\sigma}=\frac{1}{\sqrt{p_{\sigma}}}(\hat{A}_{\sigma}\hat{P}_{\sigma}\hat{C}+\hat{K}_{\sigma}\hat{Q}_{\sigma}\hat{D}^T \), 
%with $\hat{P}_{\sigma}=E[\hat{\x}(t)\hat{\x}(t)\bu_{\sigma}^2(t)]$, $\hat{Q}_{\sigma}=E[\hat{\v}(t)\hat{\v}(t)\bu_{\sigma}^2(t)]$, $\sigma \in \Sigma$.
Note that we do not claim that $SK_{\sigma}=\hat{K}_{\sigma}$, $\sigma \in \Sigma$ or
that $D=\hat{D}$. In fact, in general it will not be true. 
\end{Remark}
\begin{Remark}[Checking minimality]
 From Theorem \ref{gen:filt:theo3.2} it follows that $R_{\BS}$ can be computed bases solely
 on the matrices of $\BS$ and the covariance of the noise. From Theorem \ref{sect:pow:theo1} it follows
 that minimality of $R_{\BS}$ can be checked effectively, by checking if $R_{\BS}$ is
 reachable and observable. Hence, minimality of a \GBS\ can be checked effectively, based
 on the knowledge of the matrices $(C,D,\{A_{\sigma},K_{\sigma},Q_{\sigma})\}_{\sigma \in \Sigma})$
\end{Remark}
%Theorem \ref{gen:filt:theo3.2} and Theorem \ref{gen:filt:theo3.3} together yield the
%following corollary.
%\begin{Corollary}
%\label{gen:filt:col1}
%If $\y$ has a realization by a \GBS\, then it has a realization by a \GBS\ in forward
%innovation form.
%\end{Corollary}
 %The corollary above yields the following interpretation: if $\y$ can be realized by a 
%\GBS, then one can construct a recursive filter for $\y$.

\subsection{Realization theory for subclasses of \GBS{s}}
\label{gbs:prob:subclass}
 We have argued before that \GBS{s} include a large number of system classes
 such as linear, bilinear stochastic systems and even jump-markov linear systems. 
 However, the solution of the realization problem for \GBS{s} does not directly
 yield solutions to the realization problems for those system classes. The reason
 for this is quite obvious: while the necessary conditions remain valid for
 subclasses of \GBS{s}, the sufficient conditions need not remain valid. After all,
 it could easily happen that even if $\y$ has a realization by a \GBS\ belonging
 to a certain subclass, the realization prescribed by Theorem \ref{gen:filt:theo3.3}
 does not fall into that subclass. Nevertheless, the results obtained for
 general \GBS{s} can be used to solve the realization for the various sub-classes
 of \GBS{s} described above. Below we will discuss this topic in more detail.

 We start with specializing the results to \GBS{s} described in Example 
 \ref{example:iid:gbs_def}. We will call such \GBS{s} \emph{\GBS{s} with 
 i.i.d. inputs}. We will show that the following conditions are necessary
 and sufficient for existence of a \GBS\ realization with i.i.d inputs.
\begin{Assumption}
\label{iid:assumption}
 \begin{enumerate}
 \item
  $\{\y(t),\z_{w}(t) \mid w \in \Sigma^{+}\}$ is zero-mean,
 wide-sense stationary.
 \item The $\sigma$-algebras generated by respectively $\{\y(t-l)\}_{l=0}^{\infty}$ and $\{\bu_{\sigma}(t+l)\}_{l=0}^{\infty}$, $\sigma \in \Sigma$ are independent.
 \item
   The family $\Psi_{\y}$ is square summable and rational.
 \end{enumerate}
\end{Assumption}
 We obtain the following corollary of Theorem \ref{gen:filt:theo3}.
 \begin{Corollary}
 \label{iid:col1}
 A process $\y$ has a realization by a \GBS\ with i.i.d input if and only if 
 $\y$ satisfies Assumption \ref{iid:assumption}.
 If $\y$ satisfies Assumption \ref{iid:assumption}, then the \GBS\ realization of
 $\y$ described in Theorem \ref{gen:filt:theo3.3} is a \GBS\ with i.i.d input.
 \end{Corollary}
 Indeed, if $\y$ satisfies Assumption \ref{iid:assumption}, then 
 $\y$ is \RC\ and hence it satisfies Assumption \ref{output:assumptions}.
 Hence, by Theorem \ref{gen:filt:theo3.3}, Assumption \ref{iid:assumption}
 implies existence of a \GBS\ realization $\BS$ of $\y$ in forward innovation form.
 The noise process of this \GBS\ is then the innovation process $\e(t)$. By
 Lemma \ref{lemma:measurability}, since
 the coordinates of $\e(t)$ belong to the Hilbert space generated by $\{\y(t), \z_{w}(t) \mid w \in \Sigma^{+}\}$,
 it is measurable w.r.t to the $\sigma$-algebra generated by
 $\{\y(t-l),\bu(t-l-1)\}_{l=0}^{\infty}$. The latter $\sigma$-algebra is 
  independent of the $\sigma$-algebra generated by
 $\{\bu_{\sigma}(t+l)\}_{l=0}^{\infty}$, since $\bu$ is an i.i.d process and 
  $\y$ satisfies Assumption \ref{iid:assumption}.
 Hence, the $\sigma$-algebras generated by $\{\e(t-l)\}_{l=0}^{\infty}$ and
 $\{\bu_{\sigma}(t+l)\}_{l=0}^{\infty}$ are independent. Hence,
 $\BS$ is a \GBS\ with i.i.d inputs.
 Conversely, if $\y$ has a realization by a \GBS\ with i.i.d inputs, then
 by Theorem \ref{gen:filt:theo3.2} 
 $\y$ satisfies Assumption \ref{output:assumptions}. Moreover, since $\x(t)$ and
 hence $\y(t)$ belongs to the Hilbert-space generated by
 $\{\v(t),\z_{w}^{\v}(t) \mid w \in \Sigma^{+}\}$ and the latter variables are
 independent of $\bu_{\sigma}(t+l)$, $l \ge 0$, from Lemma 
 \ref{lemma:measurability} it follows that the $\sigma$-algebras
 $\{\y(t-l)\}_{l=0}^{\infty}$ and $\{\bu_{\sigma}(t+l)\}_{l=0}^{\infty}$,
 $\sigma \in \Sigma$ are independent. Hence, $\y$ satisfies Assumption 
 \ref{iid:assumption}.
 %The argument above also yields the following corollary of Theorem 
 %\ref{gen:filt:min_theo}.
 \begin{Corollary}
 \label{iid:col2}
  Theorem \ref{gen:filt:min_theo} remains valid if we replace \GBS{s} by 
  \GBS{s} with i.i.d inputs.
 \end{Corollary}
  Indeed, from Theorem \ref{gen:filt:min_theo} it follows that a reachable and
  observable \GBS\ with i.i.d inputs is minimal. Conversely, by Theorem
  \ref{gen:filt:min_theo} the \GBS\ realization of $\y$ 
  described by Theorem \ref{gen:filt:theo3.3} is minimal, and by Corollary
  \ref{iid:col1} it implies that if a \GBS\ with i.i.d inputs which has the minimal
  dimension among all the \GBS{s} with i.i.d. inputs, then it has the smallest
  possible dimension among all the \GBS{s} realizations of $\y$.  Hence, minimal
  \GBS{s} with i.i.d inputs are also reachable and observable. Moreover, 
  there is a minimal \GBS\ realization of $\y$ with i.i.d inputs.  Finally,
  isomorphism of minimal \GBS\ realizations with i.i.d. inputs follows directly
  from Theorem \ref{gen:filt:min_theo}.

  Recall the linear systems (Example \ref{example:lin:gbs_def}), 
  bilinear stochastic systems (Example \ref{example:bilin:gbs_def}) and jump-markov 
  linear systems with i.i.d discrete state (Example \ref{example:jumplin:gbs_def}) arise from \GBS\ with i.i.d inputs by a specific choice of
  the input process $\bu_{\sigma}$, $\sigma$.  
  If we apply Assumption \ref{iid:assumption} to the case of linear Gaussian systems,
  then we obtain the classical results on realization theory of linear systems.
  Notice that the last part of Assumption \ref{iid:assumption}, when applied
  to the linear case, reduces to requiring that the power spectrum is stable and
  rational. If we apply Assumption \ref{iid:assumption} to bilinear stochastic
  systems, then we obtain the conditions of \cite{StochBilin,FrazhoBilin}.
  Note that in \cite{StochBilin} only the sufficiency of the condition was shown,
  not the necessity. Furthermore, \cite{FrazhoBilin} deals with weak realization (see
  Definition \ref{weak:real:def}) and it assumes that the output equation does not 
  contain a noise term. 
  If we specialize Corollary \ref{iid:col1}--\ref{iid:col2} to jump-markov
  linear systems with i.i.d state process we obtain the following results.
  We will call the \GBS\ of the type described in Example \ref{example:jumplin:gbs_def} \emph{jump-markov linear systems with i.i.d switching (abbreviated by JMLSIID)}.
  \begin{Corollary}[Realization of JMLSIID]
  \label{iid:col3}
   The process $\y$ has a realization by JMLSIID if and only if the following 
   conditions hold:
   \begin{enumerate}
    \item
  $\{\y(t),\z_{w}(t) \mid w \in \Sigma^{+}\}$ is zero-mean,
  wide-sense stationary,
   \item 
     the $\sigma$-algebras generated by 
     $\{\y(t-l)\}_{l=0}^{\infty}$ and $\{\btheta(t+l)\}_{l=0}^{\infty}$ are independent, 
   \item
    the family $\Psi_{\y}$ is square summable and rational.
   \end{enumerate}
   If $\y$ satisfies the conditions above, then it has a minimal JMLSIID realization in
   forward innovation form described in Theorem \ref{gen:filt:theo3.3}.
   Moreover, Theorem \ref{gen:filt:min_theo} holds if we replace \GBS{s} by
   JMLSIID.
  \end{Corollary} 
  To the best of our knowledge, Corollary \ref{iid:col3} represents a new result.
  That is, the framework of \GBS{s} not only extends existing results on
  bilinear stochastic systems, but also yields, as a special case, new results on
  a completely different system class. 

  Finally, we show how the results above specialize to the case of 
  \emph{jump-markov linear systems with restricted switching (abbreviated by JMLSRS)},
  described in Example \ref{example:const:gbs_def}. 
  \begin{Assumption}
  \label{const:assumption}
   \begin{enumerate}
    \item
  $\{\y(t),\z_{w}(t) \mid w \in \Sigma^{+}\}$ is jointly zero-mean,
  wide-sense stationary,
   \item  the $\sigma$-algebras generated by
  $\{\y(t-l)\}_{l=0}^{\infty}$ and $\{\btheta(t+l)\}_{l=0}^{\infty}$ are conditionally independent w.r.t to the $\sigma$-algebra $\mathcal{D}_t$ generated by
  $\{\btheta(t-l-1)\}_{l=0}^{\infty}$
   \item
    The family $\Psi_{\y}$ is square summable and rational.
   \end{enumerate}
  \end{Assumption}
  \begin{Corollary}[Realization of JMLSRS]
  \label{const:col1}
   The process $\y$ has a realization by a JMLSRS if and only if it
   satisfies Assumption \ref{const:assumption}.
   If $\y$ satisfies Assumption \ref{const:assumption}, 
   then it has a minimal JMLSRS realization in
   forward innovation form described in Theorem \ref{gen:filt:theo3.3}.
   Moreover, Theorem \ref{gen:filt:min_theo} holds if we replace \GBS{s} by
   JMLSRS.
  \end{Corollary} 
  The proof of this corollary is similar to the proof of Corollary \ref{iid:col1}.
  First, if $\y$ has a realization by a JMLSR, then, since a JMLSR is a \GBS\ 
  satisfying Assumption \ref{gbs:def}, $\y$ satisfies Assumption \ref{output:assumptions}.  Moreover, $\y(t)$ belongs to the Hilbert-space generated by
  $\{\v(t),\z_w^{\v}(t) \mid v \in \Sigma^{+}\}$, where $\v$ is the noise process of a
  JMLSR realization. From Lemma \ref{lemma:measurability} it then follows
 that  $\y(t)$ is measurable with respect to the
  $\sigma$-algebra generated by $\{\v(t-l),\btheta(t-l-1)\}_{l=0}^{\infty}$.
  By the definition of JMLSRS and the well-known properties of conditional independence, $\sigma$-algebras
  generated by $\{\y(t-l)\}_{l=0}^{\infty}$ and $\{\btheta(t+l)\}_{l=0}^{\infty}$ 
  are conditionally independent w.r.t.
  $\mathcal{D}_t$.  This, together with Assumption \ref{output:assumptions}
  implies that $\y$ satisfies Assumption \ref{const:assumption}.
  Conversely, Assumption \ref{const:assumption} implies Assumption \ref{output:assumptions}. Then there exists a minimal \GBS\ realization $\BS$ of $\y$ in forward 
  innovation form. The noise process is then the innovation process $\e$ and 
  $\e(t)$ belongs to the Hilbert-space generated by 
  $\{\y(t),\z_{w}(t) \mid w \in \Sigma^{+}\}$. Using Lemma 
  \ref{lemma:measurability} it then follows that 
  $\e(t)$ is measurable w.r.t. to the $\sigma$-algebra generated by
  $\{\y(t-l),\btheta(t-l-1)\}_{l=0}^{\infty}$. The latter $\sigma$-algebra
  and the $\sigma$-algebra generated by $\{\btheta(t+l)\}_{l=0}^{\infty}$ are
  conditionally independent w.r.t to $\mathcal{D}_t$ by Assumption \ref{const:assumption}. Hence, the $\sigma$-algebras generated by $\{\e(t-l)\}_{l=0}^{\infty}$ and
   $\{\btheta(t+l)\}_{l=0}^{\infty}$ are conditionally independent w.r.t. $\mathcal{D}_t$. That is, $\BS$ is a JMLSRS and it is a minimal one among all the \GBS\ realizations.
   Hence, if a JMSRS is minimal among all the JMLSRS realizations of $\y$,
   then it is minimal among all the \GBS\ realizations of $\y$. Then the last part
   of Corollary \ref{const:col1} is a direct consequence of Theorem 
   \ref{gen:filt:min_theo}.

  The result of Corollary \ref{const:col1} is new, to the best of our knowledge. 
  This result is another proof of versatility of the \GBS\ framework.

 \subsection{Weak realization and realization algorithms}
 \label{gbs:prob:alg}
  Below we present a realization algorithm for \GBS{s}. We only state the algorithm and
  the related results, the proofs are presented in \S \ref{gbs:prob:proof}.
  Theorem \ref{gen:filt:theo3.3} proposes a procedure for construction a \GBS\ realization
  of $\y$ using the knowledge of $\y$ and $\{\bu_{\sigma}\}_{\sigma \in \Sigma}$. In this
  construction, the noise and the state processes are constructed explicitly using 
  $\y$ and $\{\bu_{\sigma}\}_{\sigma \in \Sigma}$.
  Unfortunately, this procedure is not effective. In fact,  it cannot be made effective,
  since it presumes the knowledge of stochastic processes $\y$ and $\{\bu_{\sigma}\}_{\sigma \in \Sigma}$. The latter objects cannot be represented by finite number of data points.
  Note however, that for many application the knowledge of the state or noise
  process is not required, instead it is sufficient to know the matrices of
  the \GBS\ and covariance of the state process. These matrices can be approximated
  from finitely many data points. This prompts us to introduce the notion of a 
  \emph{weak realization}.
  \begin{Definition}[Weak realization]
  \label{weak:real:def}
  A collection $(\{A_{\sigma},K_{\sigma},P_{\sigma}, Q_{\sigma}\}_{\sigma \in \Sigma}, C,D)$,
  where $A_{\sigma},P_{\sigma} \in \mathbb{R}^{n \times n}$, 
  $K_{\sigma} \in \mathbb{R}^{n \times m}, Q_{\sigma} \in \mathbb{R}^{m \times m}$, 
  $\sigma \in \Sigma$, 
  $C \in \mathbb{R}^{p \times n}$, $D \in \mathbb{R}^{p \times m}$, 
  is called a \emph{weak realization} of $\y$, if there exists a \GBS\  $\Sigma$ of 
  the form \eqref{gen:filt:bil:def}, such that
  $\Sigma$ is a realization of $\y$ and $\Sigma$ satisfies Assumption \ref{gbs:def} and
  $E[\x(t)\x^T(t)\bu_{\sigma}^2(t)]=P_{\sigma}$, $E[\v(t)\v^T(t)\bu_{\sigma}^2(t)]=Q_{\sigma}$,  $\sigma \in \Sigma$.
  The data $(\{A_{\sigma},K_{\sigma},P_{\sigma}\}_{\sigma \in \Sigma}, C,D)$
  is said to be a weak realization of $\y$ in forward innovation form, if 
  the \GBS\ $\Sigma$ above is a realization of $\y$ in forward innovation form.

  By a slight abuse of notation, we will identify $\Sigma$ with the data 
  $(\{A_{\sigma},K_{\sigma},P_{\sigma},Q_{\sigma}\}_{\sigma \in \Sigma}, C,D)$ and write
  $\Sigma=(\{A_{\sigma},K_{\sigma},P_{\sigma},Q_{\sigma}\}_{\sigma \in \Sigma}, C,D)$.
  \end{Definition}
  %A weak realization is thus a collection of the relevant matrices of
  %a \GBS\ realization,
  That is, a \GBS\ $\Sigma$ is said to be a weak realization of 
  $\y$, if there exists a \GBS\ realization of $\y$ with the same matrices, state and
  noise covariance as those of $\Sigma$.
 It turns out that the construction of Theorem \ref{gen:filt:theo3.3} can be used to compute a 
  weak realization of
  $\y$ from finite data.

  As the first step, we construct approximations of 
  the state and noise processes from Theorem \ref{gen:filt:theo3.3}
  based on finitely many random variables.  
  %In this section we show that how to approximate $\x(t)$ by computing the projection of
  %$O_{R}^{-1}(Y_{n}(t))$ to a space generated by finitely many predictors $\z_w(t)$.
  %This will enable us to propose a realization algorithm which returns an approximation
  %of $\x(t)$ in terms of the variables $\z_{w}(t)$. In addition, it will also allow us to
  %compute an approximation of the parameters $P_{\sigma}$ and $K_{\sigma}$.
   More precisely, 
   we define a sequence of candidate state-variables $\x_N(t)$ and
   noise variables $\e_N(t)$ as
  % projecting future
  % outputs and inputs (\ie $O_R^{-1}(Y_n(t))$) onto the space
  % generated by finite past. More precisely,
  % denote by $\Sigma^{N}$ the set of all words in $\Sigma^{*}$ of  
  % length at most $N$. Define
 \[ 
   \begin{split}
    & \x_{N}(t)=E_{l}[O^{-1}_R(Y_{n}(t)) \mid \{ \z_{w}(t) \mid w \in \Sigma^{N}\} ] \\
   & \e_{N}(t)=\y(t)-E_{l}[\y(t) \mid \{ \z_{w}(t) \mid w \in \Sigma^{N} \}
   \end{split}
\]
   Recall that $\Sigma^{N}=\{w \in \Sigma^{+} \mid |w| \le N\}$.
   %The variables $\z_{w}(t)$, $w \in \Sigma^{+}$ represent the past
   %up to time $t-N$.
   Recall that the original construction of $\x(t)$ and $\e(t)$ 
  the projection of future outputs to the space generated by infinitely many past outputs and inputs. In contrast, $\x_N(t)$ and $\e_N(t)$ determined by projections of
  future outputs to finitely many past outputs and inputs.
   Intuitively, $\x_N(t)$ and $\e_N(t)$ are approximations
   of $\x(t)$ and $\e(t)$ respectively.  In fact, the following result holds.
   \begin{Lemma}
   \label{gen:filt:proof:new:lemma1}
   $\lim_{N \rightarrow \infty} \x_N(t)=\x(t)$, $\lim_{N \rightarrow \infty} \e_N(t)=\e(t)$, $\lim_{N \rightarrow \infty} \x_N(t)\bu_{\sigma}(t)=\x(t)\bu_{\sigma}(t)$,
   and $\lim_{N \rightarrow \infty} \e_N(t)\bu_{\sigma}(t)=\e(t)\bu_{\sigma}(t)$.
   \end{Lemma}
   It turns out that an analog of \eqref{gen:filt:eq2} holds for $\x_N$.
   %we show that the following holds.
   \begin{Lemma}
   \label{gen:filt:proof:new:lemma2}
    There exist $n \times p$ matrices $K_{\sigma}^N$, $\sigma \in \Sigma$
    such that 
   \begin{equation}
   \label{gen:filt:proof:eq2}
   \begin{split}
     & \x_{N+1}(t+1)=\sum_{\sigma \in \Sigma}
    (\frac{1}{\sqrt{p_{\sigma}}}A_{\sigma}\x_{N}(t)+
     K^{N}_{\sigma}\e_{N}(t))\bu_{\sigma}(t) \\
     & \y(t)=C\x_N(t)+\e_N(t).
   \end{split}
 \end{equation} 
  If $P_{\sigma}^{N}=E[\x_N(t)\x_{N}(t)\bu_{\sigma}^{2}(t)]$, 
  and $(p_{\sigma}T_{\sigma,\sigma}-CP_{\sigma}^NC^T)$ is invertible,
  then $K_{\sigma}^{N}$
  \begin{equation}  
  \label{gen:filt:proof:lemma4:eq4} 
   K_N=(\sqrt{p_{\sigma}}B_{\sigma} - \frac{1}{\sqrt{p_{\sigma}}}A_{\sigma}P_{\sigma}^NC^T)(p_{\sigma}T_{\sigma,\sigma}-CP_{\sigma}^NC^T)^{-1}.  
\end{equation}
%  can be chosen as any solution of the following equation
%  \begin{equation}
%    \label{gen:filt:proof:lemma5:eq4.1}
%     K^{N}_{\sigma}(p_{\sigma}T_{\sigma,\sigma}-CP_{\sigma}^NC^T)=
%     (\sqrt{p_{\sigma}}B_{\sigma} - \frac{1}{\sqrt{p_{\sigma}}}A_{\sigma}P_{\sigma}^NC^T).
%   \end{equation}
   %Any matrix $K$ which satisfies 
   %can be taken as $K_{\sigma}^N$
\end{Lemma}
In fact, we will show later on that
\( P_{\sigma}=\lim_{N \rightarrow \infty} P_{\sigma}^{N} \) and
\( K_{\sigma}=\lim_{N \rightarrow \infty} K_{\sigma}^N \). Hence,
 if we know $P_{\sigma}^N$ and $K_{\sigma}^N$, then 
 Lemma  \ref{gen:filt:proof:new:lemma2} yields an approximation of the
 weak \GBS\ realization described in Theorem \ref{gen:filt:theo3.3}.
%% If we use  Lemma \ref{gen:filt:proof:new:lemma1}  and we take the limit
%% of both sides of \eqref{gen:filt:proof:eq2} and \eqref{gen:filt:proof:lemma5:eq4.1}, 
%% we then obtain 
%% \begin{equation}
%% \label{gen:filt:proof:new:eq2}
%%    P_{\sigma}=\lim_{N \rightarrow \infty} P_{\sigma}^{N}
%% \end{equation}
%%  Then $p_{\sigma}T_{\sigma,\sigma}-CP_{\sigma}^NC^T$ converges to
%%  $p_{\sigma}T_{\sigma,\sigma}-CP_{\sigma}C^T=E[\e(t)\e^T(t)\bu_{\sigma}^2(t)]$.
%%  If $\y$ is full-rank, then the latter
%%  is positive definite, and hence 
%%  $p_{\sigma}T_{\sigma,\sigma}-CP_{\sigma}^NC^T$ is invertable for large enough
%%  $N$. Thus, if $\y$ is full-rank, then
%%  for large enough $N$.
%%  By taking the limit of the above expression, it follows that
%%  \begin{equation}
%% \label{gen:filt:proof:new:eq21}
%%      K_{\sigma}=\lim_{N \rightarrow \infty} K_{\sigma}^N .
%%  \end{equation}

  The computation of $P_{\sigma}^N$ and $K_{\sigma}^{N}$
  requires the knowledge of the random variables
  $\{\z_{w}(t) \mid w \in \Sigma^{N}\}$. In practice, however, one 
  has only data, i.e. samples of the random variables
  $\{\z_{w}(t) \mid w \in \Sigma^{N}\}$.  Below we present a formula on approximating
  $P^N_{\sigma}$ (and hence $K^N_{\sigma}$) from such a sample.
%%In addition, notice that $\x_N(t)$ can be expressed as a linear combination
%%of the entries of $\z_{w}(t)$, $w \in \Sigma^{N}$. More precisely,
%%  In order to prove \eqref{gen:filt:proof:eq2} we need the
%%  following notation.
%%  enumeration of $L$ based on lexicographic ordering.
%%  % If $w \in L$, then by Assumption \ref{gen:filt:assum:cond1} 
%%%any prefix of it is in $L$, hence $\NL(N) \le \NL(N+1)$ and $\{\wu_1,\ldots,\wu_{\NL(N)}\}=L^N$.
    To this end, notice that 
    $\x_N(t)$ belongs to the space spanned by the entries of
  $\{\z_{w}(t) \mid w \in \Sigma^{N}\}$.
  Recall that $M(N)=|\Sigma^N|$.
  and $v_1 \prec v_2 \prec \cdots \prec v_{M(N)}$ 
  is an enumeration of $\Sigma^{N}$ based
  on lexicographic ordering. Then there exists
  $\alpha^{N} \in \mathbb{R}^{n \times pM(N)}$,
  %$\alpha^{N}_1,\ldots,\alpha^{N}_{M(N)} \in \mathbb{R}^{n \times p}$
  such that 
%the following holds.If we denote
\begin{equation}
 \label{gen:filt:proof:lemma4:eq1.1}
  \x_N(t)=%\sum_{i=1}^{\NL(N)} \alpha^N_i \z_{\wu_i}(t)=
          \alpha^N \BZ_N(t),
 \end{equation}
 where
 \( \BZ_N(t)=\begin{bmatrix} \z_{v_1}^T(t) & \ldots & \z_{v_{M(N)}}^T(t) \end{bmatrix}^T \in \mathbb{R}^{pM(N) \times 1} \).
 %From the well-known property of orthogonal projections, it follows
 %that $\alpha_N$ satisfies
 %\begin{equation}
 %\label{gen:filt:proof:lemma4:eq1.3}
%    \alpha^{N} T_N = \widetilde{\Lambda}_N
 %\end{equation}
%%%% \textbf{How shall we do this ?}
%%%%  Note that in general $T_N$ is not invertable. However,
%%%%  let's view $T_N$ as a linear map  $\mathbb{R}^{1 \times pM(N)} \ni x  \mapsto xT_N \in \mathbb{R}^{1 \times pM(N)}$ and we restrict it to
%%%%  the subspace $X$ of $\mathbb{R}^{1 \times pM(N)}$ spanned by
%%%%  vectors $\gamma=(\gamma_1,\ldots,\gamma_{M(N)})$ with
%%%%  the property that for all $i=1,\ldots,M(N)$ for which
%%%%  $v_i \notin L$, $\gamma_i=0$. Then the image of $T(N)$ is $X$,
%%%%  since $T_{v_i,v_j}=0$ for all $v_j \notin L$. Note that
%%%%  $T_N$ is invertable on $X$ and
%%%%  its inverse can be represented by a $pM(N) \times pM(N)$ matrix $E$,
%%%%  such that if $X_{c}$ is the complement of $X$, i.e. 
%%%%  $X\oplus X_c=\mathbb{R}^{1 \times pM(N)}$, then $X_cE=0$.
%%%%  We mean by $T_N^{-1}$ such a matrix representation $E$ 
%%%%  of that inverse.
%% %The validity of \eqref{gen:filt:proof:lemma4:eq1.3} is shown in
%% %Lemma \ref{gen:filt:proof:lemma4.1}. 
If we  define
 %the matrices  $T_N \in \mathbb{R}^{pM(N) \times pM(N)}$ and
 %$\widetilde{\Lambda}^N \in \mathbb{R}^{n \times pM(N)}$ as follows:
 \[ T_N=E[\BZ_N(t)\BZ_{N}^T(t)] \mbox{ and } \widetilde{\Lambda}_N=E[O_R^{-1}(Y_{n}(t+1))\BZ_N^T(t)], \]
 %\[ T_N=\begin{bmatrix}
 %        T_{\wu_1,\wu_1} & T_{\wu_1,\wu_2} & \ldots & T_{\wu_1,\wu_{\NL(N)}} \\
 %        T_{\wu_2,\wu_1} & T_{\wu_2,\wu_2} & \ldots & T_{\wu_2,\wu_{\NL(N)}} \\
 %        \vdots      & \vdots      & \ldots & \vdots \\
 %        T_{\wu_{\NL(N)},\wu_1} & T_{\wu_{\NL(N)},\wu_2} & \ldots & T_{\wu_{\NL(N)},\wu_{\NL(N)}} \\
 %       \end{bmatrix}
 %       \mbox{, \ \ \ }
 %   \widetilde{\Lambda}_N=\begin{bmatrix}
 %                                     \widetilde{\Lambda}_{\wu_1}^T \\ \widetilde{\Lambda}_{\wu_2}^T \\
% \vdots \\ \widetilde{\Lambda}_{\wu_{\NL(N)}}^T \end{bmatrix}^T
 %\]
%  where $\widetilde{\Lambda}_{w} \in \mathbb{R}^{n \times p}$,
 %$w \in \Sigma^{+}$ is defined as follows; if $w=\sigma v$
 %for some $\sigma \in \Sigma$, $v \in \Sigma^{*}$, then
 %\[ \widetilde{\Lambda}_{w} = A_{v}B_{\sigma} \]
 %where
 %\( B_{\sigma}=\begin{bmatrix} B_{1,\sigma} & \ldots & B_{p,\sigma} \end{bmatrix} \).
  %Here we used the lexicographic ordering of $\Sigma^{*}$.
  %Note that in general \eqref{gen:filt:proof:lemma4:eq1.3} may have several solutions and
  %any $\alpha$ which satisfies $\alpha T_N=\widetilde{\Lambda}_N$ can be taken as
  %$\alpha_N$. If $T_N$ is invertable, then \eqref{gen:filt:proof:lemma4:eq1.3} has
  %a unique solution and $\alpha_N$ is uniquely defined.
  then by the well-known properties of orthogonal projection,
  $\alpha_N$ is determined by $\widetilde{\Lambda}_N$ and $T_N$. In fact, 
  if $T_N$ is invertable, then
  \( \alpha_N=\widetilde{\Lambda}_N T_{N}^{-1} \).
  %By the lemma below, this is always the case if $\y$ is full rank.
  %However, Lemma \ref{gen:filt:proof:new:lemma1} yields the following.
 From \eqref{gen:filt:proof:lemma4:eq1.1} and the assumption that $\y$ is \RC\ 
 it then follows that
 %\[ P_{\sigma}=\alpha_NE[\BZ_N(t)\BZ_N^T(t)\bu_{\sigma}^{2}(t)]\alpha_N^T
 %
%. \]
 %From the definition of $\Z_N(t)$ and the assumption that $\y$ is an \RC\ process,
 %it follows that
 %$E[\Z_N(t)\Z_N^T(t)\bu_{\sigma}^{2}(t)]=p_{\sigma}T_ND$, where 
 \begin{equation}
 \label{gen:filt:proof:lemma4:eq1.41}
    P_{\sigma}^{N}=p_{\sigma}\alpha_NT_ND\alpha_N^T.
 \end{equation}
 where $D$ is a diagonal matrix such that the $i$th diagonal entry $D_ii$ is $1$ if
 $u_{i}\sigma \in L$ or it is zero otherwise.  
 It then follows that the knowledge of $\widetilde{\Lambda}_N$ and $T_N$ yields $\alpha_N$ and $P_{\sigma}^N$. 
%The latter can be viewed as an approximation of $P_{\sigma}$ and
%  it converges to $P_{\sigma}$ as $N \rightarrow \infty$.
  Note that $\widetilde{\Lambda}_N$ can be computed from a
  minimal representation $R$ of $\Psi_{\y}$ as follows:
  $\widetilde{\Lambda}_N=\begin{bmatrix} \widetilde{\Lambda}_{\wu1} & \ldots & \widetilde{\Lambda}_{\wu_{M(N)}} \end{bmatrix}$, where $\widetilde{\Lambda}_{\sigma v}=A_vB_{\sigma}$ 
  with 
 \( B_{\sigma}=\begin{bmatrix} B_{1,\sigma} & \ldots & B_{p,\sigma} \end{bmatrix} \), 
  for all $v \in \Sigma^{*}$, $\sigma \in \Sigma$.

The discussion above yields the realization algorithm
presented in Algorithm \ref{alg2}.
%Here we  $\{\y(t),\bu_{w}(t) \mid w \in \Sigma^{+}\}$ are
%jointly ergodic, then we can replace stochastic
%processes with time series.
In Algorithm \ref{alg2} we assume that we measure the finite time series
$\{y(t),u_{\sigma}(t) \mid \sigma \in \Sigma, t=0,\ldots, N+M\}$ for some $N, M \ge 0$
and that we have a $(n,n)$-selection $(\alpha,\beta)$ at our disposal.
%In addition, in the algorithm we assume that $T_N$ is \emph{strictly positive definite}.
%naturally leads to an
%algorithm for obtaining a representation of the
%form (\ref{gen:filt:eq2}) similar to what is
%described in \cite{StochBilin}.
%Below we will briefly sketch the algorithm.
%%Due to ergodicity, we
%%can replace the process $\y$ and $\bu_{w}$ with
%%the corresponding time series and we can 
\begin{algorithm}
\caption{\label{alg2}
 \newline
  \textbf{Input:} data $\{y_t,u_{\sigma}(t) \mid t=0,\ldots, N+M, \sigma \in \Sigma\}$
   and $(n.n)$--selection $(\alpha,\beta)$.
 \newline
  \textbf{Output: } weak realization 
   $\Sigma_{N,M}=(\{{}^M F_{\sigma}, K_{\sigma}^{N,M}, P_{\sigma}^{N,M},Q_{\sigma}^{N,M} \}_{\sigma \in \Sigma}, {}^M H, I_p)$.
}
\begin{algorithmic}[1]
\STATE
  %Fix an integer $n > 0$ which represents an estimate on the dimension of a 
  %potential GJMLS realization of $\y$. 
  Approximate the covariances $\Lambda_{w}$, $w \in \Sigma^{2n-1}$,
  and the covariances $T_{v_1,v_2}$ for $v_1,v_2 \in \Sigma^{N}$
 from the time-series using the formula:
 \[
   \begin{split}
    \Lambda_w \approx \Lambda_{w}^{M} \stackrel{\mbox{def}}{=} \frac{1}{M} \sum_{t=2n-1}^{M+2n-1} y(t)z_{w}(t)  \\
    T_{v_1,v_2} \approx T_{v_1,v_2}^{M} \stackrel{\mbox{def}}{=} \frac{1}{M} \sum_{t=N}^{N+M+1} z_{v_1}(t)z_{v_2}^T(t) 
   \end{split}
 \]
 where for any $w=\sigma_1\cdots \sigma_k \in \Sigma^{2n-1}$, 
 $\sigma_1,\ldots,\sigma_k \in \Sigma$, $k \ge 2n-1$,
 $z_w(t)=y(t-k)u_{\sigma_1}(t-k)\cdots u_{\sigma_k}(t-1)$.

\STATE
 Construct the finite
Hankel matrix $H^M_{\Psi_{\y},n+1,n}$ by replacing the covariances 
$\Lambda^{\y}_w$, $w \in \Sigma^{2n-1}$ in the definition of $H_{\Psi_{\y},n,n+1}$
by the estimates $\Lambda^M_{w}$, $w \in \Sigma^{2n-1}$.
\STATE
 Choose a $n,n$-selection $(\alpha,\beta)$ such that $\Rank H_{\Psi_{\y},\alpha,\beta}=\Rank H_{\Psi,N,N}$.
Apply Algorithm \ref{alg1} Section \ref{sect:pow} to 
 $H^M_{\Psi_{\y},n+1,n}$ and the $n,n$-selection $(\alpha,\beta)$
 to obtain a representation 
$R_M=(\mathbb{R}^{n},\{{}^MF_{\sigma}\}_{\sigma \in \Sigma}, {}^MG,{}^M H)$.  
%This representation can be viewed as an approximation of a representation of
%$\Psi_{\y}$.
\STATE
   Use the estimates $T^M_{v_1,v_2}$, $v_1,v_2 \in \Sigma^{N}$ to construct the
   matrix $T_{N,M}$: the matrix $T_{N,M}$ has the same structure as 
   $T_N$, but instead of the covariances $T_{v_1,v_2}$ we use the approximations
   $T^{M}_{v_1,v_2}$.

\STATE
    Define $\widetilde{\Lambda}_{N,M}$ in the same way $\widetilde{\Lambda}_{N}$, but using
    ${}^MF_v {}^{M}G_{\sigma}$ instead of $\widetilde{\Lambda}_{\sigma v}$, where
   \( {}^MG_{\sigma}=\begin{bmatrix} {}^MG_{1,\sigma} & \ldots & {}^MG_{p,\sigma} \end{bmatrix} \).

\STATE
   Assume that $T_{N,M}$ is invertable, and find
   \[ 
     \begin{split}
     & \alpha_{N,M}=\widetilde{\Lambda}_{N,M}T_{N,M}^{-1}  \\
     & P_{\sigma}^{N,M}=p_{\sigma}\alpha_{N,M}T_{N,M}D\alpha_{N,M}^T \\
     & Q_{\sigma}^{N,M}=p_{\sigma}(T_{\sigma,\sigma}^{M}-\frac{1}{p_{\sigma}}{}^MHP_{\sigma}^{N,M}{}^MH^T) \\
  & K_{\sigma}^{N,M}= (\sqrt{p_{\sigma}}{}^M G_{\sigma} - \frac{1}{\sqrt{p_{\sigma}}}{}^{M}F_{\sigma}P_{\sigma}^{N,M}{}^M H^T) (p_{\sigma}T_{\sigma,\sigma}^{M}-{}^MHP_{\sigma}^{N,M}{}^MH^T)^{-1} 
     \end{split}
   \]
    Here $D$ is a diagonal matrix such that
 the $i$th diagonal entry $D_{ii}$ is $1$ if $v_{i}\sigma \in L$ or it is zero otherwise.  

%  Find 

\STATE
    Return the weak realization $\Sigma_{N,M}=(\{{}^M F_{\sigma}, K_{\sigma}^{N,M}, P_{\sigma}^{N,M},Q_{\sigma}^{N,M} \}_{\sigma \in \Sigma}, {}^M H, I_p)$.
\end{algorithmic}
\end{algorithm}
% From Lemma \ref{gen:filt:proof:lemma1} it then follows that as 
 %$N,M$ converges to infinity, the realization returned by the algorithm above
 %converges to a weak \GBS\ realization  of $\y(t)$. 
% The convergence is understood
 %in the following sense: the system matrices returned by the algorithms converge to the
 %system matrices of a realization, and the state and noise processes converge to the
 %state and noise processes of the realization \eqref{gen:filt:eq2} in the mean-square sense.
 %More precisely, the  following holds.
 \begin{Theorem}[Correctness of Algorithm \ref{alg1}]
 \label{alg2:theo}
  Assume that the following holds:
  \begin{enumerate}
  \item
     \label{alg2:theo:assump1}
     The process $(\y,\{\bu_{w} \mid w \in \Sigma^{+}\})$ is ergodic and
     the time series $\{y(t),u_{\sigma}(t) \mid \sigma \in \Sigma, t=0,1,\ldots, \}$
     are such that for all $v,w \in \Sigma^{+}$.
     \begin{equation}
     \label{alg2:theo:assump1:eq1}
      \begin{split}
         & E[\y(t)\z_{w}^T(t)]=\lim_{N \rightarrow \infty} \frac{1}{N} \sum_{r=|w|}^{N} y(r)z_{w}^T(r) \\
         & E[\z_{v}(t)\z_{w}^T(t)]=\lim_{N \rightarrow \infty} \frac{1}{N} \sum_{r=\max{|w|,|v|}}^{N} z_{v}(r)z_{w}^T(r) \\
      \end{split}
     \end{equation}
  \item
     \label{alg2:theo:assump2}
     The $n,n$--selection $(\alpha,\beta)$ is such that 
     $\Rank H_{\Psi_{\y},\alpha,\beta}=\Rank H_{\Psi_{\y}} \le n$.
  \item
     \label{alg2:theo:assump3}
     The representation returned by Algorithm \ref{alg1} when applied to 
     $H_{\Psi_{\y},n,n+1}$ and $(\alpha,\beta)$ is of the form
     $R=(\mathbb{R}^{r},\{A_{\sigma}\}_{\sigma \in \Sigma},B,C)$.
  \item
     \label{alg2:theo:assump4}
    The process $\y$ satisfies Assumption \ref{output:assumptions}, Assumption 
    \ref{output:assumptions:extra} and it is full rank.
  \end{enumerate} 
   Let $\Sigma$ be the \GBS\ realization of $\y$ from \eqref{gen:filt:eq2} and 
   let $Q_{\sigma}=E[\e(t)\e^T(t)\bu_{\sigma}^{2}$. Identify $\Sigma$ with the
   corresponding weak realization $\Sigma=(\{A_{\sigma},K_{\sigma},P_{\sigma},Q_{\sigma}\}_{\sigma \in \Sigma}, C,I_p)$.
 Then the following holds
   \begin{enumerate}
  \item
   For large enough 
   $N,M$, $T_{N,M}$ and $Q_{\sigma}^{N,M}$ are invertable and 
   Algorithm \ref{alg1} is well posed.
  \item
     $\lim_{M \rightarrow \infty} \Sigma_{N,M}=(\{A_{\sigma},K^{N}_{\sigma},P^{N}_{\sigma},Q^{N}_{\sigma}\}_{\sigma \in \Sigma},C,I_p)$,
     %\footnote{Here, convergence is understood as convergence of tuples of matrices}, 
    where
     $Q^{N}_{\sigma}=E[\e_N(t)\e_N(t)\bu_{\sigma}^2(t)]$ and
     $P_{\sigma}^N$ and $K_{\sigma}^N$ are defined as  Lemma \ref{gen:filt:proof:new:lemma2}.
  \item
    $\lim_{N \rightarrow \infty} \lim_{M \rightarrow \infty} \Sigma_{N,M}=\Sigma$
     %\footnote{Here, convergence is understood as convergence of tuples of matrices}.
  \end{enumerate}
 \end{Theorem}
 Informally, Theorem  \ref{alg2:theo} says the following. If we let $M$
 go to infinity, then the weak realization $\Sigma_{N,M}$ returned by 
 Algorithm \ref{alg2:theo} corresponds to the approximate realization
 described in Lemma \ref{gen:filt:proof:new:lemma2}. In that realization,
 the state process $\x(t)$ is approximated by $\x_N(t)$, the latter being 
 the (linear combination of) projection of future outputs to finitely many past
 outputs and inputs. If we let $N$ go to infinity too, then
 $\Sigma_{N,M}$ will converge (as a tuple of matrices) to the weak 
 realization which corresponds to the \GBS\ described in Theorem \ref{gen:filt:theo3.3}.
 %That is, if we take say $N=M$ and take $N$ very large, then the matrices of 
 %$\Sigma_{N,N}$ will be a close approximation of the matrices of $\Sigma$ returned by
 %Theorem \ref{gen:filt:theo3.3}.
 %
 Theorem \ref{alg2:theo} and Algorithm \ref{alg2} open up the possibility of formulating
 subspace-like realization algorithms for \GBS{s} and for analyzing existing ones
 \cite{Verhaegen:CDC04,Verhaegen:automatica,MacChenBilSub,Favoreel:PhD99}.
 Pursuing this direction remains future work.

\section{Proof of the results on realization theory of \GBS{s}}
\label{gbs:prob:proof}
\subsection{Technical preliminaries and the proofs of Lemma \ref{gbs:def:new_lemma1}--\ref{gbs:def:new_lemma1.3}}
\label{proof:gbs:tech}
 Below we will present a number of technical results on
 \RC\ processes. These results will allow us to prove
 Lemmas \ref{gbs:def:new_lemma1}--\ref{gbs:def:new_lemma1.3} and 
 the main theorems. 
 
  \begin{Notation}
   Let $I_k$ denote the $k \times k$ identity matrix. 
 \end{Notation}

 Let $\br(t) \in \mathbb{R}^{r}$ be an \RC\ process.
 \begin{Notation}
 \label{new:pf:not1}
  Denote by $H_t^{\br}$ the Hilbert-space generated by the entries of
  $\{\z_{w}^{\br}(t) \mid w \in \Sigma^{+}\}$.
 \end{Notation}
  \begin{Lemma}
  \label{new:pf:lemma-2}
   With the notation above, if
 $H_{t}^{\br} \subseteq H_{t+1}^{\br}$,
   $\br(t) \in H_{t+1}^{\br}$.
  \end{Lemma}
  \begin{proof}[Proof of Lemma \ref{new:pf:lemma-2}]
  From Assumption \ref{input:assumption} it follows that $\sum_{\sigma \in \Sigma} \alpha_{\sigma} \bu_{\sigma}(t)=1$  for any $t \in \mathrm{Z}$, and hence
  $\br(t)=\sum_{\sigma \in \Sigma} \alpha_{\sigma} \br(t)\bu_{\sigma}(t)=\sum_{\sigma \in \Sigma} \z^{\br}_{\sigma}(t+1) \in H_{t+1}^{\br}$.
  Similarly, $\z_{w}^{\br}(t)=\sum_{\sigma \in \Sigma} \alpha_{\sigma} \z^{\br}_{w}(t)\bu_{\sigma}(t)=\sum_{\sigma \in \Sigma} \alpha_{\sigma} \z^{\br}_{w\sigma}(t+1) \in H_{t+1}^{\br}$.
  \end{proof}
  \begin{Lemma}
  \label{new:pf:lemma-1}
 % Assume that $\br(t)$ is \RC and let $H_t$ be the Hilbert-space generated
  %by the entries of $\{\z^{\br}_{w}(t) \mid w \in \Sigma^{+}\}$. 
  Let
  $\z(t) \in \mathbb{R}^{d}$ be a process such that the entries of $\z(t)$
  belong to $H^{\br}_{t}$ for any $t \in \mathrm{Z}$ and
  that $E[\z(t+k)(\z_{w}^{\br}(t+k))^T]=E[\z(t)(\z_{w}^{\br}(t))^T]$.
   Then the process $\begin{bmatrix} \br(t) \\ \z(t) \end{bmatrix}$ is \RC.
 \end{Lemma}
  For the proof of this lemma we will need the following results.
\begin{Lemma}
\label{gen:filt:proof:lemma5.1}
 If $\z \in \mathbb{R}$ is a mean-square integrable random variable and 
 it belongs to the linear span of the components of
 $\z_v(t)$, $v \in L$, then
 $E[\z^2\bu_{\sigma}^{2}(t)] \le p_{\sigma} E[\z^2]$.
\end{Lemma}
\begin{proof}[Proof of Lemma \ref{gen:filt:proof:lemma5.1}]
 Assume that for some finite subset $S \subseteq L$,
 $\z=\sum_{v \in S} \alpha_v \z_v(t)$ for some 
 $\alpha_v \in \mathbb{R}^{1 \times p}$.
 Define $S_1=\{ v \in S, v\sigma \in L\}$,
 $S_2 = \{ v \in S, v\sigma \notin L\}$. Then, by noticing
 that $E[\z_v(t)\z^T_{w}(t)\bu^2_{\sigma}(t)]=E[\z_{v\sigma}(t+1)\z^T_{w\sigma}(t+1)]=T_{v\sigma,w\sigma}$ and taking into
  account Part \ref{RC4} of Definition \ref{def:RC} and that
 $T^T_{v\sigma,s}=T_{s,v\sigma}=0$ for $v\sigma \notin L$, we 
 obtain
 \begin{equation}
 \label{gen:filt:proof:lemma5.1:eq1}
   \begin{split}
    & E[\z^2\bu^2_{\sigma}(t)]=%\sum_{v,w \in S} \alpha_{v} 
    %E[\z_{v}(t)\z_{w}^T(t)\bu^2_{\sigma}(t)]\alpha^T_{w}=
    \sum_{v,w \in S_1} p_{\sigma} \alpha^T_v T_{v\sigma,w\sigma}\alpha^T_{w} =
    p_{\sigma} \sum_{v,w \in S_1} \alpha_v T_{v,w}\alpha^T_{w}
      %2\sum_{v \in S_1, w \in S_2}
       %\alpha^T_v \alpha^T_v T^T_{v\sigma,w\sigma}\alpha_{w} +
      %\sum_{v,w \in S_2} \alpha^T_v \alpha^T_v T^T_{v\sigma,w\sigma}\alpha_{w} = 
    \end{split}
  \end{equation}
  On the other hand, by Part \ref{RC4} of Definition \ref{def:RC},
  if $v \in S_1$ and $w \in S_2$ or other way around, then 
  $T_{v,w}=0$. Moreover, $(T_{v,w})_{v,w \in S_2}$ is positive 
  definite, \ie  $\sum_{v,w \in S_2} \alpha_vT_{v,w}\alpha_w^T \ge 0$.
 Hence, by noticing that $S=S_1 \cup S_2$, 
  \begin{equation}
  \label{gen:filt:proof:lemma5.1:eq2}
     \begin{split}
      & E[\z^2]=\sum_{v,w \in S} \alpha_v T_{w,v}\alpha^T_{w}=
      %\sum_{v,w \in S_1} \alpha_v T_{v,w}\alpha^T_{w}+
      %2\sum_{v \in S_1, w \in S_2} \alpha_v T_{v,w}\alpha^T_{w}+
      %\sum_{v,w \in S_2} \alpha_v T_{v,w}\alpha^T_{w} = \\
    \sum_{v,w \in S_1} \alpha_v T_{v,w}\alpha^T_{w}+
      \sum_{v,w \in S_2} \alpha_v T_{v,w}\alpha^T_{w} \ge 
      \sum_{v,w \in S_1} \alpha_v T_{v,w}\alpha^T_{w} 
    \end{split}
  \end{equation}
  Combining \eqref{gen:filt:proof:lemma5.1:eq1} and
  \eqref{gen:filt:proof:lemma5.1:eq2} yields the statement of the 
  lemma.  
\end{proof}
\begin{Lemma}
\label{gen:filt:proof:lemma5}
 %Let $M$ be the closure of the linear space generated by the
 %coordinates of
 %$\z^{\br}_{v}(t)$, $v \in \Sigma^{+}$, and let $M^N$ be the subspace
 %generated by the coordinates of $\z_{v}(t)$, $v \in \Sigma^N$.
 Assume that $\z_N \in \mathbb{R}$ is a sequence such that
 $\z_N$ is a finite linear combination of $\z_{w}^{\br}(t)$, $w \in \Sigma^{+}$ and
 $\z=\lim_{N \rightarrow \infty} \z_N$ in the mean-square sense.
 Then for each $\sigma \in \Sigma$,
 $\z\bu_{\sigma}(t)=\lim_{N \rightarrow \infty} \z_N\bu_{\sigma}(t)$.
 in the mean-square sense.
\end{Lemma}
 %The proof of Lemma \ref{gen:filt:proof:lemma5} can be found in Appendix \ref{proofs}.
\begin{proof}[Proof of Lemma \ref{gen:filt:proof:lemma5}]
 If
  $\z=\lim_{N \rightarrow \infty} \z_N$ in the mean-square sense,
 then it $\z_N\bu_{\sigma}(t)$ converges to $\z_N\bu_{\sigma}(t)$ in
 mean sense.
 %$L_1$ sense, \ie $\lim_{N \rightarrow \infty} E[|\z_N\bu_{\sigma}(t)-z\bu_{\sigma(t)}|]=0$.
 Indeed, from H\"olders inequality 
 %for $p=q=2$, $f=\z_N-\z$, 
 %$g=\bu_{\sigma}(t)$ 
  it follows
  that $E[|\z_N\bu_{\sigma}(t)-z\bu_{\sigma(t)}|] = E[|(\z_N-\z)| \dot |\bu_{\sigma}(t)|] \le \sqrt{E[|z-\z_N|^{2}]}\sqrt{E[\bu_{\sigma}(t)^2]}$.
  On the other hand, it can be shown that 
  $\z_N\bu_{\sigma}(t)$ is a Cauchy-sequence in the mean-square sense.
  Notice that by Lemma \ref{gen:filt:proof:lemma5.1},
  $\z_N\bu_{\sigma}(t)$ is in fact mean-square integrable.
  Consider $\z_{N+K}-\z_N$ for any $K >0 $. Since $\z_{N+K}-\z_N$ belongs
  to the closed linear space $M_{N+K}$ generated by the entries of $\{\z_k\}_{k \le N}$, by Lemma \ref{gen:filt:proof:lemma5.1},
  \( E[|\z_{N+K}\bu_{\sigma}(t) - \z_{N}\bu_{\sigma}(t)|^2]=
     E[|\z_{N+K}-\z_N|^2\bu^2_{\sigma}(t)] \le p_{\sigma} E[|\z_{N+K} - \z_N|^2] \).
  Since $\z_N$ is convergent, it is then a Cauchy sequence and
  hence by the inequality above so is $\z_{N}\bu_{\sigma}(t)$.
  %By the completness of the $L_2$ space of mean-square integrable
  %random variables, we then obtain that the sequence
  But by Jensen's inequality,
  $E[|\z_N(t)\bu_{\sigma}(t)-h|] \le \sqrt{E[|z_N(t)\bu_{\sigma}(t)-h|^2]}$, 
  and hence  
  $h$ is the limit of $\z_N\bu_{\sigma}(t)$  in the mean
  sense as well. 
  It then follows from the uniqueness of the limit
  in $L_1$ sense that $h=\z\bu_{\sigma}(t)$ almost surely.
  Hence, $\z\bu_{\sigma}(t)=h$ is indeed the limit of
  $\z_N\bu_{\sigma}(t)$ in the mean-square sense.
\end{proof}
 \begin{proof}[Proof of Lemma \ref{new:pf:lemma-1}]
  From $\z(t) \in H_t$ it follows that  $\z(t)=\lim_{N \rightarrow \infty} \z_N$
  where $\z_N=\sum_{w \in \Sigma^{+}, |w| \le N} \alpha_{w}\z^{\br}_{w}(t)$ for
  some $\alpha_w \in \mathbb{R}^{d \times r}$.  
  Define $\z_N(k)=\sum_{s \in \Sigma^{+}, |s| \le N} \alpha_{s}\z^{\br}_{s}(k)$
  for all $k \in \mathbb{Z}$.  From $E[\z(t)\z^T_N(t)]=E[\z(t+k)\z_N^T(t+k)]$,
  it follows that $E[||\z(t+k)-\z_N(t+k)||^{2}]=E[||\z(t)-\z_N(t)||^{2}]$ and
  hence $\z(k)=\lim_{N \rightarrow \infty} \z_N(k)$, $t,k \in \mathbb{Z}$.   
  For every $v \in \Sigma^{+}$, denote by $\z^{v}_N(t)$ the finite sum
  $\z_N^{v}(t)=\sum_{s \in \Sigma^{+}, |s| \le N} \alpha_{s}\z^{\br}_{sv}(t)$.
  By repeated application of Lemma \ref{gen:filt:proof:lemma5} we obtain that
  \[ \z_{v}^{\z}(t)=\lim_{N \rightarrow \infty} \z_{N}^{v}(t) \] 
  Since $E[\z_v^{\z}(t)(\z_{w}^{\z}(t))^T]$ and $E[\z^{\z}(t)(\z^{\z}_{w}(t))^T]$ are
  the limits of $E[\z_N^{v}(t)(\z_N^{w}(t))^T]$ and
  $E[\z_N(t)(\z_N^{w}(t))^T]$ respectively.
  Hence, if $\br(t)$ satisfies Part \ref{RC1} of
  Definition \ref{def:RC}, i.e. $\{\br(t),\z_{w}^{\br}(t) \mid w \in \Sigma^{+}\}$
  is zero mean wide-sense stationary, then so is
  $\{\z(t),\z_w^{\z}(t) \mid w \in \Sigma^{+}\}$. That is
  %$E[\z_N^{v}(t)(\z_N^{w}(t))^T]=E[\z_N^{v}(t+k)(\z_N^{w}(t+k))^T]$ and
  %$E[\z_N(t)(\z_N^{w}(t))^T]=E[\z_N(t)(\z_N^{w}(t+k))^T]$ and hence
  $\z$ satisfies Part \ref{RC1} of Definition \ref{def:RC}.
 % Since $\z_N^{w}(t)$ is a linear combination of $\z_{sw}^{\br}(t)$ for finitely many $s \in \Sigma^{+}$ and $w \notin L$ implies that $sw \notin L$ and hence $\z_{sw}^{\br}(t)=0$,
 %$\z(t)$ satisfies Part \ref{RC2} of Definition \ref{def:RC}.
  Finally, in order to prove that $\z(t)$ satisfies 
  Part \ref{RC3} of Definition \ref{def:RC}, 
  notice that $T_{w\sigma, v\sigma^{'}}^{\z}$ is
  is the limit of linear combinations of $T^{\br}_{sw\sigma,hv\sigma^{'}}$
  for $s,h \in \Sigma^{+}$.  If $\sigma \ne \sigma^{'}$, then by virtue of
  $\br$ satifying Part \ref{RC3} of Definition \ref{def:RC},
  $T_{w\sigma,v\sigma^{'}}^{\br}=0$. If
  $\sigma =\sigma^{'}$ and $w\sigma, v\sigma \in L$, 
  %$E[\z_{sw\sigma^{'}}^{\br}(t)(\z_{hv\sigma}^{\br}(t))^T]=E[\z_{sw}^{\br}(t)(\z_{hv}(t))^T]$.  $T_{sw\sigma,hv\sigma}^{\br}=T_{sw,hv}^{\br}$. 
  % Indeed, if $sw\sigma,hv\sigma \in L$, this follows from
  %Part \ref{RC3} of Assumption \ref{def:RC}. If say $sw\sigma \notin L$, then from
  %$w\sigma \in L$ and the definition of $L$ it follows $sw \notin L$, and hence
  %$T_{sw\sigma,hv\sigma}^{\br}=0=T_{sw,hv}^{\br}$. 
  % and $E[\z_{sw}^{\br}(t)(\z_{hv}(t))^T]=0$. 
  Finally, if $w\sigma \notin L$ (respectively $v\sigma \notin L$), then 
  for all $s \in \Sigma^{+}$, $sw\sigma \notin L$ (respectively $hv\sigma \notin L$ for all $v \in \Sigma^{+}$) and hence  $T^{\br}_{sw\sigma,hv\sigma}=0$.
  By combining the results above and taking limits we readily conclude that
  $\z(t)$ satisfies Part \ref{RC3} of Definition \ref{def:RC}.
  Finally, as it was remarked in Remark \ref{rem:RC}, Part \ref{RC4} of Definition
  \ref{def:RC} follows from Parts \ref{RC1}--\ref{RC3} of Definition \ref{def:RC}.
 \end{proof}
 \begin{Notation}
 \label{new:pf:not2}
 For every $w \in \Sigma^{+}$, denote by $H^{\br}_{t,w}$ the Hilbert-space generated by
 the entries of $\{\z_{vw}^{\br}(t) \mid v \in \Sigma^{+}\}$ and denote
 $H^{\br,*}_{t,w}$ the Hilbert-space generated by
 the entries of $\{\z_{vw}(t)^{\br} \mid v \in \Sigma^{*}\}$. Clearly, $H_{t,w}^{\br} \subseteq H_{t,w}^{\br,*}$.
 \end{Notation}
\begin{Lemma}
\label{new:pf:lemma-4}
 With the notation above, for every $\sigma_1,\sigma_2 \in \Sigma$, $\sigma_1 \ne \sigma_2$,
 $H^{\br,*}_{t,\sigma_1}$ and $H^{\br,*}_{t,\sigma_2}$ are orthogonal and hence
 $H^{\br}_{t,\sigma_1}$ and $H^{\br}_{t,\sigma_2}$ are orthogonal.
 Moreover, if $\z \in H_{t}^{\br}$, then $\z\bu_{\sigma}(t) \in H_{t+1,\sigma}^{\br}$.
\end{Lemma}
\begin{proof}[Proof of Lemma \ref{new:pf:lemma-4}]
 The first statement of the lemma is an immediate consequence of
 the fact that $E[\z^{\br}_{v\sigma_1}(t)(\z^{\br}_{w\sigma_2}(t))^T]=0$ for 
 all $w,v \in \Sigma^{+}$, $\sigma_1 \ne \sigma_2 \in \Sigma$.
The second statement follows by noticing that 
$\z^{\br}_{w}(t)\bu_{\sigma}(t) \in H_{t+1,\sigma}^{\br}$. 
If $\z \in H_t^{\br}$, then $\z=\lim_{N \rightarrow \infty} r_N$, 
where $r_N$ is a finite linear combination of $\z^{\br}_{w}(t)$, $w \in \Sigma^{+}$. 
 It then follows that $r_N\bu_{\sigma}(t) \in H_{t+1,\sigma}^{\br}$. 
From Lemma \ref{gen:filt:proof:lemma5} it follows that 
$\z=\lim_{N \rightarrow \infty} r_N\bu_{\sigma}(t)$ and hence 
$\z \in H_{t+1,\sigma}^{\br}$.
\end{proof}
%%\begin{Lemma}
%%\label{new:pf:lemma-3}
%% Assume that $\bh(t)$ is a \RC\ process such that 
%% $H_t^{\bh} \subseteq H_t^{\br}$ for all $t \in \mathbb{Z}$. Let 
%% $\z(t)$ be a \RC\ process such that the entries of $\z(t)$ belong to $H_t^{\br}$
%% and $\z(t)$ is orthogonal to $H_t^{\bh}$. Then $\z^{\z}_{w}(t)$ is orthogonal
%% to $H_{t,w}^{\bh}$ for all $t \in \mathbb{Z}$.
%%\end{Lemma}
%%\begin{proof}[Proof of Lemma \ref{new:pf:lemma-3}]
%% It is enough to show that $\z(t)\bu_{\sigma}(t)$ is orthogonal to $H^{\bh}_{t,\sigma}$,
%% the rest of the statement follows then by induction on the length of $w$. 
%% In order to show that $\z(t)\bu_{\sigma}(t)$ is orthogonal to $H^{\bh}_{t,\sigma}$, 
%% notice that by Lemma \ref{new:pf:lemma-4} $\bs(t)=(\z(t),z^{\bh}_{w}(t))$ is an \RC\ process
%% for all $w \in \Sigma^{+}$. Since
%% $E[\z(t)(\z^{\bh}_{w}(t))^T]$ and
%% $E[\z(t)(\z^{\bh}_{w}(t))^T\bu_{\sigma}^{2}(t)]$ are sub-matrices of
%% $\Lambda^{\bs}_{w}$ and $p_{\sigma}T_{\sigma,w\sigma}^{\bs}$ respectively,
%% %$E[\bs(t)(\z^{\bs}_{w}(t))^T]$  and $E[\z(t)(\z^{\bh}_{w}(t))^T\bu_{\sigma}^{2}(t)]$
%% respectively, it then follows that
%% $E[\z(t)(z^{\bh}_{w}(t))^T\bu_{\sigma}^{2}(t)]=p_{\sigma}E[\z(t)(\z^{\bh}_{w}(t))^T]$
%% if $w\sigma \in L$ and $E[\z(t)(z^{\bh}_{w}(t))^T\bu_{\sigma}^{2}(t)]=0$ otherwise.
%% Since by the orthogonality assumption $E[\z(t)(z^{\bh}_{w}(t))^T]=0$, 
%% it then follows that $E[\z(t)(z^{\bh}_{w}(t))^T\bu_{\sigma}^{2}(t)]=0$.
%%\end{proof}
\begin{Lemma}
\label{new:pf:lemma-3}
 Let $\bh(t) \in \mathbb{R}^{l}$, $\z(t)  \in \mathbb{R}^p$ be processes such that $\bs(t)=(\z^T(t),\bh^T(t))^T$ is \RC\ and
the coordinates of $\z(t)$ are orthogonal to $H_t^{\bh}$  for all $t \in \mathbb{Z}$. 
%% $H_t^{\bh} \subseteq H_t^{\br}$ for all $t \in \mathbb{Z}$. Let 
%% $\z(t)  \in \mathbb{R}^p$ be a process such that \textbf{(1)} 
%%  the process $(\br^T(t),\z^T(t))^T$ is  \RC\,, \textbf{(2)} there exists a matrix $\alpha \in \mathbb{R}^{p \times r}$ such that the coordinates of
%% $\z(t) - \alpha \br(t)$ belong to $H_t^{\br}$ for all $t \in \mathbb{Z}$,  and  \textbf{(3)} 
Then for all $w \in \Sigma^{+}$,  the coordinates of 
 $\z^{\z}_{w}(t)$ are orthogonal
to $H_{t,w}^{\bh}$ for all $t \in \mathbb{Z}$.
\end{Lemma}
\begin{proof}
%For any $v \in \Sigma^+$, let $\bs(t)=((\z(t)-\alpha \br(t))^T,(\z^{\bh}_{v}(t))^T)^T$. 
%% By Lemma \ref{new:pf:lemma-1} $\bs_1(t)=(\bs^T(t),\br^T(t))$ is \RC. 
%% Indeed, the coordinates of $\bs(t)$ belong to $H^{\br}_t$ and  since  $(\br^T(t),\bh^T(t))^T$ and  $(\br^T(t),\z^T(t))^T$ are both \RC,
%% $E[\bs(t+k)(\z_{v}^{\br}(t+k))^T]=E[\bs(t)(\z_{v}^{\br}(t))^T]$, for any $t,k \in \mathbb{Z}$, hence the conditions of Lemma \ref{new:pf:lemma-1}
%% are satisfied. 
It then follows that $\z(t)=C_1\bs(t)$ and $\bh(t)=C_1\bs(t)$ for suitable matrices $C_1,C_2$.
%% %$\bs_2(t)=(\z^T(t),(\z^{\bh}_{v}(t))^T)^T=C_1\bs_1(t)$, is an \RC\ process.
Note $E[\z(t)(\z^{\bh}_{v}(t))^T]=C_1\Lambda^{\bs}_{v}C_2^T$ 
and $E[\z^{\z}_w(t)(\z^{\bh}_{vw}(t))^T]=C_1T_{w,vw}^{\bs}C_2^T$,
$t \in \mathbb{Z}$. Since $\bs(t)$ is \RC, $T_{w,vw}^{\bs}=\Lambda_{v}^{\bs}$ if $vw \in L$ and $T_{w,vw}^{\bs}=0$ otherwise. 
  Hence, $E[\z^{\z}_w(t)(\z^{\bh}_{vw}(t))^T]=E[\z(t)(\z^{\bh}_{v}(t))^T]$ if $vw \in L$ and  $E[\z^{\z}_w(t)(\z^{\bh}_{vw}(t))^T]=0$ otherwise.
 Since by the orthogonality assumption $E[\z(t)(\z^{\bh}_{v}(t))^T]=0$, 
 it then follows that $E[\z^{\z}_w(t)(\z^{\bh}_{vw}(t))]=0$ for all $v \in \Sigma^{+}$. 
\end{proof}
 \begin{proof}[Proof of Lemma \ref{gbs:def:new_lemma1}]
  It is clear that if  $\bs(t)=\begin{bmatrix} \v^T(t), & \x^T(t) \end{bmatrix}^T$ is \RC, then $\x(t)=\begin{bmatrix} 0, & I_n \end{bmatrix}\bs(t)$ is \RC\ too. 
  The claim that $\begin{bmatrix} \v^T(t), & \x^T(t) \end{bmatrix}^T$ is \RC\ follows directly from Part \ref{gbs:def:prop4}, Assumption \ref{gbs:def}, and Lemma \ref{new:pf:lemma-1}, if we can
  show that $E[\x(t)(\z_{v}^{\v}(t))^T]$ does not depend on $t$ for
  any $v \in \Sigma^{+}$.
    For any $k \ge 0$, 
 \begin{equation}
  \label{gbs:def:new_lemma1:eq2}
   \x(t)=\sum_{w \in \Sigma^{+},|w|=k} \sqrt{p_w} A_{w}\z^{\x}_{w}(t)+
          \sum_{w \in \Sigma^{*}, |w| \le k-1} \sum_{\sigma \in \Sigma} \sqrt{p_{\sigma w}} A_{w}B_{\sigma}\z_{\sigma w}^{\v}(t).
 \end{equation}
  If $k=|v|$, then it follows that $E[\z^{\x}_{w}(t)(\z^{\v}_{v}(t)^T]=0$. Indeed, by
  Part \ref{gbs:def:prop4}, Assumption \ref{gbs:def} and
  Lemma \ref{gen:filt:proof:lemma5}, $\z^{\x}_{w}(t)$ belongs to the Hilbert-space generated by the components of $\z^{\v}_{sw}(t)$, $s \in \Sigma^{+}$. Since, $|w|=|v|=k$, $sw \ne v$ for all $s \in \Sigma^{+}$. Hence, $\z^{\v}_{v}(t)$ is orthogonal to the latter Hilbert-space.
  Hence, for $k=|v|$, 
  \[ 
     E[\x(t)(\z_{v}^{\v}(t))^T]
%=\sum_{w \in \Sigma^{+},|w|=k} \sqrt{p_w} A_{w}E[\z^{\x}_{w}(t)(\z^{\v}_{v}(t))^T]+
     =\sum_{w \in \Sigma^{*}, |w| \le k-1} \sum_{\sigma \in \Sigma} \sqrt{p_{\sigma w}}
 A_{w}B_{\sigma}T_{\sigma w,v}^{\v},
%E[\z_{\sigma w}^{\v}(t)\z_{v}^{\v}(t))^T].
  \]
 and the latter expression does not depend on $t$.
 %The terms $A_{w}B_{\sigma}E[\z_{\sigma w}^{\v}(t)(\z_{v}^{\v}(t))^T]=T_{\sigma w,v}$ do not depend on $t$. 
 % If we can show that
 % $E[\z^{\x}_{w}(t)(\z^{\v}_{v}(t)^T]$ does not depend on $t$, for
 % $|w|=k$, then we are done.
  %Assume $\sigma$ is the first letter of $v$. 
%$E[\z^{\x}_{w}(t)(\z^{\v}_{v}(t)^T]$ is the limit if linear combinations of terms of the form $E[\z^{\v}_{sw}(t)(\z^{\v}_{v}(t))^T]$. Each term of
 %the latter form is zero, since $sw \ne v$ for any $s,w \in \Sigma^{+}$, $|w|=|v|$ and Part \ref{gbs:def:prop2} of Assumption \ref{gbs:def} holds. Hence, $E[\z^{\x}_{w}(t)(\z^{\v}_{v}(t)^T]=0$ and thus it does not
 %depend on $t$.
  
  Using
  Part \ref{gbs:def:prop4}, Assumption \ref{gbs:def} and
  Lemma \ref{new:pf:lemma-4},  the coordinates of $\z^{\x}_{w}(t)$ belong to the Hilbert-space $H_{t,w}^{\v}$
  generated by the coordinates of $\z^{\v}_{sw}(t)$, $s \in \Sigma^{+}$. 
  Since, $|v| \le  |w|=k$, $sw \ne v$ for all $s \in \Sigma^{+}$, and hence $E[\z_{sw}^{\v}(t)(\z^{\v}_{v}(t))^T]=0$. That is,
  the coordinates of $\z^{\v}_{v}(t)$ 
  are orthogonal to $H_{t,w}^{\v}$ and hence to $\z_w^{\x}(t)$.

  In order to show \eqref{gbs:def:new_lemma1:eq1}, we go back to
  \eqref{gbs:def:new_lemma1:eq2}.
  We will  show that $r_k(t)=\sum_{w \in \Sigma^{+},|w|=k} \sqrt{p_w} A_{w}\z^{\x}_{w}(t)$
  converges to zero as $k \rightarrow \infty$. Since $\x(t)$ is \RC,
  $E[\z^{\x}_{w}(t)(\z_{v}^{\x}(t))^T]=0$ for any $w \ne v$ or $w=v \notin L$, $|w|=|v|=k$, and for all $w \in L$, 
  $E[\z^{\x}_{w}(t)(\z_{w}^{\x}(t))^T]=\frac{1}{p_{\sigma}} E[\x(t-k)\x^T(t-k)\bu^2_{\sigma}(t-k)]$, where $w=\sigma s$ for $\sigma \in \Sigma$ and $s \in \Sigma^{*}$.
  Denote by $P_{\sigma}=E[\x(t-k)\x^T(t-k)\bu^2_{\sigma}(t-k)]$. Note that by
  virtue of $\x(t)$ being \RC, the definition of $P_{\sigma}$ does not
  depend on $t$ and $k$. Moreover, from Part \ref{gbs:def:prop5} it follows that $A_{s}A_{\sigma}=0$ if $s \in \Sigma^{*}, \sigma \in \Sigma$, 
  $\sigma s \notin L$. 
  In then follows that
  \begin{equation} 
  \label{gbs:def:new_lemma1:eq3}
    E[r_k(t)r_k^{T}(t)]=\sum_{s \in \Sigma^{*}} \sum_{\sigma \in \Sigma } p_s A_{s}A_{\sigma}P_{\sigma}A_{\sigma}^TA_{s}^T.
  \end{equation}
  Define $S=\sum_{\sigma \in \Sigma} A_{\sigma}P_{\sigma}A_{\sigma}^T$ and
  define the linear map $\mathcal{R}$ on the space
  of matrices $\mathbb{R}^{n \times n}$ as 
  \[ \mathcal{R}(V)=\sum_{\sigma \in \Sigma} p_{\sigma} A_{\sigma}VA_{\sigma}^{T}.  \]
  Then $E[r_k(t)r_{k}^T(t)]=\mathcal{R}^{k-1}(S)$. 
  Notice that $\sum_{\sigma \in \Sigma} p_{\sigma}A^T_{\sigma} \otimes A_{\sigma}^T$
  is just the matrix representation of $\mathcal{R}(V)$ in the basis
  described in \cite[Chapter 2]{CostaBook}. Hence, by
  Part \ref{gbs:def:prop5} of Definition \ref{gbs:def} and 
  \cite[Proposition 2.5]{CostaBook},
  $\lim_{k \rightarrow \infty} \mathcal{R}^{k}(S)=0$. Hence, it follows that
  the limit of $E[||r_k(t)||^{2}]=\mathrm{trace} E[r_k(t)r^T_k(t)]$ equals
  zero as $k \rightarrow \infty$, which is a equivalent to saying that
  the mean-square limit of $r_k(t)$ is zero as $k$ goes to $\infty$.
 \end{proof}
\begin{proof}[Proof of Lemma \ref{gbs:def:new_lemma1.2}]
  In order to prove the statement of the lemma, we use the proof of
  Lemma \ref{gbs:def:new_lemma1}. Notice that
  \eqref{gbs:def:new_lemma1:eq2} remains valid for $t=k$, 
  if we replace $\x$ by $\hat{\x}$. The assumptions of the lemma ensure 
  that \eqref{gbs:def:new_lemma1:eq3} remains valid for
  $t=k$, where $r_k(k)=\sum_{w \in \Sigma^{+},|w|=k} \sqrt{p_w} A_{w}\z^{\hat{\x}}_{w}(k)$
  and $P_{\sigma}=E[\hat{\x}(0)\hat{\x}(0)^{T}\bu_{\sigma}^2(0)]$. With the
  argument as above, it then follows that $\lim_{k \rightarrow \infty} r_k(k)=0$ in the 
  mean-square sense. Notice that $\x(t)-\hat{\x}(t)=\sum_{v \in \Sigma^{*}, |v| \ge t} \Sigma_{\sigma \in \Sigma}  \sqrt{p_{\sigma v}} A_{v}B_{\sigma}\z^{\v}_{\sigma v}(t) - r_t(t)$.
  The first terms converges to zero in the mean-square sense as $t \rightarrow +\infty$, since
   the series on the right-hand side of \eqref{gbs:def:new_lemma1:eq1} is convergent in the
   mean-square sense. It was shown that the second term $r_t(t)$ converges to zero as $t \rightarrow \infty$. Hence,
   $\x(t)-\hat{\x}(t)$ converges to  $0$ in mean-square sense.
\end{proof}
\begin{proof}[Proof of Lemma \ref{gbs:def:new_lemma1.3}]
  First, we show that there exists at most one solution to \eqref{gen:filt:theo3.2:eq1}.
  To this end, assume that there are two solutions $\{P_{\sigma}\}_{\sigma \in \Sigma}$ and
  $\{P^{'}_{\sigma}\}_{\sigma \in \Sigma}$ to \eqref{gen:filt:theo3.2:eq1}. Define
  $\hat{P}_{\sigma}=P_{\sigma}-P_{\sigma}^{'}$. By subtracting th equation
  \eqref{gen:filt:theo3.2:eq1} for $P_{\sigma}$ and $P^{'}_{\sigma}$,
  \begin{equation} 
  \label{theorem3.1:pf:eq2}
   \hat{P}_{\sigma}=\sum_{\sigma_1 \in \Sigma, \sigma_1\sigma \in L} p_{\sigma}A_{\sigma_1}\hat{P}_{\sigma_1}A_{\sigma_1}^T.
  \end{equation}
  Using the equation above and the fact that $A_{\sigma}A_{\sigma_1}=0$ for $\sigma_1\sigma \notin L$, we obtain
  \begin{equation} 
  \label{theorem3.1:pf:eq1}
    A_{\sigma}\hat{P}_{\sigma}A^T_{\sigma}=\sum_{\sigma_1 \in \Sigma}
     p_{\sigma} A_{\sigma}A_{\sigma_1}\hat{P}_{\sigma_1}A_{\sigma_1}^TA_{\sigma} 
  \end{equation}
  Consider the map $\Z : \Re^{n\times n} \to \Re^{n\times n}$ defined as 
  \(  \Z(V)=\sum_{\sigma \in \Sigma} p_{\sigma} A_{\sigma} V A^T_{\sigma} .\)
  It is easy to see that $\sum_{\sigma \in \Sigma} p_{\sigma} A_{\sigma}^T \otimes A_{\sigma}^T$ is a matrix representation of $\Z$. 
  %From the properties of Kronecker product it follows
  %that $\widetilde{A}^T=\sum_{\sigma \in \Sigma} p_{\sigma} A_{\sigma} \otimes A_{\sigma}$
  %and hence the eigenvalues of $\widetilde{A}=\sum_{\sigma \in \Sigma} p_{\sigma} A^T_{\sigma} \otimes A_{\sigma}^T$ and $\Z$ coincide.
  Hence,
  from Part \ref{gbs:def:prop5} of Assumption \ref{gbs:def} it follows that all the
  eigenvalues of $\Z$ are inside the unit circle.
  Define $Q=\sum_{\sigma \in \Sigma} A_{\sigma}\hat{P}_{\sigma}A^T_{\sigma}$ and 
  notice that \eqref{theorem3.1:pf:eq1} implies that 
  \( Q=\sum_{\sigma \in \Sigma} p_{\sigma}A_{\sigma} (\sum_{\sigma_1 \in \Sigma} A_{\sigma_1}P_{\sigma_1}A_{\sigma_1}^T)A_{\sigma}^T=\Z(Q) \). 
  Since $1$ is not an eigenvalue of $Q$, it implies that $Q=0$. 
  But if $Q=0$, then \eqref{theorem3.1:pf:eq1} implies that
  \( A_{\sigma}\hat{P}_{\sigma}A^T_{\sigma}=p_{\sigma}A_{\sigma}QA_{\sigma}^T=0 \).
  Applying \eqref{theorem3.1:pf:eq2} yields $\hat{P}_{\sigma}=0$, and hence
  $P_{\sigma}=P^{'}_{\sigma}$ for all $\sigma \in \Sigma$.
  %.This result is obtained by identifying $\Re^{n\times n}$ with $\Re^{n^2}$,  as it is done in \cite[Section 2.1]{CostaBook}. As a consequence, the eigenvalues of $\Z$ and $\widetilde{A}$ coincide and hence all the eigenvalues of $\Z$ are inside the unit circle.
 %From \cite[Proposition 2.6]{CostaBook} it follows that the equation
 %\( V=\Z(V)+S \) has a unique positive definite solution

  Next, we show that a solution to \eqref{gen:filt:theo3.2:eq1} exists and it
  is determined by
 $P _{\sigma}=E[\x(t)\x(t)^{T}\bu_{\sigma}^{2}(t)]=p_{\sigma}E[\z_{\sigma}^{\x}(t)(\z_{\sigma}^{\x}(t))^{T}]$. 
  By Lemma \ref{new:pf:lemma-1} $\x(t)$ is \RC.
  From Part \ref{gbs:def:prop4} of Assumption \ref{gbs:def} it also follows that
  for every $w,v \in \Sigma^{+}$, $|w| \ge |v|$, $\z^{\x}_w(t)$ and
  $\z^{\v}_{v}(t)$ are orthogonal.
        %$E[\z^{\x}_{w}(t)(\z^{\v}_{v}(t))^T]=0$.
 %\end{Proposition}.
  Indeed by Lemma \ref{gen:filt:proof:lemma5},
   $\z^{\x}_{w}(t)$ belongs to the Hilbert
  space generated by $\z^{\v}_{sw}(t)$, $s \in \Sigma^{+}$ and  by
  Assumption \ref{gbs:def:prop2}, $\z^{\v}_{v}(t)$ and $\z^{\v}_{sw}(t)$ are
  orthogonal, since clearly $|sw| > |w| \ge |v|$.
  Notice the identities
  $P_{\sigma}=p_{\sigma}E[\z_{\sigma}^{\x}(t+1)(\z^{\x}_{\sigma}(t+1))^T]=p_{\sigma}T^{\x}_{\sigma,\sigma}$, 
   $\z_{\sigma_1\sigma}^{\x}(t+1)=\frac{1}{\sqrt{p_{\sigma_1\sigma}}}\x(t-1)\bu_{\sigma_1}(t-1)\bu_{\sigma}(t)$,
   $\z_{\sigma\sigma^{'}}^{\v}(t+1)=\frac{1}{\sqrt{p_{\sigma_1\sigma}}}\v(t-1)\bu_{\sigma_1}(t-1)\bu_{\sigma}(t)$ and
  \begin{equation} 
  \label{theorem3.1:pf:eq3}
    \z^{\x}_{\sigma}(t+1)=\sum_{\sigma_1 \in \Sigma}
      \sqrt{p}_{\sigma\sigma_1} (A_{\sigma_1}\z^{\x}_{\sigma_1\sigma}(t+1) +
      K_{\sigma_1}\z^{\v}_{\sigma_1\sigma}(t+1).
  \end{equation}
  Notice that
  $E[\z^{\x}_{\sigma_1\sigma}(t)(\z^{\x}_{\sigma_2\sigma}(t))^T]$ equals zero, 
 if $\sigma_1 \ne \sigma_2$ or $\sigma_1=\sigma_2, \sigma_1\sigma \notin L$, and  $p_{\sigma}P_{\sigma_1}$ otherwise.
 In a similar fashion,  
  $E[\z^{\v}_{\sigma_1\sigma}(t)(\z^{\v}_{\sigma_2\sigma}(t))^T]$ equals zero,
  if $\sigma_1 \ne \sigma_2$ or $\sigma_1=\sigma_2, \sigma_1\sigma \notin L$ and  $p_{\sigma}Q_{\sigma_1}$ otherwise.
  In addition, $E[\z^{\x}_{\sigma_1\sigma}(t+1)(\z^{\v}_{\sigma_2\sigma}(t+1))^T]=0$, $\sigma_1,\sigma_2,\sigma \in \Sigma$. 
  By noticing $P_{\sigma}=p_{\sigma} E[\z^{\x}_{\sigma}(t+1)(\z^{\x}_{\sigma}(t+1))^T]$, 
  and applying \eqref{theorem3.1:pf:eq3}, it follows that
  $\{P_{\sigma}\}_{\sigma \in \Sigma}$ satisfies \eqref{gen:filt:theo3.2:eq1}.
%%  %Part \ref{gbs:def:prop2} of Definition \ref{gbs:def}, and Proposition \ref{gbs:def:prop3}, 
%%  \[ 
%%     \begin{split}
%%     &  P_{\sigma}=p_{\sigma}E[\z_{\sigma}^{\x}(t+1)(\z_{\sigma}^{x}(t+1))^{T}]=
%%        \sum_{\sigma_1,\sigma_2 \in \Sigma} p_{\sigma} A_{\sigma_1}E[\z^{\x}_{\sigma_1\sigma}(t)(\z^{\x}_{\sigma_2\sigma}(t))^T]A_{\sigma_2}^T + \\
%%     &
%%        \sum_{\sigma_1,\sigma_2 \in \Sigma} p_{\sigma} A_{\sigma_1}E[\z^{\x}_{\sigma_1\sigma}(t+1)(\z^{\v}_{\sigma_2\sigma}(t+1))^T]K_{\sigma_2}^T + \\
%%        & \sum_{\sigma_1,\sigma_2 \in \Sigma} p_{\sigma} K_{\sigma_1}E[\z^{\v}_{\sigma_1\sigma}(t+1)(\z^{\x}_{\sigma_2\sigma}(t+1))^T]A_{\sigma_2}^T + \\
%%     & 
%%        \sum_{\sigma_1,\sigma_2 \in \Sigma} p_{\sigma} K_{\sigma_1}E[\z^{\x}_{\sigma_1\sigma}(t+1)(\z^{\v}_{\sigma_2\sigma}(t+1))^T]K_{\sigma_2}^T + \\
%%      & p_{\sigma}(\sum_{\sigma_1 \in \Sigma} A_{\sigma_1}P_{\sigma_1}A_{\sigma_1}^T+ 
%%                   K_{\sigma_1}Q_{\sigma_1}K_{\Sigma_1}^T
%%     \end{split}
%%  \]
  %Hence, $\{P_{\sigma}\}_{\sigma}$ satisfies
\end{proof}

\subsection{Proof of Theorem \ref{gen:filt:theo3.2}}
  We prove the claims one by one.
  %a \GBS\ which satisfies Assumption \ref{gbs:def}, then $\y(t)$ is \RC, 
  %In addition we will show that $R_{\BS}$ is a representation of
 %$\Psi_{\y}$ and that $R_{\BS}$ is well-defined. This will imply Part \ref{output:assumptions:part2} of Assumption \ref{output:assumptions}.
  %Finally, we will show that if 
  %$\forall \sigma \in \Sigma: DE[\v(t)\v^T(t)\bu_{\sigma}^{2}(t)]D^T > 0$, then 
  %$\y$ is full rank. 
  
  \textbf{Proof that $\y$ is \RC} 
  %To begin with, we will show that $\y(t)$ is \RC. 
  %To this end, we will need the
  %following lemma which will be used several times throughout the proof.
  From Lemma \ref{gbs:def:new_lemma1} it follows that $\bs(t)=\begin{bmatrix} \v^T(t), & \x^T(t) \end{bmatrix}^T$, and as
  Notice $\y(t)=C\x(t)+D\v(t)=\begin{bmatrix} C, & D \end{bmatrix}\bs(t)$, it then follows that $\y$ is \RC.
  %From Part \ref{gbs:def:prop4} it follows that
  %the entries of $\x(t)$ belong to the Hilbert-space $H_{t+1}^{\v}$.
  %From Lemma
  %\ref{new:pf:lemma-1} it then follows that $\y(t)$ is \RC.
  %Part \ref{RC1} of Definition \ref{def:RC} hold for $\y(t)$.  From
  %Part \ref{gbs:def:prop2}--\ref{gbs:def:prop4} of Assumption \ref{gbs:def}. 
  %Hence, $\y(t)$ satisfies Part \ref{output:assumptions:part1} of Assumption \ref{output:assumptions}.  

  \textbf{Proof that $R_{\BS}$ is well-defined and that it is a representation of $\Psi_{\y}$.} 
  %In order to prove $\Psi_{\y}$ is square summable and rational, we will construct
  %a rational representation $R_{\BS}$ of $\Psi_{\y}$.
  From Lemma \ref{gbs:def:new_lemma1.3} it follows that \eqref{gen:filt:theo3.2:eq1}
  has at most one solution.
  %To this end,  define $P_{\sigma}=E[\x(t)\x(t)^{T}\bu_{\sigma}^{2}(t)]=E[\z_{\sigma}^{\x}(t)(\z_{\sigma}^{x}(t))^{T}]$.

  Next, we show that $R_{\BS}$ is a representation of $\Psi_{\y}$.
  By induction on $|w|$ we obtain that  for all $w \in \Sigma^{*}$,
  \begin{equation}
  \label{proof:gen:filt:theo3.2:eq1}
     E[\x(t)(\z^{\x}_{\sigma w}(t))^T]= \frac{1}{\sqrt{p_{\sigma}}}
     \sqrt{p_{w}}A_{w} A_{\sigma}E[\x(t-k)\x^T(t-k)\bu_{\sigma}^{2}], 
  \end{equation}
  where $p_{w}$ is defined as in Notation \ref{ppnot1}. 
%follows: $p_{\epsilon}=1$, $p_{v\hat{\sigma}}=p_{\hat{\sigma}}p_{v}$ for all $v \in \Sigma^{*}$, $\sigma \in \Sigma$.
 Indeed, for $w=\epsilon$, $\z^{\x}_{\sigma}(t)=\frac{1}{\sqrt{p_{\sigma}}}\x(t-1)\bu_{\sigma}(t-1)$
  and using that $\x(t)=\sum_{\sigma \in \Sigma} \sqrt{p_{\sigma}}(A_{\sigma}\z^{\x}_{\sigma}(t)+ K_{\sigma}\z^{\v}_{\sigma}(t))$ and
 $E[\z^{\v}_{\sigma_1}(t)(\z^{\x}_{\sigma}(t))^T]=0$ for all $\sigma_1,\sigma \in \Sigma$ (see Lemma \ref{gbs:def:new_lemma1}) ,
 we obtain  \eqref{proof:gen:filt:theo3.2:eq1}.
  %\[
  %   \begin{split}
  %    & E[\x(t)\x^T(t-1)\bu_{\sigma}(t-1)]=\sum_{\sigma_1 \in \Sigma} p_{\sigma} A_{\sigma_1}E[z^{\x}_{\sigma_1}(t)(\z^{\x}_{\sigma}(t))^T]+p_{\sigma} K_{\sigma_1}E[(\z^{\v}_{\sigma_1}(t)(\z^{\x}_{\sigma}(t))^T] = \\
   %   & A_{\sigma}P_{\sigma}
    %\end{split}
  %\]
  If $w=v\hat{\sigma}$, then using $\x(t)=\sum_{\sigma_1 \in \Sigma} \sqrt{p_{\sigma}}(A_{\sigma_1}\z^{\x}_{\sigma_1}(t)+K_{\sigma_1}\z^{\v}_{\sigma_1}(t))$, the induction hypothesis, and the
  equalities
  $E[\z^{\x}_{\sigma_1}(t)(\z^{\x}_{\sigma v\hat{\sigma}}(t))^T]=0$ if $\sigma_1 \ne \hat{\sigma}$ or $\sigma v \hat{\sigma} \notin L$, 
  and $E[\z^{\x}_{\sigma_1}(t)(\z^{\x}_{\sigma v\hat{\sigma}}(t))^T]=\sqrt{p_{\hat{\sigma}}}E[\x(t-1)(\z^{\x}_{\sigma v}(t))^T]$ for $w=\sigma v \hat{\sigma} \in L$,
  and $E[\z^{\v}_{\sigma_1}(t)(\z^{\x}_{\sigma v\hat{\sigma}}(t))^T]=0$ (see Lemma \ref{gbs:def:new_lemma1}), and using $A_{w}A_{\sigma}=0$ if $\sigma w \notin L$, 
  we again readily obtain \eqref{proof:gen:filt:theo3.2:eq1}.
  %\[
  %   \begin{split}
  %   & E[\x(t)(\z^{\x}_{\sigma w}(t))^T]=\sum_{\sigma_1 \in \Sigma}
  %     A_{\sigma_1}E[\z^{x}_{\sigma_1}(t)\z^{\x}_{\sigma v\hat{\sigma}}(t))^T]+
  %     K_{\sigma_1}E[\z^{\v}_{\sigma_1}(t)\z^{\x}_{\sigma v\hat{\sigma}}(t))^T]= \\
  %	       & p_{\hat{\sigma}} A_{\hat{\sigma}} E[\x(t-1)\z^{\x}_{\sigma v}(t-1))^T]=
  %     p_{\hat{\sigma}}p_{v} A_{\hat{\sigma}}A_{v}A_{\sigma}P_{\sigma}=p_{w}A_{w}A_{\sigma}P_{\sigma}.
     %\end{split}
  %\]
  %Above we used the following consequences of Part \ref{gbs:def:prop1}--\ref{gbs:def:prop3} of
  %Assumption \ref{gbs:def}:
  %from $\x(t)$ being \RC\ it follows that
  %Part \ref{gbs:def:prop2} of Assumption \ref{gbs:def} it follows that
  In a similar fashion, we can show that
   \[ E[\x(t)(\z^{\v}_{\sigma w}(t))^T]=\begin{cases} 
               \frac{1}{\sqrt{p_{\sigma}}} \sqrt{p_{w}}A_wK_{\sigma}Q_{\sigma}, & \sigma w \in L \\
                  0 & \mbox{ otherwise } 
              \end{cases}, 
  \]
  where $Q_{\sigma}=E[\v(t)\v^T(t)\bu_{\sigma}^2(t)]$.
  %\[ E[\x(t)\z^{\v}_{\sigma w}(t)]=p_{w}K_{\sigma}Q_{\sigma} \]
  %where $Q_{\sigma}=E[\v(t)\v(t)\bu_{\sigma}^2(t)]$.
  Finally, notice that
  \( \z_{w}(t)=C\z^{\x}_{w}(t)+D\z^{\v}_{w}(t) \), and $\v(t)$ is orthogonal to the
   variables
   $\z^{\x}_{w}(t)$ and $\z^{\v}_{w}(t)$. 
   Using the definition $\Lambda^{\y}_{\sigma w} = E[\y(t)\z_{\sigma w}^{T}(t)]$, $A_wA_{\sigma}=0$, $A_{w}K_{\sigma}Q_{\sigma}=0$ for $\sigma w \notin L$,
   and \eqref{proof:gen:filt:theo3.2:eq1}, we derive
  \[
     \begin{split}
      & \Lambda^{\y}_{\sigma w} = %E[\y(t)\z_{\sigma w}^{T}(t)] = 
    % =CE[\x(t)(\z^{\x}_{\sigma w}(t))^T]C^{T}+
    %                          CE[\x(t)(\z^{\v}_{\sigma w}(t))^T]D^{T}+ \\
                              %& DE[\v(t)(\z^{\x}_{\sigma w}(t))^T]C^T+DE[\v(t)(\z^{\v}_{\sigma w}(t))^T] = \\
      CE[\x(t)(\z^{\x}_{w}(t))^T]C^{T}+CE[\x(t)(\z^{\v}_{w}(t))^T]D^{T} =
          \sqrt{p_{w}}CA_{w}\frac{1}{\sqrt{p_{\sigma}}}(A_{\sigma}P_{\sigma}C^T+K_{\sigma}Q_{\sigma}D^T)
     \end{split}
   \]
  %Define now the
  %rational representation
  %$R_{\BS}=(\mathbb{R}^{n},\{\sqrt{p_{\sigma}}A_{\sigma}\}_{\sigma \in \Sigma}, \widetilde{B},C)$ where
  %$\widetilde{B}=\{B_{(\sigma,i)} \mid \sigma \in \Sigma, i=1,\ldots,p\}$ are
  %defined as follows: $B_{\sigma,i}$ is the $i$th column of
  %\[ B_{\sigma}=A_{\sigma}P_{\sigma}C^T+K_{\sigma}Q_{\sigma}D^{T}. \]
  That is, 
  \( \Lambda^{\y}_{\sigma w}=CA_{w}B_{\sigma} \), i.e.
 $R_{\BS}$ is a representation of $\Psi_{\y}$.

  Finally, from Part \ref{gbs:def:prop5} of Definition \ref{gbs:def}
  it follows that $R_{\BS}$ is a stable
  representation.

 \textbf{Proof that $\y$ satisfies Assumption \ref{output:assumptions}}
  From the discussion above it follows that $\y$ is \RC\ and 
 $R_{\BS}$ is a stable representation of $\Psi_{\y}$.
  Hence, $\Psi_{\y}$ is rational and by 
  Theorem \ref{hscc_pow_stab:theo2} $\Psi_{\y}$ is square-summable too.

  \textbf{Proof that $\y$ is full rank} 
  To this end, notice that $\z_{w}(t)=C\z^{\x}_{w}(t)+D\z^{\v}_{w}(t)$.
  Part \ref{gbs:def:prop4} of Assumption \ref{gbs:def} and repeated application of
  Lemma \ref{new:pf:lemma-4} implies the coordinates of $\z^{\x}_{w}(t)$ belong to $H_{t,w}^{\v} \subseteq H_t^{\v}$ and $H_t^{\y} \subseteq H_t^{\v}$.
  %Denote the latter Hilbert space by $V_t$ and denote by $H_t$ the Hilbert-space
  %generated by the entries of $\{\z_{w}(t) \mid w \in \Sigma^{+}\}$. 
  Let $H^{\bot}_t$ be the orthogonal complement of $H_t^{\y}$ in $H_t^{\v}$.
  From Definition \ref{gbs:def} it follows that
  $E[\v(t)h]=0$ for any $h \in H^{\v}_t$. Hence, $\v(t)$ is orthogonal to $H^{\y}_t$.
  Notice that the entries of $\x(t)$ belong to $H^{\v}_t$ and hence it can be written as 
  $\x(t)=\x_1(t)+\x_2(t)$ such that the entries of $\x_2(t)$ belong to
  $\mathcal{H}_t^{\bot}$. It then follows that $E_l[\x(t) \mid \{\z_{w}(t) \mid w \in \Sigma^{+}\}]=\x_1(t)$ since
  for all $w \in \Sigma^{+}$, $E[\x_2(t)\z_{w}^T(t)]=0$ and hence $E[\x(t)\z_{w}^T(t)]=E[\x_1(t)\z_{w}^{T}(t)]$.
 Then   $E_l[\y(t) \mid \{\z_{w}(t) \mid w \in \Sigma^{+}\}]=C\x_1(t)$, since 
  $E[\y(t)\z_{w}^T(t)]=CE[\x_1(t)\z_{w}^{T}(t)]$ and the entries of $C\x_1(t)$ 
  belong to $H_t$.  Moreover, from Lemma \ref{new:pf:lemma-1} it follows that $(\y^T(t),\x^T_2(t))^T$ is \RC.
   %Moreover, $(\v^T(t),\x^T(t))^T$ is \RC, and thus $(\v^T(t),\y^T(t))^T$ is \RC.
  %Since $\x_2(t)$, $\v(t)$ are orthogonal to $H^{\y}_t$, 
  %from Lemma \ref{new:pf:lemma-3} it follows $\x_2(t)\bu_{\sigma}(t)$, $\v(t)\bu_{\sigma}(t)$ 
  %are  orthogonal to
  %$H_{t+1,\sigma}^{\y}$. 
  Similarly, since $\v(t)$ is orthogonal to $H_{t}^{\v}$,
  by Lemma \ref{new:pf:lemma-3} $\v(t)\bu_{\sigma}(t)$ is orthogonal to  $H_{t+1,\sigma}^{\v}$.
  Since by Lemma \ref{new:pf:lemma-4} the entries of $\x_i(t)\bu_{\sigma}(t)$, $i=1,2$  belong to
  $H_{t+1,\sigma}^{\v}$, it then follows that $\x_i(t)\bu_{\sigma}(t)$, $i=1,2$ and
  $\v(t)\bu_{\sigma}(t)$ are orthogonal.
  Notice that  $\e(t)=\y(t)-C\x_1(t)=C\x_2(t)+D\v(t)$.
  Hence, for all $\sigma \in \Sigma$, $E[\e(t)\e^T(t)\bu_{\sigma}^2(t)]=CE[\x_2(t)\x^T_2(t)\bu_{\sigma}^2(t)]C^T+DE[\v(t)\v^T(t)\bu_{\sigma}^2(t)]D^{T} > 0$, i.e. $\y$ is full rank.
  %the entries of $\x_{2},\y(t)$  belong to $H^{\v}_{t}$ and hence by Lemma \ref{new:pf:lemma-1}
  %the process $\z(t)=(\x_2(t),\y(t))$ is \RC.  Hence, 
  %$E[\z^{\z}(t+1)(\z^{\z}(t)_{w\sigma}(t+1))^T\bu_{\sigma}(t)]=E[\z(t)(\z_{w}^{\z}(t))^T]$ if $w\sigma \in L$ and $E[\z^{\z}(t+1)(\z^{\z}(t)_{w\sigma}(t+1))^T\bu_{\sigma}(t)]=0$
  %otherwise. But there exists a choice of row and column indices 
  %such $E[\x_{2}(t)\bu_{\sigma}(t)\z^{T}_{w\sigma}(t)]$ and $E[\x_2(t)\z_{w}(t)]$
  %are the sub-matrices of 
  %$E[\z^{\z}(t+1)(\z^{\z}(t)_{w\sigma}(t+1))^T\bu_{\sigma}(t)]$ and 
  %$E[\x_2(t)\z^T_{w}(t)]$ respectively which are formed by the intersections of
  %the rows and columns indexed by those row and column indices. 
  %Hence $E[\x_{2}(t)\bu_{\sigma}(t)\z^{T}_{w\sigma}(t)]=p_{\sigma}E[\x_2(t)\z^{T}_{w}(t)]$
  %if $w\sigma \in L$ and zero otherwise.
  %Since $E[\x_2(t)\z_{w}^T(t)]=0$ for all $w \in \Sigma^{+}$, it then follows that
  %$E[\x_2(t)\z_{w}^T(t)\bu_{\sigma}^{2}]=0$ for all $
  %$E[\x_{1}(t)\x^T_{2}(t)\bu_{\sigma}^2(t)]=0$ for all $\sigma \in \Sigma$.
  %Moreover, by Lemma \ref{}, $E[\x_{i}(t)\v^T(t)\bu_{\sigma}^{2}(t)]=0$. 
%  From
%  this it then follows that 
%  \[ 
%    E[\e(t)\e(t)\bu_{\sigma}^{2}(t)]=E[\x_2(t)\x_2^{T}(t)\bu_{\sigma}^{2}(t)+DQ_{\sigma}D^{T} 
%   \]
%  Since $DQ_{\sigma}D^T$ is strictly positive definite, it then follows that 
%  $E[\e(t)\e(t)\bu_{\sigma}^{2}(t)] > DQ_{\sigma}D^T$ is strictly positive definite too,
%  i.e. $\y$ is full rank.

\subsection{Proof of Theorem \ref{gen:filt:theo3.3}}
 The proof of the theorem is organized as follows. First, we present a number of
 properties of the state process
 $\x(t)$ and the innovation process $\e(t)$. Then we show the existence of the matrix
 $K_{\sigma}$, $\sigma \in \Sigma$.  Finally, we prove \eqref{gen:filt:eq2}.

 %\textbf{ Definition of the state process $\x(t)$}
 %\textbf{Definition of $\e(t)$}
 % Consider the following corollary of Lemma \ref{gen:filt:proof:lemma1}
 % The proof of Lemma \ref{gen:filt:proof:lemma2} can be found in Appendix \ref{proofs}.
  %Define
  %\[ \e(t)=\y(t)-C\x(t) \]
  %and hence we get the second equation of
  %\eqref{gen:filt:eq2} for free. 
  % Moreover, by definition of $\e(t)$,
  %$\e(t+k)$ is orthogonal to any $\z_{w}(t+k)$, $w \in L$.
  %hence by Assumption \ref{gen:filt:ass1:cond6}, $\e(t+k)$ and $\e(t)$
  %are uncorrelated for $k > 0$. For a more detailed argument see
  %Lemma \ref{gen:filt:proof:lemma9}.  Hence, the only real
  %difficulty is to show the existence of $K_{\sigma}$ such that
  %\eqref{gen:filt:eq2} and \eqref{gen:filt:eq3} hold.
 
 \textbf{Properties of $\x(t)$ and $\e(t)$}
  Below we present some important
  properties of $\x(t)$ and $\e(t)$ constructed above. The exposition is 
  organized as a series of lemmas.
  %We start with stating a  simple consequence of the definition is the following lemma,
  %which is interesting in its own right.
  \begin{Lemma}
  \label{gen:filt:proof:lemma1}
 %  Let $S \subseteq \Sigma^{+}$ be  an arbitrary subset and define
 %  \[ \x_{S}(t)=E[O_{R}^{-1}(Y_{n}(t)) \mid \{ \z_{w}(t) \mid w \in S\}].  \]
   For each $w \in \Sigma^{*}$ and $\sigma \in \Sigma$ such
   %that $\sigma w \in S$,
   \[ E[\x(t)\z^T_{w\sigma}(t)]=E[O_{R}^{-1}(Y_n(t))\z^T_{w\sigma}(t)]=A_{w}B_{\sigma}, \]
  where
 \( B_{\sigma}=\begin{bmatrix} B_{1,\sigma},& \ldots,& B_{p,\sigma} \end{bmatrix}
  \).
  \end{Lemma}
  \begin{proof}[Proof of Lemma \ref{gen:filt:proof:lemma1}]
   %In fact, we will prove a more general result, which will be useful for
   %for the proof of other results.
   %\begin{Lemma}
   %\label{gen:filt:proof:lemma1.1}
    %Let $S \subseteq \Sigma^{+}$ be  an arbitrary subset and define
    %\[ \x_{S}(t)=E[O_{R}^{-1}(Y_{n}(t)) \mid \{ \z_{w}(t) \mid w \in S\}].  \]
   %For each $w \in \Sigma^{*}$ and $\sigma \in \Sigma$ such
   %%that $\sigma w \in S$,
    %\[ E[\x(t)\z^T_{w\sigma}(t)]=E[O_{R}^{-1}(Y_n(t))z^T_{w\sigma}(t)]=A_{w}B_{\sigma}, \]
    %where
    %\( B_{\sigma}=\begin{bmatrix} B_{1,\sigma},& \ldots,& B_{p,\sigma} \end{bmatrix}
     %\).
   %Lemma \ref{gen:filt:proof:lemma1} follows from the statement above
   %by taking $S=\Sigma^{+}$. 
   %Below we proceed to prove the statement above.
   Notice that because of the properties of orthogonal projection
   %of $O^{-1}_R(Y_n(t))$ to the closure of the linear
   %span of $\z_{v}(t)$, $v \in S$, 
   it holds that
   \( \forall v \in \Sigma^{+}:
        E[\x(t)\z^T_{v}(t)]=E[O^{-1}_R(Y_n(t))\z_{v}^T(t)].
   \)
    Notice that the entries of $E[Y_n(t)\z^T_{v}(t)]$ are of the
    form
       %\begin{bmatrix} \z_{v}(t)\z^{f}_{v_0}(t) & \ldots & \z_{v}(t)\z^{f}_{v_{M(n-1)}} \end{bmatrix}^T.
   %\]
   %Hence, we have to analyze 
   $E[(\z^f_{v_i})^T(t)\z^T_{v}(t)]$ 
   $i=0,\ldots,M(n-1)$. 
   By writing out the definition of 
   $\z^{f}_{s}$ and $\z_{v}(t)$, it follows that
   for any $v,s \in \Sigma^{+}$,
   %To this end, assume that
   %$s=\sigma_{k+1}\cdots \sigma_{k+l}$ and
   %$v=\sigma_{1}\cdots \sigma_k$, $k,l > 0$,
   %$\sigma_1,\ldots,\sigma_k \in \Sigma$.
   \[ 
    \begin{split}
     & E[(\z^{f}_{s})^T(t)\z^T_{v}(t)]= 
     %& E[\y(t+l)\y^T(t-k)\bu_{\sigma_1}(t-k)\bu_{\sigma_2}(t-k+1)\cdots \bu_{\sigma_k}(t-1)\bu_{\sigma_{k+1}}(t)\cdots \bu_{\sigma_{k+l}}(t+l-1)]= \\
      E[\y(t+l)\z_{vs}^T(t+l)]=\Lambda^{\y}_{vs}
    \end{split}.
  \]
   Hence,  applying the result above to 
   $v=\sigma w$ and noticing that
   %\[ 
   %  E[Y_n(t)\z_{\sigma w}^T(t)]=\begin{bmatrix}
   %     \Lambda_{\sigma w v_0}^T & \ldots & \Lambda_{\sigma v v_{M(n-1)}}^T
   %     \end{bmatrix}^T.
   %\]
   %Notice that 
   \( \Lambda^{\y}_{\sigma w v_i} = CA_{v_i}A_{w}B_{\sigma} \) we
   obtain
   \(  E[Y_n(t)\z_{\sigma w}(t)]=O_{R}A_{w}B_{\sigma}. \)
  From this, by taking into account that 
  $E[\x(t)\z_{\sigma w}^T(t)]=O_R^{-1}E[Y_n(t)\z_{w}(t)]$, the statement of
 the lemma follows.
  %\[ E[\x(t)\z_{\sigma w}^T(t)]=
     %O_{R}^{-1}E[Y_n(t)\z_{w}(t)]=A_{w}B_{\sigma}.
  %\]
  \end{proof}
   Lemma \ref{gen:filt:proof:lemma1} explains the relationship between
  states of the would-be generalized bilinear realization and the states of
  the rational representation $R$ of $\Psi_{\y}$.  In particular, it yields the
  following corollary.

  \begin{Corollary}
  \label{gen:filt:proof:lemma2}
   With the notation of Lemma \ref{gen:filt:proof:lemma1},
   %$C\x_{S}(t)=E_{l}[\y(t) \mid \{ \z_{w}(t) \mid w \in S\} ]$. 
   %In particular, for $S=\Sigma^{+}$, 
   \( C\x(t)=E_{l}[\y(t) \mid \{ \z_{w}(t) \mid w \in \Sigma^{+}\}]. \)
  \end{Corollary}
  \begin{proof}[Proof of Corollary \ref{gen:filt:proof:lemma2}]
   Indeed, from Lemma \ref{gen:filt:proof:lemma1} it follows that for
   any $v \in \Sigma^{+}$ of the form $v=\sigma w$ for some $w \in \Sigma^{*}$,
   $\sigma \in \Sigma$: 
   \( 
      E[C\x(t)\z^T_{\sigma w}(t)]=CA_{w}B_{\sigma}
      =\Lambda^{\y}_{\sigma w}=E[\y(t)\z^T_{\sigma w}(t)],
   \)
  and hence for any $v \in \Sigma^{+}$, 
  $E[(\y(t)-C\x(t))\z_{v}^T(t)]=0$, i.e. the entries of 
  $\y(t)-C\x(t)$ are orthogonal to the Hilbert-space generated by $\{\z_{v}(t) \mid v \in \Sigma^{+}\}$. Since the entries $C\x(t)$, obviously belong to that Hilbert-space, the corollary follows.
  \end{proof}
 The corollary above says that $C\x(t)$ is the projection of the current
 output to past outputs and inputs.
  \begin{Lemma}
  \label{gen:filt:proof:lemma9.4}
   The processes $\x(t)$ and $\e(t)=\v(t)$ satisfy Part \ref{gbs:def:prop1}--\ref{gbs:def:prop2} of Assumption \ref{gbs:def}. Moreover, $\bs(t)=\begin{bmatrix} \x^T(t),\y^T(t),\e^T(t) \end{bmatrix}^T$ is \RC.
  \end{Lemma}
 \begin{proof}[Proof of Lemma \ref{gen:filt:proof:lemma9.4}]

%  From the construction of $\x(t)$ and $\e(t)$ it follows that
%  $E[\z^{\x}_{w}(t)\z^{T}_{v}(t)]$ and $E[\z^{\e}_{w}(t)\z^{T}_{v}(t)]$
%  $w,v \in \Sigma^{+}$ do not depend on $t$. 
%  begin{proof}[Proof of Lemma \ref{gen:filt:proof:lemma9.4}]
  From Lemma \ref{gen:filt:proof:lemma1} it follows that
  $E[\x(t)(\z^{\y}_{v}(t))^T]$, 
  $v \in \Sigma^{+}$ does not depend on $t$. 
  By noticing that
  the entries of $\x(t)$ belong to $H_{t}^{\y}$ and applying Lemma \ref{new:pf:lemma-1}, it then follows that
  $\bs_1(t)=\begin{bmatrix} \x^T(t), & \y^T(t) \end{bmatrix}^T$ is \RC.
  By noticing that $\bs(t)=\begin{bmatrix} I_n & 0 \\ 0 & I_p \\ -C & I_p \end{bmatrix} \bs_1(t)$, it follows that $\bs(t)$ is \RC. 

 \textbf{Proof of Part \ref{gbs:def:prop1} of Assumption \ref{gbs:def}}
  Since $\x(t)$ and $\e(t)$ are components of $\bs(t)$, it follows that 
  $\x(t)$ and $\e(t)$ are \RC.
  %From Lemma \ref{new:pf:lemma-2} and
  %$\e(t)=\y(t)-C\x(t)$ it follows that the entries of $\e(t)$ belong to
  %$H_{t+1}^{\y}$ and hence by Lemma \ref{new:pf:lemma-1} $\e(t)$ is an \RC\ process. 

 \textbf{Proof of Part \ref{gbs:def:prop2} of Assumption \ref{gbs:def}}
  Assume that $w,v \in \Sigma^{+}$. Assume first that $|w| > |v|$ and $w=sv$ for
  some $s \in \Sigma^{+}$.
  Since the coordinates of $\x(t)$ belong to $H_t$, 
  from Lemma \ref{new:pf:lemma-4} it follows that the entries of $\z^{\x}_{w}(t)$ belong to $H_{t,w}^{\y} \subseteq H_{t,v}^{\y}$.
  From the construction of  $\e(t)=\y(t)-C\x$ it follows that $\z^{\e}_w(t)=\z_{w}(t)-C\z^{\x}_w(t)$. Hence, as the coordinates of 
  $\z_w(t)$ belong to $H_{t,v}^{\y}$, the coordinates of $\z^{\e}_w(t)$ belong to $H_{t,v}^{\y}$.
  Note that the coordinates of $\e(t)$ are orthogonal to $H_t^{\y}$. 
  Moreover, recall that $\bs(t)=\begin{bmatrix} \x^T(t), & \e^T(t), & \y^T(t) \end{bmatrix}^T$ is \RC. 
  By applying Lemma \ref{new:pf:lemma-3} to $\z(t)=\e(t)$, $\bh(t)=\y(t)$, it follows that $\z^{\e}_{v}(t)$ is
  orthogonal to $H_{t,v}^{\y}$. 
  Hence, it follows that $E[\z^{\e}_{v}(t)(\z^{\e}_{w}(t))^T]=0$.
  If $|w| > |v|$ but $w$ does not end with $v$, then from the fact that $\e(t)$ is \RC\ and
  Part \ref{RC4} of Definition \ref{def:RC} it follows that 
  $E[\z^{\e}_{v}(t)(\z^{\e}_{w}(t))^T]=0$. If $|w| < |v|$, then $E[\z^{\e}_{v}(t)(\z^{\e}_{w}(t))^T]=0$ follows from
  the discussion above by considering the transpose of $E[\z^{\e}_{w}(t)(\z^{\e}_{v}(t))^T]$. If $|w|=|v|$ but $w \ne v$, $E[\z^{\e}_{v}(t)(\z^{\e}_{w}(t))^T]=0$
  follows from the fact that $\e(t)$ is \RC, by repeated application of Part \ref{RC4} of Definition \ref{def:RC} to $\br(t)=\e(t)$. 
 \end{proof}
 \begin{Lemma}
  \label{gen:filt:proof:lemma9.1}
   For any $w \in \Sigma^{+}$ and $\sigma \in \Sigma$ such that
   $w\sigma \in L$,  $E[\x(t)\z_{w}^T(t)\bu_{\sigma}^{2}(t)]=p_{\sigma}E[\x(t)\z^T_{w}(t)]$.
  \end{Lemma} 
  %For the proof of Lemma \ref{gen:filt:proof:lemma9.1} and some of the subsequent results 
  %we will need the following lemma.
  %In order to prove the rest of the lemma, we need the following lemma.
 \begin{proof}[Proof of Lemma \ref{gen:filt:proof:lemma9.1}]
  From Lemma \ref{new:pf:lemma-1} it follows that
  $\bs(t)=(\x^T(t),\y^T(t))^T$ is a \RC\ process and hence
  $T_{\sigma,w\sigma}^{\bs}=\Lambda^{\bs}_{w}$ 
  for $w\sigma \in L$.   Since
  $E[\x(t)\z_{w}^T(t)\bu_{\sigma}^{2}(t)]$ and $E[\x(t)\z^T_{w}(t)]$  are
  the sub-matrices of $p_{\sigma}T_{\sigma, w\sigma}^{\bs}$ and respectively $\Lambda^{\bs}_{w}$, the statement of the 
lemma follows.
 \end{proof}

 \textbf{Definition of $K_{\sigma}$}
  In order to define $K_{\sigma}$, we need the following auxiliary result.
  %In order to show existence of $K_{\sigma}$ which satisfies
  %\eqref{gen:filt:eq2} and \eqref{gen:filt:eq3}, we proceed as follows.  
  %Recall the notation of Lemma \ref{gen:filt:proof:lemma9.2}, i.e.
  %Let $H_t^{\sigma}$ be the Hilbert-space generated by
  %$\{\z_{w}(t)\bu_{\sigma}(t)=\z_{w\sigma}(t+1) \mid w \in \Sigma^{+}\}$.
  %By noticing that $\z_{w}(t)\bu_{\sigma}(t)=\z_{w\sigma}(t+1)$ 
  %it then follows that
  \begin{Lemma}
  \label{new:pf:lemma1}
   \begin{equation}
    H_{t+1}^{\y}=\bigoplus_{\sigma \in \Sigma} H_{t,\sigma}^{\y} \oplus \bigoplus_{\sigma \in \Sigma} <\e(t)\bu_{\sigma}(t)> 
   \end{equation}
   where $\bigoplus$ denotes the direct sum and $<\e(t)\bu_{\sigma}(t)>$ denoted the
   Hilbert-space generated by the entries of $\e(t)\bu_{\sigma}(t)$.
   Here we used Notation \ref{new:pf:not1}--\ref{new:pf:not2}.
  \end{Lemma}
  \begin{proof}[Proof of Lemma \ref{new:pf:lemma1}]
   Indeed, from the definition of $H_{t+1}^{\y}$ and
   Lemma \ref{new:pf:lemma-1} it is clear that $H_{t+1}^{\y}$ is the closure of
   the space $\sum_{\sigma \in \Sigma}(H_{t,\sigma}^{\y} + <\e(t)\bu_{\sigma}(t)>)$.
   From Lemma \ref{new:pf:lemma-4} it follows that
   $H_{t,\sigma}^{\y}$ and $H_{t,\hat{\sigma}}^{\y}$ are orthogonal for all $\sigma,\hat{\sigma} \in \Sigma$. 
  %Since $H_t^{\sigma} \subseteq H_{t+1,\sigma}^{\y}$ and
  %$H_t^{\\hat{sigma}} \subseteq H_{t+1,\hat{\sigma}}^{\y}$, it follows that
   From 
   Lemma \ref{new:pf:lemma-3} it follows that
   $H_{t,\sigma}$, $<\e(t)\bu_{\sigma}(t)>$ are orthogonal. Finally,
   we will show that $H_{t+1,\sigma}$, $<e_{t}\bu_{\hat{\sigma}}(t)>$ are
   orthogonal for $\sigma \ne \hat{\sigma}$. To this end, notice that
   for all $\sigma,\hat{\sigma} \in \Sigma$, $\e(t)\bu_{\hat{\sigma}}(t) \in H_{t+1,\hat{\sigma}}^{\y,*}$ and $H_{t+1,\sigma}^{\y} \subseteq H_{t+1,\sigma}^{\y,*}$.
   From Lemma \ref{new:pf:lemma-4}
   it then follows that $H_{t+1,\sigma}^{\y,*}$ and $H_{t+1,\hat{\sigma}}^{\y,*}$ are 
   orthogonal for each $\sigma,\hat{\sigma}$. Hence, $\e(t)\bu_{\hat{\sigma}}(t)$ is 
   orthogonal to $H_{t+1,\sigma}^{\y} \subseteq H_{t+1,\sigma}^{\y,*}$. 
   Since all the involved spaces 
   $H_{t+1,\sigma}^{\y}$ and $<\e(t)\bu_{\sigma}(t)>$ are mutually orthogonal, the closure of
   their sum equals their direct sum.
  \end{proof}
  From Lemma \ref{new:pf:lemma1} it follows that 
  \begin{equation}
  \label{new:pf:eq1}
   \begin{split}
     & \x(t+1)=%E[O_R^{-1}(Y_n(t+1))\mid H_{t+1}]= 
     \sum_{\sigma \in \Sigma} 
     E_l[O_R^{-1}(Y_n(t+1))\mid H_{t}^{\sigma}]+E_l[O_R^{-1}(Y_n(t+1))\mid <\e(t)\bu_{\sigma}(t)>]
   \end{split}
  \end{equation}
  Define now $K_{\sigma}$ as a $n \times p$ matrix such that
  \[ E_l[O_R^{-1}(Y_n(t+1))\mid <\e(t)\bu_{\sigma}(t)>]=K_{\sigma}\e(t)\bu_{\sigma}(t). \]
  If $\y$ is full rank, then $K_{\sigma}$ is unique and 
  \( K_{\sigma}=
     E[O_R^{-1}(Y_n(t+1))\e^T(t)\bu_{\sigma}](E[\e(t)\e^T(t)\bu_{\sigma}^2(t)])^{-1}.
  \)

 \textbf{Proof of \eqref{gen:filt:eq2}}
   The second equation of \eqref{gen:filt:eq2} follows directly from the definition of
   $\e(t)=\y(t)-C\x(t)$.
   We will concentrate on the first equation. To this end, we will show that
   \begin{equation} 
   \label{new:pf:eq2}
     E_l[O_R^{-1}(Y_n(t+1))\mid H_{t+1,\sigma}^{\y}]=\frac{1}{\sqrt{p_{\sigma}}} A_{\sigma}\x(t)\bu_{\sigma}(t).   \end{equation}
   From this and from \eqref{new:pf:eq1} the first equation of \eqref{gen:filt:eq2} follows.
   From Lemma \ref{new:pf:lemma-4} and the fact that
   the entries of $\x(t)$ belong to $H_{t}$ it follows that the entries of
  $\x(t)\bu_{\sigma}(t)$ 
   belongs to $H_{t+1,\sigma}^{\y}$.  Hence, by Lemma \ref{new:pf:lemma1} 
   in order to show that \eqref{new:pf:eq2}, it is enough to show that $w \in \Sigma^{+}$, $\sigma \in \Sigma$, 
   \begin{equation}
   \label{new:pf:eq3}
      E[O_R^{-1}(Y_n(t+1))\z^T_{w\sigma}(t+1)]=\frac{1}{\sqrt{p_{\sigma}}}A_{\sigma}E[\x(t)\bu_{\sigma}(t)\z^T_{w\sigma}(t+1)].
    \end{equation}
   If $w\sigma \notin L$, then $\z^T_{w\sigma}(t+1)=0$ and hence \eqref{new:pf:eq3}
   trivially holds. Hence, in the sequel we assume that $w\sigma \in L$.
   Then from Lemma \ref{gen:filt:proof:lemma9.1} and Lemma \ref{gen:filt:proof:lemma1}
   it follows that
   \begin{equation} 
   \label{new:pf:eq3.1}
        \frac{1}{\sqrt{p}_{\sigma}}A_{\sigma}E[\x(t)\bu_{\sigma}(t)\z^T_{w\sigma}(t+1)]=\frac{1}{\sqrt{p_{\sigma}}}A_{\sigma}(\sqrt{p_{\sigma}}E[\x(t)\z^T_{w}(t)])=
        A_{\sigma}A_{w}B_{\hat{\sigma}}.
   \end{equation}
   where $\hat{\sigma}$ is the first letter of $w$.
   %In order to prove \eqref{new:pf:eq3}, we will use the following lemma, which is
   %interesting on its own right.
   By applying Lemma \ref{gen:filt:proof:lemma1} to $t+1$ instead of $t$, we obtain
   %that if $w=\hat{\sigma}v \in \Sigma$, then
   %\begin{equation}  
   %\label{new:pf:eq3.2}
   \(  E[O_R^{-1}(Y_n(t+1))\z^T_{w\sigma}(t+1)]=A_{\sigma}A_{v}B_{\hat{\sigma}} \). 
   %\end{equation}
   Hence, by combining this with \eqref{new:pf:eq3.1}, 
   \eqref{new:pf:eq3} follows.
%%   recall from Lemma \ref{gen:filt:proof:lemma9.1}  that
%%   \begin{equation}
%%   \label{new:pf:eq4}
%%       \frac{1}{p_{\sigma}}A_{\sigma}E[\x(t)\bu_{\sigma}(t)\z^T_{\hat{\sigma}v\sigma}(t+1)]=
%%       A_{\sigma}A_{v}B_{\hat{\sigma}}
%%   \end{equation}
%%   for all $v \in \Sigma^{*}$, $\hat{\sigma} \in \Sigma$. To this end, 
%%   recall from Lemma \ref{gen:filt:proof:lemma9.1}  that
%%   \begin{equation}
%%   \label{new:pf:eq5}
%%    E[\x(t)\bu_{\sigma}(t)\z^T_{\hat{\sigma}v\sigma}(t+1)]=p_{\sigma}E[\x(t)\z^T_{\hat{\sigma}v}(t)].
%%   \end{equation}
%%   From \eqref{new:pf:eq5} and Lemma \ref{gen:filt:proof:lemma1} it then follows that
%%   $E[\x(t)\z^T_{\hat{\sigma}v\sigma}(t+1)]=p_{\sigma}A_{v}B_{\hat{\sigma}}$, and hence
%%   \eqref{new:pf:eq4} follows.

\textbf{Proof of \eqref{gen:filt:eq3.1}, \eqref{gen:filt:eq3}} % \eqref{gen:filt:eq32}}
   Notice that on the one hand, by Lemma \ref{gen:filt:proof:lemma1}
   %\begin{equation} 
   %\label{new:pf:eq7}
   \(   E[\x(t+1)\z^T_{\sigma}(t+1)]=B_{\sigma}  \)
   %\end{equation}
   and on the other hand, if we use \eqref{gen:filt:eq2},
   \begin{equation}
   \label{new:pf:eq8}
      E[\x(t+1)\z^T_{\sigma}(t+1)]=\frac{1}{\sqrt{p_{\sigma}}} A_{\sigma}E[\x(t)\bu_{\sigma}\z^T_{\sigma}(t+1)]+K_{\sigma}E[\e(t)\bu_{\sigma}(t)\z_{\sigma}^T(t+1)].
   \end{equation}
   Here we used the corollary of Lemma \ref{new:pf:lemma1} that 
   $\x(t)\bu_{\sigma}(t)$, $\e(t)\bu_{\sigma}(t)$ are orthogonal to
   $\z_{\hat{\sigma}}(t+1)$ for $\hat{\sigma} \ne \sigma$, $\sigma,\hat{\sigma} \in \Sigma$ and that
   %From Lemma \ref{new:pf:lemma1} it also follows that
   $\x(t)\bu_{\sigma}(t)$ and $\e(t)\bu_{\sigma}(t)$ are orthogonal. 
  Indeed, from Lemma \ref{new:pf:lemma1} it follows that
   $\x(t)\bu_{\sigma}(t)$ and $\e(t)\bu_{\sigma}(t)$ are orthogonal. 
   From Lemma \ref{new:pf:lemma-4} it follows that the entries of $\x(t)\bu_{\sigma}(t)$ belong to $\mathcal{H}^{\y}_{t+1,\sigma} \subseteq \mathcal{H}^{\y,*}_{t+1,\sigma}$ and
   the entries of  $\e(t)\bu_{\sigma}(t)=\y(t)\bu_{\sigma}(t)-C\x(t)\bu_{\sigma}(t)$ belong to $\mathcal{H}^{\y,*}_{t+1,\sigma}$.
   The entries of $\z^{\y}_{\hat{\sigma}}(t+1)$ belong to $\mathcal{H}^{\y, *}_{t+1,\hat{\sigma}}$. From Lemma \ref{new:pf:lemma-4} it then follows that
   the spaces $\mathcal{H}^{\y,*}_{t+1,\sigma}$,  $\mathcal{H}^{\y,*}_{t+1,\hat{\sigma}}$ are orthogonal. 
   Using this remark and the equality $\z_{\sigma}(t+1)=\frac{1}{\sqrt{p_{\sigma}}}\y(t)\bu_{\sigma}(t)=\frac{1}{\sqrt{p_{\sigma}}}(C\x(t)\bu_{\sigma}(t)+\e(t)\bu_{\sigma}(t))$, we obtain
   \begin{equation}
   \label{new:pf:eq9}
      E[\x(t)\bu_{\sigma}(t)\z^T_{\sigma}(t+1)]=\frac{1}{\sqrt{p_{\sigma}}}P_{\sigma}C^{T}.
   \end{equation}
   %From $\y(t)=C\x(t)+\e(t)$ and orthogonality of $\x(t)\bu_{\sigma}(t)$ and $\e(t)\bu_{\sigma}(t)$ we can derive
   In addition, from the discussion above it follows that
  \(
    T_{\sigma,\sigma}=E[\z_{\sigma}(t)\z_{\sigma}^T(t)]=
     \frac{1}{p_{\sigma}}CP_{\sigma}C^T+\frac{1}{p_{\sigma}}E[\e(t)\e^T(t)\bu_{\sigma}^2(t)])
  \)
  and hence
  \begin{equation}
   \label{new:pf:eq10}
   E[\e(t)\e^T(t)\bu_{\sigma}^2(t)]=p_{\sigma} T_{\sigma,\sigma}-CP_{\sigma}C^T.
  \end{equation}
  Combining \eqref{new:pf:eq9},\eqref{new:pf:eq10} and \eqref{new:pf:eq8} yield
  \[ B_{\sigma}= \frac{1}{p_{\sigma}}A_{\sigma}P_{\sigma}C^T+\frac{1}{\sqrt{p_{\sigma}}}K_{\sigma}(p_{\sigma} T_{\sigma,\sigma}-CP_{\sigma}C^T),
  \]
  from which \eqref{gen:filt:eq3.1} follows easily. If $\y$ is full rank,
  then the existence of the inverse of
  $(p_{\sigma} T_{\sigma,\sigma}-CP_{\sigma}C^T)$ follows from \eqref{new:pf:eq10} and 
  the invertibility of $E[\e(t)\e^T(t)\bu_{\sigma}^2(t)]$.
  If the inverse of $(p_{\sigma} T_{\sigma,\sigma}-CP_{\sigma}C^T)$ exists, then 
  \eqref{gen:filt:eq3.1} implies \eqref{gen:filt:eq3}.
  %Finally, in order to show \eqref{gen:filt:eq32}, notice that $\BS_{R}$ satisfies
  %Assumption \ref{gbs:def}. Hence, by Lemma \ref{gbs:def:new_lemma1.3}, 
  %$P_{\sigma}$ is the unique solution of \eqref{gen:filt:theo3.2:eq1}, where
  %$Q_{\sigma}=E[\e(t)\e^T(t)\bu_{\sigma}^2(t)]$. By substituting 
  %\eqref{new:pf:eq10} and \eqref{gen:filt:eq3} into \eqref{gen:filt:theo3.2:eq1}, we obtain
  %\eqref{gen:filt:eq32}.
 %\end{proof}
%  We conclude by the stating a collorary of Theorem \ref{} which will be useful for
%  formulating a realization algorithm.
% \begin{Corollary}[Computation of $P_{\sigma}$ and $K_{\sigma}$]
%  With the notation of Theorem \ref{}, $P_{\sigma}$, $\sigma \in \Sigma$ is the
%  unique solution of 
%  \[ P_{\sigma}=\frac{1}{p_{\sigma}} A_{\sigma}  
%% \end{Corollary}

\textbf{Proof that the system satisfies Assumption \ref{gbs:def}}
  We have already shown that Parts \ref{gbs:def:prop1}--\ref{gbs:def:prop2}
  of Assumption \ref{gbs:def} are satisfied.  
 Since $R$ is a minimal representation
  of $\Psi_{\y}$ and $\Psi_{\y}$ is square-summable, from Theorem \ref{hscc_pow_stab:theo2}
  it follows that $\sum_{\sigma \in \Sigma} A^T_{\sigma} \otimes A^T_{\sigma}$ is
  stable, 
  %Finally, $D$ is the idenity matrix. Since $\y$ is full rank, 
  %$E[\e(t)\e(t)\bu_{\sigma}^{2}(t)]$ is strictly positive definite, hence
  %Part \ref{gbs:def:prop6} of Assumption \ref{gbs:def} holds too. 
  i.e. Part \ref{gbs:def:prop5} of Assumption \ref{gbs:def} holds.
  In order to show that Part \ref{gbs:def:prop7} of Assumption \ref{gbs:def} holds, let 
  $\sigma_1,\sigma_2 \in \Sigma$ such that $\sigma_1\sigma_2 \notin L$.
  Notice that for all  $\sigma \in \Sigma$, $v,w \in \Sigma^{*}$,
  $CA_wA_{\sigma_2}A_{\sigma_1}A_vB_{\sigma}=\Lambda_{\sigma v\sigma_1\sigma_2w}$.
  Notice that if $\sigma_1\sigma_2 \notin L$, then $\sigma v\sigma_1\sigma_2w \notin L$, and hence
  $CA_wA_{\sigma_2}A_{\sigma_1}A_vB_{\sigma}=\Lambda_{\sigma v\sigma_1\sigma_2w}=0$. Since $w$ is arbitrary, it then follows that
  $A_{\sigma_2}A_{\sigma_1}A_vB_{\sigma} \in O_R$. As $R$ is observable, i.e. $O_R=\{0\}$,
  it follows that $(A_{\sigma_2}A_{\sigma_1})A_vB_{\sigma}=0$.
  Since $v$ and $\sigma$ are arbitrary, the latter implies that $A_{\sigma_2}A_{\sigma_1}W_R=0$.  As $R$ is reachable, i.e. $W_R=\mathbb{R}^n$,
  it then follows that $A_{\sigma_2}A_{\sigma_1}=0$.
  With a similar reasoning, we can show that if $\sigma_1\sigma_2 \notin L$, then for any $w \in \Sigma^{*}$,
  $CA_wA_{\sigma_2}B_{\sigma_1}=\Lambda_{\sigma_1\sigma_2w}=0$. Hence, observability of $R$ implies $A_{\sigma_2}B_{\sigma_1}=0$.
  Finally, from \eqref{gen:filt:eq3.1}, $A_{\sigma_2}A_{\sigma_1}=0$,  $A_{\sigma_2}B_{\sigma_1}=0$ it follows that
  $A_{\sigma_2}K_{\sigma_1}Q_{\sigma_1}=0$ and thus by recalling that  $T_{\sigma_1,\sigma_1}^{\e}=\frac{1}{p_{\sigma_1}}Q_{\sigma_1}$,
   Part \ref{gbs:def:prop7} of Definition \ref{gbs:def} follows. 

  It is left to show that  Part \ref{gbs:def:prop4} of Assumption \ref{gbs:def} holds.
  To this end, we have to show that 
  for all $v,w \in \Sigma^{+}$, $|w| \ge |v|$, $E[\z^{\x}_{w}(t)(\z^{\e}_v(t))^T]=0$. 
  In order to show that  for all $v,w \in \Sigma^{+}$, $|w| \ge |v|$, $E[\z^{\x}_{w}(t)(\z^{\e}_v(t))^T]=0$, we proceed as follows. Since $|w| \ge |v|$, 
  $w=s\hat{v}$ for some $s \in \Sigma^*$, $\hat{v} \in \Sigma^{+}$, $|\hat{v}| = |v|$. 
  From Lemma \ref{gen:filt:proof:lemma9.4} it follows that $\br(t)=\begin{bmatrix} \x^T(t), & \e^T(t) \end{bmatrix}^T$ is \RC.
  Moreover, $E[\z^{\x}_{w}(t)(\z^{\e}_v(t))^T]$ is a suitable sub-matrix of $T^{\br}_{w,v}$. 
  If $\hat{v} \ne v$, then from  Part \ref{RC4} of Definition \ref{def:RC} it follows that $T^{\br}_{w,v}=0$ and hence $E[\z^{\x}_{w}(t)(\z^{\e}_v(t))^T]=0$.
  Assume now that $v=\hat{v}$, i.e. $w=sv$.   %Furthermore, notice that $H_t^{\y} \subseteq  H_{t}^{\br}$ (since $\y(t)=C\x(t)+\e(t)$), and 
  Recall that he entries of $\e(t)$ are orthogonal to
  $H_t^{\y}$ (since $\e(t)=\y(t)-E_{l}[\y(t) \mid H_t^{\y}]$).  Moreover,  from Lemma \ref{gen:filt:proof:lemma9.4} it follows that $(\y^T(t),\e^T(t))^T$ is \RC.
  Hence, the conditions of Lemma \ref{new:pf:lemma-3} are satisfied for $\z(t)=\e(t)$, $\bh(t)=\y(t)$. Therefore, 
  the entries of $\z_v^{\e}(t)$ are orthogonal to $H_{t,v}^{\y}$. 
  Since   the entries of $\x(t)$ belong to $H_{t}^{\y}$,
  from Lemma \ref{new:pf:lemma-4} it follows that the entries of $\z_w^{\x}(t)$ belong to  $H_{t,w}^{\y}$.
  Notice that if $w=sv$, then $H_{t,w}^{\y} \subseteq H_{t,v}^{\y}$. Hence in this case the entries of $\z_v^{\e}(t)$ are
  orthogonal to $H_{t,w}^{\y}$ and thus to the entries of $\z_w^{\x}(t)$. 
  Using that $\e(t),\x(t)$ are \RC, \eqref{gen:filt:eq2} holds and $\v(t)=\e(t)$ satisfies  Part \ref{gbs:def:prop2}
  of Assumption \ref{gbs:def}, and $\sum_{\sigma \in \Sigma} A^T_{\sigma} \otimes A^T_{\sigma}$ is  stable and that the system satisfies Part \ref{gbs:def:prop5} of Assumption \ref{gbs:def}, we can show that
  \begin{equation} 
  \label{newlimeq1}
    \x(t)=\lim_{k \rightarrow \infty} \sum_{w \in \Sigma^{*}, |w| \le k-1} \sum_{\sigma \in \Sigma} A_{w}B_{\sigma}\z_{\sigma w}^{\e}(t), 
   \end{equation}
   where the limit is taken in the mean square sense.   From \eqref{newlimeq1}
   Part \ref{gbs:def:prop4} of Assumption \ref{gbs:def} follows.
   In order to show  \eqref{newlimeq1}, we can  use a reasoning similar to the  proof of
   Lemma \ref{gbs:def:new_lemma1}. To this end, notice that \eqref{gen:filt:eq2} implies that
  \[
    \x(t)=\sum_{w \in \Sigma^{+},|w|=k} A_{w}\z^{\x}_{w}(t)+ \\
          \sum_{w \in \Sigma^{*}, |w| \le k-1} \sum_{\sigma \in \Sigma} A_{w}K_{\sigma}\z_{\sigma w}^{\v}(t),
  \]
  Hence, it is enough to show that $r_k(t)=\sum_{w \in \Sigma^{+},|w|=k} A_{w}\z^{\x}_{w}(k)$ converges to zero in the mean square sense.
  Like in the proof of  Lemma \ref{gbs:def:new_lemma1}, it can be shown that $E[r_k(t)r_k^{T}(t)]=\sum_{s \in \Sigma^{*}, |s|=k-1} \sum_{\sigma \in \Sigma }  A_s S A_{s}^T=\mathcal{Z}^{k-1}(S)$, where $S=\sum_{\sigma \in \Sigma} A_{\sigma} T^{\x}_{\sigma,\sigma}A_{\sigma}^T$,
  where $\mathcal{Z}$ is the linear map on $\mathbb{R}^{n \times n}$ defined by
  $\mathcal{Z}(V)=\sum_{\sigma \in \Sigma} A_{\sigma}VA_{\sigma}^T$. Since 
  $\sum_{\sigma \in \Sigma} A^T_{\sigma} \otimes A^T_{\sigma}$ is the matrix representation of $\mathcal{Z}$ with respect to the basis  defined in 
  \cite[Section 2.1]{CostaBook},
  \cite[Proposition 2.5]{CostaBook} implies that $\lim_{k \rightarrow \infty} E[||r_k(t)||^{2}]=\lim_{k \rightarrow \infty} \mathrm{trace} E[r_k(t)r^T_k(t)]=\lim_{k \rightarrow \infty} \mathrm{trace} \mathcal{R}_{k-1}(S)=0$.

\subsection{Proof of Theorem \ref{gen:filt:min_theo}}
\label{gbs:min}
  Assume that $\BS$ is a minimal minimal realization of $\y$ and $\BS$ satisfies Assumption \ref{gbs:def}. Then by Theorem 
  \ref{gen:filt:theo3.2},  $R_{\BS}$ is a representation of $\Psi_{\y}$ and
  $\y$ satisfies Assumptions \ref{output:assumptions}. Assume that $R_{\BS}$ is not a minimal
  representation of $\Psi_{\y}$.
  Consider a minimal representation  $R$ of $\Psi_{\y}$. From Theorem \ref{gen:filt:theo3.3}
  it then follows that there exists a \GBS\ realization $\BS_R$ of $\y$ such that
  the dimension of $\BS_R$ equals $\dim R$ and $\BS_R$ satisfies Assumption \ref{gbs:def}. 
  From the construction of $R_{\BS}$ it
  follows that $\dim \BS=\dim R_{\BS}$, hence $\dim \BS < \dim R=\dim \BS_{R}$.
  This contradicts to the minimality of $\BS$, hence $R_{\BS}$ has to be minimal.
  Conversely, assume that $R_{\BS}$ is minimal, and consider a \GBS\ realization
  $\BS_1$ of $\y$ such that $\BS_1$ satisfies Assumption \ref{gbs:def}. Then $\dim \BS_1 = \dim R_{\BS_1} \le \dim R_{\BS}=\dim \BS$.
  Since $\BS_1$ was an arbitrary realization of $\y$, it then follows that $\BS$
  is a minimal realization of $\y$. 
  The second statement of the theorem is a direct consequence of the first one and of
  Theorem \ref{sect:form:theo1}.

   %The proof of Lemma \ref{gen:filt:proof:new:lemma1} is routine and

\subsection{Proof of the results related to the realization algorithm}
   Below we present the proofs of Lemmas \ref{gen:filt:proof:new:lemma1} -- \ref{gen:filt:proof:new:lemma2} and Theorem \ref{alg2:theo}.
   To this end, we will need the following auxiliary result.
   \begin{Lemma}
   \label{gen:filt:proof:lemma3.1}
   Let $M_N$ be a sequence of closed subspaces such that
   $M_{N} \subseteq M_{N+1}$ and let $M$ be the closure of
   the space $\bigcup_{k=1}^{\infty} M_k$. Let $h$ be
   a mean-square integrable scalar random variable, and
   let $\z_N=E_{l}[h \mid M_N]$ and $\z=E_{l}[h \mid M]$.
   Then $\lim_{N \rightarrow \infty} \z_N=\z$ in the
   mean-square sense.
  \end{Lemma}
 \begin{proof}[Proof of Lemma \ref{gen:filt:proof:lemma3.1}]
   It is clear that if $\lim_{N \rightarrow \infty} \z_N$ exists and equals
   $z$, then $z=E_{l}[h \mid M]$. Indeed,
   $z=E_{l}[h \mid M]$ if and only if $h-z$ is orthogonal to $M$.
   If $z=\lim_{N \rightarrow} \z_N$, then, since $h-z_N$ is orthogonal to $M_N$,
   it follows that $h-z$ is orthogonal to $M_N$ for all $N$. Hence, $h-z$ is orthogonal
   to $M$, as the latter is the closure of $\bigcup_{N=0}^{\infty} M_N$.

   In order to show that $\z_N$ is convergent in the mean-square sense, it is enough to
   show that $z_N$ is a Cauchy sequence.
   To this end, define $d_{N}=||h-\z_N||$ and
   notice that for any $s \in M_{N}$, 
   $d_{N} \le ||h-s||$, due the the well-known properties
   of orthogonal projections onto $M_{N}$. Since
   $M_{N} \subseteq M_{N+1}$, it then follows that
   $0 \le d_{N+1} \le d_{N}$, and hence the limit 
   $\lim_{N \rightarrow \infty} d_{N}=\alpha$ exists.
   Notice that
   $<(h-\z_{N+k}),\z_{N+l}>=0$ for all $0 \le l \le k$. Hence,
   \[ ||h-\z_{N+k}||^{2}=<h-\z_{N+k},h-\z_{N+k}>=<h-\z_{N+k},h>=<h-\z_{N+k},h-\z_{N}>, \]
    and thus
   \[ 
     \begin{split}
     & ||\z_{N+k}-\z_{N}||^{2}=||(\z_{N+k}-h)+(h-\z_N)||^{2}=
      ||h-\z_N||^{2}-||h-\z_{N+k}||^{2} \\
     & d^2_{N} - d^2_{N+k} 
    \end{split}
   \]
   %By taking limits as $N \rightarrow \infty$, we obtain that
   %\( \lim_{N \rightarrow \infty} ||\z_{N+k}-\z_N||^2 = 0\), 
   Since $d_{N}^2$ is convergent, it is a Cauchy sequence, and hence for any $\epsilon > 0$ there exists $N_{\epsilon} > 0$ such that
   for any $N > N_{\epsilon}$ and for any $k \ge 0$, $0 < d^2_{N}-d^2_{N+k} < \epsilon$, and hence 
   $||\z_{N+k}-\z_{N}||^{2} < \epsilon$, i.e. $\z_N$ is indeed a Cauchy sequence.

 \end{proof}
  \begin{proof}[Proof of Lemma \ref{gen:filt:proof:new:lemma1}]
   From Lemma \ref{gen:filt:proof:lemma3.1} it follows that
   $\lim_{N \rightarrow \infty} \x_N(t)=\x(t)$. From this and $\e_{N}(t)=\y(t)-C\x_N(t)$ and
   $\e(t)=\y(t)-C\x(t)$ it follows that $\lim_{N \rightarrow \infty} \e_N(t)=\e(t)$.
   Finally, $\lim_{N \rightarrow \infty} \x_N(t)\bu_{\sigma}(t)=\x(t)\bu_{\sigma}(t)$,
   $\lim_{N \rightarrow \infty} \e_N(t)\bu_{\sigma}(t)=\e(t)\bu_{\sigma}(t)$  
   follows from Lemma \ref{gen:filt:proof:lemma5}.
  \end{proof}
  % it is presented in Appendix \ref{proofs}.
% The proof of Lemma \ref{gen:filt:proof:new:lemma2} is analogous to the 
%%% proof of \eqref{gen:filt:eq2} and \eqref{gen:filt:eq3} and it 
%%% is presented in Appendix \ref{proofs}.
\begin{proof}[Proof of Lemma \ref{gen:filt:proof:new:lemma2}]
 The proof is analogous to the proof of \eqref{gen:filt:eq2}. More precisely, we define
 $H_{t}^{N}$ as the linear space generated by $\z_{w}(t)$, $w \in \Sigma^{N}$, and define
 $H_{t}^{\sigma,N}$ as the linear space generated by $\z_{w}(t)\bu_{\sigma}(t)$,
 $w \in \Sigma^{N}$.  Similarly to Lemma \ref{new:pf:lemma1}, it then follows that
   \begin{equation}
   \label{gen:filt:proof:new:lemma2:eq1}
    H_{t+1}^{N+1}=\bigoplus_{\sigma \in \Sigma} H_t^{\sigma,N} \oplus \bigoplus_{\sigma \Sigma} <\e_N(t)\bu_{\sigma}(t)>
   \end{equation}
   and hence
   \begin{equation}
   \label{gen:filt:proof:new:lemma2:eq2.1}
    \begin{split}
     & \x_{N+1}(t+1)=%%E[O_R^{-1}(Y_n(t+1)) \mid H_{t+1}^{N+1}]=  \\
      \sum_{\sigma \in \Sigma} (E[O_R^{-1}(Y_n(t+1)) \mid H_{t}^{\sigma, N}] + E[O_R^{-1}(Y_n(t+1)) \mid <\e_N(t)\bu_{\sigma}(t)>]).
     \end{split}
   \end{equation}
   Define $K^{N}_{\sigma}$ such that
   \begin{equation}
   \label{gen:filt:proof:new:lemma2:eq2.2}
    E[O_R^{-1}(Y_n(t+1)) \mid <\e_N(t)\bu_{\sigma}(t)>]=K^{N}_{\sigma}\e_N(t)\bu_{\sigma}(t).
   \end{equation}
   If we show that 
   \begin{equation}
   \label{gen:filt:proof:new:lemma2:eq2}
     E[O_R^{-1}(Y_n(t+1))\mid H_{t}^{\sigma,N}]=\frac{1}{\sqrt{p_{\sigma}}} A_{\sigma}\x_N(t)\bu_{\sigma}(t) ,  
   \end{equation}
   then combining this with \eqref{gen:filt:proof:new:lemma2:eq2.2} and
   \eqref{gen:filt:proof:new:lemma2:eq2.1} we obtain \eqref{gen:filt:proof:eq2}.
    It is left to show that \eqref{gen:filt:proof:new:lemma2:eq2} holds. To this end,
    notice that $E[\x_{N+1}(t+1)\z^T_{w}(t+1)]=A_{v}B_{\hat{\sigma}}$ where
    $\hat{\sigma}$ is the first letter of $w$, $w \in \Sigma^{N+1}$. 
    The proof of this equality is analogous to that of
    Lemma \ref{gen:filt:proof:lemma1}.
    The proof of \eqref{gen:filt:proof:new:lemma2:eq2.1} is then analogous to
    that of \eqref{new:pf:eq2}.
 %Moreover, 
 %   if $w\sigma \notin L$, then $E[\x_{N+1}(t+1)\z^T_{w\sigma}(t+1)]=E[\x_N(t)\bu_{\sigma}(t)\z^T_{w\sigma}(t+1)]=0$. If $w\sigma \in L$, then
 %   $E[\x_N(t)\bu_{\sigma}(t)\z_{w\sigma}(t+1)]=\sqrt{p_{\sigma}}E[\x_N(t)\z_{w}(t)]$. Indeed
 %   $\x_N(t)\bu_{\sigma}(t)$ is a linear combination of $\z_{s}(t)\bu_{\sigma}(t)=\sqrt{p_{\sigma}}\z_{s\sigma}(t+1)$ and $\z_{s\sigma}(t+1)=0$ for $s\sigma \notin L$ and hence 
    %$\x_N(t)\bu_{\sigma}(t)$ is a linear combination of $\z_{s\sigma}(t+1)$, $s\sigma \in L$.
    %Since $E[\z_{s\sigma}(t+1)\z^T_{w\sigma}(t+1)]=T_{s\sigma,w\sigma}]=T_{s,w}=E[\z_s(t)\z_{w}(t)]$, we get that $E[\x_N(t)\bu_{\sigma}(t)\z^T_{w\sigma}(t+1)]=\sqrt{p_{\sigma}}E[\x_N(t)\z^T_{w}(t)]$. Finally, $E\x_{N}(t)\z_{w}(t)]=A_{v}B_{\hat{\sigma}}$ and hence
%    \[ E[\x_{N+1}(t+1)\z^T_{w\sigma}(t+1)]=A_{\sigma}A_{v}B_{\hat{\sigma}}=
%       \frac{1}{\sqrt{p_{\sigma}}}A_{\sigma}E[\x_N(t)\z^T_{w\sigma}(t+1)]
%   \]
%   for all $w \in \Sigma^{N}$. Since $E[\x_{N+1}(t+1)\z^T_{w\sigma}(t+1)]=E[O_R^{-1}(Y_n(t+1))\z^T_{w\sigma}(t+1)]$ by the properties of orthogonal projection, we conclude that
%   $\frac{1}{\sqrt{p_{\sigma}}}A_{\sigma}E[\x_N(t)\z^T_{w\sigma}(t+1)]=E[O_R^{-1}(Y_n(t+1))\z^T_{w\sigma}(t+1)]$ \eqref{gen:filt:proof:new:lemma2:eq2.2} follows easily.
   Finally \eqref{gen:filt:proof:lemma4:eq4} follows from 
   \eqref{gen:filt:proof:new:lemma2:eq2.1} in a way similar to the proof of
   \eqref{gen:filt:eq3}. 
%More precisely, 
%   $B_{\sigma}=E[\x_{N+1}(t+1)\z_{\sigma}(t+1)]=A_{\sigma}P^{N}_{\sigma}C^T\frac{1}{p_{\sigma}}+\alpha^{\sigma}_{N+1}\frac{1}{\sqrt{p_{\sigma}}}E[\e_N(t+1)\e_N^T(t)\bu_{\sigma}^{2}]$
%   and $T_{\sigma,\sigma}=CP_{\sigma}^NC^T\frac{1}{p_{\sigma}}+\frac{1}{p_{\sigma}}E[\e_N(t+1)\e_N^T(t)\bu_{\sigma}^{2}]$.  Combining these two equations yields
%  \eqref{gen:filt:proof:lemma5:eq4.1}.
\end{proof}
  \begin{Lemma} 
  \label{gen:filt:proof:new:lemma2v2} 
   If $\y$ is a full-rank process and satisfies Assumption \ref{output:assumptions} and Assumption \ref{output:assumptions:extra}, then for large enough $N$, $T_N$ is invertable.  
\end{Lemma}
  \begin{proof}[Proof of Lemma \ref{gen:filt:proof:new:lemma2v2}]
   Since the underlying assumption of this section is that Assumption \ref{output:assumptions}-   and Assumption \ref{output:assumptions:extra} 
   hold, it follows that  $\y$ has a \GBS\ realization of the form \eqref{gen:filt:eq2}.
   From $\y(t)=C\x(t)+\e(t)$, $\forall t \in \mathrm{Z}$
   it follows that $\z_w(t)=C\z^x_{w}(t)+\z^{\e}(t)$, $\forall t \in \mathrm{Z}_m$.
   $T_{v,w}=CE[\z^{\x}_{v}(t)(\z^{\x}_{w}(t))^T]C^T+E[\z^{\e}_{v}(t)(\z^{\e}_{w}(t))^T]$.
   Hence, $T_N=R_N+Q_N$, where $R_N$ is the block matrix
   $R_N=(CE[\z^{\x}_{v_i}(t)(\z^{\x}_{v_j}(t))^T]C^T)_{i,j=1,\ldots,M(N)}$ and
   $Q_N$ is the block-diagonal matrix, whose $i$th diagonal $p \times p$ block, 
   $i=1,\ldots,M(N)$ equals $p_{\hat{v}_i}E[\e(t)\e^T(t)\bu_{\sigma_i}^2(t)]$,
   $v_i=\sigma_i\hat{v}_i$. Since $\y$ is full rank, it then follows that
   $Q_N$ is strictly positive definite. Notice that $R_N$ is positive semi-definite by
   definition (as a covariance matrix of $((\z^{\x}_{v_1}(t))^T, \ldots, (\z^{\x}_{v_{M(N)}}(t))^T)^T$).
   Hence, $T_N$ is strictly positive definite.
  \end{proof}
 %$T_N$ is invertable for large enough $N$. Indeed,
  %$E[\e_N(t)\e_N^T(t)]=T_N$
%%by Assumption \ref{gen:filt:ass1:cond3}, $T_N$ is invertable, so $T_N^{-1}$ used in 
%% \eqref{gen:filt:proof:lemma4:eq1.3} is well-defined.
 \begin{proof}[Proof of Theorem \ref{alg2:theo}]
  From Part \ref{alg2:theo:assump1} of assumptions of the theorem it follows that
  $\lim_{M \rightarrow \infty} T_{N,M}=T_N$ and
  $\lim_{M \rightarrow \infty} Q_{\sigma}^{N,M}=Q_{\sigma}^N$. Since by Lemma \ref{gen:filt:proof:new:lemma2v2}
  $T_N$ is invertable, it follows that $T_{N,M}$ is invertable for large enough $M$.
  Moreover, since $\lim_{N \rightarrow \infty} Q_{\sigma}^{N}=Q_{\sigma} > 0$,
  for large enough $N$ and $M$, $Q_{\sigma}^{N,M}$ is invertable for all $\sigma \in \Sigma$.
  Hence, Algorithm \ref{alg2} is indeed well posed.

  Moreover, Part \ref{alg2:theo:assump1} implies that
  $\lim_{M \rightarrow \infty} \Lambda_{w}^M = \Lambda^{\y}_{w}$ for all $w \in \Sigma^{+}$ and
  hence $\lim_{M \rightarrow \infty} H_{\Psi_{\y},n,n+1}^M=H_{\Psi_{\y},n,n+1}$.
  Hence, for large enough $M$, 
  $\Rank H^{M}_{\Psi_{\y},\alpha,\beta}=\Rank H_{\Psi_{\y},\alpha,\beta}=\Rank H_{\Psi_{\y}}=n$.
  From Algorithm \ref{alg1} it is clear that its outcome is continuous in 
  the input matrix $H_{\Psi,N,N+1}$, i.e.
  $\lim_{M \rightarrow \infty} {}^M F_{\sigma}=A_{\sigma}$,
  $\lim_{M \rightarrow \infty} {}^M G=B$,
  $\lim_{M \rightarrow \infty} {}^M H=C$.
  Hence, $\lim_{M \rightarrow \infty} \widetilde{\Lambda}_{M,N}=\widetilde{\Lambda}_N$ and
  hence,
  $\lim_{M \rightarrow \infty} \alpha_{N,M} =\widetilde{\Lambda}_NT_N^{-1}=\alpha_N$.
  This and \eqref{gen:filt:proof:lemma4:eq1.41} yields that
  $P_{\sigma}^{N}=\lim_{M \rightarrow \infty} P_{\sigma}^{M,N}$.
Using \eqref{gen:filt:proof:lemma4:eq4} yields
  $\lim_{M \rightarrow \infty} K_{\sigma}^{M,N}=K_{\sigma}^N$.
  Moreover, notice that $\lim_{M \rightarrow \infty} Q_{N,M}=(T_{\sigma,\sigma}-CP_{\sigma}^NC^T)$ and the latter equals $E[\e_N(t)\e_N^T(t)\bu_{\sigma}^2(t)]=Q^{N}_{\sigma}$.
  Finally, from Lemma \ref{gen:filt:proof:new:lemma1} 
  %\eqref{gen:filt:proof:new:eq2}--\eqref{gen:filt:proof:new:eq21} and hence
  it follows that $\lim_{N \rightarrow \infty} P_{\sigma}^N=P_{\sigma}$
  $\lim_{N \rightarrow \infty} Q_{\sigma}^N=Q_{\sigma}$, and by virtue of
  \eqref{gen:filt:proof:lemma4:eq4}, 
   $\lim_{N \rightarrow \infty} K_{\sigma}^N=K_{\sigma}$.
 \end{proof}

\section{Jump Markov Linear Systems}
\label{sect:GJMLS}

The goal of this paper is to present a realization theory for a class of discrete-time stochastic hybrid systems known as jump Markov linear systems (JMLS) \cite{CostaBook}. In reality, however, we will look at stochastic hybrid systems of a slightly more general form than the ones defined in \cite{CostaBook}. The reason is that the more general class generates the same class of output processes as classical JMLS, but it is easier to establish necessary and sufficient conditions for the existence of a realization for the more general class.

%In Section \ref{sect:GJMLS} we introduce the class of generalized jump Markov linear systems (GJMLS) and show that they generate the same class of output processes as classical JMLS. However, it is easier to establish necessary and sufficient conditions for the existence of a realization for the more general class. Section \ref{sect:GJMLS_problem_form} defines the realization problem for GJMLS. Section \ref{real:necc} presents necessary conditions for the existence of a realization of JMLS with partially measured discrete states. Section \ref{real:suff} presents necessary and sufficient conditions for the existence of a realization of GJMLS with fully measured discrete state, as well as a characterization of minimality.

\begin{Definition}[Generalized jump Markov linear system]
A \emph{generalized jump Markov linear system} (GJMLS), $H$, is a 
system of the form
\begin{equation}
\label{stoch_jump_lin:geq1}
H: 
\begin{cases}
   \x(t+1) & = ~ M_{\btheta(t),\btheta(t+1)}\x(t)+
   B_{\btheta(t), \btheta(t+1)}\v(t)  \\
   \y(t) &= ~ C_{\btheta(t)}\x(t) + D_{\btheta}\v(t)    \\
%  \o(t) & = ~ \lambda(\btheta(t)) 
\end{cases} .
\end{equation}
Here $\btheta$, $\x$, $\y$ and $\v$ are stochastic processes defined on the whole set of integers, \ie $t\in\mathbb{Z}$.
The process $\btheta$ is called the \emph{discrete state process} and takes 
values in the \emph{set of discrete states} $Q=\{1,2,\ldots, d\}$. The process $\btheta$ is a stationary ergodic finite-state Markov process, with state-transition probabilities $p_{i,j}=\Prob(\btheta(t+1)= j \mid \btheta(t)= i) > 0$ 
for all $i,j \in Q$. Moreover, the probability distribution
of the discrete state $\btheta(t)$ is given by
the vector $\pi=(\pi_{1},\ldots, \pi_{d})^{T}$, where
$\pi_{i}=\Prob(\btheta(k)=i)$.
The process $\x$ is called the \emph{continuous state process} 
and takes values in one of the \emph{continuous-state spaces} 
$\mathcal{X}_q=\mathbb{R}^{n_q}$, $q \in Q$. More precisely,
for any time $t \in \mathbb{Z}$, the continuous
state $\x(t)$ lives in the state-space component
$\mathcal{X}_{\btheta(t)}$.
The process $\y$ is the \emph{continuous output process} and takes
values in the \emph{set of continuous outputs} $\mathbb{R}^{p}$.
%
%The process $\o$ is the \emph{discrete output process} and takes
%values in the \emph{set of discrete outputs} $O  = \{1,2,\dots, l\}$.
%
The process $\v$ is the \emph{continuous noise} and takes
values in $\mathbb{R}^{m}$.
%
%The map $\lambda: Q \rightarrow O$ is called the
%\emph{readout map} and it assigns a discrete output to
%each discrete state. When $\lambda$ is the identity map, we will say
%that the discrete state is \emph{fully observed}, otherwise we will say
%it is \emph{partially observed}.
%
The matrices $M_{q_{1},q_{2}}$ and $B_{q_{1},q_{2}}$
are of the
form $M_{q_{1},q_{2}} \in \mathbb{R}^{n_{q_{2}} \times n_{q_{1}}}$
and $B_{q_{1},q_{2}} \in \mathbb{R}^{n_{q_{2}} \times m}$
for any pair of discrete states $q_{1},q_{2} \in Q$.
Finally, the matrices $C_{q}$ and $D_q$ are of the form 
$C_{q} \in \mathbb{R}^{p \times n_{q}}$ 
and $D_{q}\in\Re^{p \times m}$ 
for each discrete state $q \in Q$.
\end{Definition}
We will make a number of assumptions on the stochastic processes involved.
\begin{Assumption}
\label{Assumption0}
\label{gjmls:assumptions}
Let $\mathcal{D}_{t}$ be the $\sigma$-algebra generated by
$\{\btheta(t-k)\}_{k \ge 0}$, and let $\mathcal{D}_{t_1,t_2}$, $t_1 \ge t_2$ be
the $\sigma$-algebra generated by $\{\btheta(t)\}_{t=t_2}^{t_1}$.
% be the collection of
%past discrete states, and denote by $E[\z|\mathcal{D}_{t}]$
%the conditional expectation of $\z$ given
%$\mathcal{D}_{t}$. 
 We assume that for all $t \in \mathbb{Z}$, 
\begin{enumerate}
\item
\label{Assumption0:part1}
   $\v(t)$ is mean square integrable, it is conditionally zero mean
   given $\mathcal{D}_{t,t+k}$ for all $k \ge 0$, \ie
   $E[\v(t) \mid \mathcal{D}_{t+k}]=0$, and
   for all $l > 0$,  $\v(t)$ and $\v(t-l)$ are
   conditionally uncorrelated given $\mathcal{D}_{t,t-l}$, \ie for all $l > 0$
   $E[\v(t)\v^T(t-l)^{T}\mid \mathcal{D}_{t,t-l}]=0$.
   Moreover, $Q_{q}=E[\v(t)\v^T(t)\chi(\btheta(t)=q)]$ does not
   depend on $t$.
\item
\label{Assumption0:part2}
   The $\sigma$-algebras generated by the random variables
   $\{\v(t-l), l \ge 0\}$ and
   $\{\btheta(t+l), l > 0 \}$
   are conditionally independent given 
   $\mathcal{D}_{t}$.
%\item
%\label{Assumption0:part3}
%   for all $l \ge 0$, 
%   $\x(t)$ and $\v(t+l)$ are
%   conditionally uncorrelated given 
%   $\mathcal{D}_{t+l,t}$, \ie for all $l \ge 0$,
%   $E[\x(t)\v^T(t+l)\mid \mathcal{D}_{t+l,t}]=0$.
%\item
%\label{Assumption2}
  % Assume that Assumptions \ref{Assumption0}--\ref{Assumption1.5} hold and
  % let $\{P_{q}\}_{q \in Q}$, be the unique collection of
  % matrices satisfying (\ref{stoch_jump_lin:eq2}). Recall also
  % the definition of $\mathcal{D}_{t}$ from
   %Assumption \ref{Assumption0}.
%   We assume that  for all $t \in \mathbb{Z}$, 
%   $\x(t)$ is conditionally zero mean given $\mathcal{D}_{t}$, 
%   \ie  $E[\x(t) \mid \mathcal{D}_{t}]=0$, and that 
%   $\x(t)$ is wide-sense stationary in the senset that
%   for all $q \in Q$, the expression
%   $E[\x(t)\x(t)^{T}\chi(\btheta(t)=q)]$ does not depend on $t$.
\item
\label{Assumption0:part4}
  For any $t \in \mathrm{Z}$, $\x(t)$ belongs to the Hilbert-space generated
  by the variables $\v(t-k)\chi(\theta(t-k)=q_0,\ldots,\theta(t)=q_k)$ for all
  $q_0,\ldots,q_k \in Q$, $k > 0$.
\item
\label{Assumption1}
The Markov process $\btheta$ is stationary and ergodic. Therefore, for all $q \in Q$ 
\begin{equation}
\sum_{s \in Q} \pi_{s} p_{s,q}=\pi_{q}.
\end{equation}
%\end{Assumption}
%\begin{Assumption}
\item
\label{Assumption1.5}
Let $n=n_{1}+n_{2}+\cdots + n_{d}$. The matrix 
  \begin{equation}
    \widetilde{M}=\begin{bmatrix}
                p_{1,1}M^T_{1,1} \otimes M_{1,1}^{T}  & 
		p_{1,2}M^T_{1,2} \otimes M_{1,2}^T & 
		\cdots &
		p_{1,d}M^T_{1,d} \otimes M_{1,d}^{T} \\
                p_{2,1}M^T_{2,1} \otimes M_{2,1}^{T} & 
		p_{2,2}M^T_{2,2} \otimes M_{2,2} & 
		\cdots &
		p_{2,d}M^T_{2,d} \otimes M_{2,d}^{T} \\
		\vdots & \vdots &
		\cdots & \vdots \\
                p_{d,1}M^T_{d,1} \otimes M_{d,1}^{T} &
		 p_{d,2}M^T_{d,2} \otimes M_{d,2}^T & 
		\cdots &
		p_{d,d}M_{d,d}^T  \otimes M_{d,d}^{T}
		\end{bmatrix}	
		\in \Re^{n^2 \times n^2}
  \end{equation}
is \emph{stable}. That is, for any eigenvalue $\lambda$ of $\widetilde{M}$, we have $|\lambda| < 1$.
% \end{Assumption}
%\begin{Assumption}
\item
\label{Assumption1.6}
  For each $q \in Q$, the matrix $D_qQ_{q}D^T_q$, where $Q_q=E[\v(t)\v^T(t)\chi(\btheta(t)=q)]$, is strictly positive definite.
\end{enumerate}
\end{Assumption}
Assumption \ref{Assumption0} implies that 
future discrete states interact with past noises and continuous states only through
the past discrete states. It also implies that for any fixed sequence
of discrete states, the noise process is a colored noise and
the future noise and the current continuous
state are uncorrelated. 
In addition, Assumption \ref{Assumption0} imply that the state process $\x(t)$ is
wide-sense stationary and the following holds.
\begin{Lemma}
 \label{stoch_jump_stab:lemma1}
   If Assumption \ref{Assumption0} holds, then
   there exists a unique collection of
   $n_{q}\times n_{q}$ matrices $P_{q}$ with
   $q \in Q$, such that $P_{q}$ satisfy
    \begin{equation}
 \label{stoch_jump_lin:eq2}
 P_{q}=\sum_{s \in Q}
                     p_{q,s}(M_{s,q}P_{s}M_{s,q}^{T}+B_{s,q}Q_{s,q}B_{s,q}^{T}), %%p_{q,s}\pi_{s},
  \end{equation}
  where $Q_{s,q}=E[\v(t)\v(t)^{T}\chi(\btheta(t)=s, \btheta(t+1)=q)]$.
  In addition $P_{q}=E[\x(t)\x^T(t)\chi(\btheta(t)=q)$ for all $q \in Q$ and 
  $t \in \mathrm{Z}$.
\end{Lemma}
 We present the proof of Lemma \ref{stoch_jump_stab:lemma1} later on in the text.  In fact, Lemma \ref{stoch_jump_stab:lemma1}
 follows from Lemma \ref{gbs:def:new_lemma1.3} and the relationship between GJMLS{s} and \GBS{s} which will
 be presented in the sequel.
realization by a GJLS system as follows.
Next, we define the notion of dimension for GJMLS.
\begin{Definition}[Dimension of a GJMLS]
The \emph{dimension} of a GJMLS $H$ with discrete state process $\btheta$ taking values on $Q=\{1,2,\dots,d\}$ is defined as
\begin{equation} 
\dim H =n_{1}+n_{2}+\cdots+n_{d},
\end{equation}
where $n_i$ is the dimension of the continuous state space associated with discrete state $q$, \ie $n_q = \dim \mathcal{X}_{q}$, for $q\in Q$.
%
%We call a realization $H$ of $\widetilde{\y}$
%\emph{minimal}, whenever $\dim H \le \dim H'$ for all JMLSs $H'$ 
%that are %%a weak or square-mean 
%realizations of $(\widetilde{\y},\widetilde{\lambda})$.
\end{Definition}
\begin{Remark}
Notice that two GJLSs can have the same dimension even if the dimensions of the individual continuous components are completely different.
\end{Remark}
The main motivation for the definitions above is that it allows us to formulate a 
neat characterization of minimality. In addition, it is intuitively appealing, as 
the definition of dimension reflects the amount of date which is required
to store the state information.
Next, we define when a GJMLS is a realization of a given process.
For ease of notation, \emph{in the sequel we will keep the discrete state process $\btheta$ fixed and whenever we speak of a GJMLS realization of the process $\y$, we will always mean a GJMLS of $\y$ with discrete state process $\btheta$}.
More precisely, let $\y$ be a stochastic process taking values in 
$\mathbb{R}^{p}$.
\begin{Definition}[Realization by GJMLS]
The GJMLS $H$ with discrete state process $\btheta$ 
is said a \emph{realization} of $\y$, if the output process of $H$ equals $\y$.
%if the continuous output process $\y$ of $H$ is equals to $\widetilde{\y}$.
We call a realization $H$ of $\y$
\emph{minimal}, whenever $\dim H \le \dim H'$ for all GJMLSs $H'$
that are %%a weak or square-mean
realizations of $\y$.
\end{Definition}
This section will be devoted to the solution of the following realization problem
for GJMLSs with fully observed discrete states.
\begin{Problem}[Realization problem for jump-markov systems]
 Given a process $\y$ and find conditions for existence of a
 GJMLS which is a realization of $\y$ and characterize minimality
 of GJMLS realizations of $\y$.
\end{Problem}

%\textbf{Comment more on this definition, why is good, why is bad}
%\begin{Remark} 
\subsection{Relationship between JMLS and GJMLS}
\label{gjmls:jmls}
Note that the classical definition of discrete-time JMLS \cite{CostaBook} differs from
\eqref{stoch_jump_lin:geq1}. The main difference is that in
our framework the continuous state transition rule depends
not only on the current, but also on the next discrete state.
More specifically, a JMLS according to \cite{CostaBook} is a GJMLS system of the form
\begin{equation}
\label{stoch_jump_lin:eq11}
\mathbf{S}:
\begin{cases}
   \x(t+1) & = A_{\btheta(t)}\x(t)+
    B_{\btheta(t)}\v(t)  \\
   \y(t) &= ~ C_{\btheta(t)}\x(t) + D_{\btheta(t)}\v(t)   \\
\end{cases} .
\end{equation}
where $\x(t) \in \mathbb{R}^{n}$ is the state process, $\v(t) \in \mathbb{R}^{m}$
is the noise process, $\y(t) \in \mathbb{R}^{p}$ is the output process and 
$A_{q} \in \mathbb{R}^{n \times n}$, $C_q \in \mathbb{R}^{p \times n}$,
$B_q \in \mathbb{R}^{n \times m}$ and $D_q \in \mathbb{R}^{p \times m}$ for
all $q \in Q=\{1,\ldots,d\}$. In other words a JMLS is just a GMLJS of
the form \eqref{stoch_jump_lin:geq1} such that $n_q=n$ for all $q \in Q$ and
$M_{q_1,q_2}=A_{q_2}$, i.e. $M_{q_1,q_2}$ depends only on $q_2$ for all 
$q_1,q_2 \in Q$.
In case of JMLS, one does not speak of state-spaces belonging to different
discrete states and the most natural candidate for the 
state-space of a JMLS $\mathbf{S}$ is the space $\mathbb{R}^{n}$.
Therefore, the most natural definition of dimension for a 
JMLS is the dimension $n$ of its state-space.

The classes of GJMLS and JMLS are equivalent in the following sense. 
First, it is clear that a classical JMLS 
also satisfies our definition. Conversely, a GJMLS 
of the form \eqref{stoch_jump_lin:geq1} can be rewritten
as a classical JMLS with the same noise and
output processes, but with
% the discrete-state process
%$\btheta$ replaced by $\widetilde{\btheta}(t)=(\btheta(t+1),\btheta(t))$ and the
the continuous state process and the system matrices are replaced
by a continuous state process and system matrices living in
the continuous space $\mathbb{R}^{n_1+n_2+\cdots +n_d}$.
More precisely, if $H$ is a GJMLS of the form \eqref{stoch_jump_lin:geq1}, then
define the JMLS
\begin{equation}
\label{stoch_jump_lin:eq1}
\mathbf{S}(H): 
\begin{cases}
   \hat{\x}(t+1) & =  ~ \hat{A}_{\btheta(t)}\hat{\x}(t)+
   \hat{B}_{\btheta(t)}\v(t)  \\
   \hat{\y}(t) & =  ~ \hat{C}_{\btheta(t)}\hat{\x}(t) + D_{\btheta(t)}\v(t)   \\
\end{cases},
\end{equation}
 where 
%\begin{enumerate}
%\item
\( \hat{\x}(t)= \begin{bmatrix} \hat{\x}^T_1(t), & \ldots, & \hat{\x}_{d}^{T}(t) \end{bmatrix}^T \), $\hat{\x}^T_q(t)=M_{q,\btheta(t-1)}\x(t-1)+B_{q,\btheta(t-1)}\v(t-1)$, 
  $q \in Q$, and
%\item
\[ 
  \begin{split}
  & \hat{A}_{q}=\begin{bmatrix} 
   \delta_{1,q}M_{1,1}, & \delta_{2,q}M_{1,2}, & \ldots & \delta_{d,q}M_{1,d} \\
   \delta_{1,q}M_{2,1}, & \delta_{2,q}M_{2,2}, & \ldots & \delta_{d,q}M_{2,d} \\
   \vdots &     \vdots & \ldots &  \vdots \\
   \delta_{1,q}M_{d,1}, & \delta_{2,q}M_{d,2}, & \ldots & \delta_{d,q}M_{d,d} \\
  \end{bmatrix} \mbox{\ \ }
   \hat{B}_{q} = \begin{bmatrix} B_{1,q} \\ B_{2,q} \\ \vdots \\ B_{d,q} \end{bmatrix} \\
  & \hat{C}_{q}=\begin{bmatrix} \delta_{1,q} C_1, & \delta_{2,q}C_2, & \ldots &
            \delta_{d,q}C_d \end{bmatrix}, 
 \end{split}
\]
 where $\delta_{i,j}=1$ if $i=j$ and $\delta_{i,j}=0$ if $i \ne j$ for
all $i,j \in Q$.
%\end{enumerate}
It is easy to see that the output of $\mathbf{S}_H$ and $H$ coincide, i.e.
$\hat{\y}(t)=\y(t)$. Hence, a process can be realized by a GJMLS if and only if
it can be realized by a JMLS. In addition, notice that is we define the
dimension of a JMLS as the dimension $n$ of its state-space, then
$\dim \mathbf{S}_H=\dim H$. In other words, the definition of the dimension
for a GJMLS becomes the natural definition, once the GJMLS is converted to a JMLS.
This is a further argument in favor of the definition of dimension of GJMLS adopted
in this paper.
%The reason why we choose to work with GJMLS of the
%form \eqref{stoch_jump_lin:geq1} instead of classical
%JMLSs is that, as we will show in the next subsection, 
%systems of the form \ref{stoch_jump_lin:geq1} 
%admit a nice realization theory. However,
%it is unclear whether one can also obtain such results
%for classical JMLS. Nevertheless,
%since any GJMLS can be transformed
%into a jump-linear system producing the same output processes,
%identification and realization algorithms for GJMLSs
%can be combined with the transformation sketched above to
%obtain classical jump-linear realizations.
%\end{Remark}

%\section{Realization of Generalized Jump Markov Linear Systems with Fully Observed Discrete State}
%\label{sect:GJMLS_real}

%Recall the notion of \emph{realization} for linear stochastic systems \cite{Caines1}.
%In this section we will address the realization problem for fully
%observed GJMLSs. That it throughout the section we restrict attention to GJMLS
%such  that $o$ is identity, i.e. 
%$\lambda(t)=\btheta(t)$ and hence the discrete state can be observed.
% In this subsection, we will formulate a similar concept for GJMLS with fully observed discrete state, \ie $\o(t) = \btheta(t)$. Section \ref{sect:GJMLS_problem_form} defines the realization problem for GJMLS with fully observed discrete states. 
%%Section \ref{real:necc} presents necessary conditions for the existence of a realization of JMLS with partially measured discrete states. 
\S \ref{real:sol} presents conditions for the existence of a realization of GJMLS and a
 characterization of minimal  GJMLS{s}.
The proofs of \S the results of \ref{real:sol} are presented in \S \ref{real:proof}.
%Section \ref{real:min} presents a characterization of minimality.

\subsection{Solution to the realization problem for GJMLS}
\label{real:sol}
Below we will present the solution to the realization problem for GJMLS. We 
will only state the results, their proofs will be presented in \S \ref{real:proof}.
We start with formulating conditions for existence of a realization by a GJMLS.
To this end, we fix a process $\y(t) \in \mathbb{R}^{p}$ GJMLS  and a 
Markov-process $\btheta(t) \in Q=\{1,\ldots,d\}$.
We will formulate sufficient and necessary conditions for
$\y(t)$ to admit a GJMLS realization.
%We will make the following assumptions on $\y(t)$ and $\btheta$.
 %In the sequel, we will will assume that
%\begin{Assumption}[Ergodicity and strong connectedness of $\btheta$]
 %\label{real:suff:ass1.3}
   %the Markov process $\btheta$ is stationary, ergodic and hence
   %for each $q_{1},q_{2}$ the transition probability
    %$p_{q_{1},q_{2}} > 0$ is nonzero.
%\end{Assumption}
In order to formulate the assumptions on $\y$ which characterize realizability, we will recall the terminology of
Section \ref{gen:filt} and we will try to interpret $\y(t)$ as a potential output 
process of a \GBS. More precisely, we define the alphabet $\Sigma$ to be the set of 
pairs of discrete states, \ie $\Sigma=Q \times Q$. For each letter $(q_{1},q_{2}) \in \Sigma$ let the input processes of $B$ be defined as
  \begin{equation}
  \bu_{(q_{1},q_{2})}(t)=
     \chi(\btheta(t+1)=q_{2},\btheta(t)=q_{1}). 
     \end{equation}
Define $p_{(q_1,q_2)}=p_{q_1,q_2}$. Notice that Assumption \ref{input:assumption} holds
with $\alpha_{\sigma}=1$ for all $\sigma \in \Sigma$. 
 We define the
   set of admissible sequences $L$ (see Definition~\ref{gen:filt:ass1:cond1:def1}) as 
  \begin{equation} 
  \label{real:suff:admiss1}
    L=\{ w=(q_{1},q_{2})(q_{2},q_{3})\cdots (q_{k-1},q_{k}) \mid   k \ge 0, q_{1},q_{2},\ldots, q_{k} \in Q \}.
  \end{equation}
 Notice that if $w=\sigma_1\sigma_2 \cdots \sigma_k \notin L$,
then $\bu_{\sigma_1}(t-k)\cdots \bu_{\sigma_k}(t)=0$.
 Using the correspondence described above, we can interpret the 
 process $\z^{\y}_{w}(t)$  defined in \eqref{gen:filt:eqzdef1}, i.e.
 if $w=\sigma_1\cdots \sigma_k \in \Sigma^{+}$, 
 $\sigma_1,\ldots,\sigma_k \in \Sigma$, with $\sigma_i=(q_{2i-1},q_{2i})$,
 for $q_{2i-1},q_{2i} \in Q$, $i=1,\ldots,k$, then
%\[ \z_{w}(t)=\y(t-k)\chi(\btheta(t-k)=q_1,\btheta(t-k+1)=q_2)\chi(\btheta(t-k+1)=q_3,\btheta)t-k+2)=q_4) \cdots \chi(\btheta(t-1)=q_{2k-1},\btheta(t)=q_{2k}), \]
  if $w \notin L$, i.e. $q_{2i}\ne q_{2i+1}$ for some $i=1,\ldots,k$, then
  $\z_{w}(t)=0$, and if $q_{2i}=q_{2i+1}$ for all $i=1,\ldots,k$, i.e. 
  if $w \in L$, then 
  \[ \z^{\y}_{w}(t)=\y(t-k)\chi(\btheta(t-k)=s_{1},\ldots,\btheta(t)=s_k) \]
  where $s_i=q_{2i-1}$, $i=1,\ldots,k$.
  In accordance with Notation \ref{gen:filt:eqzdef1:not} we drop the superscript $\y$ and we
  denote $\z^{\y}_{w}(t)$ by $\z_{w}(t)$.
 The terminology above allows us to apply Definition \ref{output:assumptions:part2} to 
 $\y$ and speak of $\y$ being full rank. 

Now we can formulate the assumptions which are necessary and sufficient for
existence of a GJMLS realization  of $\y$.
\begin{Assumption}
\label{gjmls:output:assumptions}
\begin{enumerate}
\item 
  $\{\y(t),\z_{w}(t) \mid w \in \Sigma^{+}\}$ is jointly zero-mean,
  wide-sense stationary, i.e. $E[\y(t)]=0$, $E[\z_{w}(t)]=0$ for all
  $t \in \mathbb{Z}$ and the covariances
  $E[\z_{w}(t)\z^T_{v}(t)]$, $E[\y(t)\z^T_{w}(t)]$,
  $w,v \in \Sigma^{+}$ are independent of $t \in \mathbb{Z}$,

\item 
\label{real:suff:ass1.1}
 the $\sigma$-algebras generated by
  $\{\y(t-l)\}_{l=0}^{\infty}$ and $\{\btheta(t+l)\}_{l=0}^{\infty}$ are conditionally independent w.r.t to the $\sigma$-algebra $\mathcal{D}_t$ generated by
  $\{\btheta(t-l)\}_{l=0}^{\infty}$
\item
  $y(t)$ if a full rank process.
%\item
%[Conditional independence of $\y$ and $\btheta$]
%For each $t \in \mathbb{Z}$, the collection of random variables
%$\{\y(t-l), l \ge 0\} $ and $\{\btheta(t+l) \mid l > 0 \}$ are
%conditionally independent given $\{ \btheta(t-l) \mid l \ge 0 \}$. 
\end{enumerate}
\end{Assumption}
 In fact, Assumption \ref{gjmls:output:assumptions} not only guarantees existence of a 
 GJMLS realization, but it also guarantees existence of a GJMLS realization which is
 its own Kalman-filter, i.e. the best possible estimate of its state based on observable
 is the state itself. In order to state the existence of such a GJMLS, we need additional
 terminology.
 \begin{Definition}[GJMLS in forward innovation form]
 We will call a GJMLS $H$ of the form \eqref{stoch_jump_lin:geq1}
 a GJMLS in \emph{forward innovation form}, if 
 the noise process $\v(t)$ equals the innovation process 
$\e(t)=\y(t)-E_l[\y(t)\mid \{\z_{w}(t) \mid w \in \Sigma^{+}\}]$ and $D_q$ is the $p \times p$ identity matrix for all $q \in Q$.
  \end{Definition}
 With the definitions  above, we can state the main result of existence of a 
 GMJLS realization.
\begin{Theorem}[Existence of a GJMLS Realization]  
 \label{real:suff:theo1}
  The process $\y$ satisfies
  Assumption \ref{gjmls:output:assumptions} if and only if there exists a GJMLS
  $H$ of the form \eqref{stoch_jump_lin:geq1} 
  which is a realization of $\y$ and which satisfies
  Assumptions \ref{gjmls:assumptions}. Moreover, $H$ can be chosen to be
  in forward innovation form.
  %Assumptions ref{real:suff:ass1.3}, \ref{real:suff:ass1.1} and \eqref{output:assumptions}, then $\y$ has a GJMLS realization in forward innovation form.
\end{Theorem}  
%In fact, from the proof of Theorem \ref{real:suff:theo1} it follows. 
 From the discussion in \S \ref{gjmls:jmls} and Theorem \ref{real:suff:theo1} 
 we can also deduce the following condition for an existence of a realization
 by JMLS.
 \begin{Corollary}
  Theorem \ref{real:suff:theo1} remains valid if we replace the word GJMLS by JMLS.
 \end{Corollary}
The second claim of Theorem \ref{real:suff:theo1} is
important for filtering. Notice that if $H$ is a GJMLS (respectively JMLS) is 
in forward innovation form, then it is easy to see that 
$\x(t)=E_l[\x(t) \mid \{z_{w}(t) \mid w \in \Sigma^{+}\}]$, i.e.
the Kalman-filter of the $H$ is $H$ itself.  Recall that Kalman-filtering of
JMLS is a well-established topic \cite{CostaBook}.

Theorem \ref{real:suff:theo1} follows from Theorem \ref{gen:filt:theo3} by
establishing a correspondence between GJMLS{s} and \GBS{s}. This correspondence
is interesting on its own right. Moreover, it will help us to formulate 
the characterization of minimality for GJMLS{s}.
The definition of this correspondence will also explain our choice of working with GJMLS{s} rather than JMLS{s}: the correspondence is much simpler for GJMLS{s} than for JMLS{s}.
 For this reason, we will present this correspondence below.

%\subsection{Relationship between GJMLS{s} and \GBS{s}}
 %\label{gjmls:gbs}
%In order to prove Theorem \ref{real:suff:theo1} we will establish a
%correspondence between \GBS{s}\ and GJMLSs.
We will use the following notation.
\begin{Notation}[Identity and zero matrices]
\label{zero:mat:not}
 In the sequel, we denote by $\mathbf{O}_{k,l}$ the $k \times l$ matrix with all
 zero entries and we denote by $I_k$ the $k \times k$ identity matrix.
\end{Notation}
 In addition, we will introduce an \emph{auxiliary output process} $\widetilde{\y}(t) \in \mathbb{R}^{pd} $ which is defined as follows
 \begin{equation}
 \label{gjmls:gbs:aux:output}
  \widetilde{\y}(t)=\begin{bmatrix} \y^T(t)\chi(\btheta(t)=1), & \y^T(t)\chi(\btheta(t)=2), & \ldots, & \y^T(t)\chi(\btheta(t)=d) \end{bmatrix}^T
 \end{equation} 
 Below we will show that GJMLS{s} realization of $\y$ yield \GBS{s} realizations of
 $\widetilde{\y}$ and vice versa. Moreover, these transformations preserve minimality
 and isomorphisms. This will enable us to use the existing results on existence of a
 \GBS\ realization and its minimality to prove the corresponding results for GJMLS{s}.
 Notice that for 
 \[ \mathcal{E}=\begin{bmatrix} I_{p}, & \ldots, & I_{p} \end{bmatrix} \in \mathbb{R}^{p \times pd}, \]
 $\y(t)=\mathcal{E}\widetilde{\y}(t)$. Hence, if $\BS$ of the form \eqref{gen:filt:bil:def}
 is a realization of $\widetilde{\y}(t)$, then by replacing the matrices $C$ and $D$ of $\BS$ 
 with $\mathcal{E}C$ and $\mathcal{E}D$, we obtain a \GBS\ realization of $\y$.

 In fact, from the definition of $\widetilde{\y}$ we can conclude the following.
 \begin{Lemma}
 \label{new:stoch_jump_lin:gbs0}
  If the process $\y$ satisfies Assumption \ref{gjmls:output:assumptions}, then 
  $\widetilde{\y}$ also satisfies Assumption \ref{output:assumptions}.
  In addition, if we define
  $\widetilde{\e}(t)=\widetilde{\y}(t)-E_l[\widetilde{\y}(t) \mid \{\z_{w}^{\widetilde{\y}}(t) \mid w \in \Sigma^{+}\}]$, then
  \[ \widetilde{\e}(t)=\begin{bmatrix} \e^T(t)\chi(\btheta(t)=1), & \ldots, & \e^T(t)\chi(\btheta(t)=d) \end{bmatrix}^T.
  \]
  Moreover, the Hilbert-space spanned by the entries of $\{ \z_{w}(t) \mid w \in \Sigma{+}\}$ coincides with that of spanned by the elements of $\{\z_{w}^{\y}(t) \mid w \in \Sigma^{+}\}$.
 \end{Lemma}
Next, we associate a generalized bilinear system $\BS_{H}$ with a GBJMLS $H$.
\begin{Definition}[\GBS\ associated with a GJMLS]
  Assume that $H$ is a GJLS of the form (\ref{stoch_jump_lin:geq1}) and $H$ satisfies 
  Assumptions
  \ref{Assumption0}. 
  We will define the \GBS\ 
  $\BS_{H}$, referred to as
  the \emph{\GBS\ associated with $H$} as follows.
  \begin{equation}
   \label{real:gjmls_repr}
     \BS_H \left \{\begin{split}
      \widetilde{\x}(t+1)&=\sum_{\sigma \in \Sigma} (A_{\sigma}\widetilde{\x}(t)+
       \widetilde{K}_{\sigma}\widetilde{\v(t)})\bu_{\sigma}(t) \\
      \widetilde{\y}(t)&=C\widetilde{\x}(t)+D\widetilde{\v}(t),
     \end{split}\right.
     \end{equation}
  %such that $R_{H}$  is a representation of 
  %$\Psi_{\widetilde{\y}}$.
 % The construction goes as follows. Let
  %\begin{equation}
  %%\label{real:gjmls_repr}
  %R_{H}=(\mathbb{R}^{n}, \{ A_{(q_1,q_2)} \}_{(q_1,q_{2}) \in \Sigma}, C, B), 
 % \end{equation}
  In order to define the parameters of $\BS$, we define $n=n_1+\cdots+n_d$ and
  for each $q \in Q$ define the matrices
   $\mathrm{I}_q \in \mathbb{R}^{n \times n_q}$, 
   $\mathbf{S}_{q} \in \mathbb{R}^{m \times dm}$ 
   \[ 
    \begin{split}
     & \mathbf{S}_q=\begin{bmatrix} \mathbf{O}_{m, (q-1)m}, I_{m \times m}, & \mathbf{O}_{m, (d-q)m} \end{bmatrix} \\
     & \mathbf{I}_q=\begin{bmatrix} \mathbf{O}_{n_q, n_1}, & \ldots & \mathbf{O}_{n_q, n_{q-1}} &, I_{n_q}, & \mathbf{O}_{n_{q},n_{q+1}}, & \ldots, \mathbf{O}_{n_q, n_d} \end{bmatrix}^T.
   \end{split}
\]
  Using the matrices above, we define  the parameters of $\BS_H$ as follows.

  \textbf{ State $\widetilde{\x}(t)$. } 
   $\widetilde{\x}(t)=\begin{bmatrix} \x^T(t)\chi(\btheta(t)=1), & \ldots, & \x^{T}(t)\btheta(t)=d) \end{bmatrix}^T \in \mathbb{R}^{n}$, $n=n_1+\cdots+n_d$.

 \textbf{ Noise $\widetilde{\v}(t)$.} 
   $\widetilde{\v}(t)=\begin{bmatrix} \v^T(t)\chi(\btheta(t)=1), & \ldots, & \v^{T}(t)\btheta(t)=d) \end{bmatrix}^T \in \mathbb{R}^{dm}$.
     
  \textbf{ Matrices $A_{(q_1,q_2)}$. } 
   %Let
    %$n=n_{1}+n_{2}+\cdots +n_{d}$ and define 
 %of the linear map which maps $x \in \mathbb{R}^{n_q}$ to $z=(0_{n_1},\ldots,0_{n_{q-1}},x, 0_{n_{q+1}} \ldots,0_{n_{d}})^T$, where $0_{k}$ denotes the all zero $1 \times k$ vector.
    %I.e. the only non-zero block of rows of $\mathbf{I}_q$ is the one indexed by
    %$i=(\sum_{r=1}^{n_{q-1}} n_r)+1,\ldots,\sum_{r=1}^{n_{q}} n_r$ and the matrix formed
    %by these rows is the
    %$n_q \times n_q$ matrix identity matrix.
    Define
    for each $q_1,q_{2}$ let $A_{q_2,q_1}$ be the 
    $n \times n$ matrix 
    \begin{equation*}
    A_{(q_1,q_2)}=\mathbf{I}_{q_2}M_{q_1,q_2}\mathbf{I}_{q_1}^T
    %\begin{bmatrix}
    %             \delta_{(1,1),(q_1,q_2)}M_{1,1} &
    %		  \cdots & \delta_{(1,d),(q_1,q_2)}M_{1,d} \\
    % 		  \vdots & \vdots  \vdots \\
    %		 \delta_{(d,1),(q_1,q_2)}M_{d,1} & \cdots
    %		 & \delta_{(d,d),(q_1,q_2)}M_{d,d}
    %		 \end{bmatrix}
    \end{equation*}
    %where $\delta_{(i,j),(k,l)}=1$ if $k=i$ and $j=l$ and $\delta_{(i,j),(k,l)}=0$ otherwise.
    %That is, all the entries of $A_{q_2,q_1}$ are
    %zero except the $n_{q_1} \times n_{q_2}$ block 
    %indexed by row indices
    %$i \in [n_1+\cdots + n_{q_1-1}+1,n_{1}+\cdots +n_{q_1}]$
    %and by column indices
    %$j \in [n_1+\cdots + n_{q_2-1}+1,n_{1}+\cdots+n_{q_2}]$,
    %which equals $M_{q_1,q_2}$.

\textbf{The matrix $\widetilde{K}_{(q_1,q_2)}$. } 
 The $n \times md$ matrix $\widetilde{K}_{(q_1,q_2)}$ is defined as
 \begin{equation*}
    K_{(q_1,q_2)}=\mathbf{I}_{q_2}B_{q_1,q_2}\mathbf{S}_{q_1}.
%%                 \begin{bmatrix}
%%                 \delta_{(1,1),(q_1,q_2)}K_{1,1} &
%%		  \cdots & \delta_{(1,d),(q_1,q_2)}K_{1,d} \\
%%		  \vdots & \vdots  \vdots \\
%%		 \delta_{(d,1),(q_1,q_2)}K_{d,1} & \cdots
%%		 & \delta_{(d,d),(q_1,q_2)}K_{d,d}
		 %\end{bmatrix}
    \end{equation*}
    %where $\delta_{(i,j),(k,l)}=1$ if $k=i$ and $j=l$ and 
   %$\delta_{(i,j),(k,l)}=0$ otherwise.
   %In other words, the all the blocks of $K_{(q_1,q_2)}$ are zero except the 
   %one indexed by row indices
    %$i \in [n_1+\cdots + n_{q_1-1}+1,n_{1}+\cdots +n_{q_1}]$ and 
   %column indices $j \in [(q-1)m,qm]$. The latter block equals $B_{(q_1,q_2)}$.

 \textbf{Matrix $C$.}  
    The $p \times n$ matrix $C$ is defined by
    \begin{equation*}
     C=\begin{bmatrix} \mathbf{I}_{1}C_1^T, & \mathbf{I}_{2}C_2^T, &
             \ldots, & \mathbf{I}_dC_d^T
     \end{bmatrix}^T.
    \end{equation*}
  That is, $C$ is a diagonal matrix, such that for all $q \in Q$ its
  diagonal block indexed by row indices
  $i=(q-1)p,\ldots,qp$ and column indices $j=[n_1+\cdots+n_{q-1}+1,\ldots,n_1+\cdots+n_q]$
  equals $C_q$.

\textbf{Matrix $D$} 
    The $p \times md$ matrix $D$ is defined by
    \begin{equation*}
     D=\begin{bmatrix} \mathbf{S}_1^TD_1^T, & \mathbf{S}_2^TD_2^T, & \ldots, & \mathbf{S}_d^TD_d^T  
     \end{bmatrix}^T.
    \end{equation*}

\end{Definition}
\begin{Lemma}
\label{new:stoch_jump_lin:gbs1}
 The output process of $\BS_{H}$ equals $\widetilde{\y}$. If $H$ satisfies
 Assumptions \ref{Assumption0}, then $\BS_H$ satisfies Assumption \ref{gbs:def}.
  Moreover, if we define $\hat{D}=\mathcal{E}D=\begin{bmatrix} D_1, & \ldots, & D_d \end{bmatrix}$, then for any $\sigma=(q_1,q_2) \in \Sigma$,
  $\hat{D}E[\hat{\v}^T(t)\hat{\v}(t)\chi(\btheta(t)=q_1,\btheta(t+1)=q_2]\hat{D}^T$ is
  strictly positive definite.
\end{Lemma}
\begin{Remark}
\label{JMLS2GBS}
 If the GJMLS $H$ is a jump-Markov linear system of the type studied in \cite{CostaBook}, i.e.
 \begin{equation}
\label{stoch_jump_lin:eq11.2}
\mathbf{H}:
\begin{cases}
   \x(t+1) & = F_{\btheta(t)}\x(t)+
    G_{\btheta(t)}\v(t)  \\
   \y(t) &= ~ H_{\btheta(t)}\x(t) + L_{\btheta(t)}\v(t)   \\
\end{cases}.
\end{equation}
where $F_q \in \mathbb{R}^{n \times n}$, 
$G_q \in \mathbb{R}^{n \times m}$, $H_q \in \mathbb{R}^{p \times n}$, $G_q \in \mathbb{R}^{p \times m}$, $q \in \Theta$, 
then we can directly construct a \GBS\ 
\begin{equation}
   \label{real:gjmls_repr1}
     \BS \left \{\begin{split}
      \widetilde{\x}(t+1)&=\sum_{\sigma \in \Sigma} (\widetilde{A}_{\sigma}\widetilde{\x}(t)+
       \widetilde{K}_{\sigma}\widetilde{\v}(t))\bu_{\sigma}(t) \\
      \y(t)&=\widetilde{C}\widetilde{\x}(t)+\widetilde{D}\widetilde{\v}(t),
     \end{split}\right.
 \end{equation}
whose output is $\y$. In this case, $\widetilde{A}_{(q_1,q_2)}$ is a $nd \times nd$ matrix, all elements of
which are zero, except the $n \times n$ block at location $(q_1,q_2)$ which equal $F_{q_2}$. Similarly,
$\widetilde{K}_{(q_1,q_2)}$ is an $nd \times md$ matrix, all elements of which are zero, except the $n \times m$ block at location $(q_1,q_2)$ which equals $G_{q_2}$.  That is, 
\[
    \widetilde{A}_{(q_1,q_2)}= \begin{bmatrix}
                \delta_{(1,1),(q_1,q_2)}F_{1} &
    		  \cdots & \delta_{(1,d),(q_1,q_2)}F_{d} \\
     		  \vdots & \vdots &  \vdots \\
    		 \delta_{(d,1),(q_1,q_2)}F_{1} & \cdots
    		 & \delta_{(d,d),(q_1,q_2)}F_{d}
    		 \end{bmatrix}, \quad 
    \widetilde{K}_{(q_1,q_2)} = \begin{bmatrix}
                \delta_{(1,1),(q_1,q_2)}G_{1} &
    		  \cdots & \delta_{(1,d),(q_1,q_2)}G_{d} \\
     		  \vdots & \vdots & \vdots \\
    		 \delta_{(d,1),(q_1,q_2)}G_{1} & \cdots
    		 & \delta_{(d,d),(q_1,q_2)}G_{d}
    		 \end{bmatrix}
 \]
 where $\delta_{(i,j),(k,l)}=1$ if $k=i$ and $j=l$ and $\delta_{(i,j),(k,l)}=0$ otherwise.
The matrices $\widetilde{C}$ and $\widetilde{D}$ are
$\widetilde{C}=\begin{bmatrix} C_1, & \ldots, & C_d \end{bmatrix}^T$, $\widetilde{D}=\begin{bmatrix} L_1, & \ldots, & L_d \end{bmatrix}^T$.
The processes $\widetilde{\x}$ and $\widetilde{\v}$ are defined as 
 $\widetilde{\x}(t)=\begin{bmatrix} \x^T(t)\chi(\btheta(t)=1), & \ldots, & \x^{T}(t)\btheta(t)=d) \end{bmatrix}^T$, $\widetilde{\v}(t)=\begin{bmatrix} \v^T(t)\chi(\btheta(t)=1), & \ldots, & \v^{T}(t)\btheta(t)=d) \end{bmatrix}^T$.
  If $H$ satisfies Assumptions \ref{Assumption0}, then $\BS$ defined above satisfies Assumption \ref{gbs:def}.
\end{Remark}

We can reverse the construction above, by associating with every 
\GBS\ $\BS$ a GJMLS $H$.
\begin{Definition}[GJMLS associated with \GBS]
 Let $\BS$ be a \GBS\ of the form 
 \begin{align*}
  \x(t+1) &=\sum_{\sigma \in \Sigma} (A_{\sigma}\x(t)+K_{\sigma}\widetilde{\e}(t))\bu_{\sigma}(t) \\
   \widetilde{\y}(t)& =C\x(t)+\widetilde{\e}(t)
 \end{align*}
 where $\widetilde{\e}(t)$ is the innovation process of 
 $\widetilde{\y}(t)$ defined in Lemma \ref{new:stoch_jump_lin:gbs0}.
Define the \emph{GJMLS $H_{\BS}$ associated with $\BS$} as follows.
\begin{equation}
  H_{\BS}: \left \{
  \begin{split}
   \hat{\x}(t+1)&=M_{\btheta(t+1),\btheta(t)}\hat{\x}(t)+
                 \widetilde{K}_{\btheta(t+1),\btheta(t)}\e(t)  \\
   \y(t)&=C_{\btheta(t)}\hat{\x}(t) + \e(t), 		 
  \end{split}		 
  \right.
  \end{equation}
  where 
 In order to define the parameters of $H_{\BS}$, we use the following notation.
 
    For each $q \in Q$, define the matrix $\mathbf{M}_q \in \mathbb{R}^{p \times pd}$ as
    \[ \mathbf{M}_q=\begin{bmatrix} \mathbf{O}_{p, p(q-1)}, & I_{p}, & \mathbf{O}_{p, p(d-q-1)} \end{bmatrix}.  \]
  For each $q \in Q$ define 
  $\mathcal{X}_{q} \subseteq  \mathbb{R}^{n}$ 
  as the subspace spanned
  by all the elements
  belonging to $\IM A_{(q_{1},q)}A_{w}K_{(q_2,q_3)}\mathbf{M}_{q_2}^T$ and
  $\IM K_{(q_{1},q)}\mathbf{M}_{q_1}^T$ for all
  $q_{1},q_2,q_3 \in Q$, $w \in \Sigma^{*}$, $\sigma \in \Sigma$,
  $i=1,\ldots, p$. Let $n_q=\dim \mathcal{X}_q$.
  Let $\Pi_q \in \mathbb{R}^{n \times n_q}$ be such that the columns of $\Pi_q$ are
  orthogonal and they span $\mathcal{X}_q$, i.e. $\Pi_q^T\Pi_q=I_{n_q}$ and
  $\IM \Pi_q=\mathcal{X}_q$.
 Then $\Pi_q$ is the matrix representation of the inclusion
 $\mathcal{X}_q \subseteq \mathbb{R}^{n}$ and $\Pi_q^T$ is the matrix representation of the
 projection of
 elements of $\mathbb{R}^{n}$ to $\mathcal{X}_q$.
\begin{enumerate} 
\item \textit{Continuous state-space for $q \in Q$:} $\mathbb{R}^{n_q}$, $n_q=\dim \mathcal{X}_q$.
\item  
\textit{State process.} 
The continuous state process $\hat{\x}(t)$ of the GJMLS is  obtained from the continuous state $\x(t)$ of the generalized bilinear system \ref{gen:filt:eq2} as follows.
Then let $\hat{\x}(t)=\Pi^T_{\btheta(\theta)}(\x(t))$, i.e.
$\hat{\x}(t)$ is obtained from $\x(t)$
by viewing it as an element of $\mathcal{X}_{\btheta(t)}$ and identifying it with the corresponding vector in $\mathbb{R}^{n_{q}}$ for $q=\btheta(t)$.
\item  
\textit{System matrices.} 
For each $q_1,q_2 \in Q$ the matrix $M_{q_{1},q_{2}} \in \mathbb{R}^{n_{q_{2}} \times n_{q_1}}$ is defines as 
 \[ M_{q_1,q_2}=\Pi_{q_2}^TA_{(q_1,q_2)}\Pi_{q_1} \]
 i.e. $M_{q_1,q_2}$ is the matrix representation of the ma[
%
%\begin{equation}
\( \mathcal{X}_{q_{1}} \ni x \mapsto 
     A_{(q_{1},q_{2})}x \in  \mathcal{X}_{q_{2}}  \).
%\end{equation}
%
For each $q \in Q$ the matrix $C_{q} \in \mathbb{R}^{p \times n_{q}}$ as
\[ C_q=\mathbf{M}_{q}C\Pi_q. \]
%$C_{q}$ is the restriction of $C$ to $\mathcal{X}_{q}$.
%
%\item 
%\textit{Noise process. } 
%The noise  process $\hat{\v}(t)$ equals the noise process $\v(t)$ of $\BS$.
\item 
\textit{Noise gain $\widetilde{K}_{q_1,q_2}$} 
 Let $\widetilde{K}_{q_1,q_2}=\Pi_{q_2}^TK_{(q_1,q_2)}\mathbf{M}_{q_1}^T$.
 %$\Pi_{q_2}^TK_{(q_1,q_2)}$ which are indexed by $j=p(q_1-1)+1,\ldots,pq_1$.
 %Here the columns of $K_{(q_1,q_2)}$
 %are viewed as elements of $\mathcal{X}_q$. 
%Let $D_{q}=D$ for all $q \in Q$.

\end{enumerate} 
\end{Definition}
%% In order to state the next result, we define isomorphism between 
%% \GBS{s}. Let $\BS$ be a \GBS\ of the form \eqref{gen:filt:bil:def}
%% and let $\bar{\BS}$ be another \GBS\ of the form
%% \[
%% \bar{\x}(t+1)&=\sum_{\sigma \in \Sigma} (\bar{A}_{\sigma}\x(t)+
%%       \bar{K}_{\sigma}\bar{\v}(t))\bu_{\sigma}(t) \\
%%      \y(t)&=\bar{C}\bar{\x}(t)+\bar{D}\bar{\v}(t),
%%     \end{split}\right.
%%     \end{equation}
%% \]
%% Then $\BS$ and $\bar{\BS}$ are said to be \emph{isomorphic}, if
%% $\dim \BS=\dim \bar{\BS}$, $\hat{\v}=\v$, $\hat{D}=D$
%%  and there exists a non-singular matrix $T \in \mathbb{R}^{n \time n}$ such that
%% \[ \bar{C}=CT^{-1}, \bar{A}_{\sigma}=TA_{\sigma}T^{-1}, TK_{\sigma}=\bar{K}_{\sigma},
%%    \sigma \in \Sigma. \]
%% Note that isomorphism of $\BS$ and $\bar{\BS}$ is a much stronger relation than
%% isomorphism of the respective representations associated with $\BS$ and $\bar{\BS}$.
%% The latter was already described in Theorem \ref{gen:filt:min_theo}.
%% It is clear that isomorphism preserves such properties as observability and
%% reachability.
\begin{Lemma}
\label{new:stoch_jump_lin:gbs2}
 Assume that $\BS$ is in forward innovation form, it satisfies 
 Assumptions \ref{gbs:def}, and it is a realization of $\widetilde{\y}$.
 Assume moreover that $\y$ satisfies Part \ref{real:suff:ass1.1} of Assumption \ref{gjmls:output:assumptions}.
 %and $\btheta$ satisfies Assumption \ref{real:suff:ass1.3}.
 Then $H_{\BS}$ is also a 
 realization of $\y$, it is in forward innovation form, and it
  satisfies Assumptions \ref{Assumption0}.
 Moreover, if the representation $R_{\BS}$ associated with $\BS$ is reachable
 and observable, then $\mathbb{R}^{n}=\bigoplus_{q\in Q} \mathcal{X}_q$ and hence
 $\dim \BS=\dim H_{\BS}$.
\end{Lemma}
\begin{Remark}
  In fact, we can convert any \GBS\ $\BS$ of the form 
 \begin{align*}
  \widetilde{x}(t+1) &=\sum_{\sigma \in \Sigma} (A_{\sigma}\widetilde{\x}(t)+K_{\sigma}\widetilde{\v}(t))\bu_{\sigma}(t) \\
   \y(t)& =C\widetilde{\x}(t)+D\widetilde{\v}(t)
 \end{align*}
 to a jump-Markov linear system of the type defined in \cite{CostaBook}:
 \begin{equation}
\label{stoch_jump_lin:eq11.3}
\mathbf{H}:
\begin{cases}
   \x(t+1) & = F_{\btheta(t)}\x(t)+
    G_{\btheta(t)}\widetilde{\v}(t)  \\
   \y(t) &= ~ H_{\btheta(t)}\x(t) + L_{\btheta(t)}\widetilde{\v}(t)   \\
\end{cases}.
\end{equation}
where $\x(t)=\begin{bmatrix} \z_1^T(t), & \ldots, \z_d^T(t) \end{bmatrix}^T$, 
$\z_q(t)=A_{(q,\btheta(t-1))}\widetilde{\x}(t-1)+K_{(q,\btheta(t-1))}\widetilde{\v}(t-1)$, 
$q \in Q$, and 
\[ 
  \begin{split}
  & L_q=D \\
  & H_q=\begin{bmatrix} \delta_{q,1}C,& \delta_{q,2}C, & \cdots & \delta_{q,d}C \end{bmatrix} \\
  & F_{q}=\begin{bmatrix} 
   \delta_{1,q}A_{(1,1)}, & \delta_{2,q}A_{(1,2)}, & \ldots & \delta_{d,q}A_{(1,d)} \\
   \delta_{1,q}A_{(2,1)}, & \delta_{2,q}A_{(2,2)}, & \ldots & \delta_{d,q}A_{(2,d)} \\
   \vdots &     \vdots & \ldots &  \vdots \\
   \delta_{1,q}A_{(d,1)}, & \delta_{2,q}A_{(d,2)}, & \ldots & \delta_{d,q}A_{(d,d)} \\
  \end{bmatrix} \mbox{\ \ }
   G_{q} = \begin{bmatrix} K_{(1,q)} \\ K_{(2,q)} \\ \vdots \\ K_{(d,q)} \end{bmatrix} \\
 \end{split}
\]
 where $\delta_{i,j}=1$ if $i=j$ and $\delta_{i,j}=0$ if $i \ne j$ for
all $i,j \in Q$.
If $\BS$ satisfies 
 Assumptions \ref{gbs:def}, and it is a realization of $\y$, 
%$\y$ satisfies Part \ref{real:suff:ass1.1} of Assumption \ref{gjmls:output:assumptions}.
 %and $\btheta$ satisfies Assumption \ref{real:suff:ass1.3}.
 then $H$ is also a  ealization of $\y$ and it
  satisfies Assumptions \ref{Assumption0}.
\end{Remark}

Recall the notion of \emph{minimality} of a linear system realization.
In particular, recall that a realization by a linear system is minimal if and only if it is reachable and observable. In this subsection, we will formulate similar concepts for GJMLS with fully observed discrete. We first define the notions of reachability and observability for a GJMLS. We then show that a realization by a GJMLS is minimal if and only if it is reachable and observable. 

In order to formulate the conditions more precisely, we will need to introduce some notation. In particular, we need to define reachability and observability matrices for GJMLS.
To that end, let $H$ be a given GJMLS of the form \eqref{stoch_jump_lin:geq1}
that satisfies Assumptions \ref{Assumption0}.
Let $N$ be the dimension of $H$, i.e. $N=\dim H$, and for all $(q_1,q_2)\in Q\times Q = \Sigma$
let 
\begin{equation}
\label{min:gjmsl:eq1}
 \begin{split}
 G_{q_1,q_2}= & E[\x(t)\y^{T}(t-1)\chi(\btheta(t)=q_2,\btheta(t-1)=q_1)]=\\
 = & p_{q_1,q_2}(M_{q_1,q_2}P_{q_1}C_{q_1}^T+B_{q_1,q_2}Q_{q_1}D_{q_1}^T) \in \Re^{q_2 \times p}.
 \end{split}
\end{equation}
Recall the definition of $L \subset Q \times Q=\Sigma$ from (\ref{real:suff:admiss1}).
\begin{Notation}[Matrix products]
 We define the following notation for the products of
  matrices $M_{q_1,q_2} \in \mathbb{R}^{n_{q_1} \times n_{q_2}}$.
For any admissible word $w=(q_1,q_2)\cdots (q_{k-1},q_k) \in L$,
where $k > 2$ and $q_1,\ldots, q_k \in Q$,
let
\begin{equation}
M_{w}=M_{q_{k-1},q_k}M_{q_{k-2},q_{k-1}}\cdots
   M_{q_1,q_2} \in \Re^{n_{q_k}\times n_{q_1}}
\end{equation}
If $w=\epsilon$, then $M_{\epsilon}$ is an identity matrix,
dimension of which depends on the context it is used in.
If $w \notin L$, then $M_{w}$ denotes the zero matrix.
\end{Notation}
\begin{Notation}
For each $q\in Q$, $L^{q}(N)$ be 
%the set of all words in $L$ of length at most $N$ that end in some pair whose second component is $q$, \ie $L^{q}(N)$ is
 the set of all words in $w \in L$ such that $|w| \le N$ and $w=v(q_1,q)$ for some $q_1 \in Q$ and $v \in L$.
\end{Notation}

%%\begin{equation}
%%L_q^N = \{ w \in L \mid |w| \le N 
%%\text{   and   } 
%%\exists q_1 \in Q, v \in L, \qquad  w=v(q_1,q) \}
%%\end{equation}
%For any admissible word $w=(q_1,q_2)v \in L$, let
%$R_{w}=M_{v}P_{q_1}$, where $P_{q_1}\in\Re^{q_1\times q_1}$ is the matrix defined
%in Assumption \ref{Assumption2}. Then 
%for any $w \in L^{N}_q$ the matrix $R_{w}$ is an
%$n_{q} \times p$ matrix.

\begin{Definition}[Reachability of a GJMLS]
For each discrete state $q \in Q$, define the matrix
\begin{equation}
%R_{H,q}=[ R_{w} \mid w \in L^{N}_{q} ]  \in \Re^{n_q \times |L^{N}_{q}|p}.
\mathcal{R}_{H,q}=[ M_{v}G_{q_1,q_2} \mid q_1 \in Q , q_2 \in Q, (q_1,q_2)v \in L^q(N) ]  \in \Re^{n_q \times |L^q(N)|p}.
\end{equation}
We will say that the GJMLS $H$ is \emph{reachable}, if for each discrete
state $q \in Q$, $\Rank (\mathcal{R}_{H,q}) = n_{q}$. 
\end{Definition}
Notice that the matrix $R_{H,q}$ is analogous to the controllability matrix for linear systems. 
\begin{Notation}
For each $q\in Q$, let $L_{q}(N)$ be the set of all words in $L$ of length at most $N$ that begin in some pair whose first component is $q$, \ie $L_{q}(N)$ is the set of all words in $w \in L$ such that $|w| \le N$ and $w=(q,q_2)v$ for some $q_2 \in Q$ and $v \in L$.
\end{Notation}
\begin{Definition}[Observability of a GJMLS]
For each discrete state $q \in Q$, define the matrix
\begin{align}
%O_{H,q}= \bigcap_{q_{k-1} \in Q, q_k \in Q, v(q_{k-1}, q_k) \in L_q(N)]^T  } \ker (C_{q_k} M_v) \\
\mathcal{O}_{H,q}=[ (C_{q_k} M_v )^T  \mid q_{k-1} \in Q, q_k \in Q, v(q_{k-1}, q_k) \in L_q(N)]^T  \in \Re^{| L_q(N) |p \times n_q}.
\end{align}

%subspace
%$O_{H,q}$ of $\mathcal{X}_{q}$ as the intersection of
%all the kernels $\ker C_{q_3}M_{w}M_{q_1,q}$, for all
%$(q,q_1)w(q_2,q_3) \in L^{N}$,
%where $L^{N}$ denotes the set of all
%words of $L$ of length at most $N$, i.e.
%%
%\begin{equation}
%O_{H,q}=\bigcap_{(q,q_1)w(q_2,q_3) 
%           \in L^{N},q_1,q_2,q_3 \in Q}
%            \ker C_{q_3}M_{q_3,q_2}M_{w}M_{q_1,q} .
%\end{equation}
%
We will say that a GJMLS $H$ is \emph{observable}, if for each discrete state $q \in Q$, 
$\Rank (O_{H,q}) = n_q$.
%$O_{H,q}=\{0\}$, i.e. the space $O_{H,q}$ is trivial.
\end{Definition}
Notice that the matrix $O_{H,q}$ plays a role similar to the observability matrix for linear systems. 
%Notice that the space $O_{H,q}$ plays a role similar to the kernel of the observability matrix for linear systems. 

Recall from (\ref{real:gjmls_repr}) the definition of the \GBS\ $\BS_{H}$ associated with a GJLS $H$. 
Recall from Definition \ref{gbs2repr} the definition of the representation $R_{\BS_H}$ associated with the \GBS\ $\BS_H$.  We will denote $R_{\BS_H}$ by $R_H$ and we will call it the representation associated with the GJMLS $H$.
 Recall the definition of reachability of a representation along with the definition of the space $O_{R_H}$ defined in (\ref{sect:pow:form1.1}). Observability and reachability of a GJMLS $H$ can be characterized in terms of the observability and reachability of the corresponding representation $R_{H}$ as follows. 

\begin{Lemma}
\label{real:lemma:reachobs}
The GJMLS
$H$ is reachable if and only if $R_{H}$ is reachable, and
$H$ is observable if and only if  $R_H$ is observable.
\end{Lemma}

The lemma above implies that observability and reachability of a GJMLS can 
be checked by a numerical algorithm.
\begin{Definition}[Morphism of GJMLS{s}]
 Let $H$ be a GJMLS of the form \eqref{stoch_jump_lin:geq1} and let  
$\hat{H}$ is another GJMLS realization of $\y$ given by
 \begin{equation}
 \label{gjlms:iso}
 \begin{split}
 & \hat{\x}(t+1)=\hat{M}_{\btheta(t),\btheta(t+1)}\hat{\x}(t)+
         \hat{B}_{\btheta(t),\btheta(t+1)}\hat{\v}(t) \\
 & \hat{\y}(t)=\hat{C}_{\btheta(t)}\hat{\x}(t) + \hat{D}_{\btheta(t)}\hat{\v}(t), 	 
 \end{split}
 \end{equation}
 where the dimension of the continuous state-space of $\hat{H}$
 corresponding to the discrete state $q$ is
 $\hat{n}_{q}$.
 A morphism from $H$ to $\hat{H}$ is a collection of matrices 
 $T=\{T_q \in \mathbb{R}^{\hat{n}_q \times n_q}\}_{q \in Q}$ such that 
 for all $q_1,q_2 \in Q$.
\begin{equation}
\label{isomorph}
T_{q_2}M_{q_1,q_2}=\hat{M}_{q_1,q_2}T_{q_1}, \quad
C_{q_1}=\hat{C}_{q_1}T_{q_1}, \quad 
T_{q_2}G_{q_1,q_2}=\hat{G}_{q_1,q_2},
\end{equation}
where $G_{q_1,q_2}$ is defined in \eqref{min:gjmsl:eq1}, and 
\begin{align} 
\hat{G}_{q_1,q_2} &=\sqrt{p_{q_1,q_2}}(\hat{M}_{q_1,q_2}\hat{P}_{q_1}C_{q_1}^T+\hat{B}_{q_1,q_2}\hat{Q}_{q_1}\hat{D}_{q_1}^T,
\end{align}
where $\hat{P}_{q_1}=E[\hat{\x}(t)\hat{\x}^T(t)\chi(\btheta(t)=q_1)]$ and
$\hat{Q}_{q_1}=E[\hat{\v}(t)\hat{\v}^T(t)\chi(\btheta(t)=q_1)]$.

$T$ will be called an isomorphism, if for all $q \in Q$, $n_q=\hat{n}_q$ and
$T_q$ is invertible.
\end{Definition}
Note that $T=(T_q)$ is an GJMLS isomorphism, if and only if the map
$\mathbf{S}_T:R_H \rightarrow R_{\hat{H}}$ is a representation isomorphism, where
$\mathbf{S}_T=\sum_{q \in Q} \mathbf{I}_{q}T_q\mathbf{I}_q^T$.  
%Conversely, it is easy to see that if
%$\mathbf{S}:R_H \rightarrow R_{\hat{H}}$ is a representation isomorphism, then
%GJMLS isomorpishm between $H$ and $\hat{H}$.

We are now ready to state the theorem on minimality of a GJMLS realization.
\begin{Theorem}[Minimality of a realization by a GJMLS]
\label{real:theo:min}
Let the GJMLS $H$ be a realization of $\y$ of the form 
\eqref{stoch_jump_lin:geq1} and assume that $H$ satisfies Assumption \ref{Assumption0}.
 Then, the GJMLS $H$ is a minimal realization of $\y$ if and only if it is reachable and observable. 
If $\hat{H}$ is another minimal GJMLS realization of $\y$ such that $\hat{H}$
satisfies Assumption \ref{Assumption0}, then
$\hat{H}$ and $H$ are isomorphic.
\end{Theorem}
\begin{Remark}
 Notice that in (\ref{isomorph}) we do not require 
 any relationship between $B_{q_1,q_2}$ and $\hat{K}_{q_1,q_2}$.
 This is consistent with the situation for linear stochastic
 systems.
%% where two minimal realizations can have completely
 %%different $B$ matrices and only the $A$, $C$ and $G$ matrices
 %%are related.
\end{Remark}  
\begin{Remark}[Realization Algorithms]
 It is clear that reachability and observability, and hence 
 minimality, of a GJLS can be checked
 numerically. It is also easy to see that the
 Algorithm \ref{alg2}
 can be adapted to obtain a weak realization 
 $H$ of $\y$.
\end{Remark}

\subsection{Proofs of the results on realization theory of GJMLS{s}}
\label{real:proof}
%\subsection{Proof of Theorem \ref{real:suff:theo1}}
  Below we present the proofs of the statements presented in 
  \S \ref{real:sol}. In addition, we present the proof of
  Lemma \ref{stoch_jump_stab:lemma1}.

 \begin{proof}[Proof of Lemma \ref{new:stoch_jump_lin:gbs0}]
   We show that $\widetilde{\y}$ satisfies the parts of Assumptions \ref{output:assumptions}
   one by one and then we show that the statement of the lemma for the innovation
   process of $\widetilde{\y}$ is true.

  \textbf{$\widetilde{\y}$ is an \RC\ process}
  Define the matrix $\mathbf{M}_q \in \mathbb{R}^{p \times dp}$ as
  \[ \mathbf{M}_{q}= \begin{bmatrix} \mathbf{O}_{p,(q-1)p}, & I_p, &  \mathbf{O}_{p,(d-q-1)p} \end{bmatrix}.
  \]
  It then follows that 
  $\z_{w}^{\widetilde{\y}}(t)=\mathbf{M}_q^T\z_w^T(t)$ if $w=(q,q_1)v$ for some $q,q_1 \in Q$, $v \in \Sigma^{+}$.
  Moreover, notice that $\widetilde{\y}(t)(\z_{w}^{\widetilde{\y}}(t))^T=\mathbf{M}^T_{q_2}\y(t)\z^T_{w}(t)\mathbf{M}_q$, 
  where $q \in Q$ is the first component of the first letter of $w$ and $q_2 \in Q$ is the second 
  component of the last letter of $w$. It is then easy to check that if
  $\y$ is a \RC\ process, then so is $\widetilde{\y}$.
  In order to see that $\y$ is an \RC\ process, notice that the first
  requirement of Assumption \ref{gjmls:output:assumptions} implies  that $\y$
  satisfies
  Part \ref{RC1} of Definition \ref{def:RC}. That
  $\y$ satisfies Part \ref{RC3} of Definition \ref{def:RC} can be shown as follows.
  If $w \notin L$, then $\bu_{w}(t)=0$ by definition of $L$.
  Let $w,v \in \Sigma^{*}$ be such that $w\sigma,v\sigma^{'} \in L$ and $|w| > 0$. 
  It is clear that
  $\z_{w\sigma}(t)\z_{v\sigma^{'}}(t)$ contains a term $\bu_{\sigma}(t)\bu_{\sigma^{'}}(t)$ and
  the latter term is zero, if $\sigma \ne \sigma^{'}$. Assume that 
  $\sigma=\sigma^{'}=(q_1,q_2)$. Then, using the definition of $\z_{w}(t),\z_{v}(t)$, 
  $E[\z_{w\sigma}(t)\z^T_{v\sigma}(t)]=\frac{1}{p_{q_1,q_2}}E[\z_{w}(t-1)\z^T_{v}(t-1)\chi(\btheta(t-1)=q_1,\btheta(t)=q_2)]$. Here, for $v=\epsilon$, $\z_{v}(t-1)=\y(t-1)$. 
   Using the assumption on conditional independence, 
   \( E[\z_{w}(t-1)\z^T_{v}(t-1)\chi(\btheta(t-1)=q_1,\btheta(t)=q_2) \mid \mathcal{D}_{t-1}]=
      E[\chi(\btheta(t-1)=q_1,\btheta(t)=q_2) \mid \mathcal{D}_{t-1}]
      E[\z_{w}(t-1)\z^T_{v}(t-1) \mid \mathcal{D}_{t-1}] =
     p_{q_1,q_2}\chi(\btheta(t-1)=q_1)E[\z_{w}(t-1)\z^T_{v}(t-1) \mid \mathcal{D}_{t-1}]
   \).
   Note that $E[\z_{w}(t-1)\z^T_{v}(t-1)]=E[E[\z_{w}(t-1)\z^T_{v}(t-1) \mid \mathcal{D}_{t-1}]]$. Moreover, $w\sigma \in L$, $|w| > 0$ implies that $q_1$ is the last component of 
   the last letter of $w$. Hence, $\z_{w}(t-1)\chi(\btheta(t-1)=q_1)=\z_{w}(t-1)$.
  From the properties of conditional expectation it follows then that
  \( \chi(\btheta(t-1)=q_1)E[\z_{w}(t-1)\z^T_{v}(t-1) \mid \mathcal{D}_{t-1}]=E[\z_{w}(t-1)\chi(\btheta(t-1)=q_1)\z^T_{v}(t-1) \mid \mathcal{D}_{t-1}]=E[\z_{w}(t-1)\z_{v}(t-1)\mid \mathcal{D}_{t-1}]$. Combining all these remarks, it follows that 
  \( E[\z_{w}(t-1)\z^T_{v}(t-1)\chi(\btheta(t-1)=q_1,\btheta(t)=q_2)]=p_{q_1,q_2}E[\z_{w}(t-1)\z_{v}^T(t-1)] \) and hence
  \( E[\z_{w\sigma}(t)\z^T_{v\sigma}(t)]=E[\z_{w}(t-1)\z^T_{v}(t-1)] \).
  That is, $\y$ satisfies Part \ref{RC3} of Definition \ref{def:RC}. By
  Remark \ref{rem:RC}, $\y$ then satisfies Part \ref{RC4} of Definition \ref{def:RC} too.

 \textbf{$\Psi_{\widetilde{\y}}$ is rational and square summable} \\
  From the discussion above it follows that 
  if $\Psi_{\widetilde{\y}}=\{ T_{(\sigma,i)} \mid \sigma \in \Sigma, i=1,\ldots,dp\}$,
  then for all $q \in Q$, $l=1,\ldots,p$,
   $T_{\sigma,p(q-1)+l}(v)$ can be written as follows.
  If $q$ is the first components of $\sigma$, 
  then $T_{\sigma,p(q-1)+l}(v)=\mathbf{M}^T_{q_2} S_{\sigma,l}(v)\mathbf{M}_{q}$
  for all $v \in \Sigma^{*}$  where $\sigma v = s(q_1,q_2)$ for some $s \in \Sigma^{*}$,
  $q_1 \in Q$. If $q$ is not the first components of $\sigma$, then
  $T_{\sigma,p(q-1)+l}(v)=0$.
  It is not difficult to construct a rational representation of 
  $\Psi_{\widetilde{\y}}$ based on such a representation of
  $\Psi_{\y}$.  Indeed, assume that
  $R=(\mathbb{R}^{n}, \{A_{\sigma}\}_{\sigma \in \Sigma},B,C)$ is a 
  representation of $\Psi_{\y}$. Define
  $\hat{\mathcal{X}}=\mathbb{R}^{dn}$ and define 
  $\mathbf{H}_{q} \in \mathbb{R}^{n \times nd}$ by
  \[ \mathbf{H}_{q}= \begin{bmatrix} \mathbf{O}_{n,(q-1)n}, & I_n, & \mathbf{O}_{n,(d-q-1)n} \end{bmatrix}.
  \]
 Let $\hat{A}_{(q_1,q_2)}=\mathbf{H}_{q_2}^TA_{(q_1,q_2)}\mathbf{H}_{q_1}$,
 $\hat{B}_{(q_1,q_2),p(q_1-1)+i}=\mathbf{H}_{q_2}B_{(q_1,q_2),i}$, $i=1,\ldots,p$ and
 let $\hat{B}_{(q_1,q_2),l}=0$ for all $l \ne p(q_1-1)+i$ for some $i=1,\ldots,p$.
 Finally, define $\hat{C}=\begin{bmatrix} \mathbf{H}_1^TC^T, & \ldots, & \mathbf{H}^T_dC^T \end{bmatrix}^T$, i.e. $\hat{C}$ is a block diagonal matrix, whose $(q,q)$th $p \times n$ block equals $C$. It is then easy to see that $\hat{R}=(\mathbb{R}^{nd}, \{\hat{A}_{\sigma}\}_{\sigma \in \Sigma},\hat{B},\hat{C})$ is a representation of $\Psi_{\widetilde{\y}}$.
 Square summability of $\Psi_{\widetilde{\y}}$ follows easily from that of
 $\Psi_{\y}$, by taking into account the relationship 
 $T_{\sigma,p(q-1)+l}(v)=\mathbf{M}^T_{q_2} S_{\sigma,l}(v)\mathbf{M}_{q}$,
 $v \in \Sigma^{*}$, $l=1,\ldots,p$, $q \in Q$, $\sigma \in \Sigma$, $q$ is the first
 letter of $\sigma$.

 \textbf{Proof of the formula for $\widetilde{\e}(t)$} \\
  Finally, from the discussion above it follows that the 
Hilbert-space spanned by the entries of $\{ \z_{w}(t) \mid w \in \Sigma^{+}\}$ coincides with that of spanned by the elements of $\{\z_{w}^{\widetilde{\y}}(t) \mid w \in \Sigma^{+}\}$.
  If $\z(t)=E_l[\y(t) \mid \{\z_{w}(t) \mid w \in \Sigma^{+}\}]$, then define
  $s(t)=\begin{bmatrix} \z^T(t)\chi(\btheta(t)=1), & \ldots, & \z^T(t)\chi(\btheta(t)=d) \end{bmatrix}^T$. 
   We claim that $s(t)=E_l[\widetilde{\y}(t) \mid \{\z_{w}^{\widetilde{\y}}(t) \mid w \in \Sigma^{+}\}]$. Indeed, $s(t)$ belongs to the Hilbert-space spanned by the entries of 
  $\{\z_{w}^{\widetilde{\y}}(t) \mid w \in \Sigma^{+}\}$. Moreover, if $q$ is the first
  component of the first letter of $w$ and $q_1$ is the second component of the
  last letter of $w$, then
   $E[\widetilde{\y}(t)(\z^{\widetilde{\y}}_{w}(t))^T]=\mathbf{M}^T_{q_1}E[\y(t)\z_{w}^T(t)]\mathbf{M}_q=\mathbf{M}^T_{q_1}E[\z(t)\z^T_{w}(t)]\mathbf{M}_{q}=E[s(t)(\z^{\widetilde{\y}}_{w}(t))^T]$. 
  From this, the claim of the lemma regarding $\widetilde{\e}(t)$ follows easily.
 \end{proof}

\begin{proof}[Proof of Lemma \ref{new:stoch_jump_lin:gbs1}]
 First, we show that $\BS_H$ is well-defined and the output of 
 $\BS_H$ equals $\widetilde{\y}$.
 For this, we have to show that
 $\hat{\x}(t)$ indeed satisfy \eqref{real:gjmls_repr}. 
 From this and the definition of $\widetilde{\y}(t)$ it follows easily that the outputs of
 $H$ and $\BS_H$ are equal.
 We show that the various parts of Definition \ref{gbs:def} hold one by one. First of all,
 Assumption \ref{Assumption1} on ergodicity of $\btheta$ means that the
 framework of Section \ref{gen:filt} can be used as it was explained before.
 
 \textbf{$\widetilde{\v}(t)$ satisfies Part \ref{gbs:def:prop1} of Assumption \ref{gbs:def}} \\
  First, we will show that $\v$ is an \RC\ process.
  From Part \ref{Assumption0:part1}, Assumption \ref{Assumption0} it follows that
  $\z_{w}^{\v}(t)$ is zero-mean. Moreover, for any $w,v \in \Sigma^{+}$,
  $|w|=k < |v|=l$,
  \( E[\z^{\v}_{w}(t)(\z^{\v}_{v}(t))^T]=E[E[\z^{\v}_{w}(t)(\z^{\v}_{v}(t))^T \mid D_{t-k,t}]]$.
   If $v=ss^{'}$ for some $s,s^{'} \in \Sigma^{+}$, $|s^{'}|=|w|$ and $w \ne s^{'}$  
   then clearly $\bu_{v}(t)\bu_{w}(t)=0$ and hence 
     $E[\z^{\v}_{w}(t)(\z^{\v}_{v}(t))^T]=0$.  Otherwise,  if $w=s^{'}$,  then 
     notice that $\bu_{v}(t)$ is a product of variables $\chi(\btheta(t-r)=q)$ for
     some $q \in Q$ and $r=0,\ldots,l-1$ multiplied by a constant. Hence, 
     by Part \ref{Assumption0:part1} of Assumption \ref{Assumption0}
     $E[\z^{\v}_{w}(t)\z^{\v}_{v}(t) \mid D_{t-l,t}]=\frac{1}{p_{w}} E[\bu_{v}(t) \mid D_{t-k}]E[\v(t-k)\v(t-l))^T \mid D_{t-l,t}] = 0$.
%\sqrt{p_{w}} E[\v(t-l)(\z^{\v}_{s}(t-l))^T \mid D_{t-k,t}]$.  As the next step, notice that $\bu_{s}(t-l)$ belongs to $D_{t-l,t-k}$ and hence
  %$E[\v(t-l)(\z^{\v}_{s}(t-l))^T \mid D_{t-k,}]=\frac{1}{\sqrt{p_s}} E[\v(t-k)\v^T(t-l) \mid D_{t-l}]=0$.
 Hence, $E[\z^{\v}_{w}(t)(\z^{\v}_{v}(t))^T]=0$ for any $w \ne v$, $|w| \ne |v|$.
  If $w \ne v$ but $|w|=|v|$, the $\bu_{w}(t)\bu_{v}(t)=0$ and
  hence $E[\z^{\v}_{w}(t)(\z^{\v}_{v}(t))^T]=0$. Finally, if $w=v$ and $|w|=|v|=k$, then
  using Assumption \ref{Assumption0}, Part \ref{Assumption0:part2} yields
  $E[\z^{\v}_{w}(t)(\z^{\v}_{w}(t))^T \mid D_{t-k,t}]=\frac{1}{p_{w}} E[\bu_{w}(t) \mid D_{t-k,t}]E[\v(t-k)\v^{T}(t-k) \mid D_{t-k}] = \chi(\btheta(t-k)=q)E[\v(t-k)\v^{T}(t-k) \mid D_{t-k,t}]=E[\v(t-k)\v^{T}(t-k)\chi(\btheta(t-k)=q) \mid D_{t-k,t}]$
 where is assumed to be of the form $w=(q,q_1)s$ for some $q,q_1 \in Q$,
 $s \in \Sigma^{*}$.
 Hence, $E[\z^{\v}_{w}(t)\z^{\v}_{w}(t))^T]=E[E[\z^{\v}_{w}(t)(\z^{\v}_{w}(t))^T \mid D_{t-k,t}]]=E[\v(t-k)\v^T(t-k)\chi(\btheta(t-k))]$ and the latter does not
  depend  $t$ by Assumption \ref{Assumption0}, Part \ref{Assumption0:part2}.
  Hence, we have shown that $E[\z^{\v}_{w}(t)\z^{\v}_{w}(t))^T]$ does not depend on
  $t$. Finally, notice that 
  $E[\v(t)(\z_{w}^{\v}(t))^T | D_{t-k,t}]=\frac{1}{\sqrt{p_{w}}} \bu_{w}(t) E[\v(t)\v^T(t-k) \mid D_{t-k,t}]=0$ does not depend on $t$ and hence
  $E[\v(t)(\z_{w}^{\v}(t))^T]=0$ also does not depend on $t$.
  Hence, $\v(t)$ satisfies Part \ref{RC1} of Definition \ref{def:RC}. 
  %Since $\bu_{w}(t)=0$ for $w \notin L$, $\v(t)$ automatically satisfies Part \ref{RC2} of Definition \ref{def:RC}. 
  Finally,
  from the discussion above it follows that $T_{w,v}=0$ for $w \ne v$,
  and $T_{w,w}=0$ for $w \notin L$ and $T_{w,w}=E[\v(t)\v^T(t)\chi(\btheta(t)=q)]$
  where $q \in Q$ is such that $w=(q,q_1)s$ for some $q_1 \in Q$, $s \in \Sigma^{*}$.
  This implies that Part \ref{RC3} of Definition \ref{def:RC} is satisfied.
  By Remark \ref{rem:RC} this already implies Part \ref{RC4} of Definition \ref{def:RC}.
  Hence, $\v$ is indeed an \RC\ process.  
  Next, we show that $\widetilde{\v}$ is an \RC\ process too.
  %Since clearly the entries of $\widetilde{\v}(t)$ are of the form
  %$\v(t)\chi(\btheta(t)=q)=\sum_{q_2 \in Q} \sqrt{p_{q,q_2}} \z^{\v}_{(q,q_2)}(t+1)$,
  It follows that all the entries of $\z_{w}^{\widetilde{\v}}(t)$ are zero except the one  which corresponds to the $q$th block of $p$ rows, where $q$ is the first components of
  the first letter of $w$. The latter entry equals
  $\z_{w}^{\v}(t)$,
  It is then easy to see that Part \ref{RC1}
  of Definition \ref{def:RC} hold. Consider any two
  $w,v \in \Sigma^{+}$, and let the first component of the first letter of $w$ and $v$
  be respectively $q_1,q_2 \in Q$.
  Then non-zero $p \times p$ block of 
   $E[\widetilde{\v}(t)(\z^{\widetilde{\v}}_{w}(t))^T]$ is the one indexed by
  $(q_1,q_1)$.
  Similarly, the only non-zero $p \times p$ block of
  $E[z_{v}^{\widetilde{\v}}(t)(\z^{\widetilde{\v}}_{w}(t))^T]$ is the
  is the one indexed by $q_1 \times q_2$. Here, we viewed both matrices as
  $d \times d$ matrices of $p \times p$ block.
  The respective non-zero entries are
  $E[\z^{\v}(t)(\z_{w}^{\v}(t)^T]$ and 
  $E[\z_{v}^{\v}(t)(\z_{w}^{\v}(t))^T]$.
  Since $\v$ is \RC, it follows that $\widetilde{\v}$ satisfies 
  Part \ref{RC3} and \ref{RC4} of Definition \ref{def:RC}.

 \textbf{$\widetilde{\v}(t)$ satisfies Part \ref{gbs:def:prop2} of Assumption \ref{gbs:def}} \\
  The orthogonality of $\z_{w}^{\widetilde{\v}}(t)$ and
  $\z_{v}^{\widetilde{\v}}(t)$ for $w \ne v$ follows from the proof that
  $\widetilde{\v}(t)$ satisfies Part \ref{gbs:def:prop1} of Assumption \ref{gbs:def}.

\textbf{$\x(t)$ and $\v(t)$ satisfy Part  \ref{gbs:def:prop4} of Assumption \ref{gbs:def}}
   The first statement of Part \ref{gbs:def:prop4} of Assumption \ref{gbs:def}
   is a direct consequence of Part \ref{Assumption0:part4} of
   Assumption \ref{Assumption0} and the fact that the sum of entries of
   $\widetilde{\v}(t)$ equals $\v(t)$.

\textbf{Part \ref{gbs:def:prop7} of Assumption \ref{gbs:def} holds}
  From the construction of $A_{(q_1,q_2)}$ it follows that
  the only non-zero column of $A_{(q_1,q_2)}$ is the one indexed by
  $j=(\sum_{q=1}^{q_1-1}n_q)+1, \ldots, \sum_{q=1}^{q_1} n_q$, and the only
  non-zero rows are the ones indexed by $i=(\sum_{q=1}^{q_2-1}n_q)+1, \ldots, \sum_{q=1}^{q_2} n_q$.  
Hence, $A_{(q_3,q_4)}A_{(q_1,q_2)}$ is necessarily zero if $q_2 \ne q_3$.
  The latter condition is equivalent to $(q_1,q_2)(q_3,q_4) \notin L$.
  Similarly, the only non-zero rows of $\hat{K}_{(q_1,q_2)}$ are the ones indexed by
  $i=(\sum_{q=1}^{q_2-1}n_q)+1, \ldots, \sum_{q=1}^{q_2} n_q$, so again
  $A_{(q_3,q_4)}\hat{K}_{(q_1,q_2)}=0$ for $q_3 \ne q_2$.
 
\textbf{Part \ref{gbs:def:prop5} of Assumption \ref{gbs:def} holds}
  It is easy to see that $\sum_{(q_1,q_2) \in \Sigma} p_{q_1,q_2} A^T_{(q_1,q_2)} \otimes A^T_{(q_1,q_2)}=\widetilde{M}$ and hence Part \ref{gbs:def:prop5} of Assumption \ref{gbs:def} follows directly from Part \ref{Assumption1.5} of Assumption \ref{Assumption0}.

\textbf{Proof that $\hat{D}E[\hat{\v}^T(t)\hat{\v}(t)\chi(\btheta(t)=q_1,\btheta(t+1)=q_2)]\hat{D}^T > 0$. }
  Notice that $\hat{D}E[\hat{\v}^T(t)\hat{\v}(t)\chi(\btheta(t)=q_1,\btheta(t+1)=q        _2)]\hat{D}^T=D_{q_1}E[\v(t)\v^T(t)\chi(\btheta(t)=q_1,\btheta(t+1)=q_2)]D_{q_1}^T$.
  From Part \ref{Assumption0:part1} it follows that 
  $E[\v(t)\v^T(t)\chi(\btheta(t)=q_1,\btheta(t+1)=q_2) \mid D_t]=E[\v(t)\v^T(t)\chi(\btheta(t)=q_1)] E[\chi(\btheta(t+1)=q_2) \mid D_t] \mid D_t] = p_{q_1,q_2} E[\v(t)\v^T(t)\chi(\btheta(t)=q_1)\mid D_t]$ and hence
  $E[\v(t)\v^T(t)\chi(\btheta(t)=q_1,\btheta(t+1)=q_2)]=p_{q_1,q_2}E[\v(t)\v^T(t)\chi(\btheta(t)=q)]=p_{q_1,q_2}Q_{q_1}$.
  Hence, $DQ_{(q_1,q_2)}D^T=p_{q_1,q_2}D_{q_1}Q_{q_1}D_{q_1}^T$. Since
  $p_{q_1,q_2} > 0$, by Part \ref{Assumption1.6} of Assumption \ref{Assumption0}, the above matrix is strictly positive definite.
\end{proof}
\begin{proof}[Proof of Lemma \ref{new:stoch_jump_lin:gbs2}]
 The first, we argue that $H_{\BS}$ is well-defined and its output equals $\y$.
 The only non-trivial thing is to prove that $\hat{\x}(t)$ is well defined and
 that the output of $H_{\BS}$ is $\y$.
 First, notice that Lemma \ref{new:stoch_jump_lin:gbs0} implies that
 \[ K_{(\btheta(t),\btheta(t+1)}\widetilde{\e}(t)=K_{(\btheta(t),\btheta(t+1))}\mathbf{M}_{\btheta(t)}\e(t) \]
 %$\widetilde{K}_{q_1,q_2}\e(t)\chi(\btheta(t)=q_1,\btheta(t+1)=q_2}=\Pi^T_{q_2}K_{(q_1,q_2)}\widetilde{\e}\chi(\btheta(t)=q_1,\btheta(t+1)=q_2)$ and hence
  %$\widetilde{K}_{\btheta(t),\btheta(t+1)}\e(t)=\Pi_{\btheta(t+1)}K_{(\btheta(t),\btheta(t+1)}\widetilde{\e}(t)$.
 It then follows that 
 \begin{equation} 
 \label{new:stoch_jump_lin:gbs2:pf1}
   \hat{\x}(t+1)=\Pi_{\btheta(t+1)}^T\x(t+1)=\Pi^T_{\btheta(t+1)}A_{(\btheta(t),\btheta(t+1))}\x(t)+\Pi_{\btheta(t+1)}^TK_{(\btheta(t),\btheta(t+1)}\widetilde{\e}(t)
 \end{equation}
 Note that
 from Lemma \ref{gbs:def:new_lemma1} it follows that 
 $\x$ is an \RC\ process and that  
 $\x(t)\chi(\btheta(t)=q)$ belongs to $\mathcal{X}_{q}=\IM \Pi_q$ almost surely. Hence,
 $\Pi_{\btheta(t)}\Pi^T_{\btheta(t)}\x(t)=\x(t)$ and thus $A_{\btheta(t+1),\btheta(t)}\x(t)=A_{\btheta(t+1),\btheta(t)}\Pi_{\btheta(t)}\hat{\x}(t)$. Substituting this into
 \eqref{new:stoch_jump_lin:gbs2:pf1} yields that
 \[ \hat{\x}(t+1)=M_{\btheta(t+1),\btheta(t)}\hat{x}(t)+\widetilde{K}_{\btheta(t+1),\btheta(t)}\e(t).
 \]
 Hence, the first equation of $H_{\BS}$ holds.
 Notice that  $\mathbf{M}_{\btheta(t)}\widetilde{\y}(t)=\y(t)$ and
 $\mathbf{M}_{\btheta(t)}\widetilde{\e}(t)=\e(t)$. Moreover, by the discussion above it follows
 that $C\x(t)=C\Pi_{\btheta(t)}\hat{\x}(t)$.
 By multiplying $\widetilde{\y}(t)=C\hat{x}(t)+\widetilde{\e}(t)$ with $\mathbf{M}_{\btheta(t)}$ we obtain
 \[ \y(t)=C_{\btheta(t)}\hat{\x}(t)+\e(t). \]
 That is, $\y$ is indeed the output of $H_{\BS}$.
 %$C_{\btheta(t)}\hat{\x}(t)=C\x(t)$. Hence, $\y(t)=C_{\btheta(t)}\hat{\x}(t)+\e(t)=C\x(t)+\e(t)$, i.e. the outputs of $\BS$ and $H_{\BS}$ coincide.
 %Here, $\e(t)=\y(t)-E_{l}[\y(t) \mid \{ \z_{w}(t) \mid w \in \Sigma^{+}\}]$ is
 %the innovation process of $\y$.

 Next, we show that $H_{\BS}$ satisfies each of the assumptions of
 Assumption \ref{Assumption0}.

 \textbf{Part \ref{Assumption0:part1} of Assumption \ref{Assumption0}}
  Since $\y(t)$ is the output of $\BS$, by Theorem \ref{gen:filt:theo3} it is
  \RC. Moreover, because $\BS$ satisfies Assumption \ref{gbs:def},
  the innovation process is \RC\ too. Hence,
  $E[\z^{\e}_{w}(t+|w|)]=0$ for any $w \in\Sigma^{+}$, which implies that
  $E[\e(t) \mid \mathcal{D}_{t+k}]=0$ for any $k \ge 0$.
  Notice that for any $w \in \Sigma^{+}$, $|w|=l$ the variables 
  $\bu_{w}(t+1)$ generate the $\sigma$-algebra $D_{t-l,t}$.
  Notice that for any $w \in \Sigma^{+}$, $|w|=l-1$, $\sigma \in \Sigma$
  $E[\e(t)\e^T(t-l)\bu_{w\sigma}(t+1)]=\sqrt{p_{w\sigma}} E[\z^{\e}_{\sigma}(t+1)(\z^{\e}_{w\sigma}(t+1))^T]=0$. Hence,
  $E[\e(t)\e^T(t-l) \mid \mathcal{D}_{t,t-l}]=0$.
  Finally $E[\e(t)\e^T(t)\chi(\btheta(t)=q)]=\sum_{q_2 \in Q} E[\z_{(q,q_2)}^{\e}(t+1)(\z_{(q,q_2)}^{\e}(t+1))^T]$  and the latter does not depen on $t$ due to the
  fact that $\e(t)$ is \RC.

 \textbf{Part \ref{Assumption0:part2} of Assumption \ref{Assumption0}}
   %In order to show that $\e(t)$ and  $\{\btheta(t+l)\}_{l > 0}$
   %are conditionaly independent with respect to $\mathcal{D}_t$, it is enough to show that 
   %for any integrable function $g$ and any integrable $\mathcal{F}_2$ measurable
   %random variable, $E[g(\e(t))x \mid \mathcal{D}_{t}]=E[g(\e(t))\mid \mathcal{D}_t]E[x\mid \mathcal{D}_t]$.  To this end, 
 %By Assumption \ref{real:suff:ass1.1},  $\mathcal{F}_1$
   %is conditionally independent of the
   %$\sigma$-algebra 
   Let $\mathcal{F}_1$ be the $\sigma$-algebra generated by the variables $\{y(t-l)\}_{l \ge 0}$ and denote by
   $\mathcal{F}_1 \lor \mathcal{D}_t$ the smallest $\sigma$-algebra which contains 
   $\mathcal{F}_1$ and $\mathcal{D}_t$. 
   Let $\mathcal{F}_2$ be the $\sigma$-algebra generated by $\{\btheta(t+l)\}_{l > 0}$ and
   notice that by assumption $\mathcal{F}_2$ and $\mathcal{F}_1$ are conditionaly 
   independent w.r.t. $\mathcal{D}_t$.
  From the elementary properties of conditional independence 
  and the fact that $\mathcal{F}_1$ and $\mathcal{F}_2$ are conditionally independent w.r.t. $\mathcal{D}_t$ it follows
  that $\mathcal{F}_1 \lor \mathcal{D}_t$ and $\mathcal{F}_2$ are also conditionally
  independent w.r.t. $\mathcal{D}_t$.  
   
   Hence, it is enough to show that for $l \ge 0$,
   $\e(t-l)$ is $\mathcal{F}_1 \lor \mathcal{D}_t$ measurable. 
   From this and the discussion above it then follows that the $\sigma$-algebra
   generated by 
   $\{e(t-l)\}_{l=0}^{\infty}$ and $\mathcal{F}_2$ are conditionaly independent 
   w.r.t $\mathcal{D}_t$. 
   Notice that $\e(t-l)$ belongs to the Hilbert-space
   generated by $\{\y(t-l),\z_{w}(t-l)\mid w \in \Sigma^{+}\}$, 
   and hence by Lemma \ref{lemma:measurability}, 
   $\e(t-l)$ is measurable w.r.t the $\sigma$-algebra generated by
   $\{y(t-l),\z_{w}(t-l) \mid w \in \Sigma^{+}\}$.
   The latter $\sigma$-algebra is contained in 
   $\mathcal{F}_1 \lor \mathcal{D}_t$ and 
   hence $\y$ is $\mathcal{F}_1 \lor \mathcal{D}_t$ measurable, as required.

 \textbf{Part \ref{Assumption0:part4} of Assumption \ref{Assumption0}}
  This is a direct consequence of Part \ref{gbs:def:prop4} of Assumption \ref{gbs:def}. 

 \textbf{Part \ref{Assumption1} of Assumption \ref{Assumption0}}
   This a direct consequence of Assumption \ref{real:suff:ass1.1}.

 \textbf{Part \ref{Assumption1.5} of Assumption \ref{Assumption0}}
    It then follows that $M_{q_1,q_2} \otimes M_{q_1,q_2}=(\Pi^T_{q_2} \otimes \Pi^T_{q_2}) (A_{(q_1,q_2)} \otimes A_{(q_1,q_2)})(\Pi_{q_1} \otimes \Pi_{q_1})$.
   Let $P=(P_1,\ldots,P_d)$ a $d$ tuple of matrices $P_q \in \mathbb{R}^{n_q \times n_q}$   such that if $P$ is interpreted as a $\sum_{q \in Q} n_q^2$ vector $\phi(P)$, then
   $\widetilde{M}^T\phi(P)=\lambda \phi(P)$ 
    for some $\lambda \in \mathbb{C}$. It then follows that
   $\lambda P_q=\sum_{r \in Q} p_{r,q}M_{r,q}P_rM_{r,q}^T$.
   Notice that with $\hat{P}_q=\Pi_{q}P_q\Pi_{q}^T=(\Pi_q \otimes \Pi_q)\phi(P)$,
   $\lambda P_q=\sum_{r \in Q} p_{r,q}\Pi_q^TA_{(r,q)}\hat{P}_rA^T_{(r,q)}\Pi_q$.   
   By applying from the left $\Pi_q$ and from the right $\Pi_q^T$ to both sides of the
   equation, we get $\lambda \hat{P}_q= \sum_{r \in Q} p_{r,q}\Pi_q\Pi^T_qA_{(r,q)}\hat{P}_rA^T_{(r,q)}\Pi_q\Pi^T_q$.
   Notice that $\mathcal{X}_q=\IM \Pi_q$ and that $A_{(r,q)}\mathcal{X}_r \subseteq \mathcal{X}_q$. Hence, $A_{(r,q)}\Pi_{r}=\Pi_qS$ for some $S \in \mathbb{R}^{n_r \times n}$. By exploiting $\Pi_q^T\Pi_q=I_{n_q}$, 
  it follows that $\Pi_q\Pi^T_qA_{(r,q)}\Pi_{r}=A^T_{(r,q)}\Pi_r$.
  Thus, by taking into account that $\hat{P}_r=\Pi_rP_r\Pi^T_r$, $r \in Q$,
   \[ \lambda \hat{P}_q= \sum_{r \in Q} p_{r,q}A_{(r,q)}\hat{P}_rA_{(r,q)}^T. \]
    Note that $A_{(r,q)}\Pi_{r_1}=0$ for $r_1 \ne r$, since 
   $A_{(r,q)}|_{\mathcal{X}_{r_1}}=0$, since $\mathcal{X}_{r_1}$ belongs to the linear span of
   elements of $\IM A_{(r_1,q_1)}$ and $\IM K_{(r_1,q_1)}$, $q_1 \in Q$, and 
    Part \ref{gbs:def:prop7} of
   Assumption \ref{gbs:def}. Hence, if $\hat{P}=\sum_{q \in Q} \hat{P}_q$,
   then $A_{(r,q)}\hat{P}_rA^T_{(r,q)}=A_{(r,q)}\hat{P}A^T_{(r,q)}$.
   Denote by $\mathcal{Z}$ the linear map $\mathbb{R}^{n^2 \times n^2} \mapsto \sum_{(r,q) \in Q \times Q} p_{r,q} A_{(r,q)}VA_{(r,q)}^T$.
   From the discussion above it follows that $\hat{P}$ is an eigenvector of $\Z$ 
  corresponding to the eigenvalue $\lambda$.
   %\[ \lambda \hat{P}= \Z(\hat{P}). \]
  %In other word, $\hat{P}$ is an eigenvector of the map $\Z$. 
  From \cite[Chapter 2]{CostaBook} it follows  $\sum_{(r,q) \in Q \times Q} p_{r,q} A_{(r,q)} \otimes A_{(r,q)}$ is just a matrix representation of $\Z$.
  Then Part \ref{gbs:def:prop5} of Assumption \ref{gbs:def}
  implies that the eigenvalues $(\sum_{(r,q) \in Q \times Q} p_{r,q} A_{(r,q)} \otimes A_{(r,q)})^T=\sum_{(r,q) \in Q \times Q} p_{r,q} A_{(r,q)}^T \otimes A_{(r,q)}^T$
  all inside the unit disk. Since takings transposes does not change the eigenvalues, 
  it then follows that all
the eigenvalues of $\sum_{(r,q) \in Q \times Q} p_{r,q} A_{(r,q)} \otimes A_{(r,q)}$, and hence
 of $\Z$, are inside the unit disk as well. 
 Since $\lambda$ was
  an arbitrary eigenvalue of $\widetilde{M}^T$, and $\widetilde{M}$ and $\widetilde{M}^T$ have
  the same eigenvalues,  it follows that
  Part \ref{Assumption1.5} of Assumption \ref{Assumption0} holds.

 \textbf{Part \ref{Assumption1.6} of Assumption \ref{Assumption0}}
  A direct consequence of Part \ref{gbs:def:prop7} of Definition \ref{gbs:def}.

 \textbf{Proof that $\mathbb{R}^{n}=\bigoplus_{q\in Q} \mathcal{X}_q$}

  Consider the matrix $B_{\sigma}$ of $\BS$ defined in \eqref{gsb2repr:eq1}.
  It then follows that
  $K_{(q_1,q_2)}Q_{(q_1,q_2)}=B_{(q_1,q_2)}-A_{(q_1,q_2)}P_{(q_1,q_2)}C^T$.
  where $Q_{q_1,q_2}=E[\widetilde{\e}(t)\widetilde{\e}^T(t)\chi(\btheta(t)=q_1,\btheta(t+1)=q_2)]$. From Lemma \ref{new:stoch_jump_lin:gbs0} it follows that
  $Q_{q_1,q_2}=\mathbf{M}_{q_2}^TE[\e(t)\e^T(t)\chi(\btheta(t)=q_1,\btheta(t+1)=q_2)]\mathbf{M}_{q_1}$. Since $\e(t)$ and $\btheta(t+1)$ are conditionally independent given $\mathcal{D}_t$,
  it follows that $E[\e(t)\e^T(t)\chi(\btheta(t)=q_1,\btheta(t+1)=q_2)]=E[\e(t)\e^T(t)\chi(\btheta(t)=q_1)]p_{q_1,q_2} > 0$. Notice, moreover, that $\mathbf{M}_{q_1}\mathbf{M}^T_{q_1}=I_{q_1}$.
  Hence, by multiplying $K_{(q_1,q_2)}Q_{(q_1,q_2)}=B_{(q_1,q_2)}-A_{(q_1,q_2)}P_{(q_1,q_2)}C^T$ by $\mathbf{M}_{q_1}^T(E[\e(t)\e^T(t)\chi(\btheta(t)=q_1)])^{-1}p^{-1}_{q_1,q_2}$
 from the right, we obtain that
  $K_{(q_1,q_2)}\mathbf{M}_{q_1}^T$ belongs to the linear span of elements of the form
  $\IM B_{(q_1,q_2)}$ and $A_{(q_1,q_2)}z$, $z \in \mathbb{R}^{n}$.

  Also notice that $\mathbf{M}_qCA_{w}B_{\sigma}=E[\widetilde{\y}(t)(\z_{\sigma w}^{\widetilde{\y}}(t))^T]=0$, if the last component of the last letter of 
$\sigma w$ is not $q \in Q$. 
  Since $\BS$ is reachable, any $z \in \mathbb{R}^{n}$ is a linear combination of
  vectors from $\IM A_{w}B_{\sigma}$ for some $w \in \Sigma^{*}$, $\sigma \in \Sigma$.
  Hence, $\mathbf{M}_qCA_{(q_1,q_2)}=0$ and $\mathbf{M}_{q}CB_{(q_1,q_2)}=0$
  for all $q_1,q_2,q \in Q$ such that $q_2 \ne q$.
  Combining this with the definition of $\mathcal{X}_q$, $q \in Q$ and the fact
  derived above that $K_{(q_1,q_2)}$ is spanned by elements $\IM A_{(q_1,q_2)}$, $\IM B_{(q_1,q_2)}$, it follows that $\mathbf{M}_qCx=0$ for all $x \in \mathcal{X}_{q_1}$, $q_1 \ne q$.

  We are now ready to prove that $\mathbb{R}^{n}=\bigoplus_{q \in Q} \mathcal{X}_q$.
  From the discussion above, it follows that
  $\mathcal{X}_{q_1} \cap \mathcal{X}_{q_2} = \{0\}$.  Indeed, if
  $x \in \mathcal{X}_{q_1} \cap \mathcal{X}_{q_2}$, then
  for $q \ne q_1$, $\mathbf{M}_qCx=0$, and since $q_1 \ne q_2$ and $x \in \mathcal{X}_{q_2}$,
  $\mathbf{M}_{q_1}Cx=0$. Hence, $Cx=0$. Moreover,  notice that
  $\mathcal{X}_{q} \subseteq \ker A_{(q_3,q_4)}$ for $q \ne q_3$, since 
  $A_{(q_3,q_4)}K_{q_5,q}=0$ and $A_{(q_3,q_4)}A_{(q_5,q)}=0$ for all $q_5 \in Q$.
  By applying this result to $q=q_1$ and $q=q_2$, it follows that
  $A_{(q_3,q_4)}x=0$ for any $q_1,q_4 \in Q$ and hence $A_{w}x=0$ for any $w \in  \Sigma^{+}$.
  That is, $CA_{w}x=0$ for all $w \in \Sigma^{*}$, i.e. $x \in O_{R_{\BS}}$. Since
  $\BS$ is observable, it then follows that $x=0$.

  It is left to show that $\mathbb{R}^{n}=\sum_{q \in Q} \mathcal{X}_q$.  
  To this end,  consider the definition of $R_{\BS}$. As it was already mentioned,
  $\x(t)\chi(\btheta(t)=q)$ belongs to $\mathcal{X}_q$ for $q \in Q$ almost everywhere.
  Hence, the columns of $P_{(q_1,q_2)}=E[\x(t)\x^T(t)\chi(\btheta(t)=q_1,\btheta(t+1)=q_2)]$ 
  belong to $\mathcal{X}_{q_1}$: take any $M \in \mathbb{R}^{n-n_{q_1} \times n}$ such that $\mathcal{X}_{q_1}=\ker M$; then $M\x(t)\chi(\btheta(t)=q)=0$ almost everywhere, and hence
  $MP_{(q_1,q_2)}=0$.
  It then follows that $\IM A_{(q_1,q_2)}P_{(q_1,q_2)}C^T \subseteq \mathcal{X}_{q_2}$.
  From the previous discussion it follows that
  $K_{(q_1,q_2)}Q_{(q_1,q_2)}=p_{q_1,q_2}K_{(q_1,q_2)}\mathbf{M}_{q_1}^TE[\e(t)\e^T(t)\chi(\btheta(t)=q_1)]\mathbf{M}_{q_1}$ and hence $\IM K_{(q_1,q_2)}Q_{(q_1,q_2)} \subseteq \mathcal{X}_{q_2}$. Combining all this with the definition of $R_{\BS}$ it follows that
  $\IM B_{(q_1,q_2)} \subseteq \mathcal{X}_{q_2}$. Since $A_{(q,q_3)}(\mathcal{X}_{q}) \subseteq \mathcal{X}_{q_3}$, we obtain that $\IM A_{w}B_{(q_1,q_2)}$ always belongs to 
  $\mathcal{X}_q$, where $q$ is the last component of the last letter of $(q_1,q_2)w$.
  From reachability of $R_{\BS}$ we then obtain that 
  $\mathbb{R}^{n}=\sum_{q \in Q} \mathcal{X}_q$, as claimed.
\end{proof}
 
Now we can also easily prove Lemma \ref{stoch_jump_stab:lemma1}. In fact,
we will prove first a technical result, relating state covariances of
$H$ and $\BS_H$. From this Lemma \ref{stoch_jump_stab:lemma1}
follows easily.  
%In order to state this technical result, notice that  
%the state-space $\mathbb{R}^n$ of $R_H$ is of the form $\bigoplus_{q \in Q} \mathcal{X}_q$, where $\mathcal{X}_q$ is the set of all vectors whose entries are all zero except possibly the ones indexed by $i=(\sum_{r=1}^{q-1} n_r)+1,\ldots,\sum_{r=1}^{q} n_r$.
\begin{Lemma}
\label{stoch_jump_stab:lemma1.1}
 Assume that $H$ satisfies Assumption \ref{Assumption0}.
 Let $\hat{P}_{q_1,q_2}=E[\hat{\x}(t)\hat{\x}^T(t)\chi(\btheta(t)=q_1,\btheta(t+1)=q_2)]$.
 Then for $P_{q}=E[\x(t)\x^T(t)\chi(\btheta(t)=q)]$, 
 \[ \hat{P}_{q_1,q_2}=p_{q_1,q_2}\mathbf{I}_{q_1}P_{q_1}\mathbf{I}_{q_1}^T. \]
 Similarly, if $\hat{Q}_{q_1,q_2}=E[\hat{\v}(t)\hat{\v}(t)\chi(\btheta(t)=1,\btheta(t+1)=q_2)]$,
 and $Q_{q_1}=E[\v(t)\v^T(t)\chi(\btheta(t)=q_1)]$, then 
 \[ \hat{Q}_{q_1,q_2}=p_{q_1,q_2}\mathbf{S}^T_{q_1}Q_{q_1}\mathbf{S}_{q_1}. \]
\end{Lemma}
\begin{proof}[Proof of Lemma \ref{stoch_jump_stab:lemma1.1}]
 The second statement of the lemma was already shown in the proof of
 Lemma \ref{new:stoch_jump_lin:gbs1}, while showing that $\BS_H$ 
 satisfies Part \ref{gbs:def:prop7} of Assumption \ref{gbs:def} holds.

 We proceed with the proof of the first statement. 
 From the construction of $\hat{\x}(t)$ it follows that $\hat{P}_{(q_1,q_2)}=\mathbf{I}_{q_1}E[\x(t)\x(t)\chi(\btheta(t)=q_1,\btheta(t+1)=q_2)]\mathbf{I}_{q_1}^T$.
 Hence, it is enough to show that $E[\x(t)\x^T(t)\chi(\btheta(t)=q_1,\btheta(t+1)=q_2)]=p_{q_1,q_2}P_{q_1}$. 
  To this end, notice that \( E[\x(t)\x^T(t)\chi(\btheta(t)=q_1,\btheta(t+1)=q_2)]=E[E\x(t)\x^T(t)\chi(\btheta(t)=q_1,\btheta(t+1)=q_2 \mid \mathcal{D}_t]]$. 
 Also notice that from Part \ref{Assumption0:part4} of Assumption \ref{Assumption0} it follows that
 $\x(t)$ is measurable w.r.t. the $\sigma$-algebra generated by 
 $\{\v(t-l)\}_{l \ge 0}$.  Indeed, Part \ref{Assumption0:part4} of Assumption \ref{Assumption0}
 and Lemma \ref{lemma:measurability}
 implies that $\x(t)$ 
  measurable w.r.t. to the $\sigma$-algebra generated by $\{\v(t-l)\}_{l \ge 0}$.
  %each of which is
  %measurable w.r.t. to the $\sigma$-algebra generated by $\{\v(t-l)\}_{l \ge 0}$.
  %Then by \cite[Propostion 20.5]{Bilingsley} there is a subsequence of this sequence
  %which converges to $\x(t)$ almost surely. Since each element of this subsequence is
  %measurable w.r.t. to the $\sigma$-algebra generated by $\{\v(t-l)\}_{l \ge 0}$.
  %then so is $\x(t)$ (after a suitable modification on a set of probability zero).

 From Part \ref{Assumption0:part2} of Assumption \ref{Assumption0} it then follows that $\x(t)$ and $\btheta(t),\btheta(t+1)$ are conditionally independent given
  $\mathcal{D}_t$. Hence \( E[\x(t)\x^T(t)\chi(\btheta(t)=q_1,\btheta(t+1)=q_2) \mid \mathcal{D}_t]= p_{q_1,q_2}E[\x(t)\x^T(t)\chi(\btheta(t)=q_1)\mid \mathcal{D}_t]$.
  Combining this with the discussion above yields 
  that $E[\x(t)\x^T(t)\chi(\btheta(t)=q_1,\btheta(t+1)=q_2)]=p_{q_1,q_2}E[\x(t)\x^T(t)\chi(\btheta(t)=q_1)]$.
\end{proof}
\begin{proof}[Proof of Lemma \ref{stoch_jump_stab:lemma1}]
   Consider the \GBS\ $\BS_H$ associated with $H$.
   From the construction of the matrices of $\BS_{H}$ it follows that the solutions to
   \eqref{stoch_jump_lin:eq2} and those of \eqref{gen:filt:theo3.2:eq1} interpreted for $\BS=\BS_{H}$ can be related as follows.
  Suppose that $\{P_q\}_{q \in Q}$ is a solution to \eqref{stoch_jump_lin:eq2}.
  From Lemma \ref{stoch_jump_stab:lemma1.1} it follows that
  $\hat{Q}_{(q_1,q_2)}=p_{q_1,q_2}\mathbf{S}_{q_1}^TQ_{q_1}\mathbf{I}_{S_{q_1}}$.
  %From the proof of Lemma \ref{new:stoch_jump_lin:gbs1} it follows that
  %$E[\hat{v}(t)\hat{v}^T\chi(\btheta(t)=q_1,\btheta(t+1)=q_2)]$ is zero except the 
  %block $p_{q_1,q_2}Q_{q_1}$ lying on the intersection of the rows indexed by
  %$i=(\sum_{q=1}^{q_1-1} n_q)+1,\ldots,\sum_{q=1}^{q_1} n_q$ and columns indexed by $j=(\sum_{q=1}^{q_1-1} n_q)+1,\ldots,\sum_{q=1}^{q_1} n_q$.
  Define $\hat{P}_{(q_1,q_2)}=p_{q_1,q_2}\mathbf{I}_{q_1}P_{q_1}\mathbf{I}_{q_1}^T$.
  Notice that $\mathbf{I}_{q}^T\mathbf{I}_{q}=I_{n_{q}}$ and
  $\mathbf{S}_{q}\mathbf{S}_{q}^T=I_{m}$.
  If we multiply \eqref{stoch_jump_lin:eq2} by $\mathbf{I}_{q}$ from
  the right and by $\mathbf{I}_{q}^T$ from the right, then using the discussion 
  above and the definition of $A_{(q_1,q_2)}$, $K_{q_1,q_2}$ we readily obtain that
  $\{\hat{P}_{(q_1,q_2)}\}_{(q_1,q_2) \in Q \times Q}$ satisfies 
  \eqref{gen:filt:theo3.2:eq1}.
  In addition, notice that the correspondence between 
  $p_{q_1,q_2}P_{q_1}$ and $\hat{P}_{(q_1,q_2)}$ is injective, since
  $\mathbf{I}_{q_1}$ is full column rank for all $q_1 \in Q$.
  Since by Lemma \ref{gbs:def:new_lemma1.3} \eqref{gen:filt:theo3.2:eq1}
  has precisely one solution, this implies that \eqref{stoch_jump_lin:eq2}
  has at most one solution.

  Next, we show that \eqref{stoch_jump_lin:eq2} has a solution. To this end,
  notice that the unique solution of \eqref{gen:filt:theo3.2:eq1} is of the form
  $\hat{P}_{(q_1,q_2)}=E[\hat{\x}(t)\hat{\x}^T(t)\chi(\btheta(t)=q_1,\btheta(t+1)=q_2)]$. 
  %From Lemma \ref{stoch_jump_stab:lemma1.1} it follows that
  Notice that the  only non-zero block of $\hat{P}_{(q_1,q_2)}$ is the one which corresponds
  to $p_{q_1,p_2}E[\x(t)\x^T(t)\chi(\btheta(t)=q_1)]$.
  Define now $P_{q}=E[\x(t)\x^T(t)\chi(\btheta(t)=q)]$, $q \in Q$. From the discussion above and Lemma \ref{stoch_jump_stab:lemma1.1} and the definition of
  the matrices $A_{(q_1,q_2)}$ and $K_{(q_1,q_2)}$
  it is easy to see that $\{P_{q}\}_{q \in Q}$ satisfies 
  \eqref{stoch_jump_lin:eq2}.
\end{proof}

Now we are ready to present the proof of Theorem \ref{real:suff:theo1}.
\begin{proof}[Proof of Theorem \ref{real:suff:theo1}]
 \textbf{Necessity} \\
 If $\y$ has a realization by a GJMLS which satisfies Assumption \ref{Assumption0}, then by Lemma \ref{new:stoch_jump_lin:gbs1},
   $\y$ can be realized by a \GBS\ which satisfies
  Assumption \ref{gbs:def}. By Theorem \ref{gen:filt:theo3.2},
  the latter implies that
  $\y$ satisfies  Assumption \ref{output:assumptions}.
  Moreover, the second statement of Lemma \ref{new:stoch_jump_lin:gbs1} together with
  Theorem \ref{gen:filt:theo3.2} implies that $\y$ is full rank. Hence,
  $\y$ satisfies the first part of Assumption \ref{gjmls:output:assumptions}.

  %Ergodicity of $\btheta$ is included in Assumption \ref{Assumption0}, hence
  %Assumption \ref{real:suff:ass1.3} is automatically satisfied.
  Finally, the validity of Part
  \ref{real:suff:ass1.1} of Assumption \ref{gjmls:output:assumptions}
can be obtained as follows.
  Let $\mathcal{F}_t$ be the $\sigma$-algebra generated by $\{\v(t-l)\}_{l \le 0}$.
  Let $\mathcal{D}_t^{+}$ be the $\sigma$-algebra generated by
  $\{\btheta(t+l)\}_{l \ge 0}$.
  From Part \ref{Assumption0:part4} of Assumption \ref{Assumption0}
  and $\y(t)=C_{\btheta(t)}\x(t)+D_{\btheta(t)}\v(t)$ it follows that 
  $\y(t)$ is measurable with respect to the joint $\sigma$-algebra $\mathcal{F}_t \lor \mathcal{D}_t$. 
Hence, the $\sigma$-algebra $\mathcal{H}_t$ generated by
  $\{\y(t-l)\}_{l \ge 0}$ is a sub-algebra of
  $\mathcal{F}_t \lor \mathcal{D}_t$.
  Since by Part \ref{Assumption0:part2} of Assumption \ref{Assumption0}
  $\mathcal{F}_t$ and $\mathcal{D}_t^{+}$ are conditionaly independent given
  $\mathcal{D}_t$, from the well-known properties of conditional independence it
  follows that $\mathcal{F}_t \lor \mathcal{D}_t$ and $\mathcal{D}_t^{+}$ are
  conditionally independent too. Hence, $\mathcal{H}_t$ and $\mathcal{D}_t^{+}$   are conditionally independent given $\mathcal{D}_t$,

 \textbf{Sufficiency}
  Assume that $\y$ satisfies Assumption \ref{gjmls:output:assumptions}.
  From Theorem \ref{gen:filt:theo3} it follows that 
  $\y$ admits a \GBS\ $\Sigma$ realization in forward innovation form which 
  satisfies Assumption \ref{gbs:def}. From Lemma \ref{new:stoch_jump_lin:gbs2}
  it then follows that the GJMLS $H_{\Sigma}$ associated with $\Sigma$
  is a realization of $\y$ and it satisfies Assumption \ref{Assumption0}.
 
\end{proof}
%\subsection{Minimality of Realization by a GJMLS with Fully Observed Discrete States}
%\label{real:min}
\begin{proof}[Proof of Lemma \ref{real:lemma:reachobs}]
  Consider the \GBS\ $BS_{H}$ associated with $H$ from \eqref{real:gjmls_repr}.
  Then it is easy to see that 
  $R_H=(\mathbb{R}^{n}, \{\sqrt{p_{\sigma}}A_{\sigma}\}_{\sigma \in \Sigma}, B,C)$,
  where $B=\{ B_{(\sigma,j)} \mid \sigma \in \Sigma, j=1,\ldots,p\}$ and
  with $B_{\sigma}=\begin{bmatrix} B_{\sigma,1} & \ldots B_{\sigma,p} \end{bmatrix}$,
  \[ B_{\sigma}= \sqrt{p_{\sigma}}(A_{\sigma}\hat{P}_{\sigma}C^{T}+K_{\sigma}Q_{\sigma}D_{\sigma}^T) \]
  where $\hat{P}_{\sigma}=E[\hat{\x}(t)\hat{\x}^T(t)\bu_{\sigma}^2(t)]$. 
 % Hence,
  %for each $\sigma=(q_1,q_2) \in \Sigma$,
  %$\hat{P}_{\sigma}=E[\hat{\x}(t)\hat{\x}^T(t)\chi(\btheta(t)=q_1,\btheta(t+1)=q_2)]$,
  %$Q_{\sigma}=E[\hat{\v}(t)\hat{\v}^T(t)\bu_{\sigma}^2(t)]$.
  From Lemma \ref{stoch_jump_stab:lemma1.1} 
  it then follows that 
  %$\hat{P}_{(q_1,q_2)}=p_{q_1,q_2}\mathbf{I}_{q_1}^TP_{q_1}\mathbf{I}_{q_1}^T$ and $Q_{q_1,q_2}=p_{q_1,q_2}\mathbf{I}_{q_1}Q_{q_1}\mathbf{S}_{q_1}$,
  %where $P_{q_1}=E[\x(t)\x^T(t)\chi(\btheta(t)=q_1)]$ and 
  %$Q_{q_1}=E[\v(t)\v^T(t)\chi(\btheta(t)=q_1)]$.
  %Hence,
  \[ B_{(q_1,q_2)}=\sqrt{p_{q_1,q_2}}\mathbf{I}_{q_2}^TG_{q_1,q_2}\mathbf{M}_{q_1}, \]
  where
  \( \mathbf{M}_q=\begin{bmatrix} \mathbf{O}_{p, p(q-1)}, & I_{p}, & \mathbf{O}_{p, p(d-q-1)} \end{bmatrix} \in \mathbb{R}^{p \times pd} \).
  %Hence,  the only non-zero block of  $B_{(q_1,q_2)}$ is the one which lies on the rows
  %indexed by $i=(\sum_{q=1}^{q_2-1} n_q)+1,\ldots,\sum_{q=1}^{q_2} n_q$ and that block
  %equals $G_{q_1,q_2}$. 

  Note that $R_{H}$ is reachable if and only the elements of
  $\IM \sqrt{p}_{w}A_{w}B_{(q_1,q_2)}$,  $w \in (Q \times Q)^{*}$, 
  $|w| \le n-1$, $(q_1,q_2) \in Q \times Q$, 
  span the whole space. Notice that 
  $A_{w}B_{(q_1,q_2)}=0$ if $(q_1,q_2)w \notin L$ and that $B_{(q_1,q_2)} \in \IM \mathbf{I}_q$, and 
  $A_{w}B_{(q_1,q_2)}$ belongs to $\IM \mathbf{I}_{q}$, 
  if $w$ ends in a letter $(q_3,q)$. 
  Hence, reachability of $R_H$ is equivalent to requiring that
  the span of columns of $A_{(q,q_3)w}B_{(q_1,q_2)},B_{(q,q_4)}$ for all
  $q_1,q_2,q_3,q_4$, $|w| \le n-2$, $w \in L$ equals $\IM \mathbf{I}_{q}$ 
  for all $q \in Q$.
  Notice that $\mathbf{M}_q$ is full row rank, hence $\IM A_{(q,q_3)w}B_{(q_1,q_2)}=\IM \mathbf{I}_{q}M_{(q,q_3)w}G_{(q_1,q_2)}\mathbf{M}_{q_1}=\IM \mathbf{I}_{q}M_{(q,q_3)w}G_{(q_1,q_2)}$ and 
  $\IM B_{(q_4,q)}=\IM \mathbf{I}_qG_{q_4,q}\mathbf{M}_{q_4}=\IM \mathbf{I}_qG_{q_4,q}$ for all
  $q_1,q_2,q_3,q_4$, $|w| \le n-2$, $w \in L$. 
  It then follows that the span of those vectors
  equals $\IM \mathbf{I}_{q}\mathcal{R}_{H,q}$.
  %Moreover, $A_{w}B_{(q_1,q_2)}=\sqrt{p_{w}p_{q_1,q_2}}\mathbf{I}_{q}M_{w}G_{q_1,q_2}$.
  Since $\Rank \mathbf{I}_{q}=n_q$, reachability of $R_H$ is indeed equivalent to 
  $\Rank \mathcal{R}_{H,q}=n_q$ for all $q \in Q$.

  From the definition of $R_H$ and $\BS_H$ it follows that
  $\mathbf{M}_{r}CA_{w}A_{(q_1,q)}\mathbf{I}_{q}=0$ if $(q_1,q)w$ does not end in 
  a letter $(q_2,r)$, $q_2 \in Q$, 
  and $\mathbf{M}_{r}CA_{w}A_{(q_1,q)}\mathbf{I}_{q}=C_rM_{w(q_1,q)}$ otherwise, 
  for any $r,q_1,q \in Q$, $w \in \Sigma^{*}$.
  Hence, $\ker CA_{w}\mathbf{I}_q=\ker C_{r}M_{w}$ for all $w \in L$ such that $w$ ends
  in $(q_2,r)$. 
  Notice that $C\mathbf{I}_{q}=C_q$. 
  Finally, we remark that $w \notin L$, then $CA_{w}=0$ and if
  $w$ does not start with a letter of the form $(q,q_1)$, then
  $CA_{w}\mathbf{I}_q=0$.
  From the discussion above it then follows that
  $O_{R_H} \cap \IM \mathbf{I}_q = \mathbf{I}_{q}(\mathcal{O}_{H,q})$. 

  Assume now that $R_H$ is observable, i.e. $O_{R_H}=\{0\}$. Since $\mathbf{I}_q$ is
  full column rank, we then get that $\mathcal{O}_{H,q}=\{0\}$, $q \in Q$.
  Conversely, assume that $\mathcal{O}_{H,q}=\{0\}$ for all $q \in Q$.
  It then follows that $O_{R_H} \cap \IM \mathbf{I}_q=\{0\}$.
  Let $x=(x^T_1,\ldots,x^T_d)^T \in \mathbb{R}^{n}$, $x_q \in \mathbb{R}^{n_q}$, $q \in Q$,
  and assume that $x \in O_{R_H}$. Notice that 
  $Cx=\begin{bmatrix} (C_1x_1)^T, & \ldots, & (C_dx_d)^T \end{bmatrix}^T$ and
  $Cx_q=\mathbf{M}_{q}C_qx_q=C\mathbf{I}_{q}x$, $q \in Q$. 
  Hence,  $Cx=0$ is equivalent to $C_qx_q=0$. 
  Moreover, for any $q_1,q_2 \in Q$, $A_{(q_1,q_2)}x=A_{(q_1,q_2)}\mathbf{I}_{q_1}x_{q_1}$ and
  $A_{(q_1,q_2)}\mathbf{I}_{q}x_q=0$ for $q \ne q_1$.
  Hence, $x \in O_{R_H}$ implies that $CA_{w}\mathbf{I}_{q}x_q=0$ for any $q \in Q$, $w \in \Sigma^{*}$,
  $|w| \le n-1$.
  Hence,
  $\mathbf{I}_qx_q \in O_{R_H} \cap \IM \mathbf{I}_q$. Since we have shown above that
  $O_{R_H} \cap \IM \mathbf{I}_q=\{0\}$, it follows that $\mathbf{I}_{q}x_q=0$, $q \in Q$. 
  Since $\mathbf{I}_q$ is full column rank, it follows that $x_q=0$ for all $q \in Q$.
  Hence, $x=0$. 
\end{proof}
\begin{proof}[Proof of Theorem \ref{real:theo:min}]

 \textbf{Minimality $\implies$ reachability and observability.}
 Assume that $H$ is a minimal realization of $\y$ and assume that it is not reachable or observable. Consider the \GBS\ $\BS_H$ associated
 with $H$.  From Lemma \ref{new:stoch_jump_lin:gbs1} it follows that $\BS_H$ is a realization of $\widetilde{\y}$.
 From Lemma \ref{real:lemma:reachobs} it follows $R_{\BS_H}$ cannot be reachable and observable. 
 Then by Theorem \ref{gen:filt:min_theo} $\BS_H$
 cannot be minimal. Take a minimal realization $\BS$ of $\widetilde{\y}$ in 
 forward innovation form. Then $\dim \BS < \dim \BS_H = \dim H$.
 Construct the GJMLS $H_{\BS}$ associated with $\BS$. By Lemma \ref{new:stoch_jump_lin:gbs2},
 $H_{\BS}$ is a realization of $\y$ and $\dim H_{\BS} =\dim \BS < \dim H$.
 This contradicts to minimality of $H$ and hence a contradiction.
   
 \textbf{Reachability and observability $\implies$ minimality}
   Assume that $H$ is reachable and observable but it is not a minimal
   realization of $\y$. Consider the associated \GBS\ $\BS_H$.
   From Lemma \ref{real:lemma:reachobs}
   it follows that $R_H=R_{\BS_H}$ is reachable and observable.  
  From Theorem \ref{gen:filt:min_theo} and Lemma \ref{new:stoch_jump_lin:gbs1} it then follows that $\BS_H$ is a minimal
  realization of $\widetilde{\y}$. Assume that $H$ is not minimal. Then 
  there exists a GJMLS $\hat{H}$ such that $\dim \hat{H} < \dim H$, $\hat{H}$ is a realization of $\y$ and it
  satisfies Assumption \ref{Assumption0}. From Lemma \ref{new:stoch_jump_lin:gbs2} it then follows that
  $\BS_{\hat{H}}$ is a realization of $\widetilde{\y}$. Since $\dim \hat{H}=\dim \BS_{\hat{H}}$ and $\dim H=\dim \BS_H$, it follows that $\dim \BS_{\hat{H}} < \dim \BS_H$, which 
  contradicts the minimality of $\BS_H$.

\textbf{Minimal realizations are isomorphic}
  If $H$ and $\hat{H}$ are two minimal realizations of $\y$ such that they both satisfy
  Assumption \ref{Assumption0}, then by Lemma \ref{new:stoch_jump_lin:gbs1} the 
  \GBS{s} $\BS_H$ and $\BS_{\hat{H}}$ are minimal realizations of $\widetilde{\y}$
  which satisfy Assumption \ref{gbs:def}. From Theorem \ref{gen:filt:min_theo} 
  it then follows that 
  the representations $R_{H}=R_{\BS_H}$ and 
  $R_{\hat{H}}=R_{\BS_{\hat{H}}i}$ are isomorphic
  and they are both reachable and observable.
  Consider this isomorphism $\mathbf{S}:R_H \rightarrow R_{\hat{H}}$.
  It is easy to see that $\mathbf{S}$ is then an isomorphism between $H$ and $\hat{H}$.
\end{proof}
%%\begin{Remark}
 %%The proof of Theorem \ref{real:theo:min} resembles the
 %%proof of minimality of the so called hybrid representations
 %%described in \cite{MP:HybPow} and 
 %%\cite{MP:Phd}. 
%% In fact, instead of using rational
%% representations we could use hybrid representations
%% for representing the sequences of covariances.
%% In that case the theorem above follows
%% from the general result on minimality of hybrid representations,
%% see \cite{MP:Phd,MP:HybPow}.
%%\end{Remark}

%\input{JMLS_real_GJMLS_real}
%\label{gjmsl:basic}
 
%\input{JMLS_real_GJMLS_necc}
 
\section{ Discussion and Conclusion }
 We have presented a realization theory for stochastic 
 jump-linear systems. The theory relies on the solution 
 of a generalized bilinear filtering/realization problem. 
 This solution
%% filtering/realization problem 
 represents an extension of the
 known results on 
 linear and bilinear stochastic realization/filtering.

 We would like to extend the presented results to more general
 classes of hybrid systems. In particular, we would like
 to develope realization theory for jump-linear systems
 with partially observed discrete states. Necessary conditions
 for existence of a realization by a system of this class
 were already presented in \cite{MPRV:HSCC}. 
 Another line of research we would like to pursue is to use
 the presented theory for developing subspace identification
 algorithms
 for stochastic jump-linear systems. Note that the
 classical stochastic bilinear realization theory gave
 rise to a number of subspace identification algorithms, see
 \cite{Favoreel:PhD99,Verhaegen:automatica,Verhaegen:INJC,MacChenBilSub}.
 It is very likely that the presented results will lead to
 very similar subspace identification algorithms.

 \textbf{Acknowledgement. }
 This work was supported by grants NSF EHS-05-09101, 
 NSF CAREER IIS-04-47739,  and ONR N00014-05-1083.

%\bibliographystyle{IEEEtran}
%\bibliography{biblio/control}
%\bibliography{JMLS.bib}


%\input{JMLS_real_appendix}

\end{document}